%% file: template-jems.tex
\documentclass{article}
\usepackage[journal=APDE,lang=british]{ems-journal-jems} 

\usepackage{textcomp}
\usepackage{mathrsfs}

\usepackage{chngcntr}
\usepackage{apptools}
\usepackage{verbatim}
\usepackage{tikz-cd}

\newtheorem{theorem}{Theorem}

\newtheorem{proposition}{Proposition}
\newtheorem{lemma}{Lemma}

\theoremstyle{definition}

\newtheorem{conjecture}{Conjecture}

\theoremstyle{remark}
\newtheorem{remark}{Remark}
\newtheorem{corollary}{Corollary}

\numberwithin{equation}{section}

\newcommand\Eta{\mathbb{U}}

\newcommand\restr[2]{{
		\left.\kern-\nulldelimiterspace 
		#1 
		\vphantom{\big|} 
		\right|_{#2} 
}}

\DeclareRobustCommand{\rchi}{{\mathpalette\irchi\relax}}
\newcommand{\irchi}[2]{\raisebox{\depth}{$#1\chi$}} 

\numberwithin{equation}{section}

\numberwithin{theorem}{section}
\numberwithin{proposition}{section}
\numberwithin{lemma}{section}
\numberwithin{remark}{section}
\numberwithin{exercise}{section}
\numberwithin{corollary}{section}
\numberwithin{conjecture}{section}
\newcounter{desccount}

\begin{document}

\title{Spectral comparison of compound cocycles generated by delay equations in Hilbert spaces}
\titlemark{Spectral comparison of compound cocycles}


\emsauthor{1}{
	\givenname{Mikhail}
	\surname{Anikushin}
	\orcid{0000-0002-7781-8551}}{M.M.~Anikushin}

\Emsaffil{1}{
	\pretext{}
	\department{Department of Applied Cybernetics}
	\organisation{Faculty of Mathematics and Mechanics, St Petersburg University}
	\rorid{01a2bcd34}
	\address{Universitetskiy prospekt 28}
	\zip{198504}
	\city{Peterhof}
	\country{Russia}
	\posttext{}
	\affemail{demolishka@gmail.com}
	}
	

	


\classification[37L15, 37L45, 34K08, 34K35]{37L30}
\keywords{delay equations, frequency theorem, compound cocycles, dimension estimates, Lyapunov functionals}

\begin{abstract}
	We study linear cocycles generated by nonautonomous delay equations in a suitable Hilbert space and their extensions, called compound cocycles, to exterior powers. Using a recent version of the frequency theorem, we develop analytical techniques for comparing spectral properties, such as uniform exponential dichotomies, between such cocycles and semigroups generated by stationary equations. Most of this work is devoted to the theoretical foundations related to regularity and structure in PDEs associated with delay equations. We briefly discuss some specific applications, especially those related to the global stability problem, which are described in detail in our adjacent works.
\end{abstract}

\maketitle
\tableofcontents

\input{Introduction}
\input{Somenotations}
\input{TensorProductAdditiveCompounds}
\input{CocyclesSemigroupsAdditiveCompounds}
\input{ActionOfAOnWedges}
\input{PropertiesOfInhomogeneousProblems}
\input{NonautonomousPerturbationofAddComp}
\input{ResolventsComputation}

\appendix
\input{DiagonalTranslateSemigroups}
\input{WeigthedFunctionsCompound}



\begin{funding}
The reported study was funded by the Russian Science Foundation (Project 25-11-00147).
\end{funding}

\section*{Data availability}
Data sharing not applicable to this article as no datasets were generated or analyzed during the current study.

\section*{Conflict of interest}
The author has no conflicts of interest to declare that are relevant
to the content of this article.


\end{document}

%% file: Introduction.tex
\section{Introduction}
\subsection{Brief introduction}
An evolutionary process described by a dissipative nonlinear differential equation can be treated in terms of the theory of dynamical systems as a semiflow\footnote{For simplicity, this section assumes that $\pi$ is generated by an autonomous equation. In this paper, we also consider nonautonomous nonlinear equations, which generate skew-product semiflows (hence the name $\pi$).} $\pi$ admitting a global attractor $\mathcal{A}$. Linearized equations describe the evolution of infinitesimal lengths over trajectories of the semiflow and can be viewed as a linear cocycle $\Xi$, say, in a Hilbert space $\mathbb{H}$. For studying the evolution of infinitesimal $m$-volumes, we consider the $m$-fold extension $\Xi_{m}$, called a \textit{compound cocycle}, of the cocycle $\Xi$ to the $m$-fold exterior power $\mathbb{H}^{\wedge m}$ of $\mathbb{H}$. 

A classical result in dimension theory then states that if $\Xi_{m}$ over $(\mathcal{A},\pi)$ is uniformly exponentially stable, or, in other words, the system uniformly contracts infinitesimal $m$-volumes over the trajectories in $\mathcal{A}$, then the dimension of $\mathcal{A}$ is strictly less than $m$. In these terms, the problem of effective dimension estimates is concerned with producing conditions that are to some extent computable or at least amenable to analysis and that reveal regions in the system parameter space, where such contraction can be guaranteed. Of particular interest is the case $m=2$, since uniform exponential stability of $\Xi_{2}$ often leads to global stability of $\pi$ via the generalized Bendixson criterion for attractors.

In terms of the largest uniform Lyapunov exponent $\lambda_{1}(\Xi_{m})$ of $\Xi_{m}$, the contraction condition is just $\lambda_{1}(\Xi_{m}) < 0$. This quantity is a generalization of the spectral radius of a matrix or the growth bound of a $C_{0}$-semigroup. Just as the spectral radius can be computed by minimizing the matrix norm over all possible norms in the underlying space, one can also compute $\lambda_{1}(\Xi_{m})$ by minimization of certain quantities over metrics on $\mathbb{H}^{\wedge m}$, which are generally families of norms depending on points from $\mathcal{A}$. Thus, to obtain effective dimension estimates, it is necessary to construct appropriate classes of metrics \textit{adapted} to specific classes of systems.

In this work, we follow a perturbative approach for constructing adapted metrics. Namely, we search for a quadratic Lyapunov-like functional associated with a certain exponentially stable $C_{0}$-semigroup $G^{\wedge m}$ and provide frequency conditions that ensure that this functional can also be used to establish the uniform exponential stability of $\Xi_{m}$. In the frequency-domain analysis from control theory this is known as the quadratic criterion, although it has never been applied directly to study compound cocycles, even in the case of ODEs. A recent variant of the frequency theorem, obtained in our paper \cite{Anikushin2020FreqDelay}, allows us to study such questions for various infinite-dimensional systems.

Here we illustrate the general approach by means of delay equations in $\mathbb{R}^{n}$. Most of this work revolves around functional-analytic properties of additive compounds of delay operators (the generators of $G^{\wedge m}$) and the structure of infinitesimal generators of $\Xi_{m}$, which are necessary to apply the result of \cite{Anikushin2020FreqDelay}. Due to the theoretical nature of our research, we do not limit ourselves to only the particular cases we need and establish general results that have potential for application to other problems. In particular, we study the existence of uniform exponential dichotomies for $\Xi_{m}$ (nonautonomous analogs of spectral decompositions), not just dimension estimates; see Theorem \ref{TH: QuadraticFunctionalDelayCompoundTheorem}, which is our main result. The reader interested in concrete applications should consult our papers \cite{AnikushinRomanov2023FreqConds, AnikushinRomanovQBC2026,AnikushinRomanov2026Subroutine}, which are entirely devoted to applications of the theoretical foundations in the case $m=2$. In particular, this case is important for global stability problems, see Theorem \ref{TH: GlobalStabilityDelayEqsViaBC}. A brief discussion of ideas and quantitative results presented in these works is given in Section \ref{SEC: PerspectivesDelayCompound}.

A more detailed discussion of the history and methods is given in the next sections. Along with the present study, it should convince the reader that it is useful to investigate compound systems by objects specific to them, i.e., those that cannot be lifted from the underlying space $\mathbb{H}$.

\subsection{Historical perspective: Lyapunov dimension and effective dimension estimates for delay equations}
In the study of dissipative dynamical systems, structure of attractors takes the spotlight. A classical question of this kind, especially interesting in infinite-dimensions, is related to obtaining effective dimension estimates for global attractors. Although the initial motivation for the problem was concerned with finite-dimensional reduction based on embedding theorems for sets with finite Hausdorff or fractal dimensions (see \cite{RobinsonBook2011, ZelikAttractors2022, Zelik2014}), the volume contraction approach revealed a more relevant dimension-like characteristic called the \textit{Lyapunov dimension}\footnote{More precisely, such a quantity is called the \textit{uniform} (or \textit{global}) \textit{Lyapunov dimension} to distinguish it from the Lyapunov dimension over an ergodic measure. This distinction is not necessary for the present work where only the uniform quantity is considered.}. Roughly speaking, it is determined by the dimension threshold such that infinitesimal volumes of higher dimensions admit uniform exponential decay. It is well known that such a quantity always bounds from above the fractal dimension of the invariant set, see \cite{ChepyzhovIlyin2004, KuzReit2020, Temam1997}. However, unlike purely geometric dimensions, it is upper semicontinuous (hence more robust) and admits computation with the aid of adapted metrics, see \cite{KawanPogromsky2021, Anikushin2023LyapExp} and below. Even if an exact value of the Lyapunov dimension is known, it often reflects not any geometric dimensions of the attractor, but rather possible expansions of such dimensions under perturbations of the system; see \cite{TuraevZelik2003} for a nice example. With upper bounds for the Lyapunov dimension at our disposal, we can apply a generalized Bendixson criterion, which indicates the absence of certain invariant structures on the attractor and guarantees a simple attractor structure in the case of uniform area contraction, see \cite{Smith1986HD, LiMuldowney1995LowBounds, LiMuldowney1996SIAMGlobStab, Anikushin2026Bendixson} and Theorem \ref{TH: GlobalStabilityDelayEqsViaBC}.

There is a variant of the Lyapunov dimension (generally slightly larger in the fractional part) based on the Kaplan--Yorke formula involving uniform Lyapunov exponents \cite{Anikushin2023LyapExp}. A classical result\footnote{This result holds in a much wider (than it is stated in \cite{KatokSB1980}) generality covering infinite-dimensional noninvertible systems by virtue of the variational principle for topological entropy and the Margulis--Ruelle inequality, see \cite{Thieullen1987, BlumentalYoung2017SRBmeasures}.} of Katok \cite{KatokSB1980} bounds the topological entropy from above via the sum of positive uniform Lyapunov exponents\footnote{Among control theorists, this sum is known as the \textit{restoration entropy}, see \cite{KawanPogromsky2021}.}. Along with the above, this illustrates relations between dimension, entropy, and volume in the context of uniform characteristics.
 
In the volume contraction approach we follow here, of interest are linearized equations whose solutions generate a cocycle $\Xi$ in a Hilbert space $\mathbb{H}$ over an invariant set $\mathcal{P}$ of the original system. Thus, $\Xi$ is responsible for the evolution of infinitesimal lengths, while its $m$-fold multiplicative extensions $\Xi_{m}$, called \textit{compound cocycles}, to exterior powers $\mathbb{H}^{\wedge m}$ of $\mathbb{H}$ describe the evolution of infinitesimal $m$-volumes. The uniform asymptotic growth of the trajectories of $\Xi_{m}$ is determined by the largest uniform Lyapunov exponent $\lambda_{1}(\Xi_{m})$. In these terms, we seek effective conditions for determining whether $\lambda_{1}(\Xi_{m}) < 0$.

Effective upper estimates for $\lambda_{1}(\Xi_{m})$ can be obtained as follows. We choose a metric $\mathfrak{m}$ on $\mathbb{H}^{\wedge m}$ given by a family of inner products in $\mathbb{H}^{\wedge m}$ depending on $\wp \in \mathcal{P}$ and maximize over $\wp \in \mathcal{P}$ certain quantities $\alpha^{+}_{\mathfrak{m}}(\wp)$ arising from differentiation of the cocycle trajectories in the metric. As it is shown in \cite{Anikushin2023LyapExp}, minimizing such estimates over $\mathfrak{m}$ in a general class precisely gives the quantity $\lambda_{1}(\Xi_{m})$, provided that $\Xi_{m}$ is quasi-compact. If the metric $\mathfrak{m}$ is associated with a metric on $\mathbb{H}$, the quantities $\alpha^{+}_{\mathfrak{m}}(\wp)$ can sometimes be expressed in terms of eigenvalues of appropriate self-adjoint operators in $\mathbb{H}$. See \cite{Anikushin2023LyapExp} for relevant discussions.

In the literature, most applications of the above approach involve only associated metrics. Moreover, in infinite dimensions, these are usually given by standard inner products, see \cite{Temam1997,ChepyzhovIlyin2004,ZelikAttractors2022}. For some problems, standard metrics are irrelevant, as in the case of delay equations \cite{AnikushinRomanov2024EffEst, Anikushin2023LyapExp, Anikushin2022Semigroups} or hyperbolic equations \cite{ZelikAttractors2022, TuraevZelik2003}.

Nonstandard metrics arise in applications of the Leonov method \cite{LeonovBoi1992, Kuznetsov2016, Anikushin2023LyapExp}. On the geometric level, this method deals with variations of a constant metric in its conformal class via Lyapunov-like functions. This often allows one to improve estimates without significantly complicating the calculations or even obtain exact values of the Lyapunov dimension, as in the case of Lorenz, Lorenz-like, and H\'{e}non systems studied in \cite{KuzReit2020, KuzMokKuzKud2020, LeoAlexKuzSMsystem2015,LeoKuzKorzh2016}.

Another method, in the spirit of the perturbative approach that we develop here, is proposed by Smith \cite{Smith1986HD} for ODEs. It is based on quadratic functionals that allows to bound from below all the singular values of the derivative cocycle. For infinite-dimensional systems, it can be developed using inertial manifolds, see \cite{Anikushin2022Semigroups}. However, this reveals the impracticality and artificiality of this method for the problem of dimension estimates, since the existence of inertial manifolds is related to exponential dichotomies, which is a more delicate matter.

In the case of delay equations, the problem of obtaining effective dimension estimates remained without significant progress for a long time, see \cite{Anikushin2023LyapExp} for a more detailed discussion. This is reflected in some past \cite{HaleLunel1993} and relatively recent monographs \cite{Chueshov2015, CarvalhoLangaRobinson2012}, which essentially expand the ideas from the pioneering paper \cite{MalletParet1976} by exploiting the compactness of the derivative cocycle and hence making only qualitative conclusions on the finiteness of dimensions. In \cite{Anikushin2023LyapExp}, we satisfactorily resolved the problem by explicitly constructing constant adapted metrics for a fairy general class of delay equations in $\mathbb{R}^{n}$ with discrete delays. For global attractors arising in the Mackey--Glass equations \cite{MackeyGlass1977} and the periodically forced Suarez--Schopf delayed oscillator \cite{AnikushinRom2023SS, Suarez1988}, numerical experiments suggest that the obtained estimates are asymptotically sharp as the delay tends to infinity. In \cite{AnikushinRomanov2024EffEst}, the approach is illustrated by means of the Nicholson blowflies model. However, obtaining more accurate estimates for specific parameters remains a challenging task.

All the above results concern the use of metrics on exterior powers, which are associated with a metric on $\mathbb{H}$. For computations, this makes it possible to avoid dealing directly with compound cocycles and their infinitesimal generators, allowing one to remain at the level of linearized equations. On the other hand, as argued in \cite{Anikushin2023LyapExp} by means of Lyapunov-like metrics, it may be natural for theory\footnote{For finite-dimensional invertible systems, the result of \cite{KawanPogromsky2021}, based on what we called Shannon-like metrics in \cite{Anikushin2023LyapExp}, shows that the class of associated metrics is theoretically sufficient to compute the Lyapunov dimension. However, this theoretical possibility does not at all imply that nonassociated metrics should not be used for practical purposes.} and practice to use nonassociated adapted metrics on exterior powers. In this paper, we further confirm this point of view, since applying the frequency theorem to compound cocycles leads to Lyapunov-like metrics, see Section \ref{SEC: ContributionSpecComp}.

Besides \cite{Anikushin2023LyapExp}, a rare exception in the field, which partly motivated the present investigation, is the work of Mallet-Paret and Nussbaum \cite{MalletParretNussbaum2013}, where compound cocycles on injective tensor products of Banach spaces associated with certain scalar delay equations are studied. Such equations particularly arise after linearization of scalar delay equations with monotone feedback, which are known to satisfy the Poincar\'{e}--Bendixson trichotomy \cite{MalletParetSell1996}. In \cite{MalletParretNussbaum2013}, it is shown that the $m$-fold compound cocycle preserves a convex reproducing normal cone for either odd or even $m$ depending on the feedback sign. Based on this, the authors developed the Floquet theory for periodic equations using arguments in the spirit of the Krein--Rutman theorem. In particular, a principle was formulated for comparing Floquet multipliers for periodic equations. In our work \cite{Anikushin2022Semigroups}, it is developed to obtain effective dimension estimates. However, scalar equations that exhibit chaotic behavior and systems of equations are beyond the scope of this approach.

A more detailed discussion of attractor theory and finite-dimensional reduction can be found in the surveys \cite{ZelikAttractors2022, Zelik2014}, while issues concerning effective dimension estimates and the problem of global stability for delay equations are addressed in our papers \cite{Anikushin2023LyapExp, AnikushinRomanov2023FreqConds, AnikushinRomanov2024EffEst}.

In what follows, we present another approach to the problem of effective dimension estimates for delay equations, based on a direct examination of compound cocycles and their generators.

\subsection{Contribution of the present work}
\label{SEC: ContributionSpecComp}
In this paper, we study a fairly general class of linear nonautonomous delay equations in $\mathbb{R}^{n}$ as it is described in \eqref{EQ: DelayRnLinearized}. As in our adjacent work \cite{Anikushin2023LyapExp}, we address the problem of obtaining conditions for the exponential stability of compound cocycles corresponding to such equations. We aim to express such conditions in terms of frequency inequalities arising from a comparison between compound cocycles $\Xi_{m}$ and certain $C_{0}$-semigroups $G^{\wedge m}$ with the aid of the frequency theorem developed in our work \cite{Anikushin2020FreqDelay}. In fact, we will obtain conditions for the existence of gaps in the Sacker--Sell spectrum \cite{SackerSell1994}, i.e., uniform exponential dichotomies, and even more, see Theorem \ref{TH: QuadraticFunctionalDelayCompoundTheorem} and the subsequent remarks. As will be shown, following this program reveals novel functional-analytic properties of delay equations concerned with harmonic analysis, see below.

In our adjacent works \cite{AnikushinRomanov2023FreqConds, AnikushinRomanovQBC2026, AnikushinRomanov2026Subroutine}, we developed approximation schemes for verifying such frequency inequalities, based on finite-dimensional truncations of transfer operators. It should be noted that the most efficient schemes are available for the case of $m=2$ and equations with single delays in the linear part, for which the associated transfer operators can be expressed as integral operators with bounded kernels of the Hardy--Volterra type. In \cite{AnikushinRomanov2026Subroutine}, we presented an implementation of the approximation schemes in a Python package that allows for simple and efficient numerical experiments with specific systems. A brief overview of these works, accompanied by examples of applying our theory to the problem of global stability, is presented in Section \ref{SEC: PerspectivesDelayCompound}.

Let us outline main ideas and methods of our work. For precise preliminary definitions and notations, we refer to Sections \ref{SEC: CompOperatorsTensorProducts} and \ref{SEC: CocyclesSemigroupsAdditiveCompounds}.

First, we treat delay equations in the Hilbert space $\mathbb{H}$ defined in \eqref{EQ: HilbertSpaceDelayEqDefinition}, based on our work \cite{Anikushin2022Semigroups}. This contrasts with \cite{MalletParretNussbaum2013} and most other works, where delay equations are posed in the space of continuous functions. Such a treatment is essential for our functional-analytical study of delay equations as PDEs with nonhomogeneous boundary conditions, in the spirit of \cite{LionsMagenesBVP11972}. On the other hand, thanks to the smoothing properties of dynamics, the final conclusions concerning asymptotic behavior in the norm of $\mathbb{H}$ are also fulfilled in the usual supremum norm.

Next, we study the $m$-fold compound cocycle $\Xi_{m}$ in $\mathbb{H}^{\wedge m}$ as a nonautonomous perturbation of a $C_{0}$-semigroup $G^{\wedge m}$. In terms of \eqref{EQ: DelayRnLinearized}, the stationary linear part is distinguished, and to it corresponds an operator $A$, which generates a $C_{0}$-semigroup $G$ in $\mathbb{H}$. Then $G^{\wedge m}$ is given by the multiplicative extension of $G$ onto $\mathbb{H}^{\wedge m}$. On the infinitesimal level, $G^{\wedge m}$ is generated by an operator $A^{[\wedge m]}$, called the $m$-fold \textit{additive compound} of $A$ in $\mathbb{H}^{\wedge m}$. In Theorem \ref{TH: TensorCompoundCocycleDelayDescription}, the infinitesimal generator of $\Xi_{m}$ is described as a nonautonomous boundary perturbation of $A^{[\wedge m]}$. It is essential to use the Hilbert space setting to make sense of the boundary perturbation.

After that, we study the problem of providing conditions for preserving certain dichotomy properties of $G^{\wedge m}$ for all perturbations in a given class (for example, with a prescribed Lipschitz constant). Here, the perturbation class is described through an indefinite quadratic form, for which we consider the associated infinite-horizon quadratic regulator problem\footnote{Reformulating the problem of perturbation of spectral dichotomies in the language of the quadratic regulator problem allows us to separate the geometric constructions carried out in the language of abstract dynamical systems and analytical checks specific to particular classes of problems. In the more common approach based on Lyapunov--Perron operators, these two aspects are inextricably and uncontrollably intertwined, hindering the construction of genuine geometric theories.} posed for a proper linear control system. The latter problem is resolved using the frequency theorem developed in our paper \cite{Anikushin2020FreqDelay}. This provides frequency conditions for the existence of a quadratic Lyapunov functional associated with $G^{\wedge m}$, which can also be used to obtain the desired dichotomy properties for $\Xi_{m}$.

Note that the described approach can potentially be applied to a range of problems enjoying a kind of asymptotic compactness such as semilinear parabolic equations, damped hyperbolic equations, neutral delay equations, or parabolic equations with nonlinear boundary conditions. However, apart from our studies, we are not aware of any works devoted to this, even in the case of ODEs.

As to delay equations, they represent analytically nontrivial examples of such applications. Here, some problems arise mainly due to the unbounded nature of perturbations on the infinitesimal level. This demands distinguishing between \textit{regularity} and \textit{structure}\footnote{Roughly speaking, regularity is related to various spectral bounds or estimates for resolvents in intermediate spaces, while structure is related to tempo-spatial properties of solutions, such as the spaces to which these solutions belong as functions of time. For example, in the case of semilinear parabolic equations in bounded domains, tempo-spatial properties, known as parabolic smoothing, come from the resolvent and spectrum bounds (see \cite{Anikushin2020FreqParab, Chueshov2015}). Thus, it is quite fair to attribute these properties to regularity (not mentioning their ``structurality''), as is always done. However, for delay equations, the issue of distinguishing becomes acute, since the spaces involved into their structural properties are more sophisticated.} in linear inhomogeneous problems associated with $A^{[\wedge m]}$. In our work \cite{Anikushin2020FreqDelay}, we explored some related features that allowed us to overcome these obstacles in the case $m=1$. A significant part of this work is devoted to generalizing these properties to the case of general $m$.

On the side of structure, we establish what was called in \cite{Anikushin2020FreqDelay} a \textit{structural Cauchy formula}, see Theorem \ref{TH: StructuralCauchyFormulaCompoundDelay}. This is a decomposition of mild solutions to linear inhomogeneous problems associated with $A^{[\wedge m]}$, which differs from the usual Cauchy formula, but reveals a certain structure of solutions. More precisely, according to the formula, each component of a solution is uniquely decomposed into the sum of what we call \textit{adorned} and \textit{twisted} functions. In its turn, such a sum is called by us an \textit{agalmanated} function, and the corresponding functional spaces are introduced in Appendix \ref{SEC: MeasurementOperatorsOnAgalmanated}. To prove and understand Theorem \ref{TH: StructuralCauchyFormulaCompoundDelay}, preparatory results on diagonal translation semigroups and diagonal Sobolev spaces from Appendix \ref{SEC: DiagonalTranslationSemigroups} are also required.

We use the structural Cauchy formula to make sense of integral functionals arising in the quadratic regulator problem. Here, what we call \textit{pointwise measurement operators} naturally arise, and they are studied in Appendix \ref{SEC: MeasurementOperatorsOnAgalmanated}. Such operators are defined by pointwise application to functions of time of a certain unbounded operator\footnote{For example, a $\delta$-functional in some $L_{2}$-space.} in the space of their values. Such operators are naturally defined on what we call \textit{embracing spaces}, into which the aforementioned classes of functions can be naturally embedded. Note that for the case of adorned functions and $m=1$, which are just weighted traveling waves, the well-posedness of pointwise measurement operators reflects convolution theorems for measures, see \cite{HewittRoss1965}. However, we cannot find a general result that covers our situation for $m > 1$, let alone the other classes of functions. Another key property of embracing spaces is that the Fourier transform provides an automorphism of the embracing space over $\mathbb{R}$ and commutes with pointwise measurement operators. This constitutes Theorem \ref{TH: EmbracingFourierCommutesIGamma}, which is important for deriving frequency inequalities.

On the side of regularity, we establish certain uniform bounds for the resolvent in intermediate spaces, see Theorem \ref{TH: ResolventDelayCompoundBound}. This constitutes the second main ingredient for resolving the quadratic regulator problem and establishing our final Theorem \ref{TH: QuadraticFunctionalDelayCompoundTheorem}.

For applications of the frequency theorem to other problems also involving certain dichotomies, we refer to our works on inertial manifolds \cite{Anikushin2020FreqDelay, Anikushin2020Geom, Anikushin2020FreqParab} and almost periodic cocycles \cite{AnikushinAADyn2021,Anikushin2019Liouv,AnikushinReitRom2019}.

\subsection{Structure of the present work}
Now we will describe the structure of our work, indicating its key points related to the above discussion.

To the best of our knowledge, the widespread interest in the theory of multiplicative and additive compounds began with the work of Muldowney on ODEs \cite{Muldowneys1990CompMatr}. Recently, Muldowney and Wang \cite{MuldowneyWang2021} presented an algebraic theory of additive and multiplicative compound operators in general linear spaces. For us, the spectral theory of such operators in Hilbert spaces is important. Since we cannot find relevant presentations in the existing literature, we develop an appropriate theory in Sections \ref{SEC: CompOperatorsTensorProducts} and \ref{SEC: CocyclesSemigroupsAdditiveCompounds}.

In Section \ref{SEC: ComputationOfAdditiveCompountL2}, we study additive compounds $A^{[\otimes m]}$ and $A^{[\wedge m]}$ of delay operators $A$. This includes the description of the abstract $m$-fold tensor $\mathbb{H}^{\otimes m}$ and exterior $\mathbb{H}^{\wedge m}$ products of $\mathbb{H}$ in terms of a certain $L_{2}$-space in Theorem \ref{TH: TensorProductDelayDescription}. Next, we describe the action of $A^{[\otimes m]}$ and its domain $\mathcal{D}(A^{[\otimes m]})$ in Theorems \ref{TH: AdditiveCompoundDelayDescription} and \ref{TH: CompoundDelayDomainDescription}, respectively. Finally, we derive uniform bounds for the resolvent in intermediate spaces, see Theorem \ref{TH: ResolventDelayCompoundBound}.

In Section \ref{SEC: LinearInhomogeneousDelayCompounds}, we establish a structural Cauchy formula for linear inhomogeneous problems associated with $A^{[\otimes m]}$ (in particular, $A^{[\wedge m]}$), see Theorems \ref{TH: StructuralCauchyFormulaCompoundDelay} and \ref{TH: DelayCompoundStructuralCauchyFormulaNormEstimate}.

In Section \ref{SEC: NonautonomousPerturbationsAdditiveCompounds}, linear cocycles generated by a class of delay equations in $\mathbb{R}^{n}$ are studied. In Section \ref{SUBSEC: DelayCompoundInfinitesimalDescription}, infinitesimal generators of the corresponding multiplicative compound cocycles in $\mathbb{H}^{\otimes m}$ (resp. $\mathbb{H}^{\wedge m}$) are described as nonautonomous boundary perturbations of $A^{[\otimes m]}$ (resp. $A^{[\wedge m]}$), see Theorem \ref{TH: TensorCompoundCocycleDelayDescription}. In Section \ref{SUBSEC: AssociatedLIPquadraticConstr}, related linear inhomogeneous problems with quadratic constraints are formulated. In Section \ref{SUBSEC: DelayCompPropertiesOfComplexificatedProblem}, the associated integral quadratic functionals are interpreted and their relation with the Fourier transform is established, see Lemma \ref{LEM: FourierTransformCompoundDelay}. In Section \ref{SUBSEC: DelayCompoundFrequencyInequalities}, frequency inequalities for the preservation of certain dichotomy properties under the perturbations are derived, see Theorem \ref{TH: QuadraticFunctionalDelayCompoundTheorem}.

In Section \ref{SEC: PerspectivesDelayCompound}, we discuss numerical methods to verify frequency inequalities for specific problems. In particular, we briefly outline the ideas and experimental results from our adjacent works \cite{AnikushinRomanov2023FreqConds, AnikushinRomanovQBC2026,AnikushinRomanov2026Subroutine}.

In Appendix \ref{SEC: DiagonalTranslationSemigroups}, the theory of diagonal translation semigroups and diagonal Sobolev spaces is developed.

In Appendix \ref{SEC: MeasurementOperatorsOnAgalmanated}, pointwise measurement operators on embracing spaces are studied. In particular, the spaces of adorned, twisted, and agalmanated functions are introduced.

%% file: Somenotations.tex
\subsection{Some general notations}
Throughout the paper, $m$, $n$, $k$, $l$, and $j$ denote natural numbers. Usually, $m$ and $n$ are fixed; $j \in \{1,\ldots, m\}$; $k$ is used to denote the size of multi-indices such as $j_{1}\ldots j_{k}$ with $1 \leq j_{1} < \cdots < j_{k} \leq m$; $l$ is used for indexing sequences. Real numbers are denoted by $t$, $s$ or $\theta$, where, usually, $t, s \geq 0$ and $\theta \in [-\tau,0]$ for some $\tau > 0$ being a fixed value (delay).

For distinct integers $j_{1},\ldots,j_{k}$ and $j$, we define $J(j)=J(j;j_{1},\ldots,j_{k})$ as the integer $J \in \{1,\ldots,k+1\}$ such that $j$ is the $J$th element of the set $\{ j, j_{1},\ldots, j_{k} \}$ arranged in increasing order.

We often use the excluded index notation to denote multi-indices. For example, in the context of given $j_{1} \ldots j_{k}$ and $i \in \{j_{1},\ldots, j_{k}\}$, we denote by $j_{1} \ldots \hat{i} \ldots j_{k}$ the multi-index of size $k-1$ obtained from $j_{1}\ldots j_{k}$ by removing $i$. For brevity, we also write $\hat{i}$ instead of $j_{1} \ldots \hat{i} \ldots j_{k}$, if it is clear from the context what multi-index is meant. A similar notation is used for the exclusion of several indices.

For the compactness of formulas, it is often convenient to use $\bar{s}$ or $\bar{\theta}$ to denote vectors of real numbers. For example, $\bar{s}=(s_{1},\ldots,s_{m}) \in \mathbb{R}^{m}$ or $\bar{\theta}=(\theta_{1},\ldots,\theta_{m}) \in [-\tau,0]^{m}$. Sometimes the excluded index notation for these vectors is also used in different ways. For example, by $\bar{s}_{\hat{j}}$ we denote the $(m-1)$-vector appearing after eliminating the $j$th component from $\bar{s}$. Moreover, the same vector is denoted by $(s_{1},\ldots, \hat{s}_{j},\ldots,s_{m})$.

Given a real number $t \in \mathbb{R}$, we use $\underline{t}$ to denote the vector with identical components, all of which equal to $t$. Its dimension should be understood from the context. For example, if $\bar{s} \in \mathbb{R}^{m}$, then in the sum $\bar{s} + \underline{t}$ we have $\underline{t} \in \mathbb{R}^{m}$.

We use $\mu^{k}_{L}$ to denote the $k$-dimensional Lebesgue measure. This notation is used mainly in cases where it should be emphasized that we are dealing with $\mu^{k}_{L}$-almost all elements of a certain $k$-dimensional subset.

We use $\|\cdot\|_{\mathbb{E}}$ to denote the norm in a Banach space $\mathbb{E}$. In the case of a Hilbert space $\mathbb{H}$, we often (mainly in the context of $\mathbb{H}$-valued functions) use $|\cdot|_{\mathbb{H}}$ to denote the norm. Moreover, the inner product in $\mathbb{H}$ is denoted by $\langle \cdot, \cdot \rangle_{\mathbb{H}}$.

Given Banach spaces $\mathbb{E}$ and $\mathbb{F}$, we use $\mathcal{L}(\mathbb{E};\mathbb{F})$ to denote the space of bounded linear operators between them. If $\mathbb{E}=\mathbb{F}$, we write just $\mathcal{L}(\mathbb{E})$. For the corresponding operator norm, we use the notation $\|\cdot\|_{\mathcal{L}(\mathbb{E};\mathbb{F})}$ or simply $\|\cdot\|$, if the spaces are understood from the context. For the identity operator in $\mathbb{E}$, we use the notation $\operatorname{Id}_{\mathbb{E}}$. Often the same operator is denoted by $I$, if the space is understood from the context.

%% file: TensorProductAdditiveCompounds.tex
\section{Multiplicative compounds on tensor products of Hilbert spaces}
\label{SEC: CompOperatorsTensorProducts}
In this section, we briefly recall some facts about tensor products of Hilbert spaces.

Let $\mathbb{H}_{1}$ and $\mathbb{H}_{2}$ be real or complex Hilbert spaces with inner products\footnote{In the case of complex spaces, we use the convention that sesquilinear (Hermitian) forms are linear in the first argument and conjugate-linear in the second.} $\langle \cdot,\cdot \rangle_{\mathbb{H}_{1}}$ and $\langle \cdot,\cdot \rangle_{\mathbb{H}_{2}}$. Let $\mathbb{H}_{1} \odot \mathbb{H}_{2}$ denote the algebraic tensor product of these spaces. Recall that it is spanned by the elements $v_{1} \otimes v_{2}$, where $v_{1} \in \mathbb{H}_{1}$ and $v_{2} \in \mathbb{H}_{2}$, called decomposable tensors. As an operation, $v_{1} \otimes v_{2}$ is bilinear and called the tensor product of $v_{1}$ and $v_{2}$.

There is a natural inner product on $\mathbb{H}_{1} \odot \mathbb{H}_{2}$, which is defined for decomposable tensors $v_{1} \otimes v_{2}$ and $w_{1} \otimes w_{2}$ by
\begin{equation}
	\label{EQ: HilbertTensorProduct}
	\langle v_{1}\otimes v_{2}, w_{1} \otimes w_{2} \rangle_{\mathbb{H}_{1} \otimes \mathbb{H}_{2}} \coloneq \langle v_{1}, w_{1} \rangle_{\mathbb{H}_{1}} \langle v_{2},w_{2} \rangle_{\mathbb{H}_{2}}.
\end{equation}
Using standard arguments appealing to the universal property of algebraic tensor products, it can be shown that \eqref{EQ: HilbertTensorProduct} correctly extends to an inner product on $\mathbb{H}_{1} \odot \mathbb{H}_{2}$. Then the \textit{tensor product} $\mathbb{H}_{1} \otimes \mathbb{H}_{2}$ of Hilbert spaces $\mathbb{H}_{1}$ and $\mathbb{H}_{2}$ is defined by the completion of $\mathbb{H}_{1} \odot \mathbb{H}_{2}$ by the inner product from \eqref{EQ: HilbertTensorProduct}. Sometimes it may be convenient to emphasize the field, i.e., $\mathbb{R}$ or $\mathbb{C}$, over which the tensor product is taken. For this, we use the notations $\mathbb{H}_{1} \otimes_{\mathbb{R}} \mathbb{H}_{2}$ or $\mathbb{H}_{1} \otimes_{\mathbb{C}} \mathbb{H}_{2}$, respectively.


Let $\mathbb{W}_{1}$ and $\mathbb{W}_{2}$ be another Hilbert spaces over the same field as $\mathbb{H}_{1}$ and $\mathbb{H}_{2}$. Given bounded linear operators $L_{1} \colon \mathbb{H}_{1} \to \mathbb{W}_{1}$ and $L_{2} \colon \mathbb{H}_{2} \to \mathbb{W}_{2}$, their tensor product $L_{1} \otimes L_{2}$ is a bounded linear operator from $\mathbb{H}_{1} \otimes \mathbb{H}_{2}$ to $\mathbb{W}_{1} \otimes \mathbb{W}_{2}$ defined on decomposable tensors $v_{1} \otimes v_{2}$ by
\begin{equation}
	\label{EQ: TensorProductOperatorsDefDecomp}
	(L_{1} \otimes L_{2})(v_{1} \otimes v_{2}) \coloneq L_{1}v_{1} \otimes L_{2}v_{2}.
\end{equation}
It can be shown that this formula correctly determines a bounded linear operator on $\mathbb{H}_{1} \odot \mathbb{H}_{2}$, which extends by continuity to $\mathbb{H}_{1} \otimes \mathbb{H}_{2}$. Furthermore, we have that
\begin{equation}
	\label{EQ: TensorProductNormEquality}
	\|L_{1} \otimes L_{2}\| = \|L_{1}\| \cdot \|L_{2}\|,
\end{equation}
where $\|\cdot\|$ denotes appropriate operator norms. From \eqref{EQ: TensorProductOperatorsDefDecomp}, it is also clear that the relation $(L_{2} L_{1}) \otimes (L_{4} L_{3}) = (L_{2} \otimes L_{4}) (L_{1} \otimes L_{3})$ is satisfied for any bounded operators $L_{1}$, $L_{2}$, $L_{3}$, and $L_{4}$ defined on appropriate spaces.

Suppose that $\mathbb{H}_{1}$ is decomposed into the direct sum $\mathbb{H}_{1} = \mathbb{L}_{\sharp} \oplus \mathbb{L}_{\flat}$ of closed subspaces $\mathbb{L}_{\sharp}$ and $\mathbb{L}_{\flat}$. Then the spaces $\mathbb{L}_{\sharp} \otimes \mathbb{H}_{2}$ and $\mathbb{L}_{\flat} \otimes \mathbb{H}_{2}$ can be naturally considered as subspaces in $\mathbb{H}_{1} \otimes \mathbb{H}_{2}$, and there is the direct sum decomposition
\begin{equation}
	\label{EQ: HilbertTensorProductDirectSumDecomposition}
	\mathbb{H}_{1} \otimes \mathbb{H}_{2} = \left(\mathbb{L}_{\sharp} \otimes \mathbb{H}_{2}\right) \oplus \left(\mathbb{L}_{\flat} \otimes \mathbb{H}_{2}\right).
\end{equation}
A similar statement holds for decompositions of the second factor $\mathbb{H}_{2}$. This property is important for studying spectra of operators on tensor products, see Theorem \ref{TH: SpectrumAdditiveCompounds}.

Let $(\mathcal{X}_{1},\mu_{1})$ and $(\mathcal{X}_{2},\mu_{2})$ be two measure spaces. Given Hilbert spaces $\mathbb{F}_{1}$ and $\mathbb{F}_{2}$, we consider the Hilbert spaces\footnote{We refer to \cite{DunfordSchwartz1988} for the theory of integration involving functions with values in Banach spaces.} $\mathbb{H}_{1} \coloneq L_{2}(\mathcal{X}_{1};\mu_{1};\mathbb{F}_{1})$ and $\mathbb{H}_{2} \coloneq L_{2}(\mathcal{X}_{2};\mu_{2};\mathbb{F}_{2})$. Let $\mu_{1} \otimes \mu_{2}$ be the product measure on $\mathcal{X}_{1} \times \mathcal{X}_{2}$. The following theorem is well known, although we could not find any references for its statement in full generality in the literature, so we give a proof for the sake of completeness.
\begin{theorem}
	\label{TH: DescriptionTensorProductL2}
	For the above defined $L_{2}$-spaces $\mathbb{H}_{1}$ and $\mathbb{H}_{2}$, the mapping
	\begin{equation}
		\label{EQ: IsometricIsomorphismTensorL2}
		\mathbb{H}_{1} \otimes \mathbb{H}_{2} \ni \phi_{1} \otimes \phi_{2} \mapsto (\phi_{1} \otimes \phi_{2})(\cdot_{1},\cdot_{2}),
	\end{equation}
    where $(\phi_{1} \otimes \phi_{2})(x_{1},x_{2})\coloneq \phi_{1}(x_{1}) \otimes \phi_{2}(x_{2})$ for $(\mu_{1} \otimes \mu_{2})$-almost all $(x_{1},x_{2}) \in \mathcal{X}_{1} \times \mathcal{X}_{2}$,
	induces an isometric isomorphism between $\mathbb{H}_{1} \otimes \mathbb{H}_{2}$ and $L_{2}(\mathcal{X}_{1} \times \mathcal{X}_{2}; \mu_{1} \otimes \mu_{2};\mathbb{F}_{1} \otimes \mathbb{F}_{2})$. 
\end{theorem}
\begin{proof}
	Since the right-hand side of \eqref{EQ: IsometricIsomorphismTensorL2} is linear in $\phi_{1}$ and $\phi_{2}$, it correctly defines a mapping from $\mathbb{H}_{1} \odot \mathbb{H}_{2}$. Let us denote the $L_{2}$-space over $\mathcal{X}_{1} \times \mathcal{X}_{2}$ just by $L_{2}$. Then, directly from the definitions, we have for any $\phi_{1},\psi_{1}\in \mathbb{H}_{1}$, $\phi_{2},\psi_{2} \in \mathbb{H}_{2}$ that
	\begin{equation}
		\langle \phi_{1} \otimes \phi_{2}, \psi_{1} \otimes \psi_{2} \rangle_{\mathbb{H}_{1} \otimes \mathbb{H}_{2}} = \langle \phi_{1} \otimes \phi_{2}, \psi_{1} \otimes \psi_{2}  \rangle_{L_{2}}.
	\end{equation}
    From this, it follows that \eqref{EQ: IsometricIsomorphismTensorL2} indeed induces an isometric embedding from $\mathbb{H}_{1} \odot \mathbb{H}_{2}$ to $L_{2}$, and, consequently, it can be extended to the entire $\mathbb{H}_{1} \otimes \mathbb{H}_{2}$.
    
    It remains to show that the image of $\mathbb{H}_{1} \otimes \mathbb{H}_{2}$ under \eqref{EQ: IsometricIsomorphismTensorL2} coincides with $L_{2}$. Thanks to the isometry, the image is closed. Thus, it is sufficient to show that the image is dense in $L_{2}$. For this, let $\mathbb{L}$ be the subspace in $\mathbb{H}_{1} \otimes \mathbb{H}_{2}$ spanned by the elements $f_{1}\rchi_{\mathcal{B}_{1}} \otimes f_{2}\rchi_{\mathcal{B}_{2}}$, where $f_{1} \in \mathbb{F}_{1}$, $f_{2} \in \mathbb{F}_{2}$, and $\rchi_{\mathcal{B}_{1}}$ and $\rchi_{\mathcal{B}_{2}}$ are the characteristic (indicator) functions of measurable subsets $\mathcal{B}_{1} \subset \mathcal{X}_{1}$ and $\mathcal{B}_{2} \subset \mathcal{X}_{2}$. Clearly, the mapping from \eqref{EQ: IsometricIsomorphismTensorL2} transfers $f_{1}\rchi_{\mathcal{B}_{1}} \otimes f_{2}\rchi_{\mathcal{B}_{2}}$ into $(f_{1} \otimes f_{2})\rchi_{\mathcal{B}_{1} \times \mathcal{B}_{2}}$, where $\rchi_{\mathcal{B}_{1} \times \mathcal{B}_{2}}$ is defined analogously. Since the semiring of products $\mathcal{B}_{1} \times \mathcal{B}_{2}$ generate the $\sigma$-algebra on $\mathcal{X}_{1} \times \mathcal{X}_{2}$, linear combinations of $f_{1} \otimes f_{2}$ are dense in $\mathbb{F}_{1} \otimes \mathbb{F}_{2}$, and simple functions are dense in $L_{2}$, the image of $\mathbb{L}$ under \eqref{EQ: IsometricIsomorphismTensorL2} is dense in $L_{2}$.
\end{proof}

It can be shown that the tensor product of Hilbert spaces is associative, i.e., for any triple $\mathbb{H}_{1}$, $\mathbb{H}_{2}$, and $\mathbb{H}_{3}$ of Hilbert spaces, the tensor products $(\mathbb{H}_{1} \otimes \mathbb{H}_{2}) \otimes \mathbb{H}_{3}$ and $\mathbb{H}_{1} \otimes ( \mathbb{H}_{2} \otimes \mathbb{H}_{3} )$ are naturally isometrically isomorphic, and so they are simply denoted by $\mathbb{H}_{1} \otimes \mathbb{H}_{2} \otimes \mathbb{H}_{3}$. This allows to extend the previous constructions to tensor products of any finite number of Hilbert spaces. 

Given a Hilbert space $\mathbb{H}$ and a positive integer $m$, we denote its $m$-fold tensor product by $\mathbb{H}^{\otimes m}$. For an operator $L \in \mathcal{L}(\mathbb{H})$, its $m$-fold tensor product is denoted by $L^{\otimes m}$ and called\footnote{In \cite{Muldowneys1990CompMatr, MuldowneyWang2021}, the term (multiplicative or additive) ``compound'' is applied only to operators acting on exterior powers. It is convenient to apply this term for general tensor products.} the \textit{$m$-fold multiplicative compound} of $L$. We have the following description of the spectrum $\operatorname{spec}(L^{\otimes m})$ of $L^{\otimes m}$.
\begin{theorem}[\cite{BrownPearcy1966}]
	\label{TH: SpectrumOfMultiplicativeCompound}
	For $L \in \mathcal{L}(\mathbb{H})$, we have that
	\begin{equation}
		\label{EQ: SpectraMultComp}
		\operatorname{spec}(L^{\otimes m}) = \left\{ \lambda_{1}\cdot \ldots \cdot \lambda_{m} \mid \lambda_{j} \in \operatorname{spec}(L) \quad \text{for any} \ j \in \{1, \ldots, m\} \right\}.
	\end{equation}
\end{theorem} 
\begin{remark}
	\label{REM: SpectraOfTensorProducts}
	 For the spectrum of the tensor product $L_{1} \otimes \cdots \otimes L_{m}$ of operators $L_{j} \in \mathcal{L}(\mathbb{H}_{j})$ acting on possibly different Hilbert spaces $\mathbb{H}_{j}$, where $j \in \{1,\ldots,m\}$, we have the obvious analog of \eqref{EQ: SpectraMultComp}:
	 \begin{equation}
	 	\label{EQ: SpectraTensorProduct}
	 	\begin{split}
	 		\operatorname{spec}(L_{1} \otimes \cdots \otimes L_{m}) =\\= \left\{ \lambda_{1}\cdot \ldots \cdot \lambda_{m} \mid \lambda_{j} \in \operatorname{spec}(L_{j}) \quad \text{for any} \ j \in \{1, \ldots, m\} \right\},
	 	\end{split}
	 \end{equation}
	 which follows, for example, from \cite[Theorem 4.3]{Ichinose1970}. This stronger result is necessary for our study, see the footnote in the proof of Theorem \ref{TH: SpectrumAdditiveCompounds}.
\end{remark}
\begin{remark}
	In \cite{MalletParretNussbaum2013}, multiplicities of isolated spectral points for operators on injective tensor products of Banach spaces are described using the direct sum decomposition \eqref{EQ: HilbertTensorProductDirectSumDecomposition}, which is also applicable to our case. We do not need this result here, but similar arguments will be applied to show its analog for additive compounds in Theorem \ref{TH: AddCompoundMultiplicities}.
\end{remark}

Now let $\mathbb{S}_{m}$ be the symmetric group on $\{1,\ldots,m\}$, i.e., the set of all self-bijections (permutations) $\sigma$ of $\{1,\ldots,m\}$ with the usual composition as a group operation. For each permutation $\sigma \in \mathbb{S}_{m}$, consider the transposition operator $S_{\sigma} \in \mathcal{L}(\mathbb{H}^{\otimes m})$ defined on decomposable tensors $v_{1} \otimes \cdots \otimes v_{m}$ by
\begin{equation}
	\label{EQ: SymmetryOperatorWedgeDefinition}
	S_{\sigma}(v_{1} \otimes \cdots \otimes v_{m}) \coloneq v_{\sigma(1)} \otimes \cdots \otimes v_{\sigma(m)}.
\end{equation}
It is important to note that $S_{\sigma}$ is a well-defined bijective isometry on $\mathbb{H}^{\odot m}$ and hence can be extended by continuity to a unitary operator on $\mathbb{H}^{\otimes m}$. Furthermore, we have $S^{-1}_{\sigma} = S^{*}_{\sigma} = S_{\sigma^{-1}}$ and $S_{\sigma_{1}}S_{\sigma_{2}} = S_{\sigma_{2} \sigma_{1}}$ for all $\sigma,\sigma_{1},\sigma_{2} \in \mathbb{S}_{m}$.

We define the \textit{$m$-fold exterior product} $\mathbb{H}^{\wedge m}$ of $\mathbb{H}$ by
\begin{equation}
	\label{EQ: DefinitionWedgeHilberSpace}
	\mathbb{H}^{\wedge m} \coloneq \{ V \in \mathbb{H}^{\otimes m} \mid S_{\sigma}V = (-1)^{\sigma} V \quad \text{for any} \ \sigma \in \mathbb{S}_{m} \},
\end{equation}
where $(-1)^{\sigma}$ is the sign of $\sigma$. Equivalently, $\mathbb{H}^{\wedge m}$ can be described as the range of the orthogonal projector $\Pi^{\wedge}_{m}$ in $\mathbb{H}^{\otimes m}$ defined by
\begin{equation}
	\label{EQ: AntisymmetricProjectorDefinition}
	\Pi^{\wedge}_{m} \coloneq \frac{1}{m!} \sum_{\sigma \in \mathbb{S}_{m}}(-1)^{\sigma} S_{\sigma}.
\end{equation}
Clearly, we have $S_{\sigma}\Pi^{\wedge}_{m} = (-1)^{\sigma}\Pi^{\wedge}_{m}$, which is consistent with \eqref{EQ: DefinitionWedgeHilberSpace}. 

For all $v_{1},\ldots,v_{m} \in \mathbb{H}$, we set
\begin{equation}
	\label{EQ: AntisymmetrizationDecomposableDefinition}
	v_{1} \wedge \cdots \wedge v_{m} \coloneq \Pi^{\wedge}_{m}(v_{1} \otimes \cdots \otimes v_{m}).
\end{equation}
Then \eqref{EQ: HilbertTensorProduct}, \eqref{EQ: AntisymmetrizationDecomposableDefinition}, and \eqref{EQ: AntisymmetricProjectorDefinition}, yield the inner product formula in $\mathbb{H}^{\wedge m}$:
\begin{equation}
	\label{EQ: WedgeProductInheritedWithCoefficient}
	\langle v_{1} \wedge \cdots \wedge v_{m}, w_{1} \wedge \cdots \wedge w_{m}\rangle_{\mathbb{H}^{\otimes m}} = \frac{1}{m!} \det\left( \langle v_{k}, w_{l}\rangle_{\mathbb{H}} \right)_{k,l \geq 1}^{m},
\end{equation}
where all $v_{k}$ and $w_{l}$ belong to $\mathbb{H}$. We endow $\mathbb{H}^{\wedge m}$ with the induced inner product \eqref{EQ: WedgeProductInheritedWithCoefficient}. Some authors define the inner product in  $\mathbb{H}^{\wedge m}$ without the factor $1/m!$. For us, this is not convenient due to Theorem \ref{TH: ExteriorPowerL2DescriptionAntisymmetric} below and its use in the next sections.

For any operator $L \in \mathcal{L}(\mathbb{H})$, the operator $L^{\otimes m}$ commutes with $S_{\sigma}$ and hence with $\Pi^{\wedge}_{m}$. Therefore, there is a well-defined operator $L^{\wedge m}$ given by the restriction of $L^{\otimes m}$ to $\mathbb{H}^{\wedge m}$, which is called the \textit{$m$-fold antisymmetric multiplicative compound} of $L$ or the \textit{$m$-fold multiplicative compound} of $L$ in $\mathbb{H}^{\wedge m}$. Cocycles of such operators are the main object of our study, see Section \ref{SEC: CocyclesSemigroupsAdditiveCompounds}.

Now suppose $\mathbb{F}_{1}, \ldots, \mathbb{F}_{m}$ are Hilbert spaces. For any $\sigma \in \mathbb{S}_{m}$, we define the transposition operator $T_{\sigma}$ such that (here $f_{j} \in \mathbb{F}_{j}$ for $j \in \{1,\ldots,m\}$)
\begin{equation}
	\label{EQ: TranspoitionValuesDefinition}
	\begin{split}
		T_{\sigma} \colon \mathbb{F}_{1} \otimes \cdots \otimes \mathbb{F}_{m} \to \mathbb{F}_{\sigma(1)} \otimes \cdots \otimes \mathbb{F}_{\sigma(m)},\\
		T_{\sigma} (f_{1} \otimes \cdots \otimes f_{m}) \coloneq f_{\sigma(1)} \otimes \cdots \otimes f_{\sigma(m)}. 
	\end{split}
\end{equation}
Analogously to $S_{\sigma}$ from \eqref{EQ: SymmetryOperatorWedgeDefinition}, we have that $T_{\sigma}$ is a bijective isometry. Below, when the notation $T_{\sigma}$ is used, the spaces $\mathbb{F}_{1},\ldots,\mathbb{F}_{m}$ should be understood from the context. In this sense, the identities $T^{-1}_{\sigma} = T_{\sigma^{-1}}$ and $T_{\sigma_{2}} T_{\sigma_{1}} = T_{\sigma_{1}\sigma_{2}}$ may be understood. Note that if all the spaces $\mathbb{F}_{j}$, except possibly one, are one-dimensional, then any operator $T_{\sigma}$ is identical.

Below, we study functions taking values in a tensor product of Hilbert spaces. We often consider $T_{\sigma}$, acting in the space of values, as an operator on such functions, meaning that it is applied pointwise.

Let $\mathbb{F}$ be a Hilbert space and $\mathcal{X}$ be a set. A function $\Phi \colon \mathcal{X}^{m} \to \mathbb{F}^{\otimes m}$ is called \textit{antisymmetric} if for any $\sigma \in \mathbb{S}_{m}$ and all $x_{1},\ldots,x_{m} \in \mathcal{X}$, we have that
\begin{equation}
	\label{EQ: AntisymmetricFunctionDefinition}
	\Phi(x_{\sigma(1)},\ldots, x_{\sigma(m)} ) = (-1)^{\sigma} T_{\sigma} \Phi(x_{1},\ldots,x_{m}).
\end{equation}
In the context of a given measure $\nu$ on $\mathcal{X}^{m}$, we usually require \eqref{EQ: AntisymmetricFunctionDefinition} to be satisfied only for $\nu$-almost all $(x_{1},\ldots,x_{m}) \in \mathcal{X}^{m}$ and say that $\Phi$ is $\nu$-\textit{antisymmetric}. Note that for $\mathbb{F} = \mathbb{R}$ (or $\mathbb{C}$ in the complex case), the operator $T_{\sigma}$ is identical, and the definition coincides with the usual definition of an antisymmetric function that changes its sign according to the permutation of arguments. 

For any $\sigma \in \mathbb{S}_{m}$, it is convenient to introduce the operator
\begin{equation}
	\label{EQ: ThetaSigmaDefinition}
	(\Theta_{\sigma}\Phi)(x_{1},\ldots,x_{m}) \coloneq \Phi(x_{\sigma(1)},\ldots, x_{\sigma(m)})
\end{equation}
acting on functions $\Phi$ as above. Then \eqref{EQ: AntisymmetricFunctionDefinition} is read as $\Theta_{\sigma}\Phi = (-1)^{\sigma}T_{\sigma}\Phi$. Note that $\Theta_{\sigma} \Theta_{\sigma'} = \Theta_{\sigma \sigma'}$ and $\Theta_{\sigma}$ commutes with $T_{\sigma'}$ for any $\sigma, \sigma' \in \mathbb{S}_{m}$. It sometimes is convenient to write $\Theta^{(m)}_{\sigma}$ to emphasize the number of variables being permuted.
\begin{remark}
	Let us emphasize that the correspondence $\sigma \mapsto T_{\sigma}$ is an antihomomorphism, while $\sigma \mapsto \Theta_{\sigma}$ is a homomorphism. For the latter, note that $\Theta_{\sigma}$ is the mapping $h^{*}_{\sigma}$ on functions (a change of variables) associated with the mapping $h_{\sigma}$ of $\mathcal{X}^{m}$ permuting the arguments, i.e., $h_{\sigma}(x_{1},\ldots,x_{m}) = (x_{\sigma(1)},\ldots,x_{\sigma(m)})$. In turn, $h_{\sigma}$ is itself the mapping on $x_{j}$, considered as functions of the argument $j \in \{1,\ldots,m\}$, associated with $\sigma$, i.e., $h_{\sigma} = \sigma^{*}$ in similar terms. Thus, $\Theta_{\sigma}$ is the result of two contravariant operations (antihomomorphisms) $\sigma \mapsto \sigma^{*}=h_{\sigma}$ and $h_{\sigma} \mapsto h^{*}_{\sigma} = \Theta_{\sigma}$. On the other hand, $T_{\sigma}$ (in particular, $S_{\sigma}$) is obtained by a single contravariant operation.
\end{remark}

Suppose that $\mu$ is a measure on $\mathcal{X}$ and set $\mathbb{H} \coloneq L_{2}(\mathcal{X};\mu;\mathbb{F})$. Let $\mu^{\otimes m}$ be the $m$-fold product of $\mu$ with itself, which is a measure on $\mathcal{X}^{m}$.
\begin{theorem}
	\label{TH: ExteriorPowerL2DescriptionAntisymmetric}
     For the spaces $\mathbb{H}^{\otimes m}$ and $L_{2}(\mathcal{X}^{m};\mu^{\otimes m}; \mathbb{F}^{\otimes m})$, consider the natural isometric isomorphism induced by (see Theorem \ref{TH: DescriptionTensorProductL2})
	\begin{equation}
		\label{EQ: IsomorphismTensorProductsAntiSym}
		\phi_{1} \otimes \cdots \otimes \phi_{m} \mapsto (\phi_{1} \otimes \cdots \otimes \phi_{m})(\cdot_{1},\ldots,\cdot_{m}),
	\end{equation}
    where $(\phi_{1} \otimes \cdots \otimes \phi_{m})(x_{1},\ldots,x_{m}) \coloneq \phi_{1}(x_{1}) \otimes \cdots \otimes \phi_{m}(x_{m})$ for $\mu^{\otimes m}$-almost all \\$(x_{1},\ldots,x_{m}) \in \mathcal{X}^{m}$.
    Then its restriction to $\mathbb{H}^{\wedge m}$ is an isometric isomorphism between $\mathbb{H}^{\wedge m}$ and the subspace of $\mu^{\otimes m}$-antisymmetric functions in $L_{2}(\mathcal{X}^{m};\mu^{\otimes m}; \mathbb{F}^{\otimes m})$.
\end{theorem}
\begin{proof}
	In virtue of Theorem \ref{TH: DescriptionTensorProductL2}, it is sufficient to show that the image of $\mathbb{H}^{\wedge m}$ under \eqref{EQ: IsomorphismTensorProductsAntiSym} coincides with the subspace of $\mu^{\otimes m}$-antisymmetric functions.
	
	Let $S_{\sigma}$ be the transposition operator in $\mathbb{H}^{\otimes m}$ defined by \eqref{EQ: SymmetryOperatorWedgeDefinition}. Up to the isomorphism \eqref{EQ: IsomorphismTensorProductsAntiSym}, $S_{\sigma}$ is seen to be expressed as $S_{\sigma} = T_{\sigma}\Theta_{\sigma^{-1}}$ in terms of the operators $T_{\sigma}$ in $\mathbb{F}^{\otimes m}$ and $\Theta_{\sigma}$ in $L_{2}(\mathcal{X}^{m};\mu^{\otimes m}; \mathbb{F}^{\otimes m})$, defined by \eqref{EQ: TranspoitionValuesDefinition} and \eqref{EQ: ThetaSigmaDefinition}, respectively. Analogously, the projector $\Pi^{\wedge}_{m}$ in $\mathbb{H}^{\otimes m}$ from \eqref{EQ: AntisymmetricProjectorDefinition} corresponds to the projector in $L_{2}(\mathcal{X}^{m};\mu^{\otimes m}; \mathbb{F}^{\otimes m})$:
	\begin{equation}
		\label{EQ: AnntisymmetricProjectorFunctions}
		\Pi^{\wedge}_{m} = \frac{1}{m!}\sum_{\sigma \in \mathbb{S}_{m}} (-1)^{\sigma} T_{\sigma^{-1}} \Theta_{\sigma}.
	\end{equation}
	Furthermore, $S_{\sigma^{-1}} \Pi^{\wedge}_{m} = (-1)^{\sigma}\Pi^{\wedge}_{m}$ translates into $\Theta_{\sigma}\Pi^{\wedge}_{m} = (-1)^{\sigma} T_{\sigma}\Pi^{\wedge}_{m}$, showing that the range of $\Pi^{\wedge}_{m}$ is the subspace of $\mu^{\otimes}$-antisymmetric functions, see below \eqref{EQ: ThetaSigmaDefinition}.
\end{proof}

At the end of this section, we recall the construction of complexification. Let $\mathbb{H}$ be a real Hilbert space. Then its \textit{complexification} $\mathbb{H}^{\mathbb{C}}$ is defined as the outer Hilbert direct sum $\mathbb{H} \oplus i\mathbb{H}$ consisting of elements $v + i w$, where $v,w \in \mathbb{H}$, and equipped with the natural multiplication over $\mathbb{C}$. In $\mathbb{H}^{\mathbb{C}}$, there is a natural sesquilinear form $\langle\cdot,\cdot\rangle_{\mathbb{H}^{\mathbb{C}}}$ defined by 
\begin{equation}
	\langle v+iw,v+iw \rangle_{\mathbb{H}^{\mathbb{C}}} \coloneq \langle v,v \rangle_{\mathbb{H}} + \langle w,w \rangle_{\mathbb{H}} \qquad \text{for all} \ v,w \in \mathbb{H}.
\end{equation}
Clearly, $\mathbb{H}^{\mathbb{C}}$, equipped with $\langle \cdot,\cdot \rangle_{\mathbb{H}^{\mathbb{C}}}$, is a complex Hilbert space.

Recall that for a linear operator $L$ in $\mathbb{H}$ with domain $\mathcal{D}(L)$, the \textit{complexification} $L^{\mathbb{C}}$ of $L$ is a linear operator in $\mathbb{H}^{\mathbb{C}}$ given by $L^{\mathbb{C}}(v+iw) \coloneq Lv + i Lw$ for all $v,w \in \mathcal{D}(L)$.

For a real Hilbert space $\mathbb{H}$, we may consider $\mathbb{H} \otimes_{\mathbb{R}} \mathbb{C}$ as a complex Hilbert space by defining the scalar multiplication as $\alpha \cdot (v\otimes z)\coloneq v \otimes (\alpha z)$ for all $v \in \mathbb{H}$ and $\alpha, z \in \mathbb{C}$. The following well-known properties are readily established.
\begin{proposition}
	\label{PROP: ComplexficationsTensorProducts}
	For real Hilbert spaces $\mathbb{H}$, $\mathbb{H}_{1}$, $\mathbb{H}_{2}$, $\mathbb{F}$ and a measure space $(\mathcal{X},\mu)$, we have the following natural isometric isomorphisms:
	\begin{enumerate}
		\item[1)] $\mathbb{H}^{\mathbb{C}} \cong \mathbb{H} \otimes_{\mathbb{R}} \mathbb{C}$;
		\item[2)] $\left(\mathbb{H}_{1} \otimes_{\mathbb{R}} \mathbb{H}_{2} \right)^{\mathbb{C}} \cong \mathbb{H}^{\mathbb{C}}_{1} \otimes_{\mathbb{C}} \mathbb{H}^{\mathbb{C}}_{2}$;
		\item[3)] $L_{2}(\mathcal{X};\mu;\mathbb{F}) \otimes_{\mathbb{R}} \mathbb{C} \cong L_{2}(\mathcal{X};\mu;\mathbb{F}^{\mathbb{C}})$.
	\end{enumerate}
	Moreover, for any operators $L_{1} \in \mathcal{L}(\mathbb{H}_{1})$ and $L_{2} \in \mathcal{L}(\mathbb{H}_{2})$, the identity $(L_{1} \otimes L_{2})^{\mathbb{C}} = (L^{\mathbb{C}}_{1} \otimes L^{\mathbb{C}}_{2})$ holds up to the isomorphism from item 2).
\end{proposition}

%% file: CocyclesSemigroupsAdditiveCompounds.tex
\section{Cocycles, $C_{0}$-semigroups, and additive compounds}
\label{SEC: CocyclesSemigroupsAdditiveCompounds}

Given a time set\footnote{Here $\mathbb{R}_{+} = [0,+\infty)$.} $\mathbb{T} \in \{\mathbb{R}_{+}, \mathbb{R}\}$ and a complete metric space $\mathcal{Q}$, a \textit{dynamical system} with continuous time is a family of mappings $\vartheta^{t}(\cdot) \colon \mathcal{Q} \to \mathcal{Q}$, where $t \in \mathbb{T}$, such that
\begin{description}[before={\renewcommand\descriptionlabel[1]{\ref{##1}}}]
	\item[\textbf{(DS1)}\refstepcounter{desccount}\label{DESC:DS1}] for all $q \in \mathcal{Q}$ and $t,s \in \mathbb{T}$, we have $\vartheta^{t+s}(q) = \vartheta^{t}( \vartheta^{s}(q))$ and $\vartheta^{0}(q) = q$;
	\item[\textbf{(DS2)}\refstepcounter{desccount}\label{DESC:DS2}] the mapping $\mathbb{T} \times \mathcal{Q} \ni (t,q) \mapsto \vartheta^{t}(q)$ is continuous.
\end{description}
For the sake of brevity, it is convenient to use the notation $(\mathcal{Q},\vartheta)$ or just $\vartheta$ to denote the dynamical system. If $\mathbb{T} = \mathbb{R}_{+}$ (resp. $\mathbb{T} = \mathbb{R}$), then $\vartheta$ is called a \textit{semiflow} (resp. \textit{flow}) on $\mathcal{Q}$.

Given a Banach space $\mathbb{E}$ and a dynamical system $(\mathcal{Q},\vartheta)$, a \textit{cocycle} in $\mathbb{E}$ over $(\mathcal{Q},\vartheta)$ is a family of mappings $\psi^{t}(q,\cdot) \colon \mathbb{E} \to \mathbb{E}$, where $t \in \mathbb{R}_{+}$ and $q \in \mathcal{Q}$, such that 
\begin{description}[before=\let\makelabel\descriptionlabel]
	\item[\textbf{(CO1)}\refstepcounter{desccount}\label{DESC: CO1}] for all $v \in \mathbb{E}$, $q \in \mathcal{Q}$, and $t,s \in \mathbb{R}_{+}$, we have $\psi^{t+s}(q,v) = \psi^{t}(\vartheta^{s}(q),\psi^{s}(q,v))$ and $\psi^{0}(q,v) = v$;
	\item[\textbf{(CO2)}\refstepcounter{desccount}\label{DESC: CO2}] the mapping $\mathbb{R}_{+} \times \mathcal{Q} \times \mathbb{E} \ni (t,q,v) \mapsto \psi^{t}(q,v)$ is continuous.
\end{description}
For the sake of brevity, it is convenient to denote the cocycle simply by $\psi$. If each mapping $\psi^{t}(q,\cdot)$ is a bounded linear operator in $\mathbb{E}$, then the cocycle is called \textit{linear}, and we usually denote it by $\Xi$. Furthermore, if such $\Xi$ additionally satisfies the following properties:
\begin{description}[before=\let\makelabel\descriptionlabel]
	\item[\textbf{(UC1)}\refstepcounter{desccount}\label{DESC: UC1}] for any $t \in \mathbb{R}_{+}$, the mapping $\mathcal{Q} \ni q \mapsto \Xi^{t}(q,\cdot) \in \mathcal{L}(\mathbb{E})$ is continuous in the operator norm; 
	\item[\textbf{(UC2)\refstepcounter{desccount}\label{DESC: UC2}}] the cocycle mappings $\Xi^{t}(q,\cdot)$ are bounded uniformly in finite times\footnote{Thanks to \nameref{DESC: CO1}, it follows from \eqref{EQ: UniformContinuityCocycleBoundednessFiniteTimes} that the analogous supremum, but taken over $t \in [0,T]$, is also finite any $T>0$.}, that is
	\begin{equation}
		\label{EQ: UniformContinuityCocycleBoundednessFiniteTimes}
		\sup_{t \in [0,1]} \sup_{q \in \mathcal{Q}}\| \Xi^{t}(q,\cdot) \|_{\mathcal{L}(\mathbb{E})} < +\infty,
	\end{equation}
\end{description}
then $\Xi$ is called a \textit{uniformly continuous linear cocycle}. For such cocycles, due to uniform boundedness, it is seen that \nameref{DESC: CO2} is equivalent to the continuous dependence of the operator $\Xi^{t}(q,\cdot)$ on $(t,q)$ in the strong operator topology.

For what follows, we fix a Hilbert space $\mathbb{H}$ and a dynamical system $(\mathcal{Q}, \vartheta)$. Given a uniformly continuous linear cocycle $\Xi$ in $\mathbb{H}$ over $(\mathcal{Q},\vartheta)$, we define the \textit{$m$-fold multiplicative compound} $\Xi_{m}$ of $\Xi$ in $\mathbb{H}^{\otimes m}$ as the cocycle with the mappings $\Xi^{t}_{m}(q,\cdot) \in \mathcal{L}(\mathbb{H}^{\otimes m})$ given by the $m$-fold multiplicative compound of $\Xi^{t}(q,\cdot)$. For convenience, the same notation is used to denote the restriction of such $\Xi_{m}$ to the $m$-fold exterior power $\mathbb{H}^{\wedge m}$, called the \textit{$m$-fold multiplicative compound of $\Xi$} in $\mathbb{H}^{\wedge m}$. It is indeed a uniform continuous cocycle, as stated in the following proposition.

\begin{proposition}
	\label{PROP: CompoundCocycleIsUniformlyContinuous}
	Let $\Xi$ be as above. Then $\Xi_{m}$ is a uniformly continuous linear cocycle in $\mathbb{H}^{\otimes m}$ (in particular, in $\mathbb{H}^{\wedge m}$).
\end{proposition}
\begin{proof}
	The cocycle property \nameref{DESC: CO1} for $\Xi_{m}$ follows from \eqref{EQ: TensorProductOperatorsDefDecomp} and the cocycle property for $\Xi$. Moreover, \nameref{DESC: UC2} for $\Xi$ and \eqref{EQ: TensorProductNormEquality} gives that $\Xi_{m}$ also satisfies \nameref{DESC: UC2}.
	
	To show \nameref{DESC: UC1} for $\Xi_{m}$, we use  \nameref{DESC: UC1} for $\Xi$ and the fact that
	\begin{equation}
		\label{EQ: TensorPowerOfSumOperators}
		(A+B)^{\otimes m} = A^{\otimes m} + R(A,B,m),
	\end{equation}
    where $\|R(A,B,m)\| \leq C \cdot \|B\|$ for $\|B\| \leq 1$ and some constant $C = C(\|A\|,m)$. This should be applied to $A \coloneq \Xi^{t}(q_{0},\cdot)$ and $B \coloneq \Xi^{t}(q,\cdot) - \Xi^{t}(q_{0},\cdot)$ with $q \to q_{0}$ in $\mathcal{Q}$.
    
    Finally, thanks to \nameref{DESC: UC2}, to show that $\Xi_{m}$ satisfies \nameref{DESC: CO2}, it is sufficient to prove that the mapping $\mathbb{R}_{+} \times \mathcal{Q} \ni (t,q) \mapsto \Xi^{t}_{m}(q,V) \in \mathbb{H}^{\otimes m}$ is continuous for a dense subset of $V \in \mathbb{H}^{\otimes m}$. But for $V$, which is a linear combination of decomposable tensors, this follows from \eqref{EQ: TensorProductOperatorsDefDecomp}.
\end{proof}

Let $\mathcal{B}_{1}$ denote the unit ball in $\mathbb{H}$. We call $\Xi$ \textit{uniformly eventually compact} for $t \geq t_{0}$ if the set $\Xi^{t}(\mathcal{Q},\mathcal{B}_{1}) = \bigcup_{q \in \mathcal{Q}} \Xi^{t}(q,\mathcal{B}_{1})$ is precompact in $\mathbb{H}$ for any $t \geq t_{0}$. Along with \nameref{DESC: UC1} and \nameref{DESC: UC2}, compactness properties are important for recovering spectral decompositions under certain cone conditions, see \cite{Anikushin2020Geom}. Fortunately, uniform eventual compactness is also inherited by compound cocycles.

\begin{proposition}
	\label{PROP: EventuallyCompactCocycles}
	Let $\Xi$ be uniformly eventually compact for $t \geq t_{0}$. Then $\Xi_{m}$ is also uniformly eventually compact for $t \geq t_{0}$.
\end{proposition}
\begin{proof}
	Let $t \geq t_{0}$ be fixed. By our assumptions, we may assume that $\Xi^{t}(\mathcal{Q},\mathbb{H})$ is contained in a separable closed subspace $\mathbb{L}$. Consider an orthonormal basis $\{e_{l} \}_{l \geq 1}$ in $\mathbb{L}$ and the orthogonal projector $P_{N}$ onto $\operatorname{Span}\{ e_{1}, \ldots, e_{N} \}$. Then for any $t \geq t_{0}$, we have that
	\begin{equation}
		\sup_{q \in \mathcal{Q}}\| \Xi^{t}(q,\cdot) - P_{N}\Xi^{t}(q,\cdot) \|_{\mathcal{L}(\mathbb{H})} \to 0 \qquad \text{as} \ N \to \infty.
	\end{equation}
   	Using this and applying arguments similar to those used below \eqref{EQ: TensorPowerOfSumOperators}, we obtain that
    \begin{equation}
    	\sup_{q \in \mathcal{Q}}\left\| \Xi^{t}_{m}(q,\cdot) - \left(P_{N}\Xi^{t}(q,\cdot) \right)^{\otimes m} \right\|_{\mathcal{L}(\mathbb{H}^{\otimes m})} \to 0 \qquad \text{as} \ N \to \infty.
    \end{equation}
    Since the operators $\left(P_{N}\Xi^{t}(q,\cdot) \right)^{\otimes m}$ have uniformly (in $q$) finite ranges, this yields that $\Xi_{m}$ is uniformly eventually compact for $t \geq t_{0}$.
\end{proof}

Now we are going to introduce additive compound operators for generators of $C_{0}$-semigroups. For the general theory of $C_{0}$-semigroups, we refer to \cite{EngelNagel2000}. Below, a $C_{0}$-semigroup in $\mathbb{H}$ is denoted by $G$ with its time-$t$ mapping denoted by $G(t)$ for $t \geq 0$. Note that any $C_{0}$-semigroup can be considered as a uniformly continuous linear cocycle over the trivial dynamical system on a one-point set. In particular, according to Proposition \ref{PROP: CompoundCocycleIsUniformlyContinuous}, there is the associated semigroup $G^{\otimes m}$ (resp. $G^{\wedge m}$) called the \textit{$m$-fold multiplicative compound} of $G$ in $\mathbb{H}^{\otimes}$ (resp. $\mathbb{H}^{\wedge m}$).

Let $A$ be the generator of a $C_{0}$-semigroup $G$. By $A^{[\otimes m]}$ (resp. $A^{[\wedge m]}$), we denote the generator of $G^{\otimes m}$ (resp. $G^{\wedge m}$), which will be called the \textit{$m$-fold additive compound} of $A$ in $\mathbb{H}^{\otimes m}$ (resp. in $\mathbb{H}^{\wedge m}$). Let $\mathcal{D}(A^{[\otimes m]})$ (resp. $\mathcal{D}(A^{[\wedge m]})$) be the domain of $A^{[\otimes m]}$ (resp. $A^{[\wedge m]}$). It is clear from the definition that $\mathcal{D}(A^{[\wedge m]}) = \mathcal{D}(A^{[\otimes m]}) \cap \mathbb{H}^{\wedge m}$ and $A^{[\wedge m]}$ is the restriction of $A^{[\otimes m]}$ to $\mathbb{H}^{\wedge m}$.


\begin{theorem}
	\label{TH: CompoundDescriptionBasicAction}
	For all $v_{1},\ldots,v_{m} \in \mathcal{D}(A)$, we have that $v_{1} \otimes \cdots \otimes v_{m} \in \mathcal{D}(A^{[\otimes m]})$ and
	\begin{equation}
		\label{EQ: AdditiveCompoundDecTensorFormula}
		A^{[\otimes m]}(v_{1} \otimes \cdots \otimes v_{m}) = \sum_{j=1}^{m} v_{1} \otimes \cdots \otimes A v_{j} \otimes \cdots \otimes v_{m}.
	\end{equation}
    In particular, $v_{1} \wedge \cdots \wedge v_{m} \in \mathcal{D}(A^{[\wedge m]})$ and
    \begin{equation}
    	\label{EQ: AdditiveCompoundAntiSymDecTensorFormula}
	    A^{[\wedge m]}(v_{1} \wedge \cdots \wedge v_{m}) = \sum_{j=1}^{m} v_{1} \wedge \cdots \wedge A v_{j} \wedge \cdots \wedge v_{m}.
    \end{equation}
    Furthermore, the algebraic $m$-fold tensor product $\mathcal{D}(A)^{\odot m}$ (resp. $\Pi^{\wedge}_{m}\mathcal{D}(A)^{\odot m}$) is dense in $\mathcal{D}(A^{[\otimes m]})$ (resp. $\mathcal{D}(A^{[\wedge m]})$) in the graph norm.
\end{theorem}
\begin{proof}
	Indeed, for $v_{0} \in \mathcal{D}(A)$, the function $[0,\infty) \ni t \mapsto G(t)v_{0} \in \mathbb{H}$ is $C^{1}$-differentiable and $G(t)v_{0} \in \mathcal{D}(A)$ and $\frac{d}{dt}(G(t)v_{0})=AG(t)v_{0}$ for any $t \geq 0$. Using this and the definition $G^{\otimes m}(t)(v_{1} \otimes \cdots \otimes v_{m}) = G(t)v_{1} \otimes \cdots \otimes G(t)v_{m}$, we obtain that
	\begin{equation}
		\lim_{t \to 0+} \frac{1}{t}\left(G^{\otimes m}(t)-\operatorname{Id}_{\mathbb{H}^{\otimes m}}\right)(v_{1} \otimes \cdots \otimes v_{m}) = \sum_{j=1}^{m} v_{1} \otimes \cdots \otimes A v_{j} \otimes \cdots v_{m}.
	\end{equation}
    Consequently, $v_{1} \otimes \cdots \otimes v_{m} \in \mathcal{D}(A^{[\otimes m]})$ and \eqref{EQ: AdditiveCompoundDecTensorFormula} is satisfied. From this, it is not hard to check \eqref{EQ: AdditiveCompoundAntiSymDecTensorFormula}. Moreover, it is clear that $\mathcal{D}(A)^{\odot m}$ (resp. $\Pi^{\wedge}_{m}\mathcal{D}(A)^{\odot m}$) is invariant with respect to $G^{\otimes m}(t)$ (resp. $G^{\wedge m}(t)$) and is dense in $\mathbb{H}^{\otimes m}$ (resp. $\mathbb{H}^{\wedge m}$) due to the density of $\mathcal{D}(A)$ in $\mathbb{H}$. Then \cite[Chapter II, Proposition 1.7]{EngelNagel2000} gives the density in the graph norm.
\end{proof}

Recall that $G$ is called \textit{eventually norm continuous} if for some $t_{0} \geq 0$ the mapping $\mathbb{R}_{+} \ni t \mapsto G(t) \in \mathcal{L}(\mathbb{H})$ is continuous at $t_{0}$ (and, consequently, at any $t \geq t_{0}$) in the operator norm. By \cite[Chapter II, Lemma 4.22]{EngelNagel2000}, if $G$ is eventually compact, i.e., $G(t_{0})$ is compact for some $t_{0} > 0$, then $G$ is eventually norm continuous.

\begin{proposition}
	\label{PROP: EventualNormContCompoundSemigroups}
	Suppose that $G$ is eventually norm continuous. Then $G^{\otimes m}$ (in particular, $G^{\wedge m}$) is also eventually norm continuous.
\end{proposition}
\begin{proof}
	This statement follows from arguments similar to those used below \eqref{EQ: TensorPowerOfSumOperators}.
\end{proof}

\begin{remark}
	\label{REM: ComplexificationOperatorsCompound}
	For semigroups in a real Hilbert space $\mathbb{H}$, we have $(G^{\otimes m}(t))^{\mathbb{C}} = \left((G(t))^{\mathbb{C}}\right)^{\otimes m}$ for any $t \geq 0$, thanks to Proposition \ref{PROP: ComplexficationsTensorProducts}. Bearing in mind that the generator of the complexification of a $C_{0}$-semigroup is the complexification of its generator, this implies that $(A^{[\otimes m]})^{\mathbb{C}} = (A^{\mathbb{C}})^{[\otimes m]}$ and $(A^{[\wedge m]})^{\mathbb{C}} = (A^{\mathbb{C}})^{[\wedge m]}$.
\end{remark}

Below, we consider the spectra of operators, and this requires consideration of complex spaces. From this point of view, Remark \ref{REM: ComplexificationOperatorsCompound} justifies the application of these results to the case of real spaces.

 
For what follows, $\omega(G)$ denotes the growth bound of $G$, $s(A)$ denotes the spectral bound of $A$, and $\operatorname{spec}(A)$ denotes the spectrum of $A$.
\begin{proposition}
	\label{EQ: PropositionSpectrumAdditiveUnderENC}
	Let $G$ be eventually norm continuous. Then for any $t \geq 0$, we have
	\begin{equation}
		\begin{split}
			\operatorname{spec}(G(t)) \setminus \{0\} &= e^{t\operatorname{spec}(A)},\\
			\operatorname{spec}(G^{\otimes m}(t)) \setminus \{ 0 \} &= e^{t \operatorname{spec}(A^{[\otimes m]})},\\
			\operatorname{spec}(G^{\wedge m}(t)) \setminus \{ 0 \} &= e^{t \operatorname{spec}(A^{[\wedge m]})}.
		\end{split}
	\end{equation}
    In particular, the growth bound $\omega(G^{\otimes m})$ (resp. $\omega(G^{\wedge m})$) equals the spectral bound $s(A^{[\otimes m]})$ (resp. $s(A^{[\wedge m]})$).
\end{proposition}
\begin{proof}
	This follows from Proposition \ref{PROP: EventualNormContCompoundSemigroups} and the spectral mapping theorem, see \cite[Theorem 3.10, Chapter IV]{EngelNagel2000}.
\end{proof}

For any eigenvalue $\lambda$ of $A$, we use $\mathbb{L}_{A}(\lambda)$ to denote the spectral subspace associated with $\lambda$. A similar notation is used for $A^{[\otimes m]}$ and $A^{[\wedge m]}$.

For eventually compact semigroups $G$, both semigroups $G^{\otimes m}$ and $G^{\wedge m}$ are also eventually compact, thanks to Proposition \ref{PROP: EventuallyCompactCocycles}. In particular, the spectra of $A$, $A^{[\otimes m]}$, and $A^{[\wedge m]}$ consist of eigenvalues having finite algebraic multiplicities, and there is a finite number of eigenvalues in each right half-plane, see \cite[Chapter V, Theorem 3.1]{EngelNagel2000}. In this case, the spectral properties of $A^{[\otimes m]}$ and $A^{[\wedge m]}$ can be fully described in terms of the spectral properties of $A$. This is constituted by Theorems \ref{TH: SpectrumAdditiveCompounds} and \ref{TH: AddCompoundMultiplicities} below.
\begin{theorem}
	\label{TH: SpectrumAdditiveCompounds}
	Suppose that $G$ is eventually compact. Then 
	\begin{equation}
	\label{EQ: SpectralSetAdditiveCompound}
	\operatorname{spec}(A^{[\otimes m]}) = \left\{\sum_{j=1}^{m} \lambda_{j} \mid \lambda_{j} \in \operatorname{spec}(A) \quad \text{for any} \ j \in \{1, \ldots, m\} \right\}.
	\end{equation}
    Moreover, any $\lambda_{0} \in \operatorname{spec}(A^{[\otimes m]})$ is an eigenvalue, and there exist finitely many, say $N$, distinct $m$-tuples $\left(\lambda^{k}_{1},\ldots,\lambda^{k}_{m}\right) \in \mathbb{C}^{m}$, where $k \in \{1,\ldots,N\}$, such that
    \begin{equation}
    	\label{EQ: AdditiveCompoundIsolatedPointDesc}
    	\lambda_{0} = \sum_{j=1}^{m}\lambda^{k}_{j} \quad \text{and} \quad \lambda^{k}_{j} \in \operatorname{spec}(A).
    \end{equation}
    In this notation, the spectral subspace $\mathbb{L}_{A^{[\otimes m]}}(\lambda_{0})$ can be described as follows:
    \begin{equation}
    	\label{EQ: SpectralSubspaceIsolatedAdditiveCompound}
    	\mathbb{L}_{A^{[\otimes m]}}(\lambda_{0}) = \bigoplus_{k=1}^{N} \bigotimes_{j=1}^{m} \mathbb{L}_{A}(\lambda^{k}_{j}).
    \end{equation}
    Furthermore, $\lambda_{0}$ is an eigenvalue of $A^{[\wedge m]}$ if and only if $\lambda_{0}$ is an eigenvalue of $A^{[\otimes m]}$ and $\Pi^{\wedge}_{m} \mathbb{L}_{A^{[\otimes m]}}(\lambda_{0}) \not= \{ 0 \}$, where $\Pi^{\wedge}_{m}$ is the orthogonal projector onto $\mathbb{H}^{\wedge m}$ defined by \eqref{EQ: AntisymmetricProjectorDefinition}. In this case, the spectral subspaces are related as follows:
    \begin{equation}
    	\label{EQ: SpectralSubspaceAntisymmetricCompound}
    	\mathbb{L}_{A^{[\wedge m]}}(\lambda_{0})=\Pi^{\wedge}_{m}\mathbb{L}_{A^{[\otimes m]}}(\lambda_{0}) = \mathbb{L}_{A^{[\otimes m]}}(\lambda_{0}) \cap \mathbb{H}^{\wedge m}.
    \end{equation}
\end{theorem}
\begin{proof}
	By \eqref{EQ: AdditiveCompoundDecTensorFormula}, any eigenvectors $e_{1},\ldots,e_{m}$ of $A$ corresponding to some eigenvalues $\lambda_{1},\ldots,\lambda_{m}$ give rise to the eigenvector $e_{1} \otimes \cdots \otimes e_{m}$ of $A^{[\otimes m]}$ corresponding to the eigenvalue $\lambda_{0} = \lambda_{1} + \cdots + \lambda_{m}$. In particular, the inclusion ``$\supset$'' holds for the subsets in \eqref{EQ: SpectralSetAdditiveCompound}.
	
	To show the inverse inclusion, we use Proposition \ref{EQ: PropositionSpectrumAdditiveUnderENC} with $t > 0$ and Theorem \ref{TH: SpectrumOfMultiplicativeCompound} applied to $L \coloneq G(t)$.  For any $\lambda_{0} \in \operatorname{spec}(A^{[\otimes m]})$, this gives eigenvalues $\lambda_{1}(t),\ldots,\lambda_{m}(t)$ of $A$ and an integer $l(t)$ such that
	\begin{equation}
		\label{EQ: SpectraAddComp2PiL}
		t\lambda_{0} = t\sum_{j=1}^{m} \lambda_{j}(t) + i2\pi l(t).
	\end{equation}
	In particular, $\operatorname{Re}\lambda_{0} = \sum_{j=1}^{m} \operatorname{Re}\lambda_{j}(t)$. From this and since $A$ has a finite number of eigenvalues in each right half-plane, the functions $\lambda_{j}(t)$ may attain only a finite number of values. Consequently, $i2\pi l(t)/t$ must also attain a finite number of values. This implies the existence of $t$ such that $l(t) = 0$, i.e., $\lambda_{0} = \sum_{j=1}^{m} \lambda_{j}(t)$, so \eqref{EQ: SpectralSetAdditiveCompound} is justified.
	
	Now suppose $\lambda_{0} \in \operatorname{spec}(A^{[\otimes m]})$ is fixed and consider its decomposition as in \eqref{EQ: AdditiveCompoundIsolatedPointDesc}. We aim to show \eqref{EQ: SpectralSubspaceIsolatedAdditiveCompound}. This will be done similarly to \cite[Corollary 2.2]{MalletParretNussbaum2013} by constructing the complementary to $\mathbb{L}_{A^{[\otimes m]}}(\lambda_{0})$ spectral subspace. Recall that all the distinct $m$-tuples $(\lambda^{k}_{1},\ldots,\lambda^{k}_{m}) \in \mathbb{C}^{m}$ satisfying \eqref{EQ: AdditiveCompoundIsolatedPointDesc} are enumerated by $k \in \{1,\ldots,N\}$. For each $j \in \{1,\ldots,m\}$, let $q_{j}$ be the number of numerically distinct quantities $\lambda^{k}_{j}$ over $k \in \{1,\ldots,N\}$. We renumber them as $\widetilde{\lambda}^{i}_{j}$ for $i \in \{1,\ldots,q_{j}\}$, so the unordered sets $\{ \lambda^{1}_{j},\ldots,\lambda^{N}_{j} \}$ and $\{ \widetilde{\lambda}^{1}_{j}, \ldots, \widetilde{\lambda}^{q_{j}}_{j}  \}$ coincide. Let $\mathbb{L}_{j}$ be the complementary spectral subspace of $A$ with respect to $\{ \widetilde{\lambda}^{1}_{j}, \ldots, \widetilde{\lambda}^{q_{j}}_{j}  \}$. Then for any $j \in \{1,\ldots,m\}$, we have the direct sum decomposition
	\begin{equation}
		\label{EQ: SpectraAddCompDecompositionSingle}
		\mathbb{H} = \mathbb{L}_{j} \oplus \bigoplus_{i=1}^{q_{j}} \mathbb{L}_{A}(\widetilde{\lambda}^{i}_{j}).
	\end{equation}
	
	Let $\mathcal{I}$ be the set of all $m$-tuples $\mathfrak{i} = (i_{1}, \ldots, i_{m}) \in \mathbb{Z}^{m}$ satisfying $i_{j} \in \{0,\ldots, q_{j}\}$ for each $j \in \{1,\ldots,m\}$. For any $\mathfrak{i} \in \mathcal{I}$, we set
	\begin{equation}
		\label{EQ: SpectraAddCompWspace}
		\mathbb{W}^{\mathfrak{i}} \coloneq \bigotimes_{j=1}^{m} \mathbb{L}^{i_{j}}_{j}, \quad \text{where} \ \mathbb{L}^{i_{j}}_{j} = \begin{cases}
			\mathbb{L}_{j} \qquad &\text{if} \quad i_{j} = 0,\\
			\mathbb{L}_{A}(\widetilde{\lambda}^{i_{j}}_{j}) \qquad &\text{otherwise}.
			\end{cases}
	\end{equation}
	Note that each $\mathbb{W}^{\mathfrak{i}}$ is a subspace of $\mathbb{H}^{\otimes m}$. Furthermore, from \eqref{EQ: SpectraAddCompDecompositionSingle}, \eqref{EQ: SpectraAddCompWspace}, and \eqref{EQ: HilbertTensorProductDirectSumDecomposition}, we obtain that
	\begin{equation}
		\label{EQ: SpectraAddCompDecompositionTensorProd}
		\mathbb{H}^{\otimes m} = \bigoplus_{\mathfrak{i} \in \mathcal{I}} \mathbb{W}^{\mathfrak{i}}.
	\end{equation}
	
	By construction, each subspace $\mathbb{L}^{i}_{j}$ is spectral and, in particular, invariant with respect to $A$. From this and \eqref{EQ: AdditiveCompoundDecTensorFormula}, the subspace $\mathbb{W}^{\mathfrak{i}}$ is invariant with respect to $A^{[\otimes m]}$. This and \eqref{EQ: SpectraAddCompDecompositionTensorProd} show that the algebraic multiplicity of $\lambda_{0}$ as an eigenvalue of $A^{[\otimes m]}$ is equal to the sum over $\mathfrak{i} \in \mathcal{I}$ of the corresponding multiplicities counted for the operator $A^{[\otimes m]}$ restricted to $\mathbb{W}^{\mathfrak{i}}$. For the computation, we only need to consider such $\mathfrak{i} \in \mathcal{I}$ for which $\lambda_{0}$ belongs to the spectrum of the restriction, i.e.,
	\begin{equation}
		\lambda_{0} \in \operatorname{spec}\left(\restr{A^{[\otimes m]}}{\mathbb{W}^{\mathfrak{i}}}\right).
	\end{equation}
	From \eqref{EQ: SpectralSetAdditiveCompound} it is not hard to see\footnote{Here we mean that an analog of \eqref{EQ: SpectralSetAdditiveCompound} can be established for the restriction. For this, we have to define the additive compound $A^{[\otimes m]}_{1 \ldots m}$ of $A_{1}, \ldots, A_{m}$ for possibly distinct generators $A_{1}, \ldots, A_{m}$ of $C_{0}$-semigroups $G_{1}, \ldots, G_{m}$ acting in Hilbert spaces $\mathbb{H}_{1}, \ldots, \mathbb{H}_{m}$, respectively, as the generator of $G_{1} \otimes \cdots \otimes G_{m}$. For eventually compact semigroups, it can be shown by the same argument with appealing to \eqref{EQ: SpectraTensorProduct} that
	\begin{equation}
		\operatorname{spec}(A^{[\otimes m]}_{1 \ldots m}) = \left\{ \sum_{j=1}^{m}\lambda_{j} \mid \lambda_{j} \in \operatorname{spec}(A_{j}) \right\}.
		\end{equation}
	Then the restriction of $A^{[\otimes m]}$ to $\mathbb{W}^{\mathfrak{i}}$ with $\mathfrak{i}=(i_{1},\ldots,i_{m})$ is the additive compound $A^{[\otimes m]}_{1\ldots m}$, where $A_{j}$ is given by the restriction of $A$ to $\mathbb{L}^{i_{j}}_{j}$.} that there are exactly $N$ such $m$-tuples $\mathfrak{i}$, and they correspond to the decompositions from \eqref{EQ: AdditiveCompoundIsolatedPointDesc}. More precisely, for each $k \in \{1,\ldots,N\}$, there exists a unique $\mathfrak{i}^{k}=(i^{k}_{1},\ldots,i^{k}_{m}) \in \mathcal{I}$ such that 
	\begin{equation}
		(\widetilde{\lambda}^{i^{k}_{1}}_{1},\ldots,\widetilde{\lambda}^{i^{k}_{m}}_{m}) = (\lambda^{k}_{1}, \ldots, \lambda^{k}_{m}).
	\end{equation}
	Note that $i^{k}_{j} > 0$ for any $j \in \{1,\ldots,m\}$. Then we have the following:
	\begin{equation}
		\mathbb{L}_{A^{[\otimes m]}}(\lambda_{0}) = \bigoplus_{k=1}^{N} \mathbb{W}^{\mathfrak{i}^{k}} = \bigoplus_{k=1}^{N} \bigotimes_{j=1}^{m} \mathbb{L}_{A}(\widetilde{\lambda}^{i^{k}_{j}}_{j}) = \bigoplus_{k=1}^{N} \bigotimes_{j=1}^{m} \mathbb{L}_{A}(\lambda^{k}_{j}),
	\end{equation}
	which establishes \eqref{EQ: SpectralSubspaceIsolatedAdditiveCompound}.
	
	From \eqref{EQ: SpectralSubspaceIsolatedAdditiveCompound} it immediately follows that $\Pi^{\wedge}_{m} \mathbb{L}_{A^{[\otimes m]}}(\lambda_{0}) = \mathbb{L}_{A^{[\otimes m]}}(\lambda_{0}) \cap \mathbb{H}^{\wedge m}$. Clearly, any $\lambda_{0} \in \operatorname{spec}(A^{[\wedge m]})$ must be an eigenvalue of $A^{[\otimes m]}$, and hence \eqref{EQ: SpectralSubspaceAntisymmetricCompound} is satisfied. Conversely, any $\lambda_{0} \in \operatorname{spec}(A^{[\otimes m]})$ with $\mathbb{L}_{A^{[\otimes m]}}(\lambda_{0}) \cap \mathbb{H}^{\wedge m} = \Pi^{\wedge}_{m} \mathbb{L}_{A^{[\otimes m]}}(\lambda_{0}) \not= 0$ must be an eigenvalue of $A^{[\wedge m]}$.
\end{proof}

We can also describe multiplicities of eigenvalues of $A^{[\wedge m]}$ as follows.

\begin{theorem}
	\label{TH: AddCompoundMultiplicities}
	In the context of Theorem \ref{TH: SpectrumAdditiveCompounds}, consider an eigenvalue $\lambda_{0}$ of $A^{[\wedge m]}$. Then, in terms of \eqref{EQ: SpectralSubspaceIsolatedAdditiveCompound}, for any $k \in \{1,\ldots, N\}$ put
	\begin{equation}
		\label{EQ: AntiSymCompSpectralComputation1}
		\mathbb{L}_{k} \coloneq \bigotimes_{j=1}^{m} \mathbb{L}_{A}(\lambda^{k}_{j}).
	\end{equation}
	Define an equivalence relation on $\{1,\ldots,N\}$ as follows: $k \sim k'$ if and only if there exists a permutation $\sigma \in \mathbb{S}_{m}$ such that $\mathbb{L}_{k'} = S_{\sigma}\mathbb{L}_{k}$ or, equivalently, $\mathbb{L}_{A}(\lambda^{k'}_{j}) = \mathbb{L}_{A}(\lambda^{k}_{\sigma(j)})$ for any $j \in \{1,\ldots,m\}$.
	For some $r>0$, let $\mathcal{K}_{1},\ldots,\mathcal{K}_{r} \subset \{1,\ldots,N\}$ form a complete set of the equivalence classes and consider for any $i \in \{1,\ldots, r\}$ the subspace
	\begin{equation}
		\widetilde{\mathbb{L}}_{i} \coloneq \bigoplus_{k \in \mathcal{K}_{i}} \mathbb{L}_{k}.
	\end{equation}
	Then we have the direct sum decomposition
	\begin{equation}
		\mathbb{L}_{A^{[\wedge m]}}(\lambda_{0}) = \Pi^{\wedge}_{m}\mathbb{L}_{A^{[\otimes m]}}(\lambda_{0}) = \bigoplus_{i=1}^{r}\Pi^{\wedge}_{m}\widetilde{\mathbb{L}}_{i}.
	\end{equation}
	Furthermore, for any $i \in \{1,\ldots,r\}$ there exist $k^{*} \in \mathcal{K}_{i}$ and positive integers $d$ and $\kappa_{1}$,\ldots,$\kappa_{d}$ such that $\mathbb{L}_{k^{*}}$ is of the form
	\begin{equation}
		\mathbb{L}_{k^{*}} = \mathbb{V}^{\otimes \kappa_{1}}_{1} \otimes \cdots \otimes \mathbb{V}^{\otimes \kappa_{d}}_{d},
	\end{equation}
	where the factors $\mathbb{V}_{1}, \ldots, \mathbb{V}_{d}$ form the set of all distinct spectral subspaces of $A$ from \eqref{EQ: AntiSymCompSpectralComputation1} with $k=k^{*}$, so $\kappa_{1}+\cdots+\kappa_{d} = m$. Then, in these terms, we have\footnote{Here the binomial coefficient $\binom{n}{k} = C^{k}_{n}$ is assumed to be zero for $k > n$.}
	\begin{equation}
		\dim( \Pi^{\wedge}_{m} \widetilde{\mathbb{L}}_{i}) = \prod_{l=1}^{d}\binom{\dim\mathbb{V}_{l}}{\kappa_{l}}.
	\end{equation}
\end{theorem}
\begin{proof}
	The proof follows exactly the same lines as \cite[Proposition 2.4]{MalletParretNussbaum2013} and relies only on the identities \eqref{EQ: SpectralSubspaceIsolatedAdditiveCompound} and \eqref{EQ: SpectralSubspaceAntisymmetricCompound}, so we omit it.
\end{proof}

%% file: ActionOfAOnWedges.tex
\section{Description of additive compounds for delay operators}
\label{SEC: ComputationOfAdditiveCompountL2}
In our study of delay equations, we encounter the Hilbert space 
\begin{equation}
	\label{EQ: HilbertSpaceDelayEqDefinition}
	\mathbb{H} = L_{2}([-\tau,0]; \mu; \mathbb{R}^{n}),
\end{equation}
where $\tau>0$ and $\mu =\mu^{1}_{L} + \delta_{0}$ is the sum of the Lebesgue measure $\mu^{1}_{L}$ on $[-\tau,0]$  and the $\delta$-measure $\delta_{0}$ at $0$. Let $\mu^{\otimes m}$ be the $m$-fold product of $\mu$. Then Theorems \ref{TH: DescriptionTensorProductL2} and \ref{TH: ExteriorPowerL2DescriptionAntisymmetric} yield the following description of the abstract $m$-fold tensor product $\mathbb{H}^{\otimes m}$  and the $m$-fold exterior product $\mathbb{H}^{\wedge m}$ of $\mathbb{H}$.
\begin{theorem}
	\label{TH: TensorProductDelayDescription}
	For the space $\mathbb{H}$ from \eqref{EQ: HilbertSpaceDelayEqDefinition}, the mapping\footnote{Recall that $(\phi_{1} \otimes \cdots \otimes \phi_{m})(\theta_{1},\ldots,\theta_{m}) \coloneq \phi_{1}(\theta_{1}) \otimes \cdots \otimes \phi_{m}(\theta_{m})$ for $\mu^{\otimes m}$-almost all $(\theta_{1},\ldots,\theta_{m}) \in [-\tau,0]^{m}$.}
	\begin{equation}
		\label{EQ: DelayCompoundIsomorphismHTensorProd}
		\phi_{1} \otimes \cdots \otimes \phi_{m} \mapsto (\phi_{1} \otimes \cdots \otimes \phi_{m})(\cdot_{1},\ldots,\cdot_{m})
	\end{equation}
    induces a natural isometric isomorphism between $\mathbb{H}^{\otimes m}$ and
    \begin{equation}
    	\label{EQ: L2SpaceTensorCompoundDefinition}
    	\mathcal{L}^{\otimes}_{m} \coloneq L_{2}([-\tau,0]^{m};\mu^{\otimes m};(\mathbb{R}^{n})^{\otimes m}).
    \end{equation}
    Furthermore, its restriction to $\mathbb{H}^{\wedge m}$ gives an isometric isomorphism onto the subspace $\mathcal{L}^{\wedge}_{m}$ of $\mu^{\otimes m}$-antisymmetric functions\footnote{See \eqref{EQ: AntisymmetricFunctionDefinition} or \eqref{EQ: AntisymmetricRelationsPropPreparatory} for the definition.} in $\mathcal{L}^{\otimes}_{m}$.
\end{theorem}

Below we identify the spaces  $\mathbb{H}^{\otimes m}$ (resp. $\mathbb{H}^{\wedge m}$) and $\mathcal{L}^{\otimes}_{m}$ (resp. $\mathcal{L}^{\wedge}_{m}$) according to the isomorphism \eqref{EQ: DelayCompoundIsomorphismHTensorProd} and use the same notations for the operators on $\mathcal{L}^{\otimes}_{m}$ (resp. $\mathcal{L}^{\wedge}_{m}$) induced from $\mathbb{H}^{\otimes m}$ (resp. $\mathbb{H}^{\wedge m}$) by this isomorphism.

It is convenient to introduce some notation for working with the spaces $\mathcal{L}^{\otimes}_{m}$ and $\mathcal{L}^{\wedge}_{m}$. For what follows, $m$ is fixed. For each $k \in \{0,\ldots,m\}$ and any integers $1 \leq j_{1} < \cdots < j_{k} \leq m$, we often consider multi-indices $j_{1}\ldots j_{k}$. For $k=0$, it is assumed that there is a unique multi-index $j_{1}\ldots j_{k} \coloneq 0$, and we also set $\{ j_{1},\ldots, j_{k}\} \coloneq \{0\}$ in this case. Summation over multi-indices $j_{1}\ldots j_{k}$ (with fixed $k$) is always taken over all $1 \leq j_{1} < \cdots < j_{k} \leq m$. We always emphasize the limits of $k$, which may vary depending on the context.

For $j_{1}\ldots j_{k}$ as above, we define the set $\mathcal{B}^{(m)}_{j_{1}\ldots j_{k}}$, called a $k$-\textit{face} of $[-\tau,0]^{m}$ with respect to $\mu^{\otimes m}$, as follows:
\begin{equation}
	\label{EQ: DefinitionOfBoundaryFace}
	\mathcal{B}^{(m)}_{j_{1}\ldots j_{k}} \coloneq \left\{ (\theta_{1},\ldots,\theta_{m}) \in [-\tau,0]^{m} \mid \theta_{j} = 0 \quad \text{for any} \ j \notin \{j_{1},\ldots,j_{k}\} \right\}.
\end{equation}
In particular, $\mathcal{B}^{(m)}_{0} = \{0\}^{m}$ denotes the unique $0$-face with respect to $\mu^{\otimes m}$.

From the definition of $\mu^{\otimes m}$, we obtain the following decomposition: (see also \eqref{EQ: TensorSpaceDelayCompoundDecompositionBoundarySubspaces})
\begin{equation}
\label{EQ: DelayCompoundTensorMeasureDecomposition}
\mu^{\otimes m} = \sum_{k=0}^{m}\sum_{j_{1}\ldots j_{k}} \mu^{(m)}_{j_{1}\ldots j_{k}},
\end{equation}
where $\mu^{(m)}_{j_{1}\ldots j_{k}}$ denotes the $k$-dimensional Lebesgue measure on $\mathcal{B}^{(m)}_{j_{1}\ldots j_{k}}$, which reduces to the $\delta$-measure at $\mathcal{B}^{(m)}_{0}$ for $k=0$. From this, we define the \textit{restriction operator $R^{(m)}_{j_{1}\ldots j_{k}}$} as follows: (here for $k=0$ the $L_{2}$-space is just $(\mathbb{R}^{n})^{\otimes m}$)
\begin{equation}
	\label{EQ: RestrictionOperatorDelayTensor}
	 \mathcal{L}^{\otimes}_{m} \ni \Phi \mapsto R^{(m)}_{j_{1}\ldots j_{k}}\Phi \coloneq \restr{\Phi}{\mathcal{B}^{(m)}_{j_{1}\ldots j_{k}}} \in L_{2}((-\tau,0)^{k};(\mathbb{R}^{n})^{\otimes m}),
\end{equation}
where for the inclusion in $L_{2}$ we naturally identified $\mathcal{B}^{(m)}_{j_{1}\ldots,j_{k}}$ with $[-\tau,0]^{k}$ by omitting the zeroed arguments. Thus, $R^{(m)}_{j_{1}\ldots j_{m}}$ takes a function of $m$ arguments $\theta_{1}, \ldots, \theta_{m}$ to the function of $k$ arguments $\theta_{j_{1}}, \ldots, \theta_{j_{k}}$, putting $\theta_{j}=0$ for $j \notin \{j_{1},\ldots,j_{k}\}$, and then it is considered as an element of the $L_{2}$-space over $(-\tau,0)^{k}$ with the usual Lebesgue measure.

Let $\partial_{j_{1}\ldots j_{k}}\mathcal{L}^{\otimes}_{m}$ denote the subspace of $\mathcal{L}^{\otimes}_{m}$ on which all the restriction operators, except possibly $R^{(m)}_{j_{1}\ldots j_{k}}$, vanish. We call $\partial_{j_{1}\ldots j_{k}}\mathcal{L}^{\otimes}_{m}$ the \textit{boundary subspace over the $k$-face $\mathcal{B}^{(m)}_{j_{1}\ldots j_{k}}$}. Clearly, the space $\mathcal{L}^{\otimes}_{m}$ decomposes into the orthogonal inner sum as follows:
\begin{equation}
	\label{EQ: TensorSpaceDelayCompoundDecompositionBoundarySubspaces}
	\mathcal{L}^{\otimes}_{m} = \bigoplus_{k=0}^{m}\bigoplus_{j_{1}\ldots j_{k}} \partial_{j_{1}\ldots j_{k}}\mathcal{L}^{\otimes}_{m},
\end{equation}
where each boundary subspace $\partial_{j_{1}\ldots j_{k}}\mathcal{L}^{\otimes}_{m}$ is naturally isomorphic to $L_{2}((-\tau,0)^{k};(\mathbb{R}^{n})^{\otimes m})$ through the restriction operator $R^{(m)}_{j_{1}\ldots j_{k}}$.

Thus, any element $\Phi$ of $\mathcal{L}^{\otimes}_{m}$ can be uniquely determined from all of its restrictions $R^{(m)}_{j_{1}\ldots j_{k}}\Phi$ and vice versa. We often omit the upper index in $R^{(m)}_{j_{1}\ldots j_{k}}$ and $\mathcal{B}^{(m)}_{j_{1}\ldots j_{k}}$ if it is clear from the context and write simply $R_{j_{1}\ldots j_{k}}$ or $\mathcal{B}_{j_{1}\ldots j_{k}}$. Moreover, it will be convenient to use the notation $R_{j_{1}\ldots j_{k}}$ for not necessarily monotone sequence $j_{1},\ldots,j_{k}$ to mean the same operator as for the properly rearranged sequence. Sometimes we will use the excluded index notation to denote restriction operators and $k$-faces. In particular, for $j \in \{1, \ldots, m\}$ we will often use $R_{\hat{j}} \coloneq R_{1\ldots \hat{j} \ldots m}$ and $\mathcal{B}_{\hat{j}}\coloneq \mathcal{B}_{1 \ldots \hat{j} \ldots m}$, where the hat on the right-hand sides means that the index is excluded from the considered set.
\begin{remark}
	\begin{figure}
		\centering
		\includegraphics[width=0.6\linewidth]{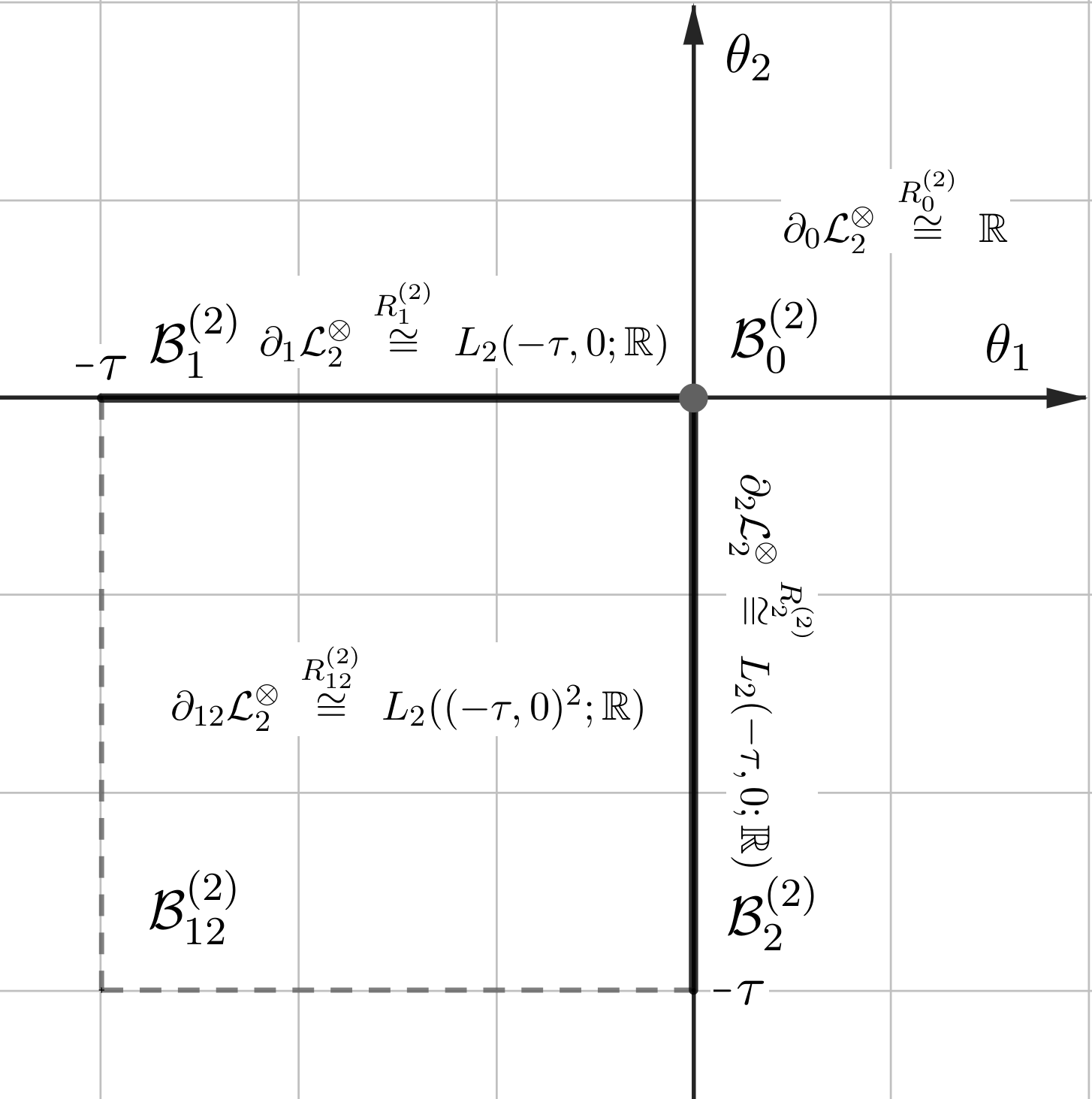}	
		\caption{An illustration to the decomposition of $L_{2}([-\tau,0]^{2};\mu^{\otimes 2};\mathbb{R})$ according to \eqref{EQ: TensorSpaceDelayCompoundDecompositionBoundarySubspaces}, where the restriction operators $R^{(2)}_{0}$, $R^{(2)}_{1}$, $R^{(2)}_{2}$, and $R^{(2)}_{12}$ provide natural isometric isomorphisms between the boundary subspaces over the faces $\mathcal{B}^{(2)}_{0}$, $\mathcal{B}^{(2)}_{1}$, $\mathcal{B}^{(2)}_{2}$, and $\mathcal{B}^{(2)}_{12}$ and appropriate $L_{2}$-spaces respectively.}
		\label{FIG: ExampleFunction}
	\end{figure}
	For $m=2$ and $n=1$, any element $\Phi \in \mathcal{L}^{\otimes}_{2}$ is determined by its four restrictions: $R^{(2)}_{0}\Phi \in \mathbb{R}$, $R^{(2)}_{1}\Phi, R^{(2)}_{2}\Phi \in L_{2}(-\tau,0;\mathbb{R})$, and $R^{(2)}_{12} \Phi \in L_{2}((-\tau,0)^{2};\mathbb{R})$, see Fig. \ref{FIG: ExampleFunction}. Note that even if $R^{(2)}_{12}\Phi$, $R^{(2)}_{1}\Phi$, and $R^{(2)}_{2}\Phi$ have continuous representations, it is not necessary that they are somehow related on intersections of faces. For example, the values $(R^{(2)}_{12}\Phi)(0,0)$, $(R^{(2)}_{1}\Phi)(0)$, $(R^{(2)}_{2}\Phi)(0)$, and $R^{(2)}_{0}\Phi$ need not be related.
\end{remark}

Relations between restrictions arise for the elements of $\mathcal{L}^{\wedge}_{m}$. This is contained in the following proposition. Recall here the operator $\Theta^{(m)}_{\sigma}$ defined in \eqref{EQ: ThetaSigmaDefinition}.
\begin{proposition}
	\label{PROP: AntisymmetricRelationsGeneral}
	An element $\Phi \in \mathcal{L}^{\otimes}_{m}$ belongs to $\mathcal{L}^{\wedge}_{m}$ if and only if for all $k \in \{0,\ldots,m\}$, integers $1 \leq j_{1} < \cdots < j_{m} \leq m$, and $\sigma \in \mathbb{S}_{m}$, it satisfies
	\begin{equation}
		\label{EQ: AntisymmetricRelationsGeneral}
		R_{j_{1}\ldots j_{k}}\Phi = (-1)^{\sigma}T_{\sigma} \Theta^{(k)}_{\bar{\sigma}} R_{\sigma(j_{1}) \ldots \sigma(j_{k})}\Phi,
	\end{equation}
    where $\bar{\sigma} \in \mathbb{S}_{k}$ is such that $\sigma( j_{\bar{\sigma}(1)} ) < \cdots < \sigma( j_{\bar{\sigma}(k)})$.
    
    As a consequence, we have that\footnote{Here in \eqref{EQ: AntisymmetricRestrictionRelationOuterPermutation} the tail of $\sigma$, i.e., $\sigma(l)$ for $l \geq k+1$ is arbitrary.}
	\begin{equation}
		\label{EQ: AntisymmetricRestrictionRelationOuterPermutation}
		R_{1\ldots k}\Phi = (-1)^{\sigma}T_{\sigma} R_{j_{1}\ldots j_{k}}\Phi \quad \text{for any} \quad \sigma = \begin{pmatrix}
			1 & \ldots & k & \ldots \\
			j_{1} & \ldots & j_{k} & \ldots
		\end{pmatrix} \in \mathbb{S}_{m},
	\end{equation}
    and, in particular, for almost all $(\theta_{1},\ldots,\theta_{k}) \in (-\tau,0)^{k}$ we have
    \begin{equation}
    	\label{EQ: AntisymmetricRestrictionRelationValues}
    	(R_{1\ldots k}\Phi)(\theta_{1},\ldots,\theta_{k}) \in (\mathbb{R}^{n})^{\otimes k} \otimes (\mathbb{R}^{n})^{\wedge (m-k)}.
    \end{equation}
\end{proposition}
\begin{proof}
	By Theorem \ref{TH: TensorProductDelayDescription}, $\Phi \in \mathcal{L}^{\otimes}_{m}$ belongs to $\mathcal{L}^{\wedge}_{m}$ if and only if it is $\mu^{\otimes m}$-antisymmetric, i.e., for any $\sigma \in \mathbb{S}_{m}$ we have
	\begin{equation}
		\label{EQ: AntisymmetricRelationsPropPreparatory}
		\Theta^{(m)}_{\sigma}\Phi = (-1)^{\sigma}T_{\sigma} \Phi \qquad \text{in} \quad \mathcal{L}^{\otimes}_{m}.
	\end{equation}
    Applying the restriction operator $R_{j_{1}\ldots j_{k}}$ in \eqref{EQ: AntisymmetricRelationsPropPreparatory}, we obtain \eqref{EQ: AntisymmetricRelationsGeneral}. For this, one should note the key identity
    \begin{equation}
    	\label{EQ: RestrictionAndThetaSigmaIdentity}
    	R_{j_{1}\ldots j_{k}}\Theta^{(m)}_{\sigma^{-1}} = \Theta^{(k)}_{\bar{\sigma}}R_{\sigma(j_{1})\ldots\sigma(j_{k})}.
    \end{equation}
    
    Using the decomposition \eqref{EQ: TensorSpaceDelayCompoundDecompositionBoundarySubspaces}, we get that \eqref{EQ: AntisymmetricRelationsGeneral}, taken over all restrictions, is the same as \eqref{EQ: AntisymmetricRelationsPropPreparatory}. This proves the necessity and sufficiency from the statement.
    
    Note that \eqref{EQ: AntisymmetricRestrictionRelationOuterPermutation} is a particular case of \eqref{EQ: AntisymmetricRelationsGeneral} with $j_{l}=l$ for $l \in \{1,\ldots,k\}$.
	
	To show \eqref{EQ: AntisymmetricRestrictionRelationValues}, we use \eqref{EQ: AntisymmetricRestrictionRelationOuterPermutation} with $j_{l}=l$ for $l \in \{1,\ldots,k\}$ and consider $\sigma$ such that
	\begin{equation}
		\sigma = \begin{pmatrix}
			1 & \ldots & k & k+1 & \ldots & m\\
			1 & \ldots & k & \widetilde{\sigma}(1)+k & \ldots & \widetilde{\sigma}(m-k)+k
		\end{pmatrix},
	\end{equation}
    where $\widetilde{\sigma} \in \mathbb{S}_{m-k}$. Note that $(-1)^{\sigma} = (-1)^{\widetilde{\sigma}}$. Summing \eqref{EQ: AntisymmetricRestrictionRelationOuterPermutation} over all such $\widetilde{\sigma}$ and dividing by $(m-k)!$, we obtain
    \begin{equation}
    	R_{1\ldots k}\Phi = \left(\frac{1}{(m-k)!}\sum_{\widetilde{\sigma} \in \mathbb{S}_{m-k}} (-1)^{\widetilde{\sigma}}T_{\sigma}\right)R_{1\ldots k}\Phi
    \end{equation}
    which shows \eqref{EQ: AntisymmetricRestrictionRelationValues}.
\end{proof}

Since $\mathbb{R}^{\wedge k} = 0$ for $k \geq 2$, from Proposition \ref{PROP: AntisymmetricRelationsGeneral} one may derive the following corollary, which is not technically important for what follows, and hence we left it for the reader as an exercise (or see \cite[Proposition 4.2]{AnikushinRomanov2023FreqConds} for related arguments).
\begin{corollary}
	\label{COR: AntisymmetricRelationsScalar}
	For $n=1$, the relations from \eqref{EQ: AntisymmetricRelationsGeneral} are equivalent to the relations
	\begin{alignat}{2}
		\label{EQ: AntisymmRelationsScalar}
			&R_{j_{1}\ldots j_{k}}\Phi = 0 \qquad &&\text{for all} \ k \in \{0,\ldots, m-2\},\notag\\
			&R_{\hat{j}}\Phi \qquad &&\text{is $\mu^{m-1}_{L}$-antisymmetric for any } j \in \{1, \ldots, m\},\\
			&R_{\hat{i}}\Phi = (-1)^{j-i} R_{\hat{j}}\Phi \qquad &&\text{for} \ i,j \in \{1, \ldots, m\},\notag\\
			&R_{1\ldots m}\Phi \qquad &&\text{is $\mu^{m}_{L}$-antisymmetric}.\notag
	\end{alignat}
\end{corollary}

Note that the antisymmetric relations \eqref{EQ: AntisymmetricRelationsGeneral} link each $\partial_{j_{1} \ldots j_{k}} \mathcal{L}^{\otimes}_{m}$ with other boundary subspaces over $k$-faces. From this, it is convenient to define for a given $k \in \{0,\ldots, m\}$ the subspace (recall $\Pi^{\wedge}_{m}$ from \eqref{EQ: AnntisymmetricProjectorFunctions})
\begin{equation}
	\label{EQ: AntisymmetricSubspaceOverKfaces}
	\partial_{k}\mathcal{L}^{\wedge}_{m} \coloneq \left\{ \Phi \in \bigoplus\limits_{j_{1}\ldots j_{k}}\partial_{j_{1} \ldots j_{k}}\mathcal{L}^{\otimes}_{m} \mid \Phi \text{ satisfies } \eqref{EQ: AntisymmetricRelationsGeneral} \right\} = \Pi^{\wedge}_{m} \bigoplus\limits_{j_{1}\ldots j_{k}}\partial_{j_{1} \ldots j_{k}}\mathcal{L}^{\otimes}_{m}.
\end{equation}
Clearly, $\mathcal{L}^{\wedge}_{m}$ decomposes into the orthogonal sum of all $\partial_{k}\mathcal{L}^{\wedge}_{m}$ as follows:
\begin{equation}
	\label{EQ: DecompositionAntisymmetricSubspaceDelayCompound}
	\mathcal{L}^{\wedge}_{m} = \bigoplus_{k=0}^{m} \partial_{k}\mathcal{L}^{\wedge}_{m}.
\end{equation}

An integer $k \in \{0,\ldots,m\}$ is called \textit{improper} if $\partial_{k}\mathcal{L}^{\wedge}_{m}$ is the zero subspace. Otherwise $k$ is called \textit{proper}. For $n=1$, Corollary \ref{COR: AntisymmetricRelationsScalar} shows that only $k=m-1$ and $k=m$ are proper. For general $n$, \eqref{EQ: AntisymmetricRestrictionRelationValues} immediately yields that any $k < m - n$ is improper, since $(\mathbb{R}^{n})^{\wedge (m-k)} = \{0\}$ in this case. It can be shown that the inverse inequality $k \geq m-n$ implies that $k$ is proper (this is not essential for what follows).

Below, we consider some operators $A$ in $\mathbb{H}$ that arise in the study of delay equations and are therefore called \textit{delay operators}. Their domain $\mathcal{D}(A)$ is given by the embedding\footnote{Recall that the Sobolev space $W^{1,2}(-\tau,0;\mathbb{R}^{n})$ is continuously embedded into $C([-\tau,0];\mathbb{R}^{n})$.} of $W^{1,2}(-\tau,0;\mathbb{R}^{n})$ into $\mathbb{H}$ such that any $\psi \in W^{1,2}(-\tau,0;\mathbb{R}^{n})$ is mapped into $\phi \in \mathbb{H}$ satisfying $R^{(1)}_{0}\phi = \psi(0)$ and $R^{(1)}_{1} \phi = \psi$. For $\phi \in \mathcal{D}(A)$, the action of $A$ is defined as follows:
\begin{equation}
	\label{EQ: OperatorAScalarDelayEquations}
	R^{(1)}_{0}(A\phi) = \widetilde{A}R^{(1)}_{1}\phi \quad \text{and} \quad R^{(1)}_{1}(A\phi) = \frac{d}{d\theta}R^{(1)}_{1}\phi,
\end{equation}
where $\widetilde{A} \colon C([-\tau,0];\mathbb{R}^{n}) \to \mathbb{R}^{n}$ is a bounded linear operator.

It can be shown that $A$ generates a $C_{0}$-semigroup $G$ in $\mathbb{H}$, see \cite{Anikushin2022Semigroups, BatkaiPiazzera2005}. We aim to describe the $m$-fold additive compound $A^{[\otimes m]}$ of $A$ in $\mathbb{H}^{\otimes m}$, discussed in Section \ref{SEC: CocyclesSemigroupsAdditiveCompounds}, in terms of the space $\mathcal{L}^{\otimes}_{m}$.

By the Riesz representation theorem, there exists an $(n\times n)$-matrix function $\alpha(\cdot)$ of bounded variation on $[-\tau,0]$ such that
\begin{equation}
	\label{EQ: DelayOperatorRieszRepresentation}
	\widetilde{A}\phi = \int_{-\tau}^{0}  d\alpha(\theta) \phi(\theta) \qquad \text{for any} \ \phi \in C([-\tau,0];\mathbb{R}^{n}).
\end{equation}
For any integer $j \in \{1,\ldots, m\}$, we set $\mathbb{R}_{1,j} \coloneq (\mathbb{R}^{n})^{\otimes(j-1)}$ and $\mathbb{R}_{2,j} \coloneq (\mathbb{R}^{n})^{\otimes (m-j)}$ and define a linear operator $\alpha_{j}(\theta)$ on $(\mathbb{R}^{n})^{\otimes m}$ by $\alpha_{j}(\theta) \coloneq \operatorname{Id}_{\mathbb{R}_{1,j}} \otimes \alpha(\theta) \otimes \operatorname{Id}_{\mathbb{R}_{2,j}}$. Clearly, $\alpha_{j}$ has bounded variation on $[-\tau,0]$.

Next, for any integers $j \in \{1,\ldots, m\}$, $k \in \{0,\ldots, m-1\}$ and $J \in \{1,\ldots,k+1\}$, we define a linear operator $\widetilde{A}^{(k)}_{j,J}$ taking a function $\Phi$ from $C([-\tau,0]^{k+1};(\mathbb{R}^{n})^{\otimes m})$ to a function from $C([-\tau,0]^{k};(\mathbb{R}^{n})^{\otimes m})$ by
\begin{equation}
	(\widetilde{A}^{(k)}_{j,J} \Phi)(\theta_{1},\ldots,\hat{\theta}_{J},\ldots \theta_{k+1}) \coloneq \int_{-\tau}^{0}d\alpha_{j}(\theta_{J})\Phi(\theta_{1}, \ldots, \theta_{k+1})
\end{equation}
for all $(\theta_{1}, \ldots, \hat{\theta}_{J}, \ldots, \theta_{k+1}) \in [-\tau,0]^{k}$.

Recall that for given integers $j_{1},\ldots,j_{k}$ with $k \in \{0,\ldots, m-1\}$ and $j \notin \{j_{1}, \ldots, j_{k}\}$, the integer $J(j) = J(j;j_{1},\ldots,j_{k})$ is defined such that $j$ is the $J(j)$th element in the set $\{ j, j_{1},\ldots,j_{k}\}$ arranged in increasing order. We usually write $J(j)$ when $j_{1},\ldots,j_{k}$ should be understood from the context.
\begin{theorem}
	\label{TH: AdditiveCompoundDelayDescription}
	Let $A$ be a delay operator as in \eqref{EQ: OperatorAScalarDelayEquations}. Then its $m$-fold additive compound $A^{[\otimes m]}$ acts on $\Phi \in \mathcal{D}(A)^{\odot m}$ as follows:
	\begin{equation}
		\label{EQ: OperatorAmFormula}
				R_{j_{1}\ldots j_{k}}\left( A^{[\otimes m]} \Phi  \right) = 
				\sum_{l=1}^{k}\frac{\partial}{\partial \theta_{j_{l}}}R_{j_{1}\ldots j_{k}}\Phi + \sum_{j \notin \{ j_{1},\ldots,j_{k} \}} \widetilde{A}^{(k)}_{j,J(j)}R_{jj_{1}\ldots j_{k}}\Phi
	\end{equation}
    for all $k \in \{0, \ldots, m\}$ and $1 \leq j_{1} < \cdots < j_{k} \leq m$. Here $R_{j_{1}\ldots j_{k}}\Phi$ is considered as a function of $(\theta_{j_{1}},\ldots,\theta_{j_{k}}) \in (-\tau,0)^{k}$.
\end{theorem}
\begin{proof}
	Due to linearity, it suffices to check \eqref{EQ: OperatorAmFormula} on decomposable tensors $\Phi = \phi_{1} \otimes \cdots \otimes \phi_{m}$ with $\phi_{j} \in \mathcal{D}(A)$ for $j \in \{1, \ldots, m\}$. Here \eqref{EQ: AdditiveCompoundDecTensorFormula} is read as
	\begin{equation}
		A^{[\otimes m]}\Phi = \sum_{j=1}^{m} \phi_{1} \otimes \cdots \otimes A\phi_{j} \otimes \cdots \otimes \phi_{m}.
	\end{equation}
	 
	Using \eqref{EQ: OperatorAScalarDelayEquations} and \eqref{EQ: RestrictionOperatorDelayTensor}, it is straightforward to verify that
	\begin{equation}
		\begin{split}
		\left(R_{j_{1}\ldots j_{k}}(\phi_{1} \otimes \cdots \otimes A\phi_{j} \otimes \cdots \otimes \phi_{m})\right)(\theta_{j_{1}},\ldots,\theta_{j_{k}}) =\\= \begin{cases}
			\frac{\partial}{\partial\theta_{j}}(R_{j_{1}\ldots j_{k}}\Phi)(\theta_{j_{1}},\ldots,\theta_{j_{k}}) \qquad &\text{if} \quad j \in \{ j_{1}, \ldots, j_{k}\},\\
			\left(\widetilde{A}^{(k)}_{j,J(j)}R_{jj_{1}\ldots j_{k}}\Phi\right)(\theta_{j_{1}},\ldots,\theta_{j_{k}}) \qquad &\text{if} \quad j \notin \{j_{1},\ldots, j_{k}\}
		\end{cases}
	    \end{split}
	\end{equation}
    for almost all $(\theta_{j_{1}},\ldots,\theta_{j_{k}}) \in (-\tau,0)^{k}$. Since $R_{j_{1}\ldots j_{k}}$ is linear, this gives \eqref{EQ: OperatorAmFormula}.
\end{proof}

Now we aim to characterize the domain $\mathcal{D}(A^{[\otimes m]})$ of $A^{[\otimes m]}$ and discuss in what sense the action \eqref{EQ: OperatorAmFormula} can be understood for general $\Phi \in \mathcal{D}(A^{[\otimes m]})$, see Remark \ref{REM: ComputationOfDelayCompoundForGeneralFunc}. For this, we recall the diagonal Sobolev space $\mathcal{W}^{2}_{D}(\Omega;\mathbb{F})$ from \eqref{EQ: W2DiagonalDefinition}, which will be used with $\mathbb{F} = (\mathbb{R}^{n})^{\otimes m}$ and $\Omega = (-\tau,0)^{k}$ for $k \in \{1,\ldots, m\}$. According to the characterization given in Proposition \ref{PROP: DiagonalSobolevCubeDescription}, the space $\mathcal{W}^{2}_{D}((-\tau,0)^{k}; \mathbb{F})$ consists of all $L_{2}$-summable $\mathbb{F}$-valued functions $\Psi$ on $(-\tau,0)^{k}$ that have $L_{2}$-summable diagonal derivative $\sum_{l=1}^{k}\frac{\partial }{\partial \theta_{l}}\Psi$ (in the generalized sense) and $L_{2}$-summable traces, denoted by $\operatorname{Tr}_{\mathcal{B}^{(k)}_{\hat{l}}}\Psi$, on the boundaries $\mathcal{B}^{(k)}_{\hat{l}}$, where $l \in \{1,\ldots, k\}$. It is a Hilbert space under the natural norm:
\begin{equation}
	\|\Psi\|^{2}_{\mathcal{W}^{2}_{D}((-\tau,0)^{k};\mathbb{F})} = \|\Psi\|^{2}_{L_{2}} + \left\|\sum_{l=1}^{k}\frac{\partial }{\partial \theta_{l}}\Psi\right\|^{2}_{L_{2}} + \sum_{l=1}^{k}\| \operatorname{Tr}_{\mathcal{B}^{(k)}_{\hat{l}}}\Psi \|^{2}_{L_{2}},
\end{equation}
where $L_{2}$ means appropriate $L_{2}$-spaces.

According to Theorem \ref{TH: OperatorCExntesionOntoWDiagonal} and Proposition \ref{PROP: EmbeddingWDiagToEmDelay}, one can uniquely interpret $\widetilde{A}^{(k)}_{j,J(j)}$ from \eqref{EQ: OperatorAmFormula} as a bounded operator,
\begin{equation}
	\label{EQ: OperatorAkjAsBoundedOperatorInWL2}
	\widetilde{A}^{(k)}_{j,J(j)} \colon \mathcal{W}^{2}_{D}((-\tau,0)^{k+1};(\mathbb{R}^{n})^{\otimes m}) \to L_{2}((-\tau,0)^{k};(\mathbb{R}^{n})^{\otimes m}),
\end{equation}
whose norm is bounded from above by the total variation of $\alpha$ from \eqref{EQ: DelayOperatorRieszRepresentation}. Indeed, $\widetilde{A}^{(k)}_{j,J(j)}$ is the operator $C^{\gamma}_{J}$ from Theorem \ref{TH: OperatorCExntesionOntoWDiagonal} with $\gamma \coloneq \alpha_{j}$ defined below \eqref{EQ: DelayOperatorRieszRepresentation}, $\mathbb{F} = \mathbb{M}_{\gamma} \coloneq (\mathbb{R}^{n})^{\otimes m}$, and $J \coloneq J(j)$. This provides the extension of $\widetilde{A}^{(k)}_{j,J(j)}$ onto the auxiliary space $\mathbb{E}^{2}_{k+1}((\mathbb{R}^{n})^{\otimes m})$ involved into Theorem \ref{TH: OperatorCExntesionOntoWDiagonal}, and the diagonal Sobolev space naturally embeds into it, thanks to Proposition \ref{PROP: EmbeddingWDiagToEmDelay}.

In the following theorem, we show that restrictions of any $\Phi \in \mathcal{D}(A^{[\otimes m]})$ belong to appropriate diagonal Sobolev spaces, and their traces agree with the restrictions of lower orders, see \eqref{EQ: DomainCompoundDelayBoundaryAndTraceIdentity}.
\begin{theorem}
	\label{TH: CompoundDelayDomainDescription}
	For any $\Phi \in \mathcal{D}(A^{[\otimes m]})$, $k \in \{1, \ldots, m\}$, and $1 \leq j_{1} < \cdots < j_{k} \leq m$, the restriction $R_{j_{1}\ldots j_{k}}\Phi$ belongs to $\mathcal{W}^{2}_{D}((-\tau,0)^{k};(\mathbb{R}^{n})^{\otimes m})$ and satisfies the identity:\footnote{Recall that $\bar{\theta}_{\hat{l}}$ is the vector obtained from $\bar{\theta}$ by omitting the $l$th component.}
	\begin{equation}
		\label{EQ: DomainCompoundDelayBoundaryAndTraceIdentity}
		\left(\operatorname{Tr}_{\mathcal{B}^{(k)}_{ \hat{l}}} R_{j_{1} \ldots j_{k}} \Phi\right)(\bar{\theta}) = (R_{j_{1} \ldots \hat{j_{l}} \ldots j_{k} } \Phi)(\bar{\theta}_{\hat{l}})
	\end{equation}
    for $\mu^{k-1}_{L}$-almost all $\bar{\theta} = (\theta_{1},\ldots, \theta_{k}) \in \mathcal{B}^{(k)}_{\hat{l}}$ and each $l \in \{1,\ldots,k\}$.
    
    Furthermore, the norm $\|\cdot\|_{\mathcal{W}^{2}_{D}}$ on $\mathcal{D}(A^{[\otimes m]})$ given by
    \begin{equation}
    	\label{EQ: EquivalentDiagonalSobolevNormInAddCompDelay}
    	\| \Phi \|^{2}_{ \mathcal{W}^{2}_{D}} \coloneq \sum_{k=1}^{m}\sum_{j_{1} \ldots j_{k}}\| R_{j_{1}\ldots j_{k}}\Phi \|^{2}_{\mathcal{W}^{2}_{D}((-\tau,0)^{k};(\mathbb{R}^{n})^{\otimes m} )}
    \end{equation}
    is equivalent to the graph norm.
\end{theorem}
\begin{proof}
	By Theorem \ref{TH: CompoundDescriptionBasicAction}, $\mathcal{D}(A)^{\odot m}$ is dense in $\mathcal{D}(A^{[\otimes m]})$ in the graph norm. Clearly, any $\Phi \in \mathcal{D}(A)^{\odot m}$ satisfies $R_{j_{1}\ldots j_{k}} \Phi \in \mathcal{W}^{2}_{D}((-\tau,0)^{k};(\mathbb{R}^{n})^{\otimes m})$ for all $j_{1}\ldots j_{k}$ as in the statement.
	
	Next, we rewrite \eqref{EQ: OperatorAmFormula} as follows:
	\begin{equation}
		\label{EQ: DelayCompoundDomainDescriptionDerivativeIdentity}
		\sum_{l=1}^{k}\frac{\partial}{\partial \theta_{j_{l}}}R_{j_{1}\ldots j_{k}}\Phi = R_{j_{1}\ldots j_{k}}\left( A^{[\otimes m]} \Phi  \right) - \sum_{j \notin \{ j_{1},\ldots,j_{k} \}} \widetilde{A}^{(k)}_{j,J(j)}R_{jj_{1}\ldots j_{k}}\Phi,
	\end{equation}
	and estimate the $L_{2}$-norm of the diagonal derivative of $R_{j_{1}\ldots j_{k}}\Phi$ using the boundedness of $\widetilde{A}^{(k)}_{j,J(j)}$, as discussed around \eqref{EQ: OperatorAkjAsBoundedOperatorInWL2}. This gives for some constant $C(k)>0$, which depends on $k$, $\tau$ and the total variation $\operatorname{Var}_{[-\tau,0]}(\alpha)$ of $\alpha$ on $[-\tau,0]$, the following estimate:
	\begin{equation}
		\label{EQ: NormsCompoundDelayDomainEquivalencePreliminary}
		\begin{split}
		\| R_{j_{1}\ldots j_{k}} \Phi \|_{\mathcal{W}^{2}_{D}((-\tau,0)^{k};(\mathbb{R}^{n})^{\otimes m})} \leq \\ \leq C(k) \cdot \left( \| \Phi \|_{A^{[\otimes m]}} + \sum_{j \notin \{j_{1},\ldots, j_{k}\} } \| R_{j j_{1}\ldots j_{k}} \Phi \|_{\mathcal{W}^{2}_{D}((-\tau,0)^{k+1};(\mathbb{R}^{n})^{\otimes m})}\right),
		\end{split}
	\end{equation}
    where $\| \cdot \|_{A^{[\otimes m]}}$ is the graph norm.
    
    Clearly, $C(m)$ in \eqref{EQ: NormsCompoundDelayDomainEquivalencePreliminary} can be taken equal to $1$. Thus, we have 
    \begin{equation}
    	\| R_{1\ldots m}\Phi  \|_{\mathcal{W}^{2}_{D}((-\tau,0)^{m};(\mathbb{R}^{n})^{\otimes m})} \leq \| \Phi \|_{A^{[\otimes m]}}.
    \end{equation}
    From this and \eqref{EQ: NormsCompoundDelayDomainEquivalencePreliminary}, acting by induction from $k=m$ to $k=1$, we obtain that the graph norm on $\mathcal{D}(A)^{\odot m}$ is stronger than the norm $\| \cdot \|_{\mathcal{W}^{2}_{D}}$ given by \eqref{EQ: EquivalentDiagonalSobolevNormInAddCompDelay}. Similarly, we can estimate the $L_{2}$-norm of $R_{j_{1}\ldots j_{k}}\left( A^{[\otimes m]} \Phi  \right)$ from \eqref{EQ: DelayCompoundDomainDescriptionDerivativeIdentity}, thereby establishing that $\| \cdot \|_{\mathcal{W}^{2}_{D}}$ is stronger than the graph norm. Thus, both norms are equivalent on $\mathcal{D}(A)^{\odot m}$. Consequently, $\mathcal{D}(A^{[\otimes m]})$ is determined by the completion by either of them.
	
	Clearly, \eqref{EQ: DomainCompoundDelayBoundaryAndTraceIdentity} is satisfied for $\Phi \in \mathcal{D}(A)^{\odot m}$. Since the trace operators on diagonal Sobolev spaces are bounded, \eqref{EQ: DomainCompoundDelayBoundaryAndTraceIdentity} is satisfied for all $\Phi \in \mathcal{D}(A^{[\otimes m]})$, due to the continuity and equivalence of norms.
\end{proof}

\begin{remark}
	In fact, $\mathcal{D}(A^{[\otimes m]})$ is characterized by the property described in Theorem \ref{TH: CompoundDelayDomainDescription}. Namely, if $\Phi \in \mathcal{L}^{\otimes}_{m}$ satisfies $R_{j_{1}\ldots j_{k}} \Phi \in \mathcal{W}^{2}_{D}((-\tau,0)^{k};(\mathbb{R}^{n})^{\otimes m})$ and \eqref{EQ: DomainCompoundDelayBoundaryAndTraceIdentity} for any $j_{1}\ldots j_{k}$ as in the statement, then it belongs to $\mathcal{D}(A^{[\otimes m]})$. This result is not essential for the present and even our adjacent works, and therefore we will limit ourselves to only indicating the idea of the proof. Namely, since $\mathcal{D}(A^{[\otimes m]})$ does not depend\footnote{Because $\mathcal{D}(A)^{\odot m}$ and the equivalent norm \eqref{EQ: EquivalentDiagonalSobolevNormInAddCompDelay} do not depend on $\widetilde{A}$.} on $\widetilde{A}$, it is sufficient to consider the case $\widetilde{A} = 0$. Here, we can explicitly construct the corresponding classical solutions with such initial data $\Phi$ by using Remark \ref{REM: CauchyFormulaCompundInverseStatement}, where the converse to the structural Cauchy formula is discussed. Thus, $\Phi \in \mathcal{D}(A^{[\otimes m]})$ by the definition of the domain.
\end{remark}

\begin{remark}
	\label{REM: ComputationOfDelayCompoundForGeneralFunc}
	Thus, any $\Phi \in \mathcal{D}(A^{[\otimes m]})$ has restrictions belonging to appropriate diagonal Sobolev spaces, and the action of $\widetilde{A}^{(k)}_{j,J(j)}$ on such restrictions can be understood as discussed around \eqref{EQ: OperatorAkjAsBoundedOperatorInWL2}. This gives an interpretation of \eqref{EQ: OperatorAmFormula} for general $\Phi \in \mathcal{D}(A^{[\otimes m]})$.
\end{remark}

We now describe a property of the resolvent of $A^{[\otimes m]}$ that is crucial for studying spectral perturbations. For this, recall the spaces $\mathbb{E}^{2}_{k}(\mathbb{F})$ from Appendix \ref{SEC: DiagonalTranslationSemigroups} for $k \in \{ 1, \ldots, m \}$ and $\mathbb{F} = (\mathbb{R}^{n})^{\otimes m}$, see near \eqref{EQ: NormInSpaceEmDelay}. Below, we also set $\mathbb{E}^{2}_{0}((\mathbb{R}^{n})^{\otimes m}) \coloneq (\mathbb{R}^{n})^{\otimes m}$. We define the Banach space $\mathbb{E}^{\otimes}_{m}$ through the outer direct sum as follows:
\begin{equation}
	\label{EQ: SpaceEmDefinitionDelayCompoundGeneral}
	\mathbb{E}^{\otimes}_{m} \coloneq \bigoplus_{k=0}^{m} \bigoplus_{j_{1}\ldots j_{k}} \mathbb{E}^{2}_{k}((\mathbb{R}^{n})^{\otimes m}),
\end{equation}
and endow it with any of standard norms. By naturally sending each element from the $j_{1}\ldots j_{k}$th summand in \eqref{EQ: SpaceEmDefinitionDelayCompoundGeneral} to $\partial_{j_{1}\ldots j_{k}}\mathcal{L}^{\otimes}_{m}$, we embed the space $\mathbb{E}^{\otimes}_{m}$ into $\mathcal{L}^{\otimes}_{m}$. From Proposition \ref{PROP: EmbeddingWDiagToEmDelay} and Theorem \ref{TH: CompoundDelayDomainDescription} it follows that\footnote{Instead of $\mathbb{E}^{\otimes}_{m}$, we may also consider its subspace, where an analog of  \eqref{EQ: DomainCompoundDelayBoundaryAndTraceIdentity} is satisfied, namely, all restrictions agree with appropriate functions of $j$th sections, see \eqref{EQ: SpaceEmDelayPhiBDefinition}. By Theorem \ref{TH: CompoundDelayDomainDescription}, we are in fact working in this subspace when considering the resolvents below. It is also clear that the embedding of $\mathcal{D}(A^{[\otimes m]})$ into this subspace is not only continuous but also dense.} 
\begin{equation}
	\label{EQ: MyLovelyScaleDelayCompoundGeneral}
	\mathcal{D}(A^{[\otimes m]}) \subset \mathbb{E}^{\otimes}_{m} \subset \mathcal{L}^{\otimes}_{m},
\end{equation}
where all the embeddings are continuous and dense in $\mathcal{L}^{\otimes}_{m}$.

In the following theorem, there is a slight abuse of notation, since we are dealing with resolvents defined on complexifications of the spaces. For convenience, we omit mention of the complexifications, but it should be understood that all the spaces involved are complex, i.e., consist of $(\mathbb{C}^{n})^{\otimes m}$-valued functions.
\begin{theorem}
	\label{TH: ResolventDelayCompoundBound}
	Let $p \in \mathbb{C}$ be a regular point of $A^{[\otimes m]}$, i.e., $p \notin \operatorname{spec}(A^{[\otimes m]})$. Then
	\begin{equation}
		\label{EQ: ResolventEstimateEmDelayCompound}
		\| (A^{[\otimes m]} - p I)^{-1} \|_{\mathcal{L}(\mathcal{L}^{\otimes}_{m};\mathbb{E}^{\otimes}_{m}) } \leq C_{1}(p) \cdot \| (A^{[\otimes m]} - p I)^{-1} \|_{\mathcal{L}(\mathcal{L}^{\otimes}_{m}) } + C_{2}(p),
	\end{equation}
	where the constants $C_{1}(p)$ and $C_{2}(p)$ depend on $\max\{1, e^{-\tau \operatorname{Re}p}\}$, $\tau$, $m$, and $\operatorname{Var}_{[-\tau,0]}(\alpha)$ in a monotonically increasing manner. Furthermore, a similar statement is valid for regular points of $A^{[\wedge m]}$.
\end{theorem}
\begin{proof}
	Suppose $(A^{[\otimes m]}-pI) \Phi = \Psi$ for some $\Psi \in \mathcal{L}^{\otimes}_{m}$ and $\Phi \in \mathcal{D}(A^{[\otimes m]})$. From Theorem \ref{TH: CompoundDelayDomainDescription} we get that $R_{j_{1}\ldots j_{k}}\Phi \in \mathcal{W}^{2}_{D}((-\tau,0)^{k};(\mathbb{C}^{n})^{\otimes m})$ for any $k \in \{1,\ldots, m\}$ and $1 \leq j_{1} < \cdots j_{k} \leq m$. We have to estimate the norm of $R_{j_{1}\ldots j_{k}}\Phi$ in $\mathbb{E}^{2}_{k}((\mathbb{C}^{n})^{\otimes m})$. This will be done by induction from $k=m$ to $k=1$.
	
	For $k=m$, we consider $\bar{\theta} \in \mathcal{B}^{(m)}_{\hat{j}}$ for some $j \in \{1,\ldots,m \}$ and define $D_{\bar{\theta}} \coloneq (-\tau,0)^{m} \cap (\mathcal{L}_{0} + \bar{\theta})$, where $\mathcal{L}_{0} = \{ \underline{t} \in \mathbb{R}^{m} \mid t \in \mathbb{R} \}$ is the diagonal line in $\mathbb{R}^{m}$. Assuming $\bar{\theta} = (\theta_{1},\ldots,\theta_{m})$, we set $\tau(\bar{\theta}) \coloneq \tau+\min_{1 \leq l \leq m}\theta_{l}$. Then the closure of $D_{\bar{\theta}}$ consist of points $\bar{\theta} + \zeta \cdot \underline{1}$, where $\underline{1} = (1,\ldots,1) \in \mathbb{R}^{m}$ and $\zeta \in [-\tau(\bar{\theta}),0]$.
	
	Then for $\mu^{m-1}_{L}$-almost all $\bar{\theta} \in \mathcal{B}^{(m)}_{\hat{j}}$ we have that $\restr{R_{1\ldots m}\Phi}{D_{\bar{\theta}}}$ is well defined and can be considered as an element of $W^{1,2}(-\tau(\bar{\theta}),0;(\mathbb{C}^{n})^{\otimes m})$. In these terms, \eqref{EQ: OperatorAmFormula} gives
	\begin{equation}
		\label{EQ: ResolventDelayCompoundEstimateRestrictedEquationStart}
		\begin{split}
			\frac{d}{d\zeta} \restr{R_{1\ldots m}\Phi}{D_{\bar{\theta}}} - p \restr{R_{1\ldots m}\Phi}{D_{\bar{\theta}}} = \restr{R_{1\ldots m}\Psi}{D_{\bar{\theta}}}.
		\end{split}
	\end{equation}
    Applying the Cauchy formula, for any $\zeta \in (-\tau(\bar{\theta}),0)$ we obtain
    \begin{equation}
    	\restr{R_{1\ldots m}\Phi}{D_{\bar{\theta}}}(\zeta) = e^{p\zeta}\restr{R_{1\ldots m}\Phi}{D_{\bar{\theta}}}(0) - \int_{\zeta}^{0}e^{p(\zeta-s)}\restr{R_{1\ldots m}\Psi}{D_{\bar{\theta}}}(s)ds.
    \end{equation}
	This and the H\"{o}lder inequality yield the estimate
	\begin{equation}
		\begin{split}
			\left| \restr{R_{1\ldots m}\Phi}{D_{\bar{\theta}}}(\zeta) \right| \leq \\ \leq C_{0}(p) \cdot \left( \left| \Phi(\bar{\theta}) \right| + \left\| \restr{R_{1\ldots m}\Psi}{D_{\bar{\theta}}} \right\|_{L_{2}(D_{\bar{\theta}};(\mathbb{C}^{n})^{\otimes m})} \right),
		\end{split}
	\end{equation}
    where $C_{0}(p) = \max\{1, \sqrt{\tau}\} \cdot \max\{1, e^{-\tau\operatorname{Re}p}\}$ and $|\cdot|$ denotes the norm in $(\mathbb{C}^{n})^{\otimes m}$.

	By combining the above estimates for any $j \in \{1, \ldots, m\}$, we get for any $l \in \{1,\ldots,m\}$ and all $\theta \in [-\tau,0]$ in appropriate $L_{2}$-norms the estimate\footnote{Here $e_{l}$ is the $l$th vector from the standard basis in $\mathbb{R}^{m}$.}
	\begin{equation} 
		\label{EQ: ResolventEstimateDelayCompoundPreparatory1}
		\left\| \operatorname{Tr}_{\mathcal{B}^{(m)}_{\hat{l}}+\theta e_{l}}\Phi \right\|_{L_{2}} \leq \widetilde{C}_{0}(p) \cdot\left(\sum_{j=1}^{m}\|R_{\hat{j}}\Phi\|_{L_{2}} + \| R_{1\ldots m}\Psi \|_{L_{2}}\right),
	\end{equation}
    where $\widetilde{C}_{0}(p)$ equals $C_{0}(p)$ multiplied by an absolute constant.
    
	From the Cauchy inequality and since $p$ is a regular point, we have
	\begin{equation}
		\label{EQ: ResolventEstimateDelayCompoundPreparatory2}
		\sum_{j=1}^{m}\|R_{\hat{j}}\Phi\|_{L_{2}} \leq \sqrt{m} \cdot \|\Phi\|_{\mathcal{L}^{\otimes}_{m}} \leq \sqrt{m} \cdot \| (A^{[\otimes m]}-pI)^{-1} \|_{\mathcal{L}(\mathcal{L}^{\otimes}_{m})} \cdot \|\Psi\|_{\mathcal{L}^{\otimes}_{m}}.
	\end{equation}
	Now, combining \eqref{EQ: ResolventEstimateDelayCompoundPreparatory1} and \eqref{EQ: ResolventEstimateDelayCompoundPreparatory2}, we get
	\begin{equation}
		\label{EQ: ResolventEstimateDelayCompoundBaseEstimateEm}
		\begin{split}
			\|R_{1\ldots m}\Phi\|_{\mathbb{E}^{2}_{m}((\mathbb{C}^{n})^{\otimes m})} = \sup_{l \in \{1,\ldots,m\}}\sup_{\theta \in [-\tau,0]}\left\|\operatorname{Tr}_{\mathcal{B}^{(m)}_{\hat{l}}+\theta e_{l}}\Phi \right\|_{L_{2}} \leq \\ \leq (\sqrt{m} \cdot \widetilde{C}_{0}(p) \cdot \| (A^{[\otimes m]}-pI)^{-1} \|_{\mathcal{L}(\mathcal{L}^{\otimes}_{m})} + 1) \cdot \|\Psi\|_{\mathcal{L}^{\otimes}_{m}}.
		\end{split}
	\end{equation}
    This is the required estimate for $k=m$.

    Now consider $k \in \{ 0, \ldots, m-1\}$ and $1 \leq j_{1} < \cdots < j_{k} \leq m$, assuming that the statement has already been proved for larger $k$. Given $j \in \{1,\ldots, k\}$, for any $\bar{\theta} \in \mathcal{B}^{(k)}_{\hat{j}}$ we similarly define $D_{\bar{\theta}} \coloneq (-\tau,0)^{k} \cap (\mathcal{L}_{0} + \bar{\theta})$, where $\mathcal{L}_{0}=\{ \underline{t} \in \mathbb{R}^{k} \mid t \in \mathbb{R} \}$ is the diagonal line in $\mathbb{R}^{k}$. Here an analog of \eqref{EQ: ResolventDelayCompoundEstimateRestrictedEquationStart}, which is also derived from \eqref{EQ: OperatorAmFormula}, is given by
    \begin{equation}
    	\label{EQ: ResolventDelayCompoundEstimateRestrictedEquationGeneral}
    	\begin{split}
    		  \frac{d}{d\zeta} \restr{R_{j_{1}\ldots j_{k}}\Phi}{D_{\bar{\theta}}} - p \restr{R_{j_{1}\ldots j_{k}}\Phi}{D_{\bar{\theta}}} =\\= - \sum_{j \notin \{ j_{1},\ldots,j_{k} \}} \restr{(\widetilde{A}^{(k)}_{j,J(j)}R_{jj_{1}\ldots j_{k}}\Phi)}{D_{\bar{\theta}}} + \restr{R_{j_{1}\ldots j_{k}}\Psi}{D_{\bar{\theta}}}.
    	\end{split}
    \end{equation}
    By applying the Cauchy formula, one obtains an analog of \eqref{EQ: ResolventEstimateDelayCompoundPreparatory1} in appropriate $L_{2}$-spaces for each $l \in \{1,\ldots, k\}$ as follows:
    \begin{equation}
    	\label{EQ: ResolventEstimateDelayCompoundPreparatory1ForK}
    	\begin{split}
    		  \left\| \operatorname{Tr}_{\mathcal{B}^{(k)}_{\hat{l}}+\theta e_{l}}\Phi \right\|_{L_{2}} \leq \widetilde{C}_{0}(p) \cdot\Biggl(\| R_{j_{1} \ldots j_{k}}\Psi \|_{L_{2}} + \sum_{l=1}^{k}\|R_{j_{1}\ldots \hat{j_{l}} \ldots j_{k}}\Phi\|_{L_{2}} + \\ + \sum_{j \notin \{j_{1}\ldots j_{k}\}}\| \widetilde{A}^{(k)}_{j,J(j)}R_{jj_{1}\ldots j_{k}}\Phi \|_{L_{2}} \Biggr),
    	\end{split}
    \end{equation}
    where $e_{l}$ is the $l$th basis vector in the standard basis of $\mathbb{R}^{k}$, and $\widetilde{C}_{0}(p)$ can be taken the same.
    
    Note that we already have an upper estimate for the $L_{2}$-norm of the new (last) term in \eqref{EQ: ResolventEstimateDelayCompoundPreparatory1ForK}, since Theorem \ref{TH: OperatorCExntesionOntoWDiagonal} provides an estimate\footnote{See above \eqref{EQ: DelayCompoundDomainDescriptionDerivativeIdentity} for details.} for the $j$th summand in terms of the norm of $R_{jj_{1}\ldots j_{k}}\Phi$ in  $\mathbb{E}^{2}_{k+1}((\mathbb{C}^{n})^{\otimes m})$, which can be further estimated from the previous step. Furthermore, similarly to \eqref{EQ: ResolventEstimateDelayCompoundBaseEstimateEm}, the resulting estimates are always of the form
    \begin{equation}
    	\begin{split}
    		 \|R_{j_{1} \ldots j_{k}}\Phi\|_{\mathbb{E}_{k}((\mathbb{C}^{n})^{\otimes m})} \leq \\ \leq C^{(k)}_{1}(p) \cdot \| (A^{[\otimes m]}-pI)^{-1} \|_{\mathcal{L}(\mathcal{L}^{\otimes}_{m})} \cdot \|\Psi\|_{\mathcal{L}^{\otimes}_{m}} + C^{(k)}_{2}(p) \cdot \|\Psi\|_{\mathcal{L}^{\otimes}_{m}},
    	\end{split}
    \end{equation}
    where the constants $C^{(k)}_{1}(p)$ and $C^{(k)}_{2}(p)$ are formed from the previous ones by addition and multiplication of $\widetilde{C}_{0}(p)$, $\sqrt{m}$, $\sqrt{\tau}$, $\operatorname{Var}_{[-\tau,0]}(\alpha)$, and some absolute constants. This shows the monotone behavior of $C_{1}(p)$ and $C_{2}(p)$ from the statement. 
    
    It should also be noted that we only used the existence of the resolvent of $A^{[\otimes m]}$, and hence the same estimates hold for $A^{[\wedge m]}$ and its regular points $p$ if we take $\Psi \in \mathcal{L}^{\wedge}_{m}$.
\end{proof}

\begin{remark}
	In contrast to the case $m=1$, the resolvent of $A^{[\wedge m]}$ (and, hence, $A^{[\otimes m]}$) ceases to be compact for $m>1$. In other words, the embedding of $\mathcal{D}(A^{[\wedge m]})$ into $\mathcal{L}^{\wedge}_{m}$ is not compact. Let us illustrate\footnote{In particular, the below arguments show that the embedding of the diagonal Sobolev space over $(-\tau,0)^{m}$ into $L_{2}$ is not compact for $m > 1$, which is more obvious.} this in the case $m=2$ and $n=1$. For any positive integer $k$, consider $\Phi_{k}(\theta_{1},\theta_{2}) \coloneq \sin( \frac{2\pi k}{\tau} (\theta_{1}-\theta_{2}) )$. Note that $\Phi_{k}$ can be viewed as an element $\Psi_{k}$ of $\mathcal{D}(A^{[\wedge 2]})$ with $R_{12}\Psi_{k} = \Phi_{k}$, $R_{1}\Psi_{k} = -R_{2}\Psi_{k} = \Phi_{k}(\cdot,0)$, and $R_{0}\Psi_{k} = 0$. Clearly, we have
	\begin{equation}
		\label{EQ: VanishingSumOfDerivativesCompoudDelay}
		\left(\frac{\partial}{\partial \theta_{1}} + \frac{\partial}{\partial \theta_{2}}\right) \Phi_{k}(\theta_{1},\theta_{2}) \equiv 0 \qquad \text{for} \ (\theta_{1},\theta_{2}) \in (-\tau,0)^{2}.
	\end{equation}
	Moreover, $R_{12} \Psi_{k}$ and $R_{12}\Psi_{l}$ are orthogonal in $L_{2}$ for $k \not= l$. However, the boundary values of $\Phi_{k}$ make the family of $\Psi_{k}$ unbounded in the graph norm. To overcome this, we use a proper truncation of $\Phi_{k}$. Given $\varepsilon>0$, let $c=c(\theta_{1},\theta_{2})$ be a scalar $C^{1}$-function of $(\theta_{1},\theta_{2}) \in [-\tau,0]^{2}$ such that\footnote{Such a function can be defined on segments parallel to the diagonal line by properly scaling the truncation function on $[0,1]$, which equals to $1$ everywhere except a small neighborhood of $1$, where it decays to zero.}
	\begin{enumerate}
		\item[1).] $c(\theta_{1},\theta_{2}) = c(\theta_{2},\theta_{1})$;
		\item[2).] The diagonal derivative $(\frac{\partial}{\partial \theta_{1}} + \frac{\partial}{\partial \theta_{2}})c(\theta_{1},\theta_{2})$ is bounded;
		\item[3).] $c(\theta_{1},0)=c(0,\theta_{2}) = 0$;
		\item[4).] $0 \leq c(\theta_{1},\theta_{2}) \leq 1$ everywhere and $c(\theta_{1},\theta_{2}) \not= 1$ on a set of measure $\leq \varepsilon$. 
	\end{enumerate}
	Then we set $\Phi_{\varepsilon, k} \coloneq c \cdot \Phi_{k}$. From \eqref{EQ: VanishingSumOfDerivativesCompoudDelay} we obtain that
	\begin{equation}
		\label{EQ: VanishingSumOfDerivativesTruncated}
		\left(\frac{\partial}{\partial \theta_{1}} + \frac{\partial}{\partial \theta_{2}}\right) \Phi_{\varepsilon,k}(\theta_{1},\theta_{2}) = \Phi_{k}(\theta_{1},\theta_{2})\left( \frac{\partial}{\partial \theta_{1}} + \frac{\partial}{\partial \theta_{2}} \right)c(\theta_{1},\theta_{2}).
	\end{equation}
	Item 3) ensures that the boundary values of $\Phi_{\varepsilon,k}$ are zero, and hence from \eqref{EQ: VanishingSumOfDerivativesTruncated} and items 1), 2), and 4) we get that the family of all $\Psi_{\varepsilon,k}$ such that $R_{12}\Psi_{\varepsilon,k} = \Phi_{\varepsilon,k}$, $R_{1}\Psi_{\varepsilon,k}=-R_{2}\Psi_{\varepsilon,k} = \Phi_{\varepsilon,k}(\cdot,0)$, and $R_{0}\Psi_{\varepsilon,k} = 0$ belongs to $\mathcal{D}(A^{[\wedge 2]})$ and bounded in the graph norm for a fixed $\varepsilon$. Furthermore, according to the definition of $\Phi_{k}$ and item 4), there exists $\delta>0$ such that the inequality
	\begin{equation}
		\| \Phi_{\varepsilon, k} - \Phi_{\varepsilon, l} \|_{L_{2}((-\tau,0)^{2};\mathbb{R})} \geq \delta \text{ for any } k \not= l.
	\end{equation}
	is satisfied for all sufficiently small $\varepsilon>0$. In particular, one cannot extract a convergent in $L_{2}$ subsequence from $\Phi_{\varepsilon, k}$. This shows that the resolvent of $A^{[\wedge 2]}$ is not compact.
\end{remark}
\begin{remark}
	\label{REM: DelayCompoundBoundActResolvent}
	In the applications discussed in Section \ref{SEC: NonautonomousPerturbationsAdditiveCompounds}, we are interested in transfer operators associated with $A^{[\wedge m]}$, and these are related to the \textit{boundary action} of the resolvent, i.e., the correspondence $\Psi \mapsto \Phi$, where $\Phi = (A^{[\wedge m]} - p I)^{-1}\Psi$ and $\Psi \in \mathcal{L}^{\wedge}_{m}$ satisfies $R_{1\ldots m}\Psi = 0$. From \eqref{EQ: OperatorAmFormula} it is seen that in this case $R_{1\ldots m}\Phi$ is determined from the restrictions of lower order. In particular, for $m=2$, the entire $\Phi$ is determined from $R_{1}\Phi = -R_{2}\Phi$ and $R_{0}\Phi$. Since $R_{1}\Phi$ belongs to the usual Sobolev space $W^{1,2}(-\tau,0;(\mathbb{C}^{n})^{\otimes m})$, which compactly embeds into $C([-\tau,0];(\mathbb{C}^{n})^{\otimes m})$, the boundary action $\Psi \mapsto \Phi \in \mathbb{E}^{\otimes}_{2}$ is a compact operator. Furthermore, at least in the single delay case, i.e., when $\widetilde{A}$ in \eqref{EQ: OperatorAScalarDelayEquations} is given by $\widetilde{A}\phi = A_{0}\phi(0) + A_{-\tau}\phi(-\tau)$ for some $(n\times n)$-matrices $A_{0}$ and $A_{-\tau}$, we know (see \cite[Section 4.5]{AnikushinRomanov2023FreqConds} for $n=1$) that the boundary action is an integral operator with a bounded kernel and expect this to hold in a wider generality. Thus, the case $m=2$ is special for computations.
\end{remark}

We finalize this section by describing the spectra of $A^{[\otimes m]}$ and $A^{[\wedge m]}$. It is well known that the semigroup $G$ generated by $A$ is eventually compact, see \cite{Anikushin2022Semigroups}. Consequently, Theorems \ref{TH: SpectrumAdditiveCompounds} and \ref{TH: AddCompoundMultiplicities} are applicable and yield the following.
\begin{proposition}
	\label{PROP: OperatorDelayCompoundSpectra}
	For any delay operator $A$ given by \eqref{EQ: OperatorAScalarDelayEquations}, all conclusions of Theorems \ref{TH: SpectrumAdditiveCompounds} and \ref{TH: AddCompoundMultiplicities} are valid.
\end{proposition}

%% file: PropertiesOfInhomogeneousProblems.tex
\section{Structural Cauchy formula for linear inhomogeneous problems}
\label{SEC: LinearInhomogeneousDelayCompounds}

Let $A$ be the delay operator given by \eqref{EQ: OperatorAScalarDelayEquations}. Recall that it acts in the Hilbert space $\mathbb{H}$ from \eqref{EQ: HilbertSpaceDelayEqDefinition}. In this section, we consider the $m$-fold additive compound $A^{[\otimes m]}$ of $A$ as an operator in the space $\mathcal{L}^{\otimes}_{m}$ from \eqref{EQ: L2SpaceTensorCompoundDefinition}, as described in Theorems \ref{TH: AdditiveCompoundDelayDescription} and \ref{TH: CompoundDelayDomainDescription}. 

We are going to study structural properties of solutions to the linear inhomogeneous evolutionary system in $\mathcal{L}^{\otimes}_{m}$:
\begin{equation}
	\label{EQ: ControlSystemMCompoundDelay}
	\dot{\Phi}(t) = (A^{[\otimes m]} + \nu I)\Phi(t) + \eta(t),
\end{equation}
where $I$ denotes the identity operator in $\mathcal{L}^{\otimes}_{m}$, $\nu \in \mathbb{R}$ is fixed, and $\eta(\cdot) \in L_{2}(0,T;\mathcal{L}^{\otimes}_{m})$ for some $T>0$. 

Recall the $C_{0}$-semigroup $G^{\otimes m}$ generated by $A^{[\otimes m]}$ and its time-$t$ mappings $G^{\otimes m}(t)$, where $t \geq 0$. Then for any $\Phi_{0} \in \mathcal{L}^{\otimes}_{m}$ there exists a unique mild solution $\Phi(t)=\Phi(t;\Phi_{0},\eta)$ to the Cauchy problem $\Phi(0)=\Phi_{0}$ for \eqref{EQ: ControlSystemMCompoundDelay} on $[0,T]$. It is delivered by the Cauchy formula:
\begin{equation}
	\label{EQ: CauchyInhomogeneousUsualDelaCompoundFormula}
	\Phi(t) = e^{\nu t}G^{\otimes m}(t) \Phi_{0} + \int_{0}^{t}e^{\nu(t-s)}G^{\otimes m}(t-s)\eta(s)ds.
\end{equation}
For brevity, we will say that the pair $(\Phi(\cdot),\eta(\cdot))$ \textit{solves} \eqref{EQ: ControlSystemMCompoundDelay} on $[0,T]$.

\begin{remark}
	Clearly, for any pair $(\Phi(t),\eta(t))=(\Phi_{\nu}(t),\eta_{\nu}(t))$ which solves \eqref{EQ: ControlSystemMCompoundDelay} on $[0,T]$, the pair $(e^{-\nu t}\Phi_{\nu}(t),e^{-\nu t}\eta_{\nu}(t))$ solves \eqref{EQ: ControlSystemMCompoundDelay} with $\nu = 0$ on $[0,T]$. 
\end{remark}

For what follows, we set $\mathbb{F} \coloneq (\mathbb{R}^{n})^{\otimes m}$ and let $L_{2,k}$ stand for $L_{2}((-\tau,0)^{k};\mathbb{F})$. For $\rho(t) = \rho_{\nu}(t) \coloneq e^{\nu t}$, we recall the spaces $\mathcal{Y}^{2}_{\rho}(0,T;L_{2,k})$ and $\mathcal{T}^{2}_{\rho}(0,T;L_{2,k})$ of special $L_{2,k}$-valued functions on $[0,T]$, defined above \eqref{EQ: NormInWindowsSpaces} and \eqref{EQ: NormOnTwistedFunctionsDef}, respectively. For brevity, during this section these spaces are denoted by $\mathcal{Y}^{2}_{\nu, k}(0,T)$ and $\mathcal{T}^{2}_{\nu, k}(0,T)$, respectively, omitting their dependence on $n$ and $m$.

Recall that the space $\mathcal{Y}^{2}_{\nu, k}(0,T)$ consists of all functions $\Phi_{X,\rho_{\nu}}(t) = e^{\nu t}X_{t}$, where $X \in L_{2}(\mathcal{C}^{k}_{T};\mathbb{F})$ with
\begin{equation}
	\label{EQ: TheSetDiagonalDomainDefinitionRecall}
	\mathcal{C}^{k}_{T} = \bigcup_{t \in [0,T]} \left( [-\tau,0]^{k} + \underline{t} \right) \subset \mathbb{R}^{k},
\end{equation}
and $X_{t}(\overline{\theta}) \coloneq X(\overline{\theta} + \underline{t})$ for almost all $\overline{\theta} \in [-\tau,0]^{k}$. For brevity, the norm in $\mathcal{Y}^{2}_{\nu, k}(0,T)$ is denoted by $\|\cdot\|_{\mathcal{Y}^{2}_{\nu, k}}$, so the other parameters should be understood from the context. By applying the Fubini theorem in \eqref{EQ: NormInWindowsSpaces}, it can be written in the following form:
\begin{equation}
	\label{EQ: NormAdornedFubini}
	\| \Phi_{X,\rho_{\nu}} \|^{2}_{\mathcal{Y}^{2}_{\nu, k}} = \int_{(-\tau,0)^{k}}\left|X(\bar{\theta})\right|^{2}_{\mathbb{F}}d\bar{\theta} + \int_{0}^{T} e^{\nu t} \left(\sum_{j=1}^{m}\int_{\mathcal{B}^{(k)}_{\hat{j}}} \left|X(\bar{\theta}+\underline{t})\right|^{2}_{\mathbb{F}}d\mu^{k-1}_{L}(\bar{\theta})\right) dt,
\end{equation}
which is more convenient for this section.

It its turn, the space $\mathcal{T}^{2}_{\nu, k}(0,T)$ consists of all functions $\Psi_{Y,\rho_{\nu}}(t) = e^{\nu t} \int_{0}^{t}T_{k}(t-s)Y(s)ds$, where $Y(\cdot) \in L_{2}(0,T;L_{2,k})$ and $T_{k}$ is the diagonal translation semigroup in $L_{2,k}$ defined in \eqref{EQ: SemigroupDiagonalTranslateFiniteCubeDef}. Similarly to the above, we denote the norm in this space by $\| \Psi_{Y,\rho_{\nu}} \|_{\mathcal{T}^{2}_{\nu, k}}$. As in \eqref{EQ: NormOnTwistedFunctionsDef}, the norm is given by
\begin{equation}
	\label{EQ: NormOnTwistedFunctionsDefRecall}
	\| \Psi_{Y,\rho_{\nu}} \|^{2}_{\mathcal{T}^{2}_{\nu, k}} = \int_{0}^{T} e^{\nu t}\| Y(t) \|^{2}_{L_{2,k}} dt.
\end{equation}

Any $\Phi_{X,\rho_{\nu}}$ as above is called the $\rho_{\nu}$-adornment of $X$, and $\Psi_{Y,\rho_{\nu}}$ is called the $\rho_{\nu}$-twisting of $Y$. We always use only such notations for elements of these spaces, possibly adding indices to $X$ and $Y$, as in \eqref{EQ: ProperFunctionSolutionAsWeightedFunction}.

Both spaces $\mathcal{Y}^{2}_{\nu, k}(0,T)$ and $\mathcal{T}^{2}_{\nu, k}(0,T)$ are naturally embedded into $L_{2}(0,T;L_{2,k})$ and linearly independent, thanks to Proposition \ref{PROP: UniquenessOfSumAdornedTwisted}. Taking their direct sum determines the space $\mathcal{A}^{2}_{\nu, k}(0,T)$ of $\rho_{\nu}$-agalmanated functions. It is endowed with the norm $\| \cdot \|_{\mathcal{A}^{2}_{\nu, k}}$ defined as follows:
\begin{equation}
	\label{EQ: AgalmanatedNormRecall}
	\| (\Phi_{X,\rho_{\nu}}, \Psi_{Y,\rho_{\nu}}) \|^{2}_{\mathcal{A}^{2}_{\nu, k}} = \| \Phi_{X,\rho_{\nu}} \|^{2}_{\mathcal{Y}^{2}_{\nu, k}} + \| \Psi_{Y,\rho_{\nu}} \|^{2}_{\mathcal{T}^{2}_{\nu, k}}.
\end{equation}

Recall also the diagonal Sobolev space $\mathcal{W}_{D}(\Omega;\mathbb{F})$, which will be used with $\Omega = (-\tau,0)^{k}$, as in the previous section, and $\Omega = \mathring{\mathcal{C}}^{k}_{T}$, where $\mathring{\mathcal{C}}^{k}_{T}$ denotes the interior of $\mathcal{C}^{k}_{T}$. By Theorem \ref{TH: DiagonalTranslatesSquareDelay}, the generator $A_{T_{k}}$ of $T_{k}$ is defined on the subspace $\mathcal{W}_{D_{0}}((-\tau,0)^{k};\mathbb{F})$ of $\mathcal{W}_{D}((-\tau,0)^{k};\mathbb{F})$ consisting of functions vanishing on each of the $(k-1)$-faces $\mathcal{B}^{(k)}_{\hat{j}}$, where $j \in \{1,\ldots,k\}$.

Now we are ready to formulate the main result of this section, which is the cornerstone of the entire paper. This is the decomposition \eqref{EQ: ProperFunctionSolutionAsWeightedFunction} of solutions to the linear inhomogeneous problem \eqref{EQ: ControlSystemMCompoundDelay}, which we call a \textit{structural Cauchy formula}. Here, the interior and boundary parts of the solution are uniquely decomposed into the sum of $\rho_{\nu}$-adorned and $\rho_{\nu}$-twisted functions. This decomposition differs from \eqref{EQ: CauchyInhomogeneousUsualDelaCompoundFormula} that can be seen from the fact that $\Phi_{X_{j_{1}\ldots j_{k}},\rho_{\nu}}$ in \eqref{EQ: ProperFunctionSolutionAsWeightedFunction} depends on the entire solution $\Phi$ (and hence $\eta$) in general, see \eqref{EQ: TheoremWeightedCompoundDelaySolutionsFunctionXForK} for an explicit construction. 

However, each formula \eqref{EQ: ProperFunctionSolutionAsWeightedFunction}, when properly read, is the usual Cauchy formula for a linear inhomogeneous problem associated with the generator $A_{T_{k}}$ of $T_{k}$, see \eqref{EQ: StructuralCauchyFormulaLastLIProblemForK}. In the proof, we justify the applicability of this formula for regular initial data and then extend it by continuity to arbitrary data.

\begin{theorem}[Structural Cauchy formula]
	\label{TH: StructuralCauchyFormulaCompoundDelay}
	Suppose that $\nu \in \mathbb{R}$, $T>0$, $\Phi_{0} \in \mathcal{L}^{\otimes}_{m}$, and $\eta_{\nu}(\cdot) \in L_{2}(0,T;\mathcal{L}^{\otimes}_{m})$ are given. Let $\Phi_{\nu}(\cdot)$ be the mild solution on $[0,T]$ to \eqref{EQ: ControlSystemMCompoundDelay} with $\eta = \eta_{\nu}$ such that $\Phi_{\nu}(0)=\Phi_{0}$. Then for any $k \in \{1,\ldots,m\}$ and $1 \leq j_{1} < \cdots < j_{k} \leq m$, there exist functions $X_{j_{1}\ldots j_{k}} \in L_{2}(\mathcal{C}^{k}_{T};(\mathbb{R}^{n})^{\otimes m})$ and $Y_{j_{1}\ldots j_{k}} \in L_{2}(0,T;L_{2}((-\tau,0)^{k};(\mathbb{R}^{n})^{\otimes m}))$ such that $R_{j_{1}\ldots j_{k}}\Phi_{\nu}$ is the sum of the $\rho_{\nu}$-adornment of $X_{j_{1}\ldots j_{k}}$ and the $\rho_{\nu}$-twisting of $Y_{j_{1}\ldots j_{k}}$ for $\rho_{\nu}(t) \coloneq e^{\nu t}$, i.e.,
	\begin{equation}
		\label{EQ: ProperFunctionSolutionAsWeightedFunction}
		R_{j_{1} \ldots j_{k}}\Phi_{\nu}(t) = \Phi_{X_{j_{1}\ldots j_{k}},\rho_{\nu}}(t) + \Psi_{Y_{j_{1}\ldots j_{k}}, \rho_{\nu}}(t) \qquad \text{for all} \ t \in [0,T].
	\end{equation}
    In particular, $R_{j_{1} \ldots j_{k}}\Phi_{\nu}$ belongs to the space $\mathcal{A}^{2}_{\nu, k}(0,T)$ of $\rho_{\nu}$-agalmanated functions, see above \eqref{EQ: AgalmanatedNormRecall}. Furthermore, 
    \begin{equation}
    	\label{EQ: StructuralCauchyFormulaDelayCompoundYformula}
    	\begin{split}
    		\rho_{\nu}(t)Y_{j_{1}\ldots j_{k}}(t) =\\= R_{j_{1} \ldots j_{k}}\eta_{\nu}(t) + \sum_{j \notin \{j_{1}, \ldots, j_{k}\}} \widetilde{A}^{(k)}_{j,J(j)} R_{j j_{1}\ldots j_{k}} \Phi_{\nu}(t) \qquad \text{for almost all} \ t \in [0,T],
    	\end{split}
    \end{equation}
    where $\widetilde{A}^{(k)}_{j,J(j)}$ as in \eqref{EQ: OperatorAmFormula}, and its pointwise action is understood in the integral sense according to Theorem \ref{TH: PointwiseMeasurementOperatorOnAgalmanatedSpace}. 
    
    In addition, for $\Phi_{0} \in \mathcal{D}(A^{[\otimes m]})$ and $\eta_{\nu}(\cdot) \in C^{1}([0,T];\mathcal{L}^{\otimes}_{m})$, we have that
    \begin{equation}
    	\label{EQ: SolutionsDelayCompoundSmoothingSpaces}
    	\begin{split}
    		X_{j_{1}\ldots j_{k}} &\in \mathcal{W}^{2}_{D}(\mathring{\mathcal{C}}^{k}_{T};(\mathbb{R}^{n})^{\otimes m}), \\ \Phi_{X_{j_{1},\ldots,j_{k}},\rho_{\nu}}(\cdot) &\in C^{1}([0,T];L_{2, k}) \cap C([0,T];\mathcal{W}^{2}_{D, k}), \\
    		\Psi_{Y_{j_{1}\ldots j_{k}},\rho_{\nu}}(\cdot) &\in C^{1}([0,T];L_{2, k}) \cap C([0,T];\mathcal{W}^{2}_{D_{0}, k}) ,
    	\end{split}
    \end{equation}
    where $L_{2, k}$, $\mathcal{W}^{2}_{D, k}$, and $\mathcal{W}^{2}_{D_{0}, k}$ stand for $L_{2}((-\tau,0)^{k};(\mathbb{R}^{n})^{\otimes m})$, $\mathcal{W}^{2}_{D}((-\tau,0)^{k};(\mathbb{R}^{n})^{\otimes m})$, and $\mathcal{W}^{2}_{D_{0}}((-\tau,0)^{k};(\mathbb{R}^{n})^{\otimes m})$, respectively.
\end{theorem}

Before proving the theorem, we establish that $\Phi_{X_{j_{1} \ldots j_{k}},\rho_{\nu}}$ and $\Psi_{Y_{j_{1}\ldots j_{k}},\rho_{\nu}}$ in \eqref{EQ: ProperFunctionSolutionAsWeightedFunction} must depend continuously on the data $(\Phi_{0},\eta_{\nu})$ from $\mathcal{L}^{\otimes}_{m} \times L_{2}(0,T; \mathcal{L}^{\otimes}_{m})$. This will be done by deriving the estimate \eqref{EQ: StructuralCauchyFormulaNormEstimatesAdornedTwisted} in terms of the data and the solution $\Phi_{\nu}$ itself. This estimate also applies to the case $T = \infty$, which will be considered in Section \ref{SUBSEC: DelayCompPropertiesOfComplexificatedProblem}. Moreover, the proof introduces an explicit construction of $X_{j_{1}\ldots j_{k}}$ in \eqref{EQ: StructuralCauchyFormulaNormEstimateXmDescription} and \eqref{EQ: StructuralCauchyFormulaNormEstimateXmDescriptionGENERALK}, which will be used to prove Theorem \ref{TH: StructuralCauchyFormulaCompoundDelay}.

\begin{theorem}
	\label{TH: DelayCompoundStructuralCauchyFormulaNormEstimate}
	In terms of Theorem \ref{TH: StructuralCauchyFormulaCompoundDelay}, assume that the decompositions \eqref{EQ: ProperFunctionSolutionAsWeightedFunction} and \eqref{EQ: StructuralCauchyFormulaDelayCompoundYformula} and the property from \eqref{EQ: SolutionsDelayCompoundSmoothingSpaces} are valid for all $k \in \{1,\ldots,m\}$ and $1 \leq j_{1} < \cdots < j_{k} \leq m$. Then the norms  $\| \Phi_{X_{j_{1}\ldots j_{k}},\rho_{\nu}} \|_{\mathcal{Y}^{2}_{\nu, k}}$ and $\|\Psi_{Y_{j_{1}\ldots j_{k}},\rho_{\nu}} \|_{\mathcal{T}^{2}_{\nu, k}}$ defined in \eqref{EQ: NormAdornedFubini} and \eqref{EQ: NormOnTwistedFunctionsDefRecall}, respectively, admit the following estimate:
	\begin{equation}
		\label{EQ: StructuralCauchyFormulaNormEstimatesAdornedTwisted}
		\begin{split}
			\left\| \Phi_{X_{j_{1}\ldots j_{k}},\rho_{\nu}} \right\|^{2}_{\mathcal{Y}^{2}_{\nu, k}} +
			\left\| \Psi_{Y_{j_{1}\ldots j_{k}},\rho_{\nu}} \right\|^{2}_{ \mathcal{T}^{2}_{\nu, k}} \leq \\ 
			\leq C_{k} \cdot \left( | \Phi_{\nu}(0) |^{2}_{\mathcal{L}^{\otimes}_{m}} + \int_{0}^{T} | \Phi_{\nu}(t) |^{2}_{\mathcal{L}^{\otimes}_{m}}dt + \int_{0}^{T}|\eta_{\nu}(t)|^{2}_{\mathcal{L}^{\otimes}_{m}}dt \right),
		\end{split}
	\end{equation}
	where the constant $C_{k}>0$ depends on $\max\{1,e^{\nu \tau}\}$, $\tau$, and the total variation $\operatorname{Var}_{[-\tau,0]}(\alpha)$ of $\alpha$ from \eqref{EQ: DelayOperatorRieszRepresentation} in a monotonically increasing manner and does not depend on $T$.
\end{theorem}
\begin{proof}
	We give a proof by induction from $k=m$ to $k=1$. 
	
	For $k=m$, \eqref{EQ: StructuralCauchyFormulaDelayCompoundYformula} reads as $\rho_{\nu}(t)Y_{1\ldots m}(t) = R_{1\ldots m}\eta_{\nu}(t)$. Consequently,
	\begin{equation}
		\label{EQ: StructuralCauchyFormulaEstimateTwistedKm}
		\begin{split}
			\left\| \Psi_{Y_{1\ldots m},\rho_{\nu}} \right\|^{2}_{ \mathcal{T}^{2}_{\nu, m}} =& \int_{0}^{T}\| \rho_{\nu}(t) Y_{1\ldots m}(t) \|^{2}_{L_{2, m}}dt =\\=& \int_{0}^{T} \| R_{1\ldots m}\eta_{\nu}(t) \|^{2}_{L_{2, m}}dt \leq \int_{0}^{T} | \eta_{\nu}(t) |^{2}_{\mathcal{L}^{\otimes}_{m}}dt.
		\end{split}
	\end{equation}
    
    Now consider $\Phi_{0} \in \mathcal{D}(A^{[\otimes m]})$ and $\eta_{\nu}(\cdot) \in C^{1}([0,T];\mathcal{L}^{\otimes}_{m})$. According to \cite[Chapter I, Theorem 6.5]{Krein1971}, for such data the solution $\Phi_{\nu}(\cdot)$ is classical and, in particular, belongs to $C([0,T];\mathcal{D}(A^{[\otimes m]}))$. Thanks to \eqref{EQ: SolutionsDelayCompoundSmoothingSpaces}, for any $j \in \{ 1,\ldots, m\}$ we may apply the trace operator $\operatorname{Tr}_{\mathcal{B}^{(m)}_{\hat{j}}}$ given by Theorem \ref{TH: TraceOperatorForDiagonalTranslatesFinite} to both sides of \eqref{EQ: ProperFunctionSolutionAsWeightedFunction}. From this, according to Theorem \ref{TH: CompoundDelayDomainDescription} and the definition of $\Phi_{X_{1\ldots m},\rho_{\nu}}$, we obtain that
    \begin{equation}
    	\label{EQ: StructuralCauchyFormulaNormEstimateXmDescription}
    	\begin{split}
    		 \rho_{\nu}(t)X_{1\ldots m}(\bar{\theta} + \underline{t}) = (\operatorname{Tr}_{\mathcal{B}^{(m)}_{\hat{j}}} \Phi_{X_{1\ldots m},\rho_{\nu}}(t))(\bar{\theta}) =\\= (\operatorname{Tr}_{\mathcal{B}^{(m)}_{\hat{j}}}R_{1 \ldots m}\Phi(t))(\bar{\theta}) = R_{\hat{j}}\Phi_{\nu}(t)(\bar{\theta}_{\hat{j}})
    	\end{split}
    \end{equation}
    for almost all $t \in [0,T]$ and $\mu^{m-1}_{L}$-almost all $\bar{\theta} \in \mathcal{B}^{(m)}_{\hat{j}}$. From this we get that
    \begin{equation}
    	\label{EQ: StructuralCauchyFormulaEstimateAdornedKm}
    	\begin{split}
    		&\| \Phi_{X_{1\ldots m},\rho_{\nu}} \|^{2}_{\mathcal{Y}^{2}_{\nu, m}} =\\&= \|R_{1\ldots m}\Phi_{0}\|^{2}_{L_{2, m}} + \int_{0}^{T}e^{2 \nu t} \left(\sum_{j=1}^{m} \int_{\mathcal{B}_{\hat{j}}} |X_{1\ldots m}(\bar{\theta}+\underline{t})|^{2} d\mu^{m-1}_{L}(\bar{\theta}) \right) dt =\\&= \|R_{1\ldots m}\Phi_{0}\|^{2}_{L_{2, m}} + \int_{0}^{T} \sum_{j=1}^{m}\|R_{\hat{j}}\Phi_{\nu}(t)\|^{2}_{L_{2, m-1}}dt \leq \\ &\leq  | \Phi_{0} |^{2}_{\mathcal{L}^{\otimes}_{m}} + \int_{0}^{T} | \Phi_{\nu}(t) |^{2}_{\mathcal{L}^{\otimes}_{m}}dt.
    	\end{split}
    \end{equation}
    
    By combining \eqref{EQ: StructuralCauchyFormulaEstimateTwistedKm} and \eqref{EQ: StructuralCauchyFormulaEstimateAdornedKm}, we obtain \eqref{EQ: StructuralCauchyFormulaNormEstimatesAdornedTwisted} with $k=m$ and $C_{m} = 1$ for solutions with regular data: $\Phi_{0} \in \mathcal{D}(A^{[\otimes m]})$ and $\eta_{\nu}(\cdot) \in C^{1}([0,T];\mathcal{L}^{\otimes}_{m})$. For arbitrary data, the estimate is established by continuity.
    
    Now consider $k \in \{1,\ldots, m-1\}$ and assume that \eqref{EQ: StructuralCauchyFormulaNormEstimatesAdornedTwisted} has already been proven for all larger $k$. From \eqref{EQ: ProperFunctionSolutionAsWeightedFunction}, we know that $R_{j j_{1} \ldots j_{k}} \Phi_{\nu}$ is a $\rho_{\nu}$-agalmanated function for $j \notin \{j_{1}, \ldots, j_{k}\}$. From this, we may apply Theorem \ref{TH: PointwiseMeasurementOperatorOnAgalmanatedSpace} to each operator $\widetilde{A}^{(k)}_{j,J(j)}$ in \eqref{EQ: StructuralCauchyFormulaDelayCompoundYformula} to estimate the terms therein as follows:
    \begin{equation}
    	\label{EQ: StructuralCauchyFormulaInductionYestimate}
    	\begin{split}
    		\left\| \Psi_{Y_{j_{1}\ldots j_{k}},\rho_{\nu}} \right\|_{ \mathcal{T}^{2}_{\nu, k}} = \left(\int_{0}^{T}\| \rho_{\nu}(t) Y_{j_{1}\ldots j_{k}}(t) \|^{2}_{L_{2, k}}dt\right)^{1/2} \leq \\ \leq \left(\int_{0}^{T} \| R_{j_{1}\ldots j_{k}}\eta_{\nu}(t) \|^{2}_{L_{2, k}}dt\right)^{1/2} + \widetilde{C} \cdot \sum_{j \notin \{j_{1},\ldots j_{k}\}} \| R_{j j_{1} \ldots j_{k}} \Phi_{\nu} \|_{\mathcal{A}^{2}_{\nu, k+1}},
    	\end{split}
    \end{equation}
    where $\widetilde{C}>0$ depends only on $\operatorname{Var}_{[-\tau,0]}(\alpha)$, $\tau$, and $\max\{1, e^{\nu \tau} \}$ (the last value is $\rho_{0}$ in terms of Theorem \ref{TH: PointwiseMeasurementOperatorOnAgalmanatedSpace}).
    
    For regular data, similarly to \eqref{EQ: StructuralCauchyFormulaNormEstimateXmDescription}, for any $l \in \{1,\ldots, k\}$ we obtain that
    \begin{equation}
    	\label{EQ: StructuralCauchyFormulaNormEstimateXmDescriptionGENERALK}
    	\rho_{\nu}(t)X_{j_{1}\ldots j_{k}}(\bar{\theta}+\underline{t}) = R_{j_{1}\ldots \hat{j_{l}} \ldots j_{k}}\Phi_{\nu}(t)(\bar{\theta}_{\hat{l}})
    \end{equation}
    for almost all $t \in [0,T]$ and $\mu^{k-1}_{L}$-almost all $\bar{\theta} \in \mathcal{B}^{(k)}_{\hat{l}}$. Then, similarly to \eqref{EQ: StructuralCauchyFormulaEstimateAdornedKm}, we derive the following estimate:
    \begin{equation}
    	\label{EQ: StructuralCauchyFormulaInductionXestimate}
    	\| \Phi_{X_{j_{1}\ldots j_{k}},\rho_{\nu}} \|^{2}_{\mathcal{Y}^{2}_{\nu, k}} \leq |\Phi_{0}|^{2}_{\mathcal{L}^{\otimes}_{m}} + \int_{0}^{T}|\Phi_{\nu}(t)|^{2}_{\mathcal{L}^{\otimes}_{m}}dt.
    \end{equation}

    Note that the norm $\| R_{j j_{1} \ldots j_{k}} \Phi_{\nu}(\cdot) \|_{\mathcal{A}^{2}_{\nu, k+1}}$ in \eqref{EQ: StructuralCauchyFormulaInductionYestimate} can be estimated from the previous step, i.e., \eqref{EQ: StructuralCauchyFormulaNormEstimatesAdornedTwisted} with $k$ replaced by $k+1$. Combining this with \eqref{EQ: StructuralCauchyFormulaInductionXestimate} results in validity of \eqref{EQ: StructuralCauchyFormulaNormEstimatesAdornedTwisted} for the given $k$.
\end{proof}

\begin{proof}[Proof of Theorem \ref{TH: StructuralCauchyFormulaCompoundDelay}]
	Set $\Phi(t) \coloneq e^{-\nu t} \Phi_{\nu}(t)$ and $\eta(t) = e^{-\nu t}\eta_{\nu}(t)$. Then $(\Phi, \eta)$ solve \eqref{EQ: ControlSystemMCompoundDelay} with $\nu = 0$ on $[0,T]$. Thus, it suffices to show the statement for $\nu = 0$. Moreover, we can also assume that the initial data are regular: $\eta(\cdot) \in C^{1}([0,T];\mathcal{L}^{\otimes}_{m})$ and $\Phi(0) = \Phi_{0} \in \mathcal{D}(A^{[\otimes m]})$. In particular, the solution $\Phi(\cdot)$ is classical and hence $\Phi(\cdot) \in C^{1}([0,T];\mathcal{L}^{\otimes}_{m}) \cap C([0,T];\mathcal{D}(A^{[\otimes m]}))$. For arbitrary data, the continuity argument can be used together with the already established estimate \eqref{EQ: StructuralCauchyFormulaNormEstimatesAdornedTwisted}.
	
	First, we give a proof for $k=m$. Define $X_{1\ldots m} \in L_{2}(\mathcal{C}^{m}_{T};(\mathbb{R}^{n})^{\otimes m})$ for almost all $\bar{s} \in \mathcal{C}^{m}_{T}$ as follows:
	\begin{equation}
		\label{EQ: TheoremWeightedCompoundDelaySolutionsFunctionX}
		X_{1\ldots m}(\bar{s}) \coloneq \begin{cases}
			(R_{1\ldots m}\Phi_{0})(\bar{s}) \qquad &\text{if} \quad \bar{s} \in (-\tau,0)^{m},\\
			\left(R_{\hat{j}}\Phi(t)\right)(\bar{s}_{\hat{j}} - \underline{t}) \qquad &\text{if} \quad (\bar{s} - \underline{t}) \in  \mathcal{B}^{(m)}_{\hat{j}},
		\end{cases}
	\end{equation}  
    where the second condition is taken over all $j \in \{1,\ldots,m\}$ and $t \in [0,T]$.
    
    Let $\mathring{\mathcal{C}}^{m}_{T}$ be the interior of $\mathcal{C}^{m}_{T}$. We aim to show that $X_{1\ldots m} \in \mathcal{W}^{2}_{D}(\mathring{\mathcal{C}}^{m}_{T};(\mathbb{R}^{n})^{\otimes m})$. For this, define for each $j \in \{1,\ldots,m\}$ the sets
    \begin{equation}
	\mathcal{C}_{j} \coloneq \bigcup_{t \in [0,T]}( \mathcal{B}^{(m)}_{\hat{j}} + \underline{t}),
    \end{equation}
    and let $\mathring{\mathcal{C}}_{j}$ be the interior of $\mathcal{C}_{j}$.
    Since $\Phi(\cdot) \in C^{1}([0,T];\mathcal{L}^{\otimes}_{m})$, we have that the mapping $[0,T] \ni t \mapsto R_{\hat{j}}\Phi(t) \in L_{2, m}$
    is $C^{1}$-differentiable. From this and \eqref{EQ: TheoremWeightedCompoundDelaySolutionsFunctionX}, it is not hard to see that $\restr{X_{1\ldots m}}{\mathring{\mathcal{C}}_{j}}$ must belong to $\mathcal{W}^{2}_{D}(\mathring{\mathcal{C}}_{j};(\mathbb{R}^{n})^{\otimes m})$, and its diagonal derivative is given by
    \begin{equation}
    	\label{EQ: StructuralCauchyFormulaDerivativeXJ}
    	(D^{j}X_{1 \ldots m})(\bar{\theta} + \underline{t}) \coloneq \left(\frac{d}{dt}R_{\hat{j}}\Phi(t)\right)(\bar{\theta}_{\hat{j}})
    \end{equation}
    for $\mu^{m-1}_{L}$-almost all $\bar{\theta} \in \mathcal{B}^{(m)}_{\hat{j}}$ and all $t \in [0,T]$. Indeed, by the Newton-Liebniz formula, for any $0 \leq a \leq b \leq T$ we have that
    \begin{equation}
    	R_{\hat{j}}\Phi(b) - R_{\hat{j}}\Phi(a) = \int_{a}^{b}\frac{d}{ds}R_{\hat{j}}\Phi(s)ds.
    \end{equation}
    Substituting $\bar{\theta}_{\hat{j}}$ with $\bar{\theta} \in \mathcal{B}^{(m)}_{\hat{j}}$ into the functions from the formula above, we obtain that
    \begin{equation}
    	R_{\hat{j}}\Phi(b)(\bar{\theta}_{\hat{j}}) - R_{\hat{j}}\Phi(a)(\bar{\theta}_{\hat{j}}) = \int_{a}^{b}\left(\frac{d}{ds}R_{\hat{j}}\Phi(s)\right)(\bar{\theta}_{\hat{j}})ds,
    \end{equation}
    which makes sense for $\mu^{m-1}_{L}$-almost all $\bar{\theta} \in \mathcal{B}^{(m)}_{\hat{j}}$ and according to \eqref{EQ: TheoremWeightedCompoundDelaySolutionsFunctionX} and \eqref{EQ: StructuralCauchyFormulaDerivativeXJ} gives
    \begin{equation}
    	X_{1\ldots m}(\bar{\theta} + \underline{b}) - X_{1\ldots m}(\bar{\theta} + \underline{a}) = \int_{a}^{b}(D^{j}X_{1\ldots m})(\bar{\theta}+\underline{s})ds.
    \end{equation}
   This implies that $\restr{X_{1\ldots m}}{\mathring{\mathcal{C}}_{j}} \in \mathcal{W}^{2}_{D}(\mathring{\mathcal{C}}_{j};(\mathbb{R}^{n})^{\otimes m})$ according to the characterization from Lemma \ref{LEM: ExtensionOperatorFromSobolevDiagonalToRm}.
   
   Next, note that $\restr{X_{1\ldots m}}{(-\tau,0)^{m}} = R_{1\ldots m}\Phi_{0}$ lies in $\mathcal{W}^{2}_{D, m}$, thanks to Theorem \ref{TH: CompoundDelayDomainDescription} and the inclusion $\Phi_{0} \in \mathcal{D}(A^{[\otimes m]})$. Since
   \begin{equation}
   	\mathcal{C}^{m}_{T} = \bigcup_{j \in \{1,\ldots, m\}} \mathcal{C}_{j} \cup [-\tau,0]^{m},
   \end{equation}
   and the trace of $\restr{X_{1\ldots m}}{(-\tau,0)^{m}}$ on $\mathcal{B}^{(m)}_{\hat{j}}$ as an element of $\mathcal{W}^{2}_{D, m}$ agrees with the trace of $\restr{X_{1\ldots m}}{\mathring{\mathcal{C}}_{j}}$ on $\mathcal{B}^{(m)}_{\hat{j}}$ as an element of $\mathcal{W}^{2}_{D}(\mathring{\mathcal{C}}_{j};(\mathbb{R}^{n})^{\otimes m})$, we get that $X_{1\ldots m}$ belongs to $\mathcal{W}^{2}_{D}(\mathring{\mathcal{C}}^{m}_{T};(\mathbb{R}^{n})^{\otimes m})$. This shows the first inclusion in \eqref{EQ: SolutionsDelayCompoundSmoothingSpaces} with $k=m$.
        
   By Lemma \ref{LEM: ExtensionOperatorFromSobolevDiagonalToRm}, there exists an element $\hat{X}_{1\ldots m}$ from $\mathcal{W}^{2}_{D}(\mathbb{R}^{m}; (\mathbb{R}^{n})^{\otimes m} )$ that extends $X_{1\ldots m}$. By Theorem \ref{TH: DiagonalTranslationInRm}, $\hat{X}_{1\ldots m}$ belongs to the domain $\mathcal{D}(A_{\mathcal{T}_{m}})$ of the generator $A_{\mathcal{T}_{m}}$ of the diagonal translation group $\mathcal{T}_{m}(t)$ in $L_{2}(\mathbb{R}^{m};(\mathbb{R}^{n})^{\otimes m})$. Consequently, the function $[0,T] \ni t \mapsto \mathcal{T}_{m}(t)\hat{X}_{1\ldots m}$ is a classical solution to the Cauchy problem associated with $A_{\mathcal{T}_{m}}$. Thus, considering $\hat{X}_{1\ldots m}$ as a function of $(s_{1},\ldots,s_{m}) \in \mathbb{R}^{m}$, we obtain that
    \begin{equation}
    	\label{EQ: WeightedSolutionsCompoundTheorem1}
    	\frac{d}{dt}(\mathcal{T}_{m}(t)\hat{X}_{1\ldots m}) = \left(\sum_{j=1}^{m} \frac{\partial}{\partial s_{j}}\right)\mathcal{T}_{m}(t)\hat{X}_{1\ldots m} \qquad \text{for all} \ t \in [0,T].
    \end{equation}
    In particular, this gives the second inclusion in \eqref{EQ: SolutionsDelayCompoundSmoothingSpaces} with $k=m$.

    Let $\mathcal{R} \colon L_{2}(\mathbb{R}^{m};(\mathbb{R}^{n})^{\otimes m}) \to L_{2}((-\tau,0)^{m};(\mathbb{R}^{n})^{\otimes m})$ be the operator that restricts functions from $\mathbb{R}^{m}$ to $(-\tau,0)^{m}$. Then we have that the function (here $\rho_{0}$ is $\rho_{\nu}$ for $\nu=0$)
    \begin{equation}
    	\label{EQ: WeigthedWindowFunctionCompoundContDiff}
    	[0,T] \ni t \mapsto \Phi_{X_{1\ldots m},\rho_{0}}(t) = \mathcal{R}\mathcal{T}_{m}(t)\hat{X}_{1\ldots m} \in L_{2, m}
    \end{equation}
    is $C^{1}$-differentiable, and it is continuous as a $\mathcal{W}^{2}_{D,m}$-valued function. Moreover, applying $\mathcal{R}$ to both sides of \eqref{EQ: WeightedSolutionsCompoundTheorem1}, we get for any $t \in [0,T]$ that
    \begin{equation}
    	\label{EQ: WeightedSolutionsCompoundTheorem2}
    	\begin{split}
    		 \frac{d}{dt}\Phi_{X_{1\ldots m},\rho_{0}}(t) &= \frac{d}{dt}( \mathcal{R}\mathcal{T}(t)\hat{X}_{1\ldots m} ) =\\= \mathcal{R}A_{\mathcal{T}_{m}}\mathcal{T}_{m}(t)\hat{X}_{1\ldots m} &= \left(\sum_{j=1}^{m}\frac{\partial}{\partial \theta_{j}}\right)\Phi_{X_{1\ldots m},\rho_{0}}(t),
    	\end{split}
    \end{equation}
    where $\Phi_{X_{1\ldots m},\rho_{0}}(t)$ is a function of $(\theta_{1},\ldots,\theta_{m}) \in (-\tau,0)^{m}$.
    
   	From \eqref{EQ: WeightedSolutionsCompoundTheorem2} and \eqref{EQ: OperatorAmFormula}, it follows that $\Delta(t) \coloneq R_{1\ldots m}\Phi(t) - \Phi_{X_{1\ldots m},\rho_{0}}(t)$ satisfies
    \begin{equation}
    	\frac{d}{dt}\Delta(t) = \left(\sum_{j=1}^{m} \frac{\partial}{\partial \theta_{j}}\right) \Delta(t) + R_{1\ldots m}\eta(t) \qquad \text{for all} \ t \in [0,T].
    \end{equation}
    According to Theorem \ref{TH: CompoundDelayDomainDescription}, we have $\operatorname{Tr}_{\mathcal{B}_{\hat{j}}}R_{1\ldots m}\Phi(t)(\bar{\theta}) = R_{\hat{j}}\Phi(t)(\bar{\theta}_{\hat{j}})$ for $\mu^{m-1}_{L}$-almost all $\bar{\theta} \in \mathcal{B}_{\hat{j}}$. Moreover, from \eqref{EQ: TheoremWeightedCompoundDelaySolutionsFunctionX} and Theorem \ref{TH: TraceOperatorForDiagonalTranslatesFinite}, we obtain that $\operatorname{Tr}_{\mathcal{B}_{\hat{j}}}\Phi_{X_{1\ldots m},\rho_{0}}(t)(\bar{\theta}) = R_{\hat{j}}\Phi(t)(\bar{\theta}_{\hat{j}})$ for $\mu^{m-1}_{L}$-almost all $\bar{\theta} \in \mathcal{B}_{\hat{j}}$. Thus, $\operatorname{Tr}_{\mathcal{B}_{\hat{j}}}\Delta(t) = 0$ for all $t \in [0,T]$.
    
    Now recall the diagonal translation semigroup $T_{m}$ in $L_{2,m}$ and its generator $A_{T_{m}}$, see above or Theorem \ref{TH: DiagonalTranslatesSquareDelay}. We have just shown $\Delta(\cdot)$ is a classical solution on $[0,T]$ to the inhomogeneous Cauchy problem associated with $A_{T_{m}}$. Since $\Delta(0) = 0$, the Cauchy formula for this problem gives that
    \begin{equation}
    	R_{1\ldots m}\Phi(t) - \Phi_{X_{1\ldots m},\rho_{0}}(t) = \int_{0}^{t}T_{m}(t-s)R_{1\ldots m}\eta(s)ds \eqcolon \Psi_{Y_{1\ldots m},\rho_{0}}(t)
    \end{equation}
    for all $t \in [0,T]$. This shows \eqref{EQ: ProperFunctionSolutionAsWeightedFunction}, \eqref{EQ: StructuralCauchyFormulaDelayCompoundYformula}, and the third inclusion in \eqref{EQ: SolutionsDelayCompoundSmoothingSpaces} with $k=m$.
    
    Now we suppose that $k \in \{1, \ldots , m-1\}$, and the statement has already been proved for all larger $k$. Analogously to \eqref{EQ: TheoremWeightedCompoundDelaySolutionsFunctionX}, we define $X_{j_{1} \ldots j_{k}} \in L_{2}(\mathcal{C}^{k}_{T};(\mathbb{R}^{n})^{\otimes m})$ for almost all $\bar{s} \in \mathcal{C}^{k}_{T}$ as follows:
    \begin{equation}
    	\label{EQ: TheoremWeightedCompoundDelaySolutionsFunctionXForK}
    	X_{j_{1}\ldots j_{k}}(\bar{s}) \coloneq \begin{cases}
    		(R_{j_{1}\ldots j_{k}}\Phi_{0})(\bar{s}) \qquad &\text{if} \quad \bar{s} \in (-\tau,0)^{k},\\
    		\left(R_{j_{1} \ldots \hat{j}_{l} \ldots j_{k}}\Phi(t)\right)(\bar{s}_{\hat{l}} - \underline{t}) \qquad &\text{if} \quad (\bar{s} - \underline{t}) \in  \mathcal{B}^{(k)}_{\hat{l}},
    	\end{cases}
    \end{equation}  
    where the second condition is taken over all $l \in \{1,\ldots,k\}$ and $t \in [0,T]$.
    
    Similarly, one can show that $X_{j_{1}\ldots j_{k}}$ belongs to $\mathcal{W}^{2}_{D}(\mathring{\mathcal{C}}^{k}_{T};(\mathbb{R}^{n})^{\otimes m})$ and further obtain that the difference $\Delta(t) = R_{j_{1}\ldots j_{k}}\Phi(t) - \Phi_{X_{j_{1}\ldots j_{k}},\rho_{0}}(t)$ is a classical solution to the inhomogeneous Cauchy problem for $A_{T_{k}}$ such that
    \begin{equation}
    	\label{EQ: StructuralCauchyFormulaLastLIProblemForK}
    	\frac{d}{dt}\Delta(t) = A_{T_{k}}\Delta(t) + R_{j_{1}\ldots j_{k}}\eta(t) + \sum_{j \notin \{j_{1}\ldots j_{k}\}}\widetilde{A}^{(k)}_{j,J(j)}R_{jj_{1}\ldots j_{k}}\Phi(t)
    \end{equation}
    and $\Delta(0) = 0$. Here the last term is a continuous function of $t$ by Proposition \ref{PROP: EmbeddingWDiagToEmDelay} and since $R_{jj_{1}\ldots j_{k}}\Phi(\cdot)$ belongs to $C([0,T]; \mathcal{W}^{2}_{D}((-\tau,0)^{k+1};(\mathbb{R}^{n})^{\otimes m}))$. Then \eqref{EQ: ProperFunctionSolutionAsWeightedFunction} is the Cauchy formula for \eqref{EQ: StructuralCauchyFormulaLastLIProblemForK}.
\end{proof}

\begin{remark}
	\label{REM: CauchyFormulaCompundInverseStatement}
	In the context of Theorem \ref{TH: StructuralCauchyFormulaCompoundDelay}, for $\Phi_{0} \in \mathcal{D}(A^{[\otimes m]})$, according to Theorem \ref{TH: CompoundDelayDomainDescription}, we have that\footnote{Here $\mathcal{I}_{\delta^{l}_{0}}$ is as in \eqref{EQ: NormEmbracingSpaceCube}, i.e., it acts by putting $0$ to the $l$th argument.}
	\begin{equation}
		\operatorname{Tr}_{\mathcal{B}^{(k)}_{\hat{l}}}R_{j_{1}\ldots j_{k}}\Phi_{\nu}(t)(\bar{\theta}) = (\mathcal{I}_{\delta^{l}_{0}}R_{j_{1}\ldots j_{k}}\Phi_{\nu})(t)(\bar{\theta}_{\hat{l}}) = R_{j_{1}\ldots \hat{j}_{l}\ldots j_{k}}\Phi_{\nu}(t)(\bar{\theta}_{\hat{l}})
	\end{equation}
	for all $t \in [0,T]$, $k \in \{1,\ldots,m\}$, $l \in \{1,\ldots,k\}$, and $\mu^{k-1}_{L}$-almost all $\bar{\theta} \in \mathcal{B}^{(k)}_{\hat{l}}$. By continuity, for general $\Phi_{0} \in \mathcal{L}^{\otimes}_{m}$, the second identity takes the form
	\begin{equation}
		\label{EQ: DelayCompoundAgreementRestrGeneral}
		\mathcal{I}_{\delta^{l}_{0}}R_{j_{1}\ldots j_{k}}\Phi_{\nu} = R_{j_{1}\ldots \hat{j}_{l}\ldots j_{k}}\Phi_{\nu} \quad \text{in} \quad L_{2}(0,T;L_{2}((-\tau,0)^{k};(\mathbb{R}^{n})^{\otimes m})).
	\end{equation}
	In this sense, the agreement of restrictions and ``traces'' for mild solutions can be understood. It can be shown the converse statement, i.e., if a continuous $\mathcal{L}^{\otimes}_{m}$-valued function $\Phi_{\nu}(\cdot)$ on $[0,T]$ has all the restrictions $R_{j_{1}\ldots j_{k}}\Phi_{\nu}(\cdot)$ satisfying \eqref{EQ: ProperFunctionSolutionAsWeightedFunction}, \eqref{EQ: StructuralCauchyFormulaDelayCompoundYformula}, and \eqref{EQ: DelayCompoundAgreementRestrGeneral} for any $k \in \{1,\ldots,m\}$ and $l \in \{1,\ldots,k\}$, and for $k=0$ the restriction $R_{0}\Phi_{\nu}$ satisfies
	\begin{equation}
		\label{EQ: GeneralizedSolutionCharacterizationEqZeroK}
		R_{0}\Phi_{\nu}(t) = R_{0}\Phi_{\nu}(0) + \int_{0}^{t}\left(\sum_{j=1}^{m}\widetilde{A}^{(1)}_{1,1}R_{j}\Phi_{\nu}(s) + R_{0}\eta_{\nu}(s) \right)ds \quad \text{for} \ t \in [0,T],
	\end{equation}
	then $\Phi_{\nu}(\cdot)$ is a mild solution to \eqref{EQ: ControlSystemMCompoundDelay} on $[0,T]$. Indeed, taking the difference of such $\Phi_{\nu}$ with the true mild solution, it is not hard to see that it suffices to consider the case $\nu = 0$, $\eta_{\nu} = 0$, and $\Phi_{\nu}(0) = 0$ and establish $\Phi_{\nu}(\cdot) \equiv 0$. In this case, from \eqref{EQ: ProperFunctionSolutionAsWeightedFunction} with $k=1$ and \eqref{EQ: DelayCompoundAgreementRestrGeneral} with $k=1$, we obtain that the right-hand side of \eqref{EQ: GeneralizedSolutionCharacterizationEqZeroK} is a bounded linear operator acting on $R_{0}\Phi_{\nu}(\cdot)$ in an appropriate $L_{2}$-space, and it is a contraction provided that the space is considered over the time interval $[0,\varepsilon]$ with sufficiently small $\varepsilon$. Thus, $R_{0}\Phi_{\nu}(t) = 0$ for $t \in [0,\varepsilon]$ and, using  \eqref{EQ: ProperFunctionSolutionAsWeightedFunction} and \eqref{EQ: DelayCompoundAgreementRestrGeneral}, by induction from $k=1$ to $k=m$, we obtain $\Phi_{\nu}(t) = 0$ for $t \in [0,\varepsilon]$. Then the same argument can be applied in the interval $[\varepsilon,2\varepsilon]$ and so on.
\end{remark}

%% file: NonautonomousPerturbationofAddComp.tex
\section{Nonautonomous perturbations of additive compounds for delay equations}
\label{SEC: NonautonomousPerturbationsAdditiveCompounds}
In this section, we study cocycles $\Xi$ generated by delay equations in $\mathbb{R}^{n}$ over a semiflow $(\mathcal{P},\pi)$ on a complete metric space $\mathcal{P}$. In Section \ref{SUBSEC: DelayCompoundInfinitesimalDescription}, after introducing the class of equations in \eqref{EQ: DelayRnLinearized} and their evolutionary form in \eqref{EQ: DelayLinearCocAbsract} with a distinguished delay operator $A$, we aim to derive similar equations generating the $m$-fold compound cocycle $\Xi_{m}$ in $\mathcal{L}^{\wedge}_{m}$, see \eqref{EQ: DelayCompoundCocyclePertWedgeOperatorForm}. To do this, it is necessary to introduce certain operators in terms of the space $\mathcal{L}^{\wedge}_{m}$. 

In \eqref{EQ: DelayCompoundCocyclePertWedgeOperatorForm}, the generator of $\Xi_{m}$ is described as a nonautonomous boundary perturbation of $A^{[\wedge m]}$. In Section \ref{SUBSEC: AssociatedLIPquadraticConstr}, we formulate linear inhomogeneous problems with quadratic constraints arising from some perturbations. In Section \ref{SUBSEC: DelayCompPropertiesOfComplexificatedProblem}, we consider associated infinite-horizon quadratic regulator problems. In Section \ref{SUBSEC: DelayCompoundFrequencyInequalities}, we introduce frequency conditions ensuring that certain dichotomy properties of the semigroup $G^{\wedge m}$ generated by $A^{[\wedge m]}$ are preserved for $\Xi_{m}$, as contained in Theorem \ref{TH: QuadraticFunctionalDelayCompoundTheorem}.

\subsection{Infinitesimal description of compound cocycles}
\label{SUBSEC: DelayCompoundInfinitesimalDescription}
Let $\Eta \coloneq \mathbb{R}^{r_{1}}$ and $\mathbb{M} \coloneq \mathbb{R}^{r_{2}}$, where $r_{1},r_{2} > 0$, be endowed with some inner products. Consider the class of nonautonomous delay equations in $\mathbb{R}^{n}$ over $(\mathcal{P},\pi)$ described by
\begin{equation}
	\label{EQ: DelayRnLinearized}
	\dot{x}(t) = \widetilde{A}x_{t} + \widetilde{B}F'(\pi^{t}(\wp))Cx_{t},
\end{equation}
where $\wp \in \mathcal{P}$; $\tau>0$ is a constant; $x(\cdot) \colon [-\tau,T] \to \mathbb{R}^{n}$ for some $T>0$ with $x_{t}(\theta) \coloneq x(t+\theta)$ for all $t \in [0,T]$ and $\theta \in [-\tau,0]$ denoting the $\tau$-history segment of $x(\cdot)$ at $t$; $\widetilde{A} \colon C([-\tau,0];\mathbb{R}^{n}) \to \mathbb{R}^{n}$ and $C \colon C([-\tau,0];\mathbb{R}^{n}) \to \mathbb{M}$ are bounded linear operators; $\widetilde{B} \colon \Eta \to \mathbb{R}^{n}$ is an $(n \times r_{1})$-times matrix, and $F' \colon \mathcal{P} \to \mathcal{L}(\mathbb{M};\Eta)$ is a continuous mapping\footnote{In fact, it suffices to consider $F'(\pi^{t}(\wp))$ as the mapping $\mathcal{P} \ni \wp \mapsto F'(\pi^{\cdot}\wp) \in L_{2}(0,T;\mathcal{L}(\mathbb{M};\Eta))$, which is continuous for any $T>0$. In other words, $F'(\cdot)$ suffices to be defined on the trajectories of $\pi$, rather than at the points from $\mathcal{P}$. Such a relaxation allows one to consider linearized equations over semiflows $\pi$ generated by delay equations in Hilbert spaces. In our case, the class of equations \eqref{EQ: DelayRnNonlinear}, which generate $\pi$ in applications, is smoothing in finite time, so any invariant set $\mathcal{P}$ lies in the space of continuous functions, where $F'(\cdot)$ is defined pointwise.} such that for some $\Lambda>0$ we have
\begin{equation}
	\label{EQ: LipschitzFprimeDelay}
	\|F'(\wp)\|_{\mathcal{L}(\mathbb{M};\Eta)} \leq \Lambda \qquad \text{for all} \ \wp \in \mathcal{P}.
\end{equation}
\begin{remark}
	\label{REM: DerivativeCocycleExample}
	Equations as \eqref{EQ: DelayRnLinearized} arise after the linearization of nonlinear nonautonomous delay equations over a semiflow $(\mathcal{Q},\vartheta)$ on a complete metric space $\mathcal{Q}$, which can be described as follows:
	\begin{equation}
		\label{EQ: DelayRnNonlinear}
		\dot{z} = \widetilde{A}z_{t} + \widetilde{B}F(\vartheta^{t}(q),Cz_{t}) + \widetilde{W}(\vartheta^{t}(q)),
	\end{equation}
    where $\widetilde{W} \colon \mathcal{Q} \to \mathbb{R}^{n}$ is a bounded continuous function (exterior forcing) and $F \colon \mathcal{Q} \times \mathbb{M} \to \Eta$ is a $C^{1}$-differentiable in the second argument continuous mapping satisfying
    \begin{equation}
    	|F(q,y_{1})-F(q,y_{2})|_{\Eta} \leq \Lambda |y_{1}-y_{2}|_{\mathbb{M}} \qquad \text{for all} \ q \in \mathcal{Q} \ \text{and} \ y_{1},y_{2} \in \mathbb{M}.
    \end{equation}
    For example, periodic equations are covered by the case when $(\mathcal{Q},\vartheta)$ is a periodic flow. In terms of \eqref{EQ: DelayRnLinearized}, we can take $\pi$ as the skew-product semiflow on $\mathcal{Q} \times C([-\tau,0];\mathbb{R}^{n})$ generated by \eqref{EQ: DelayRnNonlinear}, or its restriction to any closed positively invariant subset $\mathcal{P}$, and $F'(\wp) \coloneq F'(q,C\phi)$ for $\wp = (q,\phi) \in \mathcal{P}$, see \cite[Section 6.1]{Anikushin2023LyapExp}.
\end{remark}

Recall the Hilbert space $\mathbb{H} = L_{2}([-\tau,0];\mu;\mathbb{R}^{n})$ from \eqref{EQ: HilbertSpaceDelayEqDefinition} and consider the delay operator $A$ in $\mathbb{H}$ corresponding via \eqref{EQ: OperatorAScalarDelayEquations} to $\widetilde{A}$ from \eqref{EQ: DelayRnLinearized}. In terms of the restriction operators $R^{(1)}_{1}$ and $R^{(1)}_{0}$ defined in \eqref{EQ: RestrictionOperatorDelayTensor}, we associate with $\widetilde{B}$ from \eqref{EQ: DelayRnLinearized} a bounded linear operator $B \colon \Eta \to \mathbb{H}$, which is defined by $R^{(1)}_{0}B\eta \coloneq \widetilde{B} \eta$ and $R^{(1)}_{1}B\eta \coloneq 0$ for any $\eta \in \Eta$. 

There is a natural embedding of $\mathbb{E} \coloneq C([-\tau,0];\mathbb{R}^{n})$ into $\mathbb{H}$ that takes any $\psi \in \mathbb{E}$ to $\phi \in \mathbb{H}$ such that $R^{(1)}_{0}\phi = \psi(0)$ and $R^{(1)}_{1}\phi = \psi$. Identifying elements of $\mathbb{E}$ with their images under the embedding, we obtain $\mathcal{D}(A) \subset \mathbb{E}$. It is convenient to use the same notation for operators in $\mathbb{H}$ induced by the embedding from operators on $\mathbb{E}$. In particular, this will be used for the operator $C$, i.e., we set $C\phi \coloneq CR^{(1)}_{1}\phi$ for any $\phi \in \mathbb{E}$ considered as an element of $\mathbb{H}$ in the above sense.

Using the introduced notation, \eqref{EQ: DelayRnLinearized} can be viewed as an abstract evolution equation in $\mathbb{H}$ as follows:
\begin{equation}
	\label{EQ: DelayLinearCocAbsract}
	\dot{\xi}(t) = A\xi(t) + BF'(\pi^{t}(\wp))C\xi(t).
\end{equation}
It can be shown, see\footnote{Since \cite[Theorem 1]{Anikushin2022Semigroups} is stated only in terms of processes, it should be noted that all the required cocycle properties may be derived through the variation of constants formula and a priori integral estimates. See  \cite[Equation (1.10)]{Anikushin2022Semigroups} and its further use in \cite[Section 3]{Anikushin2022Semigroups}.} \cite[Theorem 1]{Anikushin2022Semigroups}, that \eqref{EQ: DelayLinearCocAbsract} generates a uniformly continuous and uniformly eventually compact linear cocycle $\Xi$ in $\mathbb{H}$ over $(\mathcal{P},\pi)$. Specifically, $\Xi^{t}(\wp,\xi_{0}) \coloneq \xi(t;\wp;\xi_{0})$, where $\xi(t;\wp;\xi_{0})$ for $t \geq 0$ is a generalized solution to \eqref{EQ: DelayLinearCocAbsract} on $[0,+\infty)$ with $\xi(0;\wp;\xi_{0}) = \xi_{0}$.

For what follows, we need to discuss the sense in which classical and generalized solutions exist. Regarding the existence of classical solutions in $\mathbb{H}$, we have the following lemma.
\begin{lemma}\cite[Theorem 1]{Anikushin2022Semigroups}
	\label{LEM: ClassicalSolutionsDelay}
	For any $\xi_{0} \in \mathcal{D}(A)$ and $\wp \in \mathcal{P}$, there exists a unique classical solution $\xi(\cdot)=\xi(\cdot;\wp;\xi_{0})$ to \eqref{EQ: DelayLinearCocAbsract} on $[0,+\infty)$ with $\xi(0)=\xi_{0}$, i.e., such that $\xi(\cdot) \in C^{1} ([0,+\infty);\mathbb{H}) \cap C([0,+\infty); \mathcal{D}(A))$ and $\xi(t)$ satisfies \eqref{EQ: DelayLinearCocAbsract} for all $t \geq 0$.
\end{lemma}

In \cite{Anikushin2022Semigroups}, generalized solutions are obtained by continuity from classical ones. However, a more meaningful way for understanding generalized solutions is provided by the structural Cauchy formula as follows\footnote{In this case, the formula reduces to the fact that the solutions are traveling waves, which is well known and has been used since classical works.}. For any $T>0$, let $\mathcal{Y}^{2}(0,T;L_{2}(-\tau,0;\mathbb{R}^{n}))$ be the space of $1$-adorned $L_{2}(-\tau,0;\mathbb{R}^{n})$-valued functions on $[0,T]$ defined in \eqref{EQ: NormInWindowsSpaces} with $\rho \equiv 1$ and $\mathbb{F} =\mathbb{R}^{n}$. Then, a continuous $\mathbb{H}$-valued function $\xi(\cdot)$ is a \textit{generalized solution} to \eqref{EQ: DelayLinearCocAbsract} on $[0,T]$ if $R^{(1)}_{1}\xi(\cdot) \in \mathcal{Y}^{2}(0,T;L_{2}(-\tau,0;\mathbb{R}^{n}))$, $\xi(t)$ satisfies
\begin{equation}
	\label{EQ: DelayVariationOfConstantsSingle}
	R^{(1)}_{0}\xi(t) = R^{(1)}_{0}\xi(0) + \int_{0}^{t}\left(\widetilde{A}R^{(1)}_{1}\xi(s) + BF'(\pi^{s}(\wp))CR^{(1)}_{1}\xi(s)\right)ds
\end{equation}
for any $t \in [0,T]$, and $(R^{(1)}_{1}\xi(s))(0) = R^{(1)}_{0}\xi(s)$ for almost all $s \in [0,T]$. By Theorem \ref{TH: PointwiseMeasurementOperatorOnAdornedSpace}, it is possible to interpret the functions $[0,T] \ni s \mapsto CR^{(1)}_{1}\xi(s) \in \mathbb{M}$ and $[0,T] \ni s \mapsto (R^{(1)}_{1}\xi(s))(0) \in \mathbb{R}^{n}$ as elements of appropriate $L_{2}$-spaces, so the above definition is correct. This classification of generalized solutions is a consequence of the existence of classical solutions, obviously satisfying the above definition and approximating the generalized ones. For the uniqueness of generalized solutions, one can use the contraction principle as in Remark \ref{REM: CauchyFormulaCompundInverseStatement}.

For $\xi(0) \in \mathbb{E}$, the function $x(\cdot) \colon [-\tau, T] \to \mathbb{R}^{n}$ given by 
\begin{equation}
	x(s) = \begin{cases}
		(R^{(1)}_{1}\xi_{0})(s) \qquad &\text{for} \ s \in [-\tau,0],\\
		R^{(1)}_{0}\xi(s) \qquad &\text{for} \ s \in [0,T],
	\end{cases}
\end{equation}
is a classical solution to \eqref{EQ: DelayRnLinearized} in the usual sense, see \cite{HaleLunel1993}.

By virtue of the above, \eqref{EQ: DelayLinearCocAbsract} can be viewed as an abstract evolutionary form of the equation $\frac{\partial}{\partial t}\phi(t,\theta) = \frac{\partial}{\partial \theta} \phi(t,\theta)$, where $t \in [0,T]$ and $\theta \in [-\tau,0]$, with a nonlocal nonautonomous Neumann boundary condition at $\theta = 0$ represented by \eqref{EQ: DelayRnLinearized}.

Let $\Xi_{m}$ the $m$-fold multiplicative compound of $\Xi$ in $\mathbb{H}^{\otimes m}$, as defined in Section \ref{SEC: CocyclesSemigroupsAdditiveCompounds}. Below, we aim to describe $\Xi_{m}$ on the infinitesimal level in a form similar to \eqref{EQ: DelayLinearCocAbsract}. According to Theorem \ref{TH: TensorProductDelayDescription}, $\mathbb{H}^{\otimes m}$ is naturally isomorphic to the space $\mathcal{L}^{\otimes}_{m}$ from \eqref{EQ: L2SpaceTensorCompoundDefinition}, in terms of which the description will be given. For a better understanding of the forthcoming definitions, the reader is advised to start with the proof of Theorem \ref{TH: TensorCompoundCocycleDelayDescription} below and look at the resulting abstract form \eqref{EQ: DelayCompoundCocyclePertTensorOperatorForm}.

First, using the Riesz representation theorem, similarly to \eqref{EQ: DelayOperatorRieszRepresentation}, we obtain an $(r_{2} \times n)$-matrix-valued function $c(\cdot)$ of bounded variation on $[-\tau,0]$, which represents the operator $C$ from \eqref{EQ: DelayRnLinearized} as follows:
\begin{equation}
	C\phi = \int_{-\tau}^{0}dc(\theta)\phi(\theta) \qquad \text{for any} \ \phi \in C([-\tau,0];\mathbb{R}^{n}).
\end{equation}
For any $j \in \{1,\ldots,m \}$, we set $\mathbb{R}_{1,j} \coloneq (\mathbb{R}^{n})^{\otimes(j-1)}$, $\mathbb{R}_{2,j} \coloneq (\mathbb{R}^{n})^{\otimes (m-j)}$, and $\mathbb{M}_{j} \coloneq \mathbb{R}_{1,j} \otimes \mathbb{M} \otimes \mathbb{R}_{2,j}$. Then we associate with $c(\cdot)$ the operator-valued function $c_{j}(\cdot)$ of bounded variation on $[-\tau,0]$ defined by
\begin{equation}
	c_{j}(\theta) \coloneq \operatorname{Id}_{\mathbb{R}_{1,j}} \otimes c(\theta) \otimes \operatorname{Id}_{\mathbb{R}_{2,j}} \qquad \text{for} \ \theta \in [-\tau,0].
\end{equation}
Note that $c_{j}(\theta)$ is a linear operator from $(\mathbb{R}^{n})^{\otimes m}$ to $\mathbb{M}_{j}$.

Next, for $k \in \{0,\ldots,m-1\}$ and $J \in \{1,\ldots,k+1\}$, we define a linear operator $C^{(k)}_{j,J}$, which takes elements $\Phi$ from $C([-\tau,0]^{k+1}; (\mathbb{R}^{n})^{\otimes m})$ to $C([-\tau,0]^{k}; \mathbb{M}_{j})$ according to the formula:
\begin{equation}
	\label{EQ: DelayCjDefinition}
	(C^{(k)}_{j,J}\Phi)(\theta_{1}, \ldots, \hat{\theta}_{J}, \ldots, \theta_{k+1}) \coloneq \int_{-\tau}^{0}dc_{j}(\theta_{J})\Phi(\theta_{1},\ldots,\theta_{k+1}),
\end{equation}
for all $(\theta_{1}, \ldots, \hat{\theta}_{J}, \ldots, \theta_{k+1}) \in [-\tau,0]^{k}$. 

Recall the operators $T_{\sigma}$ and $\Theta_{\sigma}$ defined in \eqref{EQ: TranspoitionValuesDefinition} and \eqref{EQ: ThetaSigmaDefinition}, respectively. Clearly, we have $c_{j}(\theta)T_{\sigma} = T_{\sigma} c_{\sigma(j)}(\theta)$ and hence $C^{(k)}_{j,J} T_{\sigma} = T_{\sigma} C^{(k)}_{\sigma(j),J}$. Furthermore, for any $\widetilde{\sigma} \in \mathbb{S}_{k+1}$ we have that
\begin{equation}
	\label{EQ: MeasurementPartialOperatorProperty}
	C^{(k)}_{j,J}\Theta^{(k+1)}_{\widetilde{\sigma}} = \Theta^{(k)}_{\widetilde{\sigma}_{J}} C^{(k)}_{j,\widetilde{\sigma}^{-1}(J)},
\end{equation}
where $\widetilde{\sigma}_{J} \in \mathbb{S}_{k}$ is obtained from $\widetilde{\sigma}$ by removing $\widetilde{\sigma}^{-1}(J)$th element in its preimage and $J$th element in its image. More rigorously, $\widetilde{\sigma}_{J} = (h^{(k)}_{J})^{-1} \circ \widetilde{\sigma} \circ h^{(k)}_{\widetilde{\sigma}^{-1}(J)}$, where $h^{(k)}_{i}$ is the order-preserving bijection from $\{1,\ldots,k\}$ to $\{ 1,\ldots,k+1 \} \setminus \{i\}$.

For any $j \in \{1,\ldots,m\}$, we also set $\Eta_{j} \coloneq \mathbb{R}_{1,j} \otimes \Eta \otimes \mathbb{R}_{2,j}$. Recall the boundary subspace $\partial_{j_{1}\ldots j_{k}}\mathcal{L}^{\otimes}_{m}$ from \eqref{EQ: TensorSpaceDelayCompoundDecompositionBoundarySubspaces}. For all $k \in \{0,\ldots,m-1\}$ and integers $1 \leq j_{1} < \cdots < j_{k} \leq m$, we associate with $\widetilde{B} \in \mathcal{L}(\mathbb{U};\mathbb{R}^{n})$ from \eqref{EQ: DelayRnLinearized} a bounded linear operator $B^{j_{1} \ldots j_{k}}_{j}$, which takes elements $\Phi_{\Eta}$ from $L_{2}( (-\tau,0)^{k}; \Eta_{j} )$ to $\partial_{j_{1} \ldots j_{k}} \mathcal{L}^{\otimes}_{m}$ according to the formula:
\begin{equation}
	\label{EQ: BoundaryOpeartorCompoundDelay}
	\left(B^{j_{1}\ldots j_{k}}_{j}\Phi_{\Eta}\right)(\theta_{1},\ldots,\theta_{m}) \coloneq (\operatorname{Id}_{\mathbb{R}_{1,j}} \otimes \widetilde{B} \otimes \operatorname{Id}_{ \mathbb{R}_{2,j} } )\Phi_{\Eta}(\theta_{j_{1}},\ldots,\theta_{j_{k}})
\end{equation}
for $\mu^{k}_{L}$-almost all $(\theta_{1},\ldots,\theta_{m}) \in \mathcal{B}_{j_{1}\ldots j_{k}}$.

We associate with $F'(\wp)$ from \eqref{EQ: DelayRnLinearized} a bounded linear operator $F'_{j}(\wp)$, which takes elements $\Phi_{\mathbb{M}}$ from $L_{2}((-\tau,0)^{k}; \mathbb{M}_{j} )$ to $L_{2}((-\tau,0)^{k};\Eta_{j})$ according to the formula:
\begin{equation}
	\label{EQ: DelayCompoundOperatorFJdefinition}
	(F'_{j}(\wp) \Phi_{\mathbb{M}})(\theta_{1},\ldots,\theta_{k}) \coloneq ( \operatorname{Id}_{\mathbb{R}_{1,j}} \otimes F'(\wp) \otimes \operatorname{Id}_{ \mathbb{R}_{2,j} } )\Phi_{\mathbb{M}}(\theta_{1},\ldots,\theta_{k})
\end{equation}
for almost all $(\theta_{1},\ldots,\theta_{k}) \in (-\tau,0)^{k}$. For convenience, we have omitted the dependence of $F'_{j}(\wp)$ on $k$ in the notation, and this should be understood from the context.

Note that either of $B^{j_{1} \ldots j_{k}}_{j}$ or $F'_{j}(\wp)$ is a bounded operator. It is only the operator $C$ that causes problems in the study of delay equations. Before we get into more details, let us describe, as promised, the compound cocycle $\Xi_{m}$ on the infinitesimal level.

For what follows, recall that any summation over multi-indices $j_{1}\ldots j_{k}$ (with fixed $k$) is always taken over all $1 \leq j_{1} < \ldots < j_{k} \leq m$, and $j_{1}\ldots j_{k} \coloneq 0$ for $k=0$. Moreover, when summing over $j \not\in \{j_{1},\ldots, j_{k}\}$ for a fixed multi-index $j_{1}\ldots j_{k}$, it is always assumed that $j \in \{1,\ldots,m\}$.
\begin{theorem}
	\label{TH: TensorCompoundCocycleDelayDescription}
	Let $\xi_{1}(t), \ldots, \xi_{m}(t)$ be solutions to \eqref{EQ: DelayLinearCocAbsract} with initial data $\xi_{1}(0)$, \ldots, $\xi_{m}(0)$ from $\mathcal{D}(A)$. Then the $\mathcal{L}^{\otimes}_{m}$-valued function $\Phi(t)$ of $t \geq 0$ given by
	\begin{equation}
		\Phi(t) \coloneq \xi_{1}(t) \otimes \cdots \otimes \xi_{m}(t) = \Xi^{t}_{m}(\wp, \xi_{1}(0)\otimes \cdots \otimes \xi_{m}(0))
	\end{equation}
    is $C^{1}$-differentiable and belongs to $C([0,\infty);\mathcal{D}(A^{[\otimes m]}))$. Furthermore, its restriction $R_{j_{1}\ldots j_{k}}\Phi$ belongs\footnote{Here we naturally embed the space $C([-\tau,0]^{k};(\mathbb{R}^{n})^{\otimes m})$ into $\mathcal{L}^{\otimes}_{m}$.} to $C([0,\infty); C([-\tau,0]^{k};(\mathbb{R}^{n})^{\otimes m}))$ for any $k \in \{0,\ldots,m\}$ and $1 \leq j_{1} < \cdots < j_{k} \leq m$, and for any $t \geq 0$ we have that\footnote{Here, as before, $J(j)=J(j;j_{1}\ldots j_{k})$ denotes the integer $J$ such that $j$ is the $J$th element of the set $\{ j,j_{1},\ldots, j_{k} \}$ arranged by increasing.}
	\begin{equation}
		\label{EQ: CompoundInfinitesimalDescription}
		\begin{split}
			\dot{\Phi}(t) = A^{[\otimes m]}\Phi(t) + \\ + \sum_{k=0}^{m-1}\sum_{j_{1} \ldots j_{k}}\sum_{j \notin \{j_{1},\ldots,j_{k}\}} B^{j_{1}\ldots j_{k}}_{j}F'_{j}(\pi^{t}(\wp))C^{(k)}_{j,J(j)} R_{j j_{1}\ldots j_{k}}\Phi(t).
		\end{split}
	\end{equation}
\end{theorem}
\begin{proof}
	By Lemma \ref{LEM: ClassicalSolutionsDelay}, $\xi_{j}(\cdot)$ is a classical solution for any $j \in \{1, \ldots, m\}$. From this, for any $t \geq 0$ we have $\Phi(t) \in \mathcal{D}(A)^{\odot m} \subset \mathcal{D}(A^{[\otimes m]})$ and hence $R_{j_{1}\ldots j_{k}}\Phi(t) \in C([-\tau,0]^{k};(\mathbb{R}^{n})^{\otimes m})$, and the functions continuously depend on $t \geq 0$ in that spaces. Furthermore, $\Phi(\cdot)$ is a $C^{1}$-differentiable $\mathcal{L}^{\otimes}_{m}$-valued function, and straightforward calculations yield
	\begin{equation}
		\label{EQ: GeneratorCompoundCocycleDescriptionAuxiliaryEq}
		\dot{\Phi}(t) = A^{[\otimes m]}\Phi(t) + \Phi_{0}(t),
	\end{equation}
    where
    \begin{equation}
    	\label{EQ: GeneratorCompoundCocycleDescription}
    	\Phi_{0}(t) = \sum_{j=1}^{m} \xi_{1}(t) \otimes \cdots \otimes BF'(\pi^{t}(\wp))C\xi_{j}(t) \otimes \cdots \otimes \xi_{m}(t).
    \end{equation}
    Note that $BF'(\pi^{t}(\wp))C\xi_{j}(t)$ as an element of $\mathbb{H} = L_{2}([-\tau,0];\mu;\mathbb{R}^{n})$ vanishes in $(-\tau,0)$ or, in other words, after applying $R^{(1)}_{1}$. Thus, the $j$th summand in \eqref{EQ: GeneratorCompoundCocycleDescription} vanishes after applying $R_{j_{1} \ldots j_{k}}$, provided that $j \in \{ j_{1}, \ldots, j_{k} \}$. It is a straightforward verification that for $j \notin \{ j_{1}, \ldots, j_{k} \}$, the operator $R_{j_{1} \ldots j_{k}}$ applied to the $j$th summand in \eqref{EQ: GeneratorCompoundCocycleDescription} corresponds to the $j$th summand from the inner sum in \eqref{EQ: CompoundInfinitesimalDescription}, according to the definitions in \eqref{EQ: DelayCjDefinition}, \eqref{EQ: BoundaryOpeartorCompoundDelay}, and \eqref{EQ: DelayCompoundOperatorFJdefinition}.
\end{proof}

Next, we aim to write \eqref{EQ: CompoundInfinitesimalDescription} in an operator form. To do this, consider the \textit{control space} $\mathbb{U}^{\otimes}_{m}$ defined by the outer orthogonal sum as follows:
\begin{equation}
	\label{EQ: ControlSpaceDelayCompoundDefinition}
	\mathbb{U}^{\otimes}_{m} \coloneq \bigoplus_{k=0}^{m-1}\bigoplus_{j_{1} \ldots j_{k}} \bigoplus_{j \notin \{ j_{1},\ldots,j_{k} \}} L_{2}((-\tau,0)^{k};\mathbb{U}_{j}).
\end{equation}
For any element $\eta \in \mathbb{U}^{\otimes}_{m}$, we write $\eta = (\eta^{j}_{j_{1}\ldots j_{k}})$, meaning that the indices vary in appropriate ranges, and each $\eta^{j}_{j_{1}\ldots j_{k}}$ belongs to the corresponding summand from \eqref{EQ: ControlSpaceDelayCompoundDefinition}. 

Recalling the operators $B^{j_{1}\ldots j_{k}}_{j}$ from \eqref{EQ: BoundaryOpeartorCompoundDelay}, we define the \textit{control operator} $B^{\otimes}_{m} \in \mathcal{L}(\mathbb{U}^{\otimes}_{m};\mathcal{L}^{\otimes}_{m})$ as follows:
\begin{equation}
	\label{EQ: OperatorBDelayCompoundDefinition}
	B^{\otimes}_{m}\eta \coloneq  \sum_{k=0}^{m-1}\sum_{j_{1} \ldots j_{k}}\sum_{j \notin \{j_{1},\ldots,j_{k}\}} B^{j_{1}\ldots j_{k}}_{j}\eta^{j}_{j_{1}\ldots j_{k}} \qquad \text{for} \ \eta = (\eta^{j}_{j_{1}\ldots j_{k}}) \in \mathbb{U}^{\otimes}_{m}.
\end{equation}

\begin{remark}
	For $\nu = 0$, the linear inhomogeneous system \eqref{EQ: ControlSystemMCompoundDelay}, where $\eta(t)$ is exchanged with $B^{\otimes}_{m} \eta(t)$, is related to \eqref{EQ: CompoundInfinitesimalDescription} via the closed feedback $\eta(t)=( \eta^{j}_{j_{1}\ldots j_{k}}(t) )$ with
	\begin{equation}
		\label{EQ: ClosedFedbackDelayCompound}
		\eta^{j}_{j_{1}\ldots j_{k}}(t) \coloneq F'_{j}(\pi^{t}(\wp))C^{(k)}_{j,J(j)} R_{j j_{1}\ldots j_{k}}\Phi(t).
	\end{equation}
\end{remark}

Similarly to \eqref{EQ: ControlSpaceDelayCompoundDefinition}, we define the \textit{measurement space} $\mathbb{M}^{\otimes}_{m}$ by the outer orthogonal sum as follows:
\begin{equation}
	\label{EQ: MeasurementSpaceCompoundTensorDefinition}
	\mathbb{M}^{\otimes}_{m}\coloneq\bigoplus_{k=0}^{m-1}\bigoplus_{j_{1}\ldots j_{k}} \bigoplus_{j \notin \{j_{1},\ldots,j_{k}\}}L_{2}((-\tau,0)^{k};\mathbb{M}_{j}),
\end{equation}
We similarly write $M = (M^{j}_{j_{1}\ldots j_{k}})$ for an element $M$ of $\mathbb{M}^{\otimes}_{m}$.

Recalling the operators $C^{(k)}_{j,J}$ from \eqref{EQ: DelayCjDefinition} and the space $\mathbb{E}^{\otimes}_{m}$ from \eqref{EQ: SpaceEmDefinitionDelayCompoundGeneral}, we define the \textit{measurement operator} $C^{\otimes}_{m} \in \mathcal{L}(\mathbb{E}^{\otimes}_{m};\mathbb{M}^{\otimes}_{m})$ by
\begin{equation}
	\label{EQ: DelayMeasumerentAntisymmetricSumSpace}
	C^{\otimes}_{m}\Phi \coloneq \sum_{k=0}^{m-1}\sum_{j_{1}\ldots j_{k}}\sum_{j \notin \{j_{1},\ldots,j_{k}\}} C^{(k)}_{j,J(j)}R_{jj_{1}\ldots j_{k}}\Phi,
\end{equation}
where the sum is taken in $\mathbb{M}^{\otimes}_{m}$ according to \eqref{EQ: MeasurementSpaceCompoundTensorDefinition}, and the action of $C^{(k)}_{j,J(j)}$ is understood in the sense of Theorem \ref{TH: OperatorCExntesionOntoWDiagonal}.

Recalling the operators $F'_{j}(\wp)$ from \eqref{EQ: DelayCompoundOperatorFJdefinition}, we define an operator $F^{\otimes}_{m}(\wp)$, which takes elements $M = (M^{j}_{j_{1}\ldots j_{k}})$ from $\mathbb{M}^{\otimes}_{m}$ to $\mathbb{U}^{\otimes}_{m}$ according to the formula:
\begin{equation}
	\label{EQ: FPrimeTensorOperatorDef}
	F^{\otimes}_{m} M \coloneq \sum_{k=0}^{m-1}\sum_{j_{1}\ldots j_{k}} \sum_{ j \notin \{j_{1}\dots j_{k}\} } F'_{j}(\wp) M^{j}_{j_{1}\ldots j_{k}},
\end{equation}
where the overall sum is taken in $\mathbb{U}^{\otimes}_{m}$ according to \eqref{EQ: ControlSpaceDelayCompoundDefinition}.

Using the above notation, we can rewrite \eqref{EQ: CompoundInfinitesimalDescription} as follows:
\begin{equation}
	\label{EQ: DelayCompoundCocyclePertTensorOperatorForm}
	\dot{\Phi}(t) = A^{[\otimes m]}\Phi(t) + B^{\otimes}_{m} F^{\otimes}_{m}(\pi^{t}(\wp))C^{\otimes}_{m}\Phi(t).
\end{equation}
From \eqref{EQ: DelayCompoundCocyclePertTensorOperatorForm} it is clear that the generator of $\Xi_{m}$ in $\mathcal{L}^{\otimes}_{m}$ is given by a nonautonomous boundary perturbation of $A^{[\otimes m]}$.

\begin{remark}
	As in the case of \eqref{EQ: DelayVariationOfConstantsSingle}, using the structural Cauchy formula, one can also establish in what sense generalized solutions to \eqref{EQ: DelayCompoundCocyclePertTensorOperatorForm}, i.e., the trajectories of $\Xi_{m}$, can be understood.
\end{remark}

Now we move on to the antisymmetric space $\mathcal{L}^{\wedge}_{m}$. First, we write an analog of \eqref{EQ: DelayCompoundCocyclePertTensorOperatorForm} in this space.

To do this, consider $\eta=(\eta^{j}_{j_{1}\ldots j_{k}}) \in \mathbb{U}^{\otimes}_{m}$, satisfying the antisymmetric relations induced by \eqref{EQ: AntisymmetricRelationsGeneral} when applying the closed feedback \eqref{EQ: ClosedFedbackDelayCompound}. This provides for all $k \in \{0, \ldots, m-1\}$, $1 \leq j_{1} < \cdots < j_{k} \leq m$, $j \notin \{ j_{1},\ldots, j_{k} \}$, and $\sigma \in \mathbb{S}_{m}$ the relations
\begin{equation}
	\label{EQ: ControlSpaceRelationsAntisymmetricCompoundDelay}
	\eta^{j}_{j_{1}\ldots j_{k}} = (-1)^{\sigma}T_{\sigma} \Theta^{(k)}_{\bar{\sigma}} \eta^{\sigma(j)}_{\sigma(j_{\bar{\sigma}(1)})\ldots \sigma(j_{\bar{\sigma}(k)})}
\end{equation}
where $\bar{\sigma} \in \mathbb{S}_{k}$ is such that $\sigma(j_{\bar{\sigma}(1)}) < \cdots < \sigma(j_{ \bar{\sigma}(k)})$.

Recall that $k \in \{ 0, \ldots, m\}$ is called improper if the subspace $\partial_{k}\mathcal{L}^{\wedge}_{m}$ from \eqref{EQ: AntisymmetricSubspaceOverKfaces} is zero. Define a subspace $\mathbb{U}^{\wedge}_{m}$ of $\mathbb{U}^{\otimes}_{m}$ as follows:
\begin{equation}
	\label{EQ: AntisymmetricControlSpaceDefinition}
	\begin{split}
		\mathbb{U}^{\wedge}_{m}\coloneq\{ \eta=(\eta^{j}_{j_{1}\ldots j_{k}}) \in \mathbb{U}^{\otimes}_{m} \mid \ &\eta \ \text{satisfies} \ \eqref{EQ: ControlSpaceRelationsAntisymmetricCompoundDelay} \ \text{and} \ \\&\eta^{j}_{j_{1}\ldots j_{k}}=0 \ \text{for improper} \ k \}.
	\end{split}
\end{equation}

Let $B^{\wedge}_{m}$ denote the restriction of the operator $B^{\otimes}_{m}$ from \eqref{EQ: OperatorBDelayCompoundDefinition} to $\mathbb{U}^{\wedge}_{m}$. 
\begin{proposition}
	\label{PROP: ControlOperatorWedgeCorectness}
	Let $\eta=(\eta^{j}_{j_{1}\ldots j_{k}}) \in \mathbb{U}^{\otimes}_{m}$ satisfy all the antisymmetric relations from \eqref{EQ: ControlSpaceRelationsAntisymmetricCompoundDelay}. Then $B^{\otimes}_{m}\eta \in \mathcal{L}^{\wedge}_{m}$. In particular, $B^{\wedge}_{m} \in \mathcal{L}(\mathbb{U}^{\wedge}_{m};\mathcal{L}^{\wedge}_{m})$.
\end{proposition}
\begin{proof}
	Given $1 \leq j_{1} < \cdots < j_{k} \leq m$ and $j \notin\{j_{1},\ldots,j_{k}\}$, let $\sigma$ and $\bar{\sigma}$ be as in \eqref{EQ: ControlSpaceRelationsAntisymmetricCompoundDelay}. Then
	\begin{equation}
		\begin{split}
			B^{j_{1}\ldots j_{k}}_{j} \eta^{j}_{j_{1}\ldots j_{k}} &= B^{j_{1}\ldots j_{k}}_{j} (-1)^{\sigma}T_{\sigma} \Theta^{(k)}_{\bar{\sigma}} \eta^{\sigma(j)}_{\sigma(j_{\bar{\sigma}(1)})\ldots \sigma(j_{\bar{\sigma}(k)})} \\
			&= (-1)^{\sigma}T_{\sigma} \Theta^{(m)}_{\sigma^{-1}} B^{\sigma(j_{\bar{\sigma}(1)})\ldots \sigma(j_{\bar{\sigma}(m)})}_{\sigma(j)}\eta^{\sigma(j)}_{\sigma(j_{\bar{\sigma}(1)})\ldots \sigma(j_{\bar{\sigma}(k)})}.
		\end{split}
	\end{equation}
	By applying the restriction operator $R_{j_{1}\ldots j_{k}}$ to both sides of the above identity, and then using \eqref{EQ: RestrictionAndThetaSigmaIdentity} and summing over all $j \notin \{ j_{1},\ldots,j_{k}\}$, we obtain the antisymmetric relations \eqref{EQ: AntisymmetricRelationsGeneral} for $\Phi = B^{\otimes}_{m}\eta$ according to \eqref{EQ: OperatorBDelayCompoundDefinition}. By Proposition \ref{PROP: AntisymmetricRelationsGeneral}, this is equivalent to $\Phi \in \mathcal{L}^{\wedge}_{m}$.
\end{proof}

\begin{remark}
	For $\eta=(\eta^{j}_{j_{1}\ldots j_{k}}) \in \mathbb{U}^{\otimes}_{m}$ satisfying \eqref{EQ: ControlSpaceRelationsAntisymmetricCompoundDelay}, it is not necessary that $\eta^{j}_{j_{1}\ldots j_{k}} = 0$ for improper $k$. However, the influence of such components on the system, i.e., $R_{j_{1}\ldots j_{k}}B^{\otimes}_{m}\eta \in \mathcal{L}^{\wedge}_{m}$, must vanish for improper $k$. This is why we exclude them from consideration in the control space $\mathbb{U}^{\wedge}_{m}$ --- otherwise they will result in rougher frequency conditions.
\end{remark}

Now consider elements $M=(M^{j}_{j_{1}\ldots j_{k}}) \in \mathbb{M}^{\otimes}_{m}$, which satisfy similar to \eqref{EQ: ControlSpaceRelationsAntisymmetricCompoundDelay} relations, i.e., for all $k \in \{0, \ldots, m-1\}$, $1 \leq j_{1} < \cdots < j_{k} \leq m$, $j \notin \{ j_{1},\ldots, j_{k} \}$, and $\sigma \in \mathbb{S}_{m}$, we have
\begin{equation}
	\label{EQ: MeasurementSpaceRelationsAntisymmetricCompoundDelay}
		M^{j}_{j_{1}\ldots j_{k}} = (-1)^{\sigma}T_{\sigma} \Theta^{(k)}_{\bar{\sigma}} M^{\sigma(j)}_{\sigma(j_{\bar{\sigma}(1)})\ldots \sigma(j_{\bar{\sigma}(k)})}
\end{equation}
where $\bar{\sigma} \in \mathbb{S}_{k}$ is such that $\sigma(j_{\bar{\sigma}(1)}) < \cdots < \sigma(j_{ \bar{\sigma}(k)})$.

We define $\mathbb{M}^{\wedge}_{m}$ as follows:
\begin{equation}
	\label{EQ: AntisymmetricMeasurementSpaceDefinition}
	\begin{split}
		\mathbb{M}^{\wedge}_{m}\coloneq\{ M=(M^{j}_{j_{1}\ldots j_{k}}) \in \mathbb{M}^{\otimes}_{m} \mid \ &M \ \text{satisfies} \ \eqref{EQ: MeasurementSpaceRelationsAntisymmetricCompoundDelay} \text{ and } \\&M^{j}_{j_{1}\ldots j_{k}}=0 \ \text{for improper} \ k \}.
	\end{split}
\end{equation}

Recall the space $\mathbb{E}^{\otimes}_{m}$ from \eqref{EQ: SpaceEmDefinitionDelayCompoundGeneral} and set $\mathbb{E}^{\wedge}_{m} \coloneq \mathcal{L}^{\wedge}_{m} \cap \mathbb{E}^{\otimes}_{m}$. Clearly, $\mathbb{E}^{\wedge}_{m}$ is a closed subspace of $\mathbb{E}^{\otimes}_{m}$. For $\Phi \in \mathbb{E}^{\wedge}_{m}$, we define $C^{\wedge}_{m}\Phi \in \mathbb{M}^{\otimes}_{m}$ as follows:
\begin{equation}
	\label{EQ: OperatorCWedgeDefinition}
	(C^{\wedge}_{m}\Phi)^{j}_{j_{1}\ldots j_{k}} \coloneq \begin{cases}
		(C^{\otimes}_{m}\Phi)^{j}_{j_{1}\ldots j_{k}} \qquad &\text{if $k$ is proper},\\
		0 \qquad &\text{if $k$ is improper},
	\end{cases}
\end{equation} 
where the indices vary as above. In fact, we must have $C^{\wedge}_{m}\Phi \in \mathbb{M}^{\wedge}_{m}$, as the following proposition confirms.
\begin{proposition}
	\label{PROP: OperatorCWedgeDelayCompoundCorrectness}
	For any $\Phi \in \mathbb{E}^{\wedge}_{m}$, the element $M = C^{\otimes}_{m}\Phi$ of $\mathbb{M}^{\otimes}_{m}$ satisfies all the antisymmetric relations from \eqref{EQ: MeasurementSpaceRelationsAntisymmetricCompoundDelay}. In particular, $C^{\wedge}_{m} \in \mathcal{L}(\mathbb{E}^{\wedge}_{m};\mathbb{M}^{\wedge}_{m})$.
\end{proposition}
\begin{proof}
	It suffices to consider only those $\Phi \in \mathbb{E}^{\wedge}_{m}$ for which all restrictions $R_{j_{1}\ldots j_{k}}\Phi$ are continuous functions, since for general $\Phi$ we can apply the approximation argument. Let $1 \leq j_{1} < \cdots < j_{k} \leq m$, $j \notin \{j_{1},\ldots, j_{k}\}$, $\sigma \in \mathbb{S}_{m}$, and $\bar{\sigma} \in \mathbb{S}_{k}$ be as in \eqref{EQ: MeasurementSpaceRelationsAntisymmetricCompoundDelay}. Consider $J=J(j;j_{1},\ldots,j_{k})$ and for any $l \in \{1,\ldots, k+1\}$ set
	\begin{equation}
		\widetilde{j}_{l} \coloneq \begin{cases}
			j_{l} \qquad &\text{if} \quad l < J,\\
			j \qquad &\text{if} \quad l=J,\\
			j_{l-1} \qquad &\text{if} \quad l > J.
		\end{cases}
	\end{equation}
	Let $\widetilde{\sigma} \in \mathbb{S}^{k+1}$ be such that $\sigma(\widetilde{j}_{\widetilde{\sigma}(1)}) < \cdots < \sigma(\widetilde{j}_{\widetilde{\sigma}(k+1)})$. Note that $\widetilde{\sigma}^{-1}(J) = J(\sigma(j)) \coloneq J(\sigma(j); \sigma(j_{1}), \ldots, \sigma(j_{k}) )$ and $\widetilde{\sigma}_{J} = \bar{\sigma}$ in terms of \eqref{EQ: MeasurementPartialOperatorProperty}. Then using \eqref{EQ: AntisymmetricRelationsGeneral} and \eqref{EQ: MeasurementPartialOperatorProperty}, we obtain the following relations:
	\begin{equation}
		\begin{split}
			M^{j}_{j_{1}\ldots j_{k}} = C^{(k)}_{j,J} R_{j j_{1}\ldots j_{k}}\Phi &= 	C^{(k)}_{j,J} (-1)^{\sigma}T_{\sigma} \Theta^{(k+1)}_{\widetilde{\sigma}} R_{\sigma(j) \sigma(j_{1}) \ldots \sigma(j_{k})}\Phi \\
			&= (-1)^{\sigma}T_{\sigma} \Theta^{(k)}_{\bar{\sigma}} C^{(k)}_{\sigma(j),J(\sigma(j))} R_{\sigma(j) \sigma(j_{1}) \ldots \sigma(j_{k})}\Phi \\
			&= (-1)^{\sigma}T_{\sigma} \Theta^{(k)}_{\bar{\sigma}} M^{\sigma(j)}_{\sigma(j_{\bar{\sigma}(1)}) \ldots \sigma(j_{\bar{\sigma}(k)})}
		\end{split}
	\end{equation}
	that gives the relations \eqref{EQ: MeasurementSpaceRelationsAntisymmetricCompoundDelay}.
\end{proof}

Finally, let $F^{\wedge}_{m}$ be the restriction of $F^{\otimes}_{m}$ to $\mathbb{M}^{\wedge}_{m}$. Clearly, $F^{\wedge}_{m} \in \mathcal{L}(\mathbb{M}^{\wedge}_{m};\mathbb{U}^{\wedge}_{m})$. Using \eqref{EQ: DelayCompoundCocyclePertTensorOperatorForm} and Propositions \ref{PROP: ControlOperatorWedgeCorectness} and \ref{PROP: OperatorCWedgeDelayCompoundCorrectness}, we obtain the infinitesimal description  for the cocycle $\Xi_{m}$ in $\mathcal{L}^{\wedge}_{m}$ as follows:
\begin{equation}
	\label{EQ: DelayCompoundCocyclePertWedgeOperatorForm}
	\dot{\Phi}(t) = A^{[\wedge m]}\Phi(t) + B^{\wedge}_{m} F^{\wedge}_{m}(\pi^{t}(\wp))C^{\wedge}_{m}\Phi(t).
\end{equation}
This system will be used below to study the cocycle $\Xi_{m}$ in $\mathcal{L}^{\wedge}_{m}$ using the frequency theorem.

\subsection{Associated linear inhomogeneous problems with quadratic constraints}
\label{SUBSEC: AssociatedLIPquadraticConstr}

We associate with \eqref{EQ: DelayCompoundCocyclePertWedgeOperatorForm} the control system as follows:
\begin{equation}
\label{EQ: CompoundDelayControlSystem}
	\dot{\Phi}(t) = (A^{[\wedge m]} + \nu I) \Phi(t) + B^{\wedge}_{m}\eta(t),
\end{equation}
where $I$ denotes the identity operator in $\mathcal{L}^{\wedge}_{m}$, $\nu \in \mathbb{R}$ is fixed, and $\eta(\cdot) \in L_{2}(0,T;\mathbb{U}^{\wedge}_{m})$ for some $T>0$.

Recall the space $\mathbb{E}^{\wedge}_{m} \coloneq \mathbb{E}^{\otimes}_{m} \cap \mathcal{L}^{\wedge}_{m}$. Similarly to \eqref{EQ: MyLovelyScaleDelayCompoundGeneral}, we have that
\begin{equation}
	\label{EQ: MyLovelyScaleDelayCompoundAntiSymmetric}
	\mathcal{D}(A^{[\wedge m]}) \subset \mathbb{E}^{\wedge}_{m} \subset \mathcal{L}^{\wedge}_{m},
\end{equation}
where all the embeddings are continuous and dense in $\mathcal{L}^{\wedge}_{m}$.

To relate \eqref{EQ: CompoundDelayControlSystem} with \eqref{EQ: DelayCompoundCocyclePertWedgeOperatorForm} appropriately, we consider the quadratic form $\mathcal{F}(\Phi,\eta)$ of $\Phi \in \mathbb{E}^{\wedge}_{m}$ and $\eta \in \mathbb{U}^{\wedge}_{m}$ defined by
\begin{equation}
	\label{EQ: DelayCompoundQuadraticFormsAbstractGeneral}
	\mathcal{F}(\Phi,\eta) \coloneq \Lambda^{2}\|C^{\wedge}_{m} \Phi\|^{2}_{\mathbb{M}^{\wedge}_{m}} - \|\eta\|^{2}_{\mathbb{U}^{\wedge}_{m}}.
\end{equation}

For any $\wp \in \mathcal{P}$ and $\Phi \in \mathbb{E}^{\wedge}_{m}$, it can be seen from \eqref{EQ: LipschitzFprimeDelay} that the following relation holds:
\begin{equation}
	\label{EQ: QuadraticConstraintCompoundDelay}
	\begin{split}
		\mathcal{F}(\Phi,\eta) \geq 0 \qquad \text{if} \ \eta = F^{\wedge}_{m}(\wp) C^{\wedge}_{m}\Phi.
	\end{split}
\end{equation}
In this case, $\mathcal{F}$ is said to define a \textit{quadratic constraint} for \eqref{EQ: CompoundDelayControlSystem} associated with the closed feedback rule $\eta = F^{\wedge}_{m}(\wp) C^{\wedge}_{m}\Phi$. Under additional assumptions on $F'(\wp)$, one may consider more subtle quadratic constraints, see \eqref{EQ: DelayCompoundMixedContraintG}.

We can generalize \eqref{EQ: DelayCompoundQuadraticFormsAbstractGeneral} as follows. Consider a bounded quadratic form $\mathcal{G}(M,\eta)$ of $M \in \mathbb{M}^{\wedge}_{m}$ and $\eta \in \mathbb{U}^{\wedge}_{m}$. Then we set
\begin{equation}
	\label{EQ: QuadraticFormRealCompoundDelayGeneral}
	\mathcal{F}(\Phi,\eta) \coloneq \mathcal{G}(C^{\wedge}_{m}\Phi,\eta) \qquad \text{for all} \ \Phi \in \mathbb{E}^{\wedge}_{m} \ \text{and} \ \eta \in \mathbb{U}^{\wedge}_{m}. 
\end{equation}
Let us describe the Hermitian extension $\mathcal{F}^{\mathbb{C}}$ of such $\mathcal{F}$. Recall that $\mathcal{F}^{\mathbb{C}}$ is a quadratic form on $\left(\mathbb{E}^{\wedge}_{m}\right)^{\mathbb{C}} \times \left(\mathbb{U}^{\wedge}_{m}\right)^{\mathbb{C}}$ given by $\mathcal{F}^{\mathbb{C}}(\Phi_{1} + i \Phi_{2},\eta_{1} + i\eta_{2}) \coloneq \mathcal{F}(\Phi_{1},\eta_{1}) + \mathcal{F}(\Phi_{2},\eta_{2})$ for any $\Phi_{1},\Phi_{2} \in \mathbb{E}^{\wedge}_{m}$ and $\eta_{1},\eta_{2} \in \mathbb{U}^{\wedge}_{m}$. First, any $\mathcal{G}$ as above can be represented as follows:
\begin{equation}
	\mathcal{G}(M,\eta) = \langle M, \mathcal{G}_{1} M \rangle_{\mathbb{M}^{\wedge}_{m}} + \langle\eta,\mathcal{G}_{2}M\rangle_{\mathbb{U}^{\wedge}_{m}} + \langle\eta, \mathcal{G}_{3}\eta\rangle_{ \mathbb{U}^{\wedge}_{m}},
\end{equation}
where $\mathcal{G}_{1} \in \mathcal{L}(\mathbb{M}^{\wedge}_{m})$ and $\mathcal{G}_{3} \in \mathcal{L}( \mathbb{U}^{\wedge}_{m})$ are self-adjoint and $\mathcal{G}_{2} \in \mathcal{L}(\mathbb{M}^{\wedge}_{m}; \mathbb{U}^{\wedge}_{m})$. Then for any $\Phi \in (\mathbb{M}^{\wedge}_{m})^{\mathbb{C}}$ and $\eta \in (\mathbb{U}^{\wedge}_{m})^{\mathbb{C}}$, the value $\mathcal{F}^{\mathbb{C}}(\Phi,\eta)$ is given by
\begin{equation}
	\label{EQ: QuadraticFormComplexificationCompoundDelayGeneral}
	\begin{split}
	\mathcal{F}^{\mathbb{C}}(\Phi,\eta) = \mathcal{G}^{\mathbb{C}}(C^{\wedge}_{m}\Phi,\eta) =\\= \langle C^{\wedge}_{m}\Phi, \mathcal{G}_{1} C^{\wedge}_{m}\Phi\rangle_{(\mathbb{M}^{\wedge}_{m})^{\mathbb{C}}} + \operatorname{Re}\langle\eta, \mathcal{G}_{2} C^{\wedge}_{m}\Phi\rangle_{(\mathbb{U}^{\wedge}_{m})^{\mathbb{C}}} + \langle \eta, \mathcal{G}_{3}\eta\rangle_{(\mathbb{U}^{\wedge}_{m})^{\mathbb{C}}},
	\end{split}
\end{equation}
where for convenience we have omitted mentioning complexifications of the operators $C^{\wedge}_{m}$, $\mathcal{G}_{1}$, $\mathcal{G}_{2}$, and $\mathcal{G}_{3}$.

For the forthcoming applications, it is important that $\mathcal{F}$ is bounded on $\mathbb{E}^{\wedge}_{m} \times \mathbb{U}^{\wedge}_{m}$, and $\mathbb{E}^{\wedge}_{m}$ is an intermediate Banach space, as in \eqref{EQ: MyLovelyScaleDelayCompoundAntiSymmetric}. Since such $\mathcal{F}$ carries the unbounded nature of perturbations in \eqref{EQ: DelayCompoundCocyclePertWedgeOperatorForm}, a certain specificity of the unperturbed problem \eqref{EQ: CompoundDelayControlSystem} is required to control its perturbations. As discussed in the introduction, this has to do with regularity and structure.

Regarding regularity, Theorem \ref{TH: ResolventDelayCompoundBound} delivers uniform bounds for the resolvent of $A^{[\wedge m]}$ in $\mathcal{L}(\mathcal{L}^{\wedge}_{m};\mathbb{E}^{\wedge}_{m})$ on vertical lines, see Corollary \ref{COR: DelayCompoundRESProperty}. In contrast to parabolic problems, we do not have similar bounds in $\mathcal{L}(\mathcal{L}^{\wedge}_{m};\mathcal{D}(A^{[\wedge m]}))$, and this can already be seen for $m=1$.

Now we are going to exploit structural properties to guarantee the well-posedness of integral quadratic functionals associated with $\mathcal{F}$.

\subsection{Properties of the complexificated problem}
\label{SUBSEC: DelayCompPropertiesOfComplexificatedProblem}
In this section, we work with the complexificated problem. For convenience, we omit mention of complexifications of spaces and operators. It can be assumed (during this section) that they are all complexificated by default. We emphasize only the Hermitian extensions $\mathcal{F}^{\mathbb{C}}$ and $\mathcal{G}^{\mathbb{C}}$ of $\mathcal{F}$ and $\mathcal{G}$ from \eqref{EQ: QuadraticFormComplexificationCompoundDelayGeneral}, respectively.

First, we have the following uniform bounds for the resolvent of $A^{[\wedge m]}$ on vertical lines.
\begin{corollary}
	\label{COR: DelayCompoundRESProperty}
	For some $\nu_{0} \in \mathbb{R}$, suppose that the operator $A^{[\wedge m]}$ does not have eigenvalues on the line $-\nu_{0} + i \mathbb{R}$. Then
	\begin{equation}
		\sup_{ \omega \in \mathbb{R}}\left\| \left((A^{[\wedge m]} + \nu_{0}I) - i\omega I\right)^{-1} \right\|_{\mathcal{L}(\mathcal{L}^{\wedge}_{m};\mathbb{E}^{\wedge}_{m}) } < \infty.
	\end{equation}
\end{corollary}
\begin{proof}
	The statement follows from the analogue of \eqref{EQ: ResolventEstimateEmDelayCompound} for $A^{[\wedge m]}$ and the fact that
	\begin{equation}
	\label{EQ: ResolventAddCompBoundMainSpace}
	\sup_{ \omega \in \mathbb{R}}\left\| \left((A^{[\wedge m]} + \nu_{0}I) - i\omega I\right)^{-1} \right\|_{\mathcal{L}(\mathcal{L}^{\wedge}_{m})} < \infty,
    \end{equation}
    which follows from spectral decompositions and the representation of the resolvent through the Laplace transform of the semigroup, see \cite[Theorem 4.2]{Anikushin2020FreqDelay} for similar arguments.
\end{proof}

Next, we consider the extended control system associated with the pair $(A^{[\wedge m]} + \nu I,B^{\wedge}_{m})$ for some $\nu \in \mathbb{R}$, defined as follows:
\begin{equation}
	\label{EQ: ExtendedControlSysDelayCompoundWedge}
	\dot{\Phi}(t) = (A^{[\wedge m]} +\nu I)\Phi(t) + B^{\wedge}_{m}\eta(t) + \zeta(t).
\end{equation}
Given $T>0$, let $\mathfrak{M}^{T}_{\Phi_{0}}(\nu)$ be the space of processes of \eqref{EQ: ExtendedControlSysDelayCompoundWedge} on $[0,T]$ through $\Phi_{0} \in \mathcal{L}^{\wedge}_{m}$, i.e., the space of all $(\Phi(\cdot), (\eta(\cdot),\zeta(\cdot)))$ such that $\eta(\cdot) \in L_{2}(0,T;\mathbb{U}^{\wedge}_{m})$, $\zeta(\cdot) \in L_{2}(0,T;\mathcal{L}^{\wedge}_{m})$, and $\Phi(\cdot)$ is the mild solution to \eqref{EQ: ExtendedControlSysDelayCompoundWedge} with $\Phi(0) = \Phi_{0}$. For $T=\infty$, we simply write $\mathfrak{M}_{\Phi_{0}}(\nu)$ and additionally require\footnote{Since $A^{[\wedge m]}$ generates the $C_{0}$-semigroup $G^{\wedge m}$, and hence the growth exponent $\omega(G^{\wedge m})$ of $G^{\wedge m}$ is finite, it is clear that $\mathfrak{M}_{\Phi_{0}}(\nu)$ is not empty. Indeed, just consider $\eta(\cdot) \equiv 0$ and $\zeta(\cdot) \equiv \varkappa \Phi(\cdot)$ for any $\varkappa \in \mathbb{R}$ such that $\varkappa + \nu + \omega(G^{\wedge m}) < 0$. This is why we study the extended control system, since for the original system the analogous space of processes may be empty.} that $\Phi(\cdot)$ belong to $L_{2}(0,\infty;\mathcal{L}^{\wedge}_{m})$.

From this, we define the space $\mathcal{Z}^{T}_{0}(\nu)$ of processes on $[0,T]$ as follows:
\begin{equation}
	\mathcal{Z}^{T}_{0}(\nu) \coloneq \bigcup_{\Phi_{0} \in \mathcal{L}^{\wedge}_{m}}\mathfrak{M}^{T}_{\Phi_{0}}(\nu)
\end{equation}
Clearly, the norm
\begin{equation}
	\begin{split}
		\|(\Phi(\cdot),(\eta(\cdot),\zeta(\cdot)))\|^{2}_{\mathcal{Z}^{T}_{0}} \coloneq\\= |\Phi(0)|^{2}_{\mathcal{L}^{\wedge}_{m}} + \| \Phi(\cdot) \|^{2}_{L_{2}(0,\infty;\mathcal{L}^{\wedge}_{m})} + \| \eta(\cdot) \|^{2}_{L_{2}(0,\infty;\mathbb{U}^{\wedge}_{m})} + \| \zeta(\cdot) \|^{2}_{L_{2}(0,\infty;\mathcal{L}^{\wedge}_{m})}
	\end{split}
\end{equation} 
determines a Hilbert space structure on $\mathcal{Z}^{T}_{0}(\nu)$. Similarly, we define such a space for $T=\infty$ and denote it simply as $\mathcal{Z}_{0}(\nu)$.

With any Hermitian form $\mathcal{F}^{\mathbb{C}}$ as in \eqref{EQ: QuadraticFormComplexificationCompoundDelayGeneral}, we associate the quadratic functional $\mathcal{J}^{T}_{\mathcal{F}^{\mathbb{C}}}$ on $\mathcal{Z}^{T}_{0}(\nu)$ as follows:
\begin{equation}
	\label{EQ: DelayCompIntQuadFuncFiniteT}
	\mathcal{J}^{T}_{\mathcal{F}^{\mathbb{C}}}(\Phi(\cdot),(\eta(\cdot),\zeta(\cdot))) \coloneq \int_{0}^{T} \mathcal{G}^{\mathbb{C}}\left((\mathcal{I}_{C^{\wedge}_{m}}\Phi)(t),\eta(t)\right)dt.
\end{equation}
Here $(\mathcal{I}_{C^{\wedge}_{m}}\Phi)(t)$ is given for almost all $t \in [0,T]$ by the following sum in $\mathbb{M}^{\wedge}_{m}$:
\begin{equation}
	\label{EQ: DelayCompoundOperatorICWedgeDefinition}
	\left(\mathcal{I}_{C^{\wedge}_{m}}\Phi\right)(t) \coloneq \sum_{k=0}^{m-1}\sum_{j_{1}\ldots j_{k}}\sum_{j \notin \{j_{1}\ldots j_{k}\}} \left(\mathcal{I}_{C^{(k)}_{j,J(j)}}R_{jj_{1}\ldots j_{k}}\Phi\right)(t),
\end{equation}
where the operators $\mathcal{I}_{C^{(k)}_{j,J(j)}}$ are delivered by Theorem \ref{TH: PointwiseMeasurementOperatorOnAgalmanatedSpace} applied to $C^{(k)}_{j,J(j)}$ from \eqref{EQ: DelayCjDefinition}, $p=2$, and $\rho=\rho_{\nu}$ with $\rho_{\nu}(t)=e^{\nu t}$. By Theorem \ref{TH: StructuralCauchyFormulaCompoundDelay}, the functional $\mathcal{J}^{T}_{\mathcal{F}^{\mathbb{C}}}$ is well defined on $\mathcal{Z}^{T}_{0}(\nu)$. Furthermore, Theorem \ref{TH: DelayCompoundStructuralCauchyFormulaNormEstimate} gives a constant $C_{\mathcal{F}}>0$, which is independent of $T$, such that
\begin{equation}
	\int_{0}^{T}\left|\mathcal{G}^{\mathbb{C}}\left((\mathcal{I}_{C^{\wedge}_{m}}\Phi)(t),\eta(t)\right)\right|dt \leq C_{\mathcal{F}} \cdot \| (\Phi(\cdot),( \eta(\cdot),\zeta(\cdot))) \|^{2}_{\mathcal{Z}^{T}_{0}}.
\end{equation}
For $\Phi_{0} \in \mathcal{D}(A^{[\wedge m]})$, $\eta(\cdot) \in C^{1}([0,T];\mathbb{U}^{\wedge}_{m})$, and $\zeta(\cdot) \in C^{1}([0,T];\mathcal{L}^{\wedge}_{m})$, we have that
\begin{equation}
	\mathcal{J}^{T}_{\mathcal{F}^{\mathbb{C}}}(\Phi(\cdot),(\eta(\cdot),\zeta(\cdot))) = \int_{0}^{T}\mathcal{G}^{\mathbb{C}}\left(C^{\wedge}_{m}\Phi(t),\eta(t)\right)dt,
\end{equation}
what follows from \eqref{EQ: SolutionsDelayCompoundSmoothingSpaces}, \eqref{EQ: OperatorIGammaIdentityContAgalmanated}, and Proposition \ref{PROP: EmbeddingWDiagToEmDelay}.

The above considerations also apply to the case $T=\infty$. This yields a quadratic functional on $\mathcal{Z}_{0}(\nu)$, denoted by $\mathcal{J}_{\mathcal{F}^{\mathbb{C}}}$. We write it as follows:
\begin{equation}
	\label{EQ: QuadraticFunctionalDelayCompoundOnProcesses}
	\mathcal{J}_{\mathcal{F}^{\mathbb{C}}}(\Phi(\cdot),(\eta(\cdot),\zeta(\cdot))) \coloneq \int_{0}^{\infty}\mathcal{G}^{\mathbb{C}}\left((\mathcal{I}_{C^{\wedge}_{m}}\Phi)(t),\eta(t) \right) dt.
\end{equation}
Note that using the same symbol for the operator $\mathcal{I}_{C^{\wedge}_{m}}$ on different time intervals is justified by Lemma \ref{LEM: CommutativeOperatorIGammaEmbr} and Theorem \ref{TH: EmbeddingAgalmanatedIntoEmbracing}. Furthermore, let $\mathcal{R}_{T} \colon \mathcal{Z}_{0}(\nu) \to \mathcal{Z}^{T}_{0}$ be the operator that restricts functions to $[0,T]$. Then it is clear that $\mathcal{J}_{\mathcal{F}^{\mathbb{C}}}$ is the pointwise limit of $\mathcal{J}^{T}_{\mathcal{F}^{\mathbb{C}}} \circ \mathcal{R}_{T}$ as $T \to \infty$. Thus, these integral quadratic functionals are well defined on the spaces of processes and agree in the limit.

For any $(\Phi(\cdot),(\eta(\cdot),\zeta(\cdot))) \in \mathfrak{M}_{0}(\nu)$, consider the Fourier transforms $\hat{\Phi}(\cdot) \in L_{2}(\mathbb{R};\mathcal{L}^{\wedge}_{m})$, $\hat{\eta}(\cdot) \in L_{2}(\mathbb{R};\mathbb{U}^{\wedge}_{m})$, and $\hat{\zeta}(\cdot) \in L_{2}(\mathbb{R};\mathcal{L}^{\wedge}_{m})$ of $\Phi(\cdot)$, $\eta(\cdot)$, and $\zeta(\cdot)$, respectively, after extending them by zero to the negative semiaxis. Since $A^{[\wedge m]}$ is the generator of a $C_{0}$-semigroup, we have that $\hat{\Phi}(\omega) \in \mathcal{D}(A^{[\wedge m]})$ for almost all $\omega \in \mathbb{R}$ and 
\begin{equation}
	\label{EQ: SolutionsCompoundDelayInhomogeneousFourierIdentity}
	i\omega \hat{\Phi}(\omega) = (A^{[\wedge m]}+\nu I)\hat{\Phi}(\omega) + B^{\wedge}_{m} \hat{\eta}(\omega) + \hat{\zeta}(\omega).
\end{equation}
We finalize this section, by establishing the following key lemma.
\begin{lemma}
	\label{LEM: FourierTransformCompoundDelay}
	For any $(\Phi(\cdot),(\eta(\cdot),\zeta(\cdot))) \in \mathfrak{M}_{0}(\nu)$ we have the identity:
	\begin{equation}
		\mathcal{J}_{\mathcal{F}^{\mathbb{C}}}(\Phi(\cdot),(\eta(\cdot),\zeta(\cdot))) = \int_{-\infty}^{+\infty}\mathcal{G}^{\mathbb{C}}(C^{\wedge}_{m}\hat{\Phi}(\omega),\hat{\eta}(\omega))d\omega.
	\end{equation}
\end{lemma}
\begin{proof}
    From the boundedness of $\mathcal{G}^{\mathbb{C}}$ and the Parseval identity, we obtain that
    \begin{equation}
    	\int_{0}^{\infty} \mathcal{G}^{\mathbb{C}}\left((\mathcal{I}_{C^{\wedge}_{m}}\Phi)(t),\eta(t) \right) dt = \int_{-\infty}^{+\infty}\mathcal{G}^{\mathbb{C}}\left( (\widehat{\mathcal{I}_{C^{\wedge}_{m}}\Phi})(\omega), \hat{\eta}(\omega) \right)d\omega,
    \end{equation}
    where $\widehat{\mathcal{I}_{C^{\wedge}_{m}}\Phi}$ is the Fourier transform in $L_{2}(\mathbb{R};\mathbb{M}^{\wedge}_{m})$ of $\mathcal{I}_{C^{\wedge}_{m}}\Phi \in L_{2}(0,\infty;\mathbb{M}^{\wedge}_{m})$ after extending the latter by zero to the negative semiaxis. It remains to show that $(\widehat{\mathcal{I}_{C^{\wedge}_{m}}\Phi})(\omega) = C^{\wedge}_{m}\hat{\Phi}(\omega)$ for almost all $\omega \in \mathbb{R}$. 
    
    From \eqref{EQ: DelayCompoundOperatorICWedgeDefinition} with $T = \infty$, we obtain that
    \begin{equation}
    	\label{EQ: DelayCompoundFourierLemmaICwedgeDecomp}
    	\left(\widehat{\mathcal{I}_{C^{\wedge}_{m}}\Phi}\right)(\omega) \coloneq \sum_{k=0}^{m-1}\sum_{j_{1}\ldots j_{k}}\sum_{j \notin \{j_{1}\ldots j_{k}\}} \widehat{\left(\mathcal{I}_{C^{(k)}_{j,J(j)}}R_{jj_{1}\ldots j_{k}}\Phi\right)}(\omega),
    \end{equation}
    where the hats denote Fourier transforms in appropriate spaces.
    
    Let $L_{2}$ stand for $L_{2}((-\tau,0)^{k+1};(\mathbb{C}^{n})^{\otimes m} )$ for a given $k \in \{0,\ldots,m-1\}$. By Theorem \ref{TH: StructuralCauchyFormulaCompoundDelay}, the restriction $R_{jj_{1}\ldots j_{k}}\Phi$ belongs to the space $\mathcal{A}^{2}_{\rho_{\nu}}(0,\infty; L_{2} )$, which is continuously embedded into $\mathcal{E}_{2}(0,\infty;L_{2})$, thanks to Theorem \ref{TH: EmbeddingAgalmanatedIntoEmbracing}. Moreover, the latter space is embedded into $\mathcal{E}_{2}(\mathbb{R};L_{2})$ by extending functions by zero to the negative semiaxis. Up to the embedding, by Theorem \ref{TH: EmbracingFourierCommutesIGamma}, the Fourier transform $R_{jj_{1}\ldots j_{k}}\hat{\Phi}$ of $R_{jj_{1}\ldots j_{k}}\Phi$ also belongs to $\mathcal{E}_{2}(\mathbb{R};L_{2})$. From \eqref{EQ: SolutionsCompoundDelayInhomogeneousFourierIdentity} we have that $\hat{\Phi}(\cdot) \in L_{2,loc}(\mathbb{R};\mathcal{D}(A^{[\wedge m]}))$, and hence, by Theorem \ref{TH: CompoundDelayDomainDescription}, $R_{jj_{1}\ldots j_{k}}\hat{\Phi}(\cdot)$ belongs to $L_{2,loc}(\mathbb{R};\mathcal{W}^{2}_{D}((-\tau,0)^{k+1};(\mathbb{C}^{n})^{\otimes m}))$ for any indices as in \eqref{EQ: DelayCompoundFourierLemmaICwedgeDecomp}. Finally, Theorem \ref{TH: EmbracingFourierCommutesIGamma}, Proposition \ref{PROP: EmbeddingWDiagToEmDelay}, and Corollary \ref{COR: RelaxedComputationMeasurementOperatorsOnEmbracing} give that
    \begin{equation}
    	\widehat{\left(\mathcal{I}_{C^{(k)}_{j,J(j)}}R_{jj_{1}\ldots j_{k}}\Phi\right)}(\omega) = \left(\mathcal{I}_{C^{(k)}_{j,J(j)}}R_{jj_{1}\ldots j_{k}}\hat{\Phi}\right)(\omega) = C^{(k)}_{j,J(j)}R_{jj_{1}\ldots j_{k}}\hat{\Phi}(\omega)
    \end{equation}
    for almost all $\omega \in \mathbb{R}$. According to \eqref{EQ: DelayCompoundFourierLemmaICwedgeDecomp}, this gives $(\widehat{\mathcal{I}_{C^{\wedge}_{m}}\Phi})(\omega) = C^{\wedge}_{m}\hat{\Phi}(\omega)$ for almost all $\omega \in \mathbb{R}$.
\end{proof}

\subsection{Frequency inequalities for spectral comparison}
\label{SUBSEC: DelayCompoundFrequencyInequalities}
In this section, we return to the context of real spaces and operators.

With any quadratic form $\mathcal{F}$, as in \eqref{EQ: QuadraticFormRealCompoundDelayGeneral}, and $\nu_{0} \in \mathbb{R}$ such that the vertical line $-\nu_{0} + i \mathbb{R}$ avoids the spectrum of $A^{[\wedge m]}$, we associate the following frequency inequality:
\begin{description}[before=\let\makelabel\descriptionlabel]
	\item[\textbf{(FI)}\refstepcounter{desccount}\label{DESC: FREQINEQ}] For some $\delta>0$ and any $p \in \mathbb{C}$ with $\operatorname{Re}p = -\nu_{0}$, we have
	\begin{equation}
		\label{EQ: FreqIneqDelayCompoundGeneralForm}
		\mathcal{F}^{\mathbb{C}}((p I - A^{[\wedge m]})^{-1}B^{\wedge}_{m}\eta , \eta ) \leq -\delta \left|\eta\right|^{2}_{(\mathbb{U}^{\wedge}_{m})^{\mathbb{C}}} \qquad \text{for all} \ \eta \in \left(\mathbb{U}^{\wedge}_{m}\right)^{\mathbb{C}}.
	\end{equation}
\end{description}

One can describe \eqref{EQ: FreqIneqDelayCompoundGeneralForm} in terms of the transfer operator 
\begin{equation}
	W(p) \coloneq C^{\wedge}_{m}(p I - A^{[\wedge m]})^{-1}B^{\wedge}_{m},
\end{equation}
which is defined at least for regular points $p \in \mathbb{C}$ of $A^{[\wedge m]}$. Note that $W(p)$ is a bounded linear operator between the complexifications $\left(\mathbb{U}^{\wedge}_{m}\right)^{\mathbb{C}} = \mathbb{U}^{\wedge}_{m} \otimes \mathbb{C}$ and $\left( \mathbb{M}^{\wedge}_{m} \right)^{\mathbb{C}} = \mathbb{M}^{\wedge}_{m} \otimes \mathbb{C}$. Again, for convenience we have omitted mentioning the complexifications of $A^{[\wedge m]}$, $B^{\wedge}_{m}$, and $C^{\wedge}_{m}$. Then it is clear that \eqref{EQ: FreqIneqDelayCompoundGeneralForm} is equivalent to 
\begin{equation}
	\label{EQ: FreqIneqDelayCompoundGeneralFormViaGW}
	\sup_{ \omega \in \mathbb{R}} \mathcal{G}^{\mathbb{C}}(W(-\nu_{0}+i\omega)\eta,\eta ) \leq -\delta \left|\eta\right|^{2}_{\left(\mathbb{U}^{\wedge}_{m}\right)^{\mathbb{C}}} \qquad \text{for all} \ \eta \in \left(\mathbb{U}^{\wedge}_{m}\right)^{\mathbb{C}}.
\end{equation}

Recall that $A^{[\wedge m]}$ generates an eventually compact semigroup $G^{\wedge m}$, and the spectrum of $A^{[\wedge m]}$ is completely described in Proposition \ref{PROP: OperatorDelayCompoundSpectra} through the spectrum of $A$. By \cite[Chapter V, Theorem 3.1]{EngelNagel2000}, for any $\nu_{0} \in \mathbb{R}$ there are the finite-dimensional spectral subspace $\mathcal{L}^{u}_{m}(\nu_{0})$ of $A^{[\wedge m]}$ corresponding to the eigenvalues $\lambda$ with $\operatorname{Re}\lambda > -\nu_{0}$ and the complementary spectral subspace $\mathcal{L}^{s}_{m}(\nu_{0})$, i.e., $\mathcal{L}^{\wedge}_{m} = \mathcal{L}^{u}_{m}(\nu_{0}) \oplus \mathcal{L}^{s}_{m}(\nu_{0})$. Both spectral subspaces are positively invariant with respect to $G^{\wedge m}$, and the growth bounds of $G^{\wedge m}$ on them are determined by the spectral bounds of the corresponding restrictions of $A^{[\wedge m]}$. In particular, for any $\varepsilon > 0$ there exists $M_{\varepsilon}>0$ such that for all $t \geq 0$ the following holds:
\begin{equation}
	\label{EQ: DelayCompoundLinearExponentialDichotomy}
	\begin{split}
		\left|e^{\nu_{0} t}G^{\wedge m}(t) \Phi_{0} \right|_{\mathcal{L}^{\wedge}_{m}} &\leq M_{\varepsilon} e^{-\varepsilon t}|\Phi_{0}|_{\mathcal{L}^{\wedge}_{m}} \qquad \text{for all} \ \Phi_{0} \in \mathcal{L}^{s}_{m}(\nu_{0}),\\
		\left|e^{-\nu_{0} t}G^{\wedge m}(-t) \Phi_{0} \right|_{\mathcal{L}^{\wedge}_{m}} &\leq M_{\varepsilon} e^{-\varepsilon t}|\Phi_{0}|_{\mathcal{L}^{\wedge}_{m}} \qquad \text{for all} \ \Phi_{0} \in \mathcal{L}^{u}_{m}(\nu_{0}),\\
	\end{split}
\end{equation}
where the past $G^{\wedge m}(-t) \Phi_{0}$ of $\Phi_{0} \in \mathcal{L}^{u}_{m}(\nu_{0})$ on $\mathcal{L}^{u}_{m}(\nu_{0})$ with respect to $G^{\wedge m}$ is uniquely determined since $\mathcal{L}^{u}_{m}(\nu_{0})$ is finite-dimensional.

For the next theorem, we assume that $\mathcal{F}$ has the form as in \eqref{EQ: QuadraticFormRealCompoundDelayGeneral}, satisfies \eqref{EQ: QuadraticConstraintCompoundDelay}, and $\mathcal{F}(\Phi,0) \geq 0$ for all $\Phi \in \mathbb{E}^{\wedge}_{m}$. For brevity, any such $\mathcal{F}$ is called \textit{admissible}.
\begin{theorem}
	\label{TH: QuadraticFunctionalDelayCompoundTheorem}
	For an admissible quadratic form $\mathcal{F}$, let the frequency inequality in
	\nameref{DESC: FREQINEQ} be satisfied for some $\nu_{0} \in \mathbb{R}$. Then there exists a bounded self-adjoint operator $P$ in $\mathcal{L}^{\wedge}_{m}$ such that for its quadratic form $V(\Phi)\coloneq\langle\Phi,P\Phi\rangle_{\mathcal{L}^{\wedge}_{m}}$ and some $\delta_{V}>0$, the cocycle $\Xi_{m}$ in $\mathcal{L}^{\wedge}_{m}$ generated by \eqref{EQ: DelayCompoundCocyclePertWedgeOperatorForm} satisfies the following inequality:
    \begin{equation}
    	\label{EQ: QuadraticIneqDelayCompound}
    	e^{2\nu_{0} t}V( \Xi^{t}_{m}(\wp, \Phi ) ) - V(\Phi) \leq -\delta_{V} \int_{0}^{t} e^{2\nu_{0} s} \left| \Xi^{s}_{m}(\wp,\Phi) \right|^{2}_{\mathcal{L}^{\wedge}_{m}}ds
    \end{equation}
    for any $t \geq 0$, $\wp \in \mathcal{P}$, and $\Phi \in \mathcal{L}^{\wedge}_{m}$. 
    
    Furthermore, $V(\cdot)$ is positive on $\mathcal{L}^{s}_{m}(\nu_{0})$ and negative on $\mathcal{L}^{u}_{m}(\nu_{0})$, i.e., $V(\Phi) > 0$ for any nonzero $\Phi \in \mathcal{L}^{s}_{m}(\nu_{0})$ and $V(\Phi) < 0$ for any nonzero $\Phi \in \mathcal{L}^{u}_{m}(\nu_{0})$.
\end{theorem}
\begin{proof}
	First, we apply \cite[Theorem 2.1]{Anikushin2020FreqDelay} to the pair $(A^{[\wedge m]}+\nu_{0} I, B^{\wedge}_{m})$ and the form $\mathcal{F}$ as follows. The statement of the referenced theorem involves abstract spaces designated therein as $\mathbb{E}_{0}$, $\mathbb{E}$, $\mathbb{H}$, and $\mathbb{W}$ and the abstract conditions designated therein as \textbf{(RES)}, \textbf{(FT)}, and \textbf{(QF)}, with the frequency inequality as the icing on the cake. To put this into our context, we set $\mathbb{E}_{0} = \mathbb{H}=\mathbb{W} \coloneq \mathcal{L}^{\wedge}_{m}$ and $\mathbb{E} \coloneq \mathbb{E}^{\wedge}_{m}$, where $\mathbb{E}^{\wedge}_{m}$ is the intermediate space from \eqref{EQ: MyLovelyScaleDelayCompoundAntiSymmetric}.
	
	Then \textbf{(RES)}, which is mainly concerned with uniform resolvent bounds in intermediate spaces, in our situation is precisely the statement of Corollary \ref{COR: DelayCompoundRESProperty} combined with the admissibility of $\mathcal{F}$. Next, \textbf{(FT)}, which is concerned with the computation of the quadratic functional via the Fourier transform, in our situation is precisely the statement of Lemma \ref{LEM: FourierTransformCompoundDelay}. Finally, \textbf{(QF)}, which requires the agreement between integral quadratic functionals over processes on finite and infinite time intervals, is discussed below \eqref{EQ: QuadraticFunctionalDelayCompoundOnProcesses}. This, together with the fulfillment of the frequency inequality in
	\nameref{DESC: FREQINEQ}, constitutes the conditions of the referenced theorem in our situation. 
	
	Now, \cite[Theorem 2.1]{Anikushin2020FreqDelay} ensures the existence of a bounded self-adjoint operator $P$ in $\mathcal{L}^{\wedge}_{m}$ such that for its quadratic form $V(\Phi) \coloneq \langle\Phi,P\Phi\rangle_{\mathcal{L}^{\wedge}_{m}}$ and some $\delta_{V}>0$ the following holds:
	\begin{equation}
		\label{EQ: DelayCompoundFreqThAppl1}
		\begin{split}
		      V(\Phi_{\nu_{0}}(t)) - V(\Phi_{0}) + \int_{0}^{t}\mathcal{F}(\Phi_{\nu_{0}}(s),\eta_{\nu_{0}}(s))ds \leq \\ \leq -\delta_{V}\int_{0}^{t}\left(|\Phi_{\nu_{0}}(s)|^{2}_{\mathcal{L}^{\wedge}_{m}} + |\eta_{\nu_{0}}(s)|^{2}_{\mathbb{U}^{\wedge}_{m}}\right)ds
		\end{split}
	\end{equation}
    for all $T>0$, $t \in [0,T]$, $\Phi_{0} \in \mathcal{L}^{\wedge}_{m}$, and $(\Phi_{\nu_{0}}(\cdot),\eta_{\nu_{0}}(\cdot))$ solving \eqref{EQ: CompoundDelayControlSystem} with $\nu \coloneq \nu_{0}$ on $[0,T]$ such that $\Phi_{\nu_{0}}(0) = \Phi_{0}$. More rigorously, the integral part with $\mathcal{F}$ in \eqref{EQ: DelayCompoundFreqThAppl1} should be interpreted similarly to \eqref{EQ: DelayCompIntQuadFuncFiniteT}.
    
    Since we have that $\Phi_{\nu_{0}}(t) = e^{\nu_{0} t}\Phi(t)$ and $\eta_{\nu_{0}}(t) = e^{\nu_{0}t}\eta(t)$, where the pair $(\Phi(\cdot), \eta(\cdot))$ solves \eqref{EQ: CompoundDelayControlSystem} with $\nu \coloneq 0$, we can deduce from \eqref{EQ: DelayCompoundFreqThAppl1} the following:
    \begin{equation}
    	\label{EQ: DelayCompoundFreqApplKeySubstitution}
    	e^{2\nu_{0}t} V(\Phi(t)) - V(\Phi_{0}) + \int_{0}^{t}e^{2\nu_{0}s}\mathcal{F}(\Phi(s),\eta(s))ds \leq -\delta_{V} \int_{0}^{t}e^{2\nu_{0} s}|\Phi(s)|^{2}_{\mathcal{L}^{\wedge}_{m}}ds.
    \end{equation}
    
    By Theorem \ref{TH: TensorCompoundCocycleDelayDescription}, if $\Phi_{0} \in \mathcal{D}(A)^{\odot m} \cap \mathcal{L}^{\wedge}_{m}$ and $\Phi(t)=\Xi^{t}_{m}(\wp,\Phi_{0})$ for $t \in [0,T]$, then $\eta(s) = F^{\wedge}_{m}(\pi^{s}(\wp))C^{\wedge}_{m}\Phi(s)$ for $s \in [0,T]$. Substituting this into \eqref{EQ: DelayCompoundFreqApplKeySubstitution} and using \eqref{EQ: QuadraticConstraintCompoundDelay}, we obtain \eqref{EQ: QuadraticIneqDelayCompound}. For general $\Phi_{0}$, the last inequality is obtained by continuity.
    
    Substituting $\eta_{\nu_{0}}(\cdot) \equiv 0$ into \eqref{EQ: DelayCompoundFreqThAppl1} and using that $\mathcal{F}(\Phi,0) \geq 0$ for all $\Phi \in \mathbb{E}^{\wedge}_{m}$, we obtain
    \begin{equation}
    	\label{EQ: DelayCompoundFreqThApplLyapIneq}
    	V(\Phi_{\nu_{0}}(t)) - V(\Phi_{0}) \leq -\delta_{V} \int_{0}^{t}|\Phi_{\nu_{0}}(s)|^{2}_{\mathcal{L}^{\wedge}_{m}}ds
    \end{equation}
    for any $\Phi_{0} \in \mathcal{L}^{\wedge}_{m}$. Using the exponential dichotomy from \eqref{EQ: DelayCompoundLinearExponentialDichotomy}, we obtain the desired sign properties of $V(\cdot)$ by passing to the limits as $t \to \pm \infty$ for appropriate initial data $\Phi_{0}$ in \eqref{EQ: DelayCompoundFreqThApplLyapIneq}, see \cite[Theorem 5]{Anikushin2020FreqParab}.
\end{proof}

\begin{remark}
	If $\operatorname{dim}\mathcal{L}^{u}_{m}(\nu_{0}) = 0$, from \eqref{EQ: QuadraticIneqDelayCompound} it can be deduced that $\Xi_{m}$ has the uniform exponential growth bounded from above by the exponent $-\nu_{0}$, i.e., for some $M(\nu_{0})>0$ we have that
	\begin{equation}
		\label{EQ: CompoundCocyclePertExponentialDecay}
		|\Xi^{t}_{m}(\wp, \Phi) |_{\mathcal{L}^{\wedge}_{m}} \leq M(\nu_{0}) e^{-\nu_{0} t} \cdot | \Phi |_{\mathcal{L}^{\wedge}_{m}}
	\end{equation}
	for all $t \geq 0$, $\wp \in \mathcal{P}$, and $\Phi \in \mathcal{L}^{\wedge}_{m}$; see \cite[Corollary 3.2]{AnikushinRomanov2023FreqConds}. In particular, if $\nu_{0} > 0$, then $\Xi_{m}$ is uniformly exponentially stable. This case with $m=2$ is important for applications to global stability problems, see \cite{AnikushinRomanov2023FreqConds} and Section \ref{SEC: PerspectivesDelayCompound} for examples.
\end{remark}
\begin{remark}
	In the case $(\mathcal{P},\pi)$ is a flow, from \eqref{EQ: QuadraticIneqDelayCompound} it can be obtained that $-\nu_{0}$ is a gap of rank $\operatorname{dim}\mathcal{L}^{u}_{m}(\nu_{0})$ in the Sacker--Sell spectrum of $\Xi_{m}$ (see \cite{SackerSell1994}), i.e., the cocycle with time $t$-mappings $e^{\nu_{0} t}\Xi^{t}_{m}(\wp, \cdot)$ admits a uniform exponential dichotomy with unstable bundle of rank $\operatorname{dim}\mathcal{L}^{u}_{m}(\nu_{0})$. To construct the corresponding bundles under \eqref{EQ: QuadraticIneqDelayCompound}, one can use the results of our work \cite{Anikushin2020Geom}, for which it is important that $\Xi_{m}$ is uniformly eventually compact.
\end{remark}
\begin{remark}
	Except for the case of scalar equations ($n=1$) or $m=1$, frequency inequalities depend on the delay scaling $t \mapsto \varkappa t$ with $\varkappa > 0$. This reflects the fact that the transfer operators of the original and rescaled systems are not unitarily equivalent, see \cite{AnikushinRomanovQBC2026}. In practice, this means that one should consider the frequency conditions for the family of rescaled systems and optimize them over $\varkappa$ to obtain sharper conditions.
\end{remark}

For $\mathcal{F}$ given by \eqref{EQ: DelayCompoundQuadraticFormsAbstractGeneral}, the frequency inequality in \eqref{EQ: FreqIneqDelayCompoundGeneralFormViaGW} takes the following form:
\begin{equation}
	\label{EQ: FrequencyInequalityCompoundCBounded}
	\sup_{\omega \in \mathbb{R}}\| W(-\nu_{0} + i \omega) \|_{\mathcal{L}((\mathbb{U}^{\wedge}_{m})^{\mathbb{C}}, (\mathbb{M}^{\wedge}_{m})^{\mathbb{C}})}  < \Lambda^{-1}.
\end{equation}
Since the left-hand side is finite and depends only on the triple $(A^{[\wedge m]}$, $B^{\wedge}_{m}$, $C^{\wedge}_{m}$), the inequality \eqref{EQ: FrequencyInequalityCompoundCBounded} is always fulfilled provided that $\Lambda$ is sufficiently small. This reflects the well-known fact that uniform exponential dichotomies are robust under small perturbations. In our case, \eqref{EQ: FrequencyInequalityCompoundCBounded} is an explicit condition for the preservation of similar properties of a stationary system in the considered class of nonautonomous perturbations. Such frequency conditions are optimal in the class of perturbations, see \cite{Anikushin2020FreqParab} for related discussions on this matter.

Let us consider a more general class of quadratic constraints arising from distinguishing several monotone or Lipschitz inputs in the feedback determined by $F'$. For this, suppose that there exists an integer $n_{\mathbb{U}} \geq 1$, and let $\mathbb{U}$ split into the inner orthogonal sum:
\begin{equation}
	\mathbb{U} = \bigoplus_{i=1}^{n_{\mathbb{U}}} \mathbb{U}^{(i)},
\end{equation}
of certain subspaces $\mathbb{U}^{(i)}$. According to this, for any $\eta \in \mathbb{U}$, we write $\eta = (\eta_{1},\ldots,\eta_{n_{\mathbb{U}}})$, where $\eta_{i} \in \mathbb{U}^{(i)}$ for each $i$.

Now for each $i \in \{1,\ldots, n_{\mathbb{U}}\}$, suppose that there exist subspaces $\mathbb{M}^{(i)}$ of $\mathbb{M}$ such that
\begin{equation}
	F'(\wp) = (F'_{1}(\wp)\Pi^{(1)}_{\mathbb{M}}, \ldots, F'_{n_{\mathbb{U}}}(\wp)\Pi^{(n_{\mathbb{U}})}_{\mathbb{M}}) \qquad \text{for any} \ \wp \in \mathcal{P},
\end{equation}
where $\Pi^{(i)}_{\mathbb{M}}$ is the orthogonal projector onto $\mathbb{M}^{(i)}$, and $F'_{i}(\wp)$ is a linear operator from $\mathbb{M}^{(i)}$ to $\mathbb{U}^{(i)}$. For different $i$, the spaces $\mathbb{M}^{(i)}$ may intersect (distinct inputs may use the same measurements).

Moreover, let each $F'_{i}$ satisfy one of the following conditions (depending on $i$):
\begin{description}[before=\let\makelabel\descriptionlabel]
	\item[\textbf{(LipC)}\refstepcounter{desccount}\label{DESC: LipC}] There exists a constant $\Lambda^{(i)} > 0$ such that
	\begin{equation}
		\|F'_{i}(\wp)\|_{\mathcal{L}(\mathbb{M}^{(i)}; \mathbb{U}^{(i)})} \leq \Lambda^{(i)} \qquad \text{for any} \ \wp \in \mathcal{P};
	\end{equation}
	\item[\textbf{(MonC)}\refstepcounter{desccount}\label{DESC: MonC}] There exist linear mappings 
	\begin{equation}
		h_{i} \colon \mathbb{M} \to \mathbb{U} \quad \text{and} \quad h^{(-1)}_{i} \colon \mathbb{U} \to \mathbb{M}
	\end{equation}
	such that $h_{i}$ takes $\mathbb{M}^{(i)}$ isometrically onto $\mathbb{U}^{(i)}$, $h^{(-1)}_{i}$ takes $\mathbb{U}^{(i)}$ isometrically onto $\mathbb{M}^{(i)}$, and the restricted mappings are mutually inverse. Furthermore, there exist constants $\Lambda^{(i)}_{1} \leq 0 \leq \Lambda^{(i)}_{2}$ such that for any $\wp \in \mathcal{P}$, the operator $h^{(-1)}_{i} F'_{i}(\wp)$ is self-adjoint in $\mathbb{M}^{(i)}$, and any of its eigenvalues belongs to the interval $[\Lambda^{(i)}_{1},\Lambda^{(i)}_{2}]$. 
\end{description}

Note that the subspaces $\mathbb{U}^{(i)}$ and $\mathbb{M}^{(i)}$ naturally determine the associated subspaces $\mathbb{U}^{(i)}_{m}$ and $\mathbb{M}^{(i)}_{m}$ in $\mathbb{U}^{\wedge}_{m}$ and $\mathbb{M}^{\wedge}_{m}$, respectively. Let $\Pi^{(i)}_{\mathbb{U}^{\wedge}_{m}}$ and $\Pi^{(i)}_{\mathbb{M}^{\wedge}_{m}}$ denote the corresponding orthogonal projectors onto $\mathbb{U}^{(i)}_{m}$ and $\mathbb{M}^{(i)}_{m}$, respectively. Similarly, let $F^{\wedge}_{m,i}(\wp)$ denote the mapping from $\mathbb{M}^{(i)}_{m}$ to $\mathbb{U}^{(i)}_{m}$  induced by $F'_{i}(\wp)$, and let $h_{m, i}$ (resp. $h^{(-1)}_{m, i}$) denote the induced by $h_{i}$ (resp.  $h^{(-1)}_{i}$) from \nameref{DESC: MonC} linear mapping that takes $\mathbb{M}^{(i)}$ (resp. $\mathbb{U}^{(i)}$) into $\mathbb{U}^{(i)}$ (resp. $\mathbb{M}^{(i)}$).

With each input $F'_{m, i}$, we associate the quadratic form $\mathcal{G}_{i}(M_{i},\eta_{i})$ of $M_{i} \in \mathbb{M}^{(i)}_{m}$ and $\eta_{i} \in \mathbb{U}^{(i)}_{m}$ either given by 
\begin{equation}
	\label{EQ: PartialInputQuadraticConstraintLip}
	\mathcal{G}_{i}(M_{i},\eta_{i}) = 	|\Lambda^{(i)}|^{2}|M_{i}|^{2}_{\mathbb{M}^{\wedge}_{m}} - |\eta_{i}|^{2}_{\mathbb{U}^{\wedge}_{m}},
\end{equation}
if $F'_{i}$ satisfies \nameref{DESC: LipC}, or by 
\begin{equation}
	\label{EQ: PartialInputQuadraticConstraintMon}
	\mathcal{G}_{i}(M_{i},\eta_{i}) = -\Lambda^{(i)}_{1} \Lambda^{(i)}_{2}|M_{i}|^{2}_{\mathbb{M}^{\wedge}_{m}} + (\Lambda^{(i)}_{1} + \Lambda^{(i)}_{2}) \langle h_{m, i} M_{i}, \eta_{i} \rangle_{\mathbb{U}^{\wedge}_{m}} - |\eta_{i}|^{2}_{\mathbb{U}^{\wedge}_{m}},
\end{equation}
if $F'_{i}$ satisfies \nameref{DESC: MonC}. It can be shown that $\mathcal{G}_{i}(M_{i},0) \geq 0$ and $\mathcal{G}_{i}(M_{i}, \eta_{i}) \geq 0$ provided that $\eta_{i} = F^{\wedge}_{m,i}(\wp)M_{i}$ for any $M_{i} \in \mathbb{M}^{(i)}_{m}$. For \eqref{EQ: PartialInputQuadraticConstraintLip}, this can be done analogously to \eqref{EQ: DelayCompoundQuadraticFormsAbstractGeneral}. For \eqref{EQ: PartialInputQuadraticConstraintMon}, one should appeal to the spectral theorem, see \cite[Section 4.1]{AnikushinRomanov2023FreqConds}.

Finally, given $\varkappa_{i} \geq 0$ for $i \in \{1,\ldots,n_{\mathbb{U}}\}$, define the quadratic form $\mathcal{G}$ on $\mathbb{M}^{\wedge}_{m} \times \mathbb{U}^{\wedge}_{m}$ as follows:
\begin{equation}
	\label{EQ: DelayCompoundMixedContraintG}
	\mathcal{G}(M, \eta) \coloneq \sum_{i=1}^{n_{\mathbb{U}}} \kappa_{i} \mathcal{G}_{i}(\Pi^{(i)}_{\mathbb{M}^{\wedge}_{m}} M, \Pi^{(i)}_{\mathbb{U}^{\wedge}_{m}} \eta).
\end{equation}
From the above, it is clear that $\mathcal{F}(\Phi, \eta) \coloneq \mathcal{G}(C^{\wedge}_{m}\Phi, \eta)$ determines an admissible quadratic form in the sense of Theorem \ref{TH: QuadraticFunctionalDelayCompoundTheorem}.

Straightforward calculations show that we can write
\begin{equation}
	\mathcal{G}^{\mathbb{C}}_{i}(\Pi^{(i)}_{\mathbb{M}^{\wedge}_{m}}W(p)\eta, \Pi^{(i)}_{\mathbb{U}^{\wedge}_{m}}\eta) = \langle S_{i}(p)\eta, \eta \rangle_{\mathbb{U}^{\wedge}_{m}},
\end{equation}
where $S_{i}(p)$ is a self-adjoint operator in $(\mathbb{U}^{\wedge}_{m})^{\mathbb{C}}$ given by
\begin{equation}
	\label{EQ: BlockSAoperatorLipschitz}
	S_{i}(p) = |\Lambda^{(i)}|^{2}W^{*}(p) \Pi^{(i)}_{\mathbb{M}^{\wedge}_{m}} W(p) - \Pi^{(i)}_{\mathbb{U}^{\wedge}_{m}},
\end{equation}
if $F'_{i}$ satisfies \nameref{DESC: LipC}, or
\begin{equation}
	\label{EQ: BlockSAoperatorMonotone}
	\begin{split}
		S_{i}(p) = -\Lambda^{(i)}_{1} \Lambda^{(i)}_{2} W^{*}(p) \Pi^{(i)}_{\mathbb{M}^{\wedge}_{m}} W(p) +\\
		\frac{1}{2}(\Lambda^{(i)}_{1} + \Lambda^{(i)}_{2})\left(W^{*}(p) \Pi^{(i)}_{\mathbb{M}^{\wedge}_{m}} h^{(-1)}_{m,i} \Pi^{(i)}_{\mathbb{U}^{\wedge}_{m}} + \Pi^{(i)}_{\mathbb{U}^{\wedge}_{m}} h_{m, i} \Pi^{(i)}_{\mathbb{M}^{\wedge}_{m}} W(p)\right) - \Pi^{(i)}_{\mathbb{U}^{\wedge}_{m}},
	\end{split}
\end{equation}
if $F'_{i}$ satisfies \nameref{DESC: MonC}. 

Consequently, the frequency inequality in \nameref{DESC: FREQINEQ} for $\mathcal{F}$ as above takes the following form:
\begin{equation}
	\label{EQ: FrequencyInequalityCompoundDelayMixedConstraint}
	\sup_{\omega \in \mathbb{R}}s^{+}(\Sigma^{(\kappa)}(-\nu_{0} + i \omega) ) < 0, \quad \Sigma^{(\kappa)}(p) \coloneq \sum_{i=1}^{n_{\mathbb{U}}}\kappa_{i} S_{i}(p).
\end{equation}
where $s^{+}$ denotes the upper spectral bound of a self-adjoint operator. As discussed in Remark \ref{REM: DelayCompoundBoundActResolvent}, the transfer operator $W(p)$ is compact for $m=2$. In this case, $\Sigma^{(\kappa)}(p)$ is a compact perturbation of $-\sum_{i=1}^{n_{\mathbb{U}}} \kappa_{i} \Pi^{(i)}_{\mathbb{U}^{\wedge}_{m}}$, and, consequently, $s^{+}=s^{+}(\omega)$ is just the largest eigenvalue. Further properties will be discussed in Section \ref{SEC: PerspectivesDelayCompound}.
\begin{remark}
	Let $\bar{p}$ denote the complex conjugate of $p \in \mathbb{C}$. Since we deal with the complexification of a real problem, it is not hard to establish that $\Sigma^{(\kappa)}(p)$ and $\Sigma^{(\kappa)}(\bar{p})$ are conjugated by the conjugate-linear isometry $J$ in $(\mathbb{U}^{\wedge}_{m})^{\mathbb{C}}$ such that $J(\eta_{1} + i \eta_{2}) = \eta_{1} - i \eta_{2}$ for any $\eta_{1},\eta_{2} \in \mathbb{U}^{\wedge}_{m}$. In particular, $s^{+}(\Sigma^{(\kappa)}(p)) = s^{+}(\Sigma^{(\kappa)}(\bar{p}))$, and the supremum in \eqref{EQ: FrequencyInequalityCompoundDelayMixedConstraint} can be taken only over $\omega \geq 0$.
\end{remark}
\begin{remark}
	Without loss of generality, the sum of all $\kappa_{i} \geq 0$ can be taken to be 1. Ideally, these parameters should be chosen to minimize the left-hand side of \eqref{EQ: FrequencyInequalityCompoundDelayMixedConstraint}. This can be achieved numerically using truncations of $W(p)$ over certain orthonormal bases, see \cite{AnikushinRomanov2026Subroutine}.
\end{remark}

%% file: ResolventsComputation.tex
\section{Computation of frequency inequalities and applications}
\label{SEC: PerspectivesDelayCompound}

In this section, to highlight the importance of the developed theoretical foundations, we briefly discuss certain computational nuances and quantitative results detailed in our works \cite{AnikushinRomanov2023FreqConds,AnikushinRomanovQBC2026,AnikushinRomanov2026Subroutine}. First, we illustrate the general approach by means of \eqref{EQ: FrequencyInequalityCompoundCBounded}, but the scheme can be adapted to the general case as well.

In \eqref{EQ: FrequencyInequalityCompoundCBounded}, it is required to compute or at least estimate the norm of the transfer operator $W(p) = C^{\wedge}_{m}(p I - A^{[\wedge m]})^{-1}B^{\wedge}_{m}$ for $p=-\nu_{0} + i \omega$ with some $\nu_{0} \in \mathbb{R}$ and all $\omega \in \mathbb{R}$ as an operator from $(\mathbb{U}^{\wedge}_{m})^{\mathbb{C}}$ to $(\mathbb{M}^{\wedge}_{m})^{\mathbb{C}}$. This problem is related to the computation of the resolvent $(A^{[\wedge m]}- p I)^{-1}$. According to Theorem \ref{TH: AdditiveCompoundDelayDescription}, it reduces to solving a first-order PDE on the $m$-cube $(-\tau,0)^{m}$ with boundary conditions posed on all $k$-faces adjoint to the vertex $\{0\}^{m}$ which involve partial derivatives and delays. Furthermore, the solutions to such problems are not the usual smooth functions. Except for special, albeit nontrivial and important, cases, which will be discussed below, we do not know whether the resolvent equations can be significantly simplified.

As already discussed in Remark \ref{REM: DelayCompoundBoundActResolvent}, in the special case $m=2$, which is of great interest due to its role in global stability problems, the transfer operator is compact. In addition to this, $W(p)$ is known to be an integral operator with bounded kernel $K$ of the Hardy--Volterra type in the case of systems with single delays in the linear part, see \cite{AnikushinRomanov2023FreqConds, AnikushinRomanovQBC2026, AnikushinRomanov2026Subroutine}. For general $n$, the kernel is obtained through resolving a boundary-value problem for a first-order linear ODE in $(\mathbb{C}^{n})^{\otimes 2} \times (\mathbb{C}^{n})^{\otimes 2}$ on $[-\tau,0]$, which reduces to a system with constant coefficients. For $n=1$, an explicit formula for the kernel is obtained in \cite[Section 4.5]{AnikushinRomanov2023FreqConds}, and it is shown that the kernel is asymptotic to a kernel $\bar{K}$ as $\omega \to \infty$. However, even the norm of the integral operator with the kernel $\bar{K}$ is unlikely to be explicitly computable.

Therefore, solving the problem solely using rigorous analytical methods is difficult even in the simplest cases. From a general perspective, it is natural to approximate the transfer operator by its finite-dimensional truncations in appropriately chosen orthonormal bases in $(\mathbb{U}^{\wedge}_{m})^{\mathbb{C}}$ and $(\mathbb{M}^{\wedge}_{m})^{\mathbb{C}}$. For \eqref{EQ: FrequencyInequalityCompoundCBounded}, this is justified by the following simple lemma, whose analogs can be established for general frequency inequalities.
\begin{lemma}
	\label{LEM: RaleighQuotientSelfAdjointOptimizationApproximation}
	Suppose that $\mathbb{H}_{1}$ and $\mathbb{H}_{2}$ are separable complex Hilbert spaces with orthonormal bases $\{ e^{1}_{k} \}_{k \geq 1}$ and $\{ e^{2}_{l} \}_{l \geq 1}$ respectively. Let $W$ be a bounded linear operator from $\mathbb{H}_{1}$ to $\mathbb{H}_{2}$. For any positive integer $N$, consider the orthogonal projectors $P^{1}_{N}$ and $P^{2}_{N}$ onto $\operatorname{Span}\{e^{1}_{1},\ldots,e^{1}_{N}\}$ and $\operatorname{Span}\{e^{2}_{1},\ldots,e^{2}_{N}\}$, respectively. Then we have that
	\begin{equation}
		\label{EQ: RaleighQuotientOptimizationApproximationConvergence}
		\alpha_{N} \coloneq \|P^{2}_{N} \circ  W \circ  P^{1}_{N}\|_{\mathcal{L}(\mathbb{H}_{1};\mathbb{H}_{2})} \to \alpha \coloneq \| W \|_{\mathcal{L}(\mathbb{H}_{1};\mathbb{H}_{2})} \qquad \text{as} \ N \to \infty.
	\end{equation}
    Furthermore, $\alpha_{N} \leq \alpha_{N+1}$ for any $N$.
\end{lemma}

Applying the lemma to $\mathbb{H}_{1} \coloneq (\mathbb{U}^{\wedge}_{m})^{\mathbb{C}}$, $\mathbb{H}_{2} \coloneq (\mathbb{M}^{\wedge}_{m})^{\mathbb{C}}$, and $W \coloneq W(p)$ with the orthonormal bases chosen independently of $p$, we obtain approximations $\alpha_{N} = \alpha_{N}(\omega)$ to the norm $\alpha = \alpha(\omega)$ of $W(-\nu_{0} + i\omega)$.
\begin{lemma}
	In the above context, the functions $\alpha_{N}$ and $\alpha$ are Lipschitz on $\mathbb{R}$ with a common Lipschitz constant.
\end{lemma}
\begin{proof}
	Consider $p_{1}=-\nu_{0}+i\omega_{1}$ and $p_{2} = -\nu_{0} + i\omega_{2}$ for some $\omega_{1}, \omega_{2}\in \mathbb{R}$. Using the first resolvent identity, we obtain
	\begin{equation}
		\begin{split}
			P^{2}_{N}C^{\wedge}_{m}(A^{[\wedge m]} - p_{1} I)^{-1}B^{\wedge}_{m} P^{1}_{N} - P^{2}_{N}C^{\wedge}_{m}(A^{[\wedge m]} - p_{2} I)^{-1}B^{\wedge}_{m} P^{1}_{N} =\\= (\omega_{1}-\omega_{2}) P^{2}_{N} C^{\wedge}_{m} (A^{[\wedge m]} - p_{1}I)^{-1} (A^{[\wedge m]} - p_{2} I)^{-1} B^{\wedge}_{m}P^{1}_{N}.
		\end{split}
	\end{equation}
    From this and since $C^{\wedge}_{m} \in \mathcal{L}((\mathbb{E}^{\wedge}_{m})^{\mathbb{C}};(\mathbb{M}^{\wedge}_{m})^{\mathbb{C}})$, the conclusion follows from Corollary \ref{COR: DelayCompoundRESProperty}.
\end{proof}
\begin{remark}
	In the above arguments, we can vary $\nu_{0}$ also, and the Lipschitz constant can be chosen uniformly in $\nu_{0}$ from compact subsets.
\end{remark}

In particular, the above lemma ensures that $\alpha_{N}(\omega)$ converges to $\alpha(\omega)$ uniformly in $\omega$ from compact segments. However, \eqref{EQ: FrequencyInequalityCompoundCBounded} requires investigating the entire $\mathbb{R}$. For this, we have the following conjecture, analogues of which can be formulated for other types of inequalities.
\begin{conjecture}
	The function $\alpha(\omega)$ is asymptotically almost periodic (in the sense of Bohr) as $|\omega| \to \infty$ or even asymptotically constant.
\end{conjecture}
\noindent According to this conjecture, frequency inequalities can be tested on a finite segment. It is well known that for $m=1$ or in some other infinite-dimensional problems, the analogues of $\alpha(\omega)$ tend to $0$ as $|\omega| \to \infty$, see \cite{Anikushin2020FreqDelay,Anikushin2020FreqParab}, which ceases to be true if $m \geq 2$. In the discussed case of single delays with $m=2$ and $n=1$, the conjecture about asymptotically constant behavior is proved in \cite[Section 4.5]{AnikushinRomanov2023FreqConds}, but the proof relies on the explicit formulas for the kernels, which are not available for $n > 1$. There should exist somewhat general methods for establishing such properties. In any case, numerical experiments point to the validity of this conjecture.

In \cite{AnikushinRomanov2023FreqConds}, we presented two approaches for computing the truncations of $W(p)$. In the first one, we consider only the case $\operatorname{Re}p > \omega(G^{\wedge m})$, but with general $m$, and use the representation of the resolvent via the Laplace transform of the semigroup applied to decomposable tensors as follows:
\begin{equation}
	(p I - A^{[\wedge m]})^{-1}(\phi_{1} \wedge \cdots \wedge \phi_{m}) = \int_{0}^{\infty} e^{-pt}G^{\wedge m}(t)(\phi_{1} \wedge \cdots \wedge \phi_{m})dt,
\end{equation}
where $\phi_{1},\ldots,\phi_{m} \in \mathbb{H}$ and $p = -\nu_{0} + i \omega$ with $\omega \in \mathbb{R}$. Then, the integral is truncated to the integral over a finite segment $[0,T]$, and the remainder is shown to be exponentially decaying in the norm of $(\mathbb{E}^{\wedge m})^{\mathbb{C}}$ uniformly on the unit ball of $(\mathcal{L}^{\wedge}_{m})^{\mathbb{C}}$ and in $\omega \in \mathbb{R}$. Since $G^{\wedge m}(t)(\phi_{1} \wedge \cdots \wedge \phi_{m}) = G(t)\phi_{1} \wedge \cdots \wedge G(t)\phi_{m}$, we need to compute only the trajectories of the semigroup $G$ generated by $A$, that is, solve the usual linear delay equations. However, despite its generality, this approach requires repeated integration in the time domain and is less suitable for rigorous accuracy control.

The most efficient approximation schemes are available in the case of $m=2$ and systems of equations with a single delay in the linear part, where the integral representation of $W(p)$ is possible. This was investigated for scalar equations in \cite{AnikushinRomanov2023FreqConds} and then extended to systems of equations in \cite{AnikushinRomanovQBC2026, AnikushinRomanov2026Subroutine}. For general frequency inequalities, one can substitute the truncations of $W(p)$ into \eqref{EQ: BlockSAoperatorLipschitz} and \eqref{EQ: BlockSAoperatorMonotone} to obtain the corresponding truncations of $\Sigma^{(\kappa)}(p)$ from \eqref{EQ: FrequencyInequalityCompoundDelayMixedConstraint}. In the case under consideration, it can be shown that the principal eigenfunction $f$ of $\Sigma^{(\kappa)}(p)$ possesses sufficient smoothness, see \cite{AnikushinRomanov2026Subroutine}. In this regard, it is advisable to construct orthonormal bases using the trigonometric basis augmented by a linear function. This ensures convergence of order $O(N^{-3/2})$. It should be possible to obtain explicit estimates for the accuracy of the approximations, expressed in terms of the kernel. Alternatively, in some cases, the Schur test can be applied, see \cite{AnikushinRomanov2025Schur}. In any case, our experiments show that choosing $N \approx 15$ is sufficient to achieve accuracy on the order of thousandths.

In specific applications, we use the developed machinery to provide frequency inequalities for the uniform exponential stability of twofold compound cocycles associated with derivative cocycles over semiflows $(\mathcal{P},\pi)$ generated by autonomous nonlinear delay equations, as discussed in Remark \ref{REM: DerivativeCocycleExample}. This is of significant interest, as demonstrated by the following theorem.
\begin{theorem}
	\label{TH: GlobalStabilityDelayEqsViaBC}
	In terms of Section \ref{SUBSEC: DelayCompoundInfinitesimalDescription}, let the nonlinear autonomous delay equation in $\mathbb{R}^{n}$:
	\begin{equation}
		\dot{z}(t) = \widetilde{A}z_{t} + \widetilde{B}F(Cz_{t}),
	\end{equation}
	generate a dissipative semiflow $\pi$ in a convex subset $\mathcal{Q}$ of $\mathbb{E}$ with a global attractor\footnote{It suffices for $\mathcal{A}$ to be an attractor for compact subsets of $\mathcal{Q}$.} $\mathcal{A}$ lying in the interior of $\mathcal{Q}$. Let $\Xi$ be the linear cocycle in $\mathbb{H}$ over $(\mathcal{A},\pi)$ generated by the linearized equation:
	\begin{equation}
		\dot{x}(t) = \widetilde{A}x_{t} + \widetilde{B}F'(\pi^{t}(\wp))Cx_{t},
	\end{equation}
	where $F'(\wp) \coloneq F'(C\wp)$ for $\wp \in \mathcal{A}$. Suppose that the twofold multiplicative compound $\Xi_{2}$ of $\Xi$ is uniformly exponentially stable. Then
	\begin{enumerate}
		\item[1)] Any chain recurrent point of $(\mathcal{A},\pi)$ is an equilibrium. In particular, $\mathcal{A}$ consists of equilibria and complete orbits connecting them, which cannot form a cycle;
		\item[2)] The $\omega$-limit set of any point of $(\mathcal{Q},\pi)$, as well as the $\alpha$-limit set of any complete trajectory of $(\mathcal{A},\pi)$, is a single equilibrium;
		\item[3)] In particular, if the number of equilibria is finite (e.g., when $F$ is an analytic function in a neighborhood of $C\mathcal{A}$), then the number of connecting orbits equals the number of equilibria minus one.
	\end{enumerate}
\end{theorem}
\begin{proof}
	By means of dimension estimates (see \cite{ChepyzhovIlyin2004, Temam1997, Anikushin2023LyapExp}), the uniform exponential stability of $\Xi_{2}$ ensures that the topological dimension of $\mathcal{A}$ does not exceed $1$. Then item 1) is a purely topological fact, see \cite[Theorem 1.1]{Anikushin2026Bendixson}. Item 3) is its particular case (discussed below the just referenced theorem), except for the part concerning the finiteness of equilibria in the analytic case. The latter follows from \cite[Theorem 5.1]{Anikushin2026Bendixson}, which is applicable since the time-1 mapping $\pi^{1}$ is analytic in a neighborhood of $\mathcal{A}$, and its derivatives uniformly contract areas, thanks to the uniform exponential stability of $\Xi_{2}$.
	
	For item 2), we apply the center manifold theory as follows. Let $\phi_{0} \in \mathcal{Q}$ be fixed. Since the $\omega$-limit set $\omega(\phi_{0})$ of $\phi_{0}$ consists of chain recurrent points, item 1) ensures that $\omega(\phi_{0})$ is a subset of equilibria. Since the set $\omega(\phi_{0})$ is connected, it either coincides with an equilibrium, or any equilibrium within it is not isolated. Suppose that $\phi^{1} \in \omega(\phi_{0})$ is not isolated. Since $\Xi_{2}$ is uniformly exponentially stable, $\phi^{1}$ has exactly one characteristic root with a zero real part (if there are no such, then $\phi^{1}$ is isolated). Using the method of \cite{Anikushin2020Geom}, we can construct a one-dimensional center manifold at $\phi^{1}$, as well as the corresponding strongly stable foliation of codimension one in a small neighborhood. Since $\phi^{1}$ is not isolated, there exist equilibria $\phi^{2}, \phi^{3} \in \omega(\phi_{0})$ that belong to the neighborhood and, consequently, lie on the center manifold. Using a suitable ordering on the manifold, we may assume that $\phi^{1} < \phi^{2} < \phi^{3}$. When the trajectory of $\phi_{0}$ comes sufficiently close to $\phi^{2}$, it is exponentially attracted by a point $\phi^{*}$ on the center manifold such that $\phi^{1} < \phi^{*} < \phi^{3}$. Since such $\phi^{*}$ never leaves the neighborhood, the trajectory of $\phi_{0}$ also never leaves it. Thus, $\omega(\phi_{0}) = \omega(\phi^{*})$. On the other hand, $\omega(\phi^{*})$ cannot contain both $\phi^{1}$ and $\phi^{3}$, which leads to a contradiction. For $\alpha$-limit sets of complete trajectories of $(\mathcal{A},\pi)$, one can apply a similar reasoning.
\end{proof}

The above theorem illustrates Smith's geometric approach to global stability \cite{Smith1986HD, LiMuldowney1996SIAMGlobStab}, extended to infinite dimensions by our recent developments. Because the core part contained in item 1) of the theorem is based on a generalization of the Bendixson criterion, we call the frequency conditions for uniform exponential stability of $\Xi_{2}$ the \textit{quadratic Bendixson criterion} \cite{AnikushinRomanovQBC2026, AnikushinRomanov2026Subroutine}. Next, we discuss some examples of its applications detailed in the works \cite{AnikushinRomanov2023FreqConds, AnikushinRomanovQBC2026, AnikushinRomanov2026Subroutine}.

First, we consider the Mackey--Glass equations \cite{MackeyGlass1977}:
\begin{equation}
	\label{EQ: MackeyGlassExample}
	\dot{x}(t) =  -\gamma x(t) + \beta \frac{ x(t-\tau) }{ 1 + |x(t-\tau)|^{\kappa} },
\end{equation}
where $\gamma$, $\beta > 0$, and $\kappa > 1$ are parameters. From the physiological perspective, it is interesting to study \eqref{EQ: MackeyGlassExample} in the invariant cone of positive functions, where the model has a unique equilibrium $\phi^{+} \equiv (\beta\gamma^{-1}-1)^{1/\kappa}$ for $\beta > \gamma$. However, for our method it suffices to consider the model \eqref{EQ: MackeyGlassExample} in the whole space, where it is also dissipative (see \cite{Anikushin2023LyapExp}). For the classical parameters $\gamma=0.1$, $\beta = 0.2$, and $\kappa = 10$, after appropriate transformations (see \cite{AnikushinRomanov2023FreqConds,AnikushinRomanov2026Subroutine}), the method indicates that $\Xi_{2}$ is uniformly exponentially stable for all $\tau \in (0, 4.63]$, and hence Theorem \ref{TH: GlobalStabilityDelayEqsViaBC} applies for such parameters. Note that the estimate is close to the bifurcation parameter $\tau = \tau^{*} = 4.708\ldots$, at which $\phi^{+}$ loses its stability and a supercritical Andronov--Hopf bifurcation occurs. Using the method from \cite{Anikushin2023LyapExp}, which also provides conditions for the uniform exponential stability of $\Xi_{2}$ by employing explicitly constructed adapted metrics, the global stability of \eqref{EQ: MackeyGlassExample} can be established over an interval close to $(0, 1]$, which is significantly smaller. Moreover, the Popov frequency criterion \cite{PopovKhalanay1962} is potentially applicable to \eqref{EQ: MackeyGlassExample} within the positive cone. However, for the parameters under consideration, it does not indicate global stability, regardless of the value of $\tau$.

A more interesting analysis is offered from the method of \cite{LizTkachenkoTrofimchuk2003}, which generalizes the famous Myshkis stability criterion to nonlinear scalar equations with a single equilibrium and unimodal nonlinearities. It can be applied to \eqref{EQ: MackeyGlassExample} in the positive cone, which is justified\footnote{We caution the reader that condition (iii) of \cite[Theorem 3.2]{LizetAl2005}, relevant to the discussion, contains an error, which is likely caused by an incorrect computation of the involved derivative. The correct condition is obtained by substituting the derivative into the general condition in \cite[Equation (1.3)]{LizetAl2005}.} in \cite{LizetAl2005}. For the classical parameters, it yields the following estimate:
\begin{equation}
	\tau < -10 \cdot \left[\ln 4 + \ln\ln(20/17)\right] = 4.306\ldots,
\end{equation}
which is smaller than our bound. We emphasize that our method uses only the derivative estimate.
\begin{remark}
	It should be mentioned that the method of \cite{LizTkachenkoTrofimchuk2003} is based on a comparison, although not on the level of compound cocycles, with a linear system with constant coefficients and variable (time-dependent) delays, and the resulting conditions are optimal in the class of systems with variable delays. Since our method is more specific to constant delays, it is not surprising that it may improve the results in some situations. However, this is not always the case, as can be demonstrated by means of the Nicholson blowflies model, see \cite{AnikushinRomanov2024EffEst}.
\end{remark}

An advantage of our method is that it is not limited to either scalar equations or equations with a unique equilibrium. Among systems of equations, the method can be illustrated by the Hopfield networks with delay, which are discussed in detail in \cite{AnikushinRomanovQBC2026}. Let us mention only the example from \cite[Section 5]{YeMichelWang1994}, which is described as follows:
\begin{equation}
	\label{EQ: HopfieldExampleComparisonPaper}
	\begin{split}
		\dot{x}_{1}(t) = -1.1x_{1}(t) - 0.2 \sigma_{1}(x_{1}(t-\tau)) + \sigma_{2}(x_{2}(t-\tau)),\\
		\dot{x}_{2}(t) = -1.2x_{2}(t) + \sigma_{1}(x_{1}(t-\tau)) -0.1 \sigma_{2}(x_{2}(t-\tau)),
	\end{split}
\end{equation}
where the nonlinearities $\sigma_{1}$ and $\sigma_{2}$ are given by
\begin{equation}
	\label{EQ: HopfieldLastExample}
	\sigma_{1}(y) = \frac{2}{\pi}\arctan\left(\frac{1.4\pi}{2}y\right) \quad \text{and} \quad \sigma_{2}(y) = \frac{2}{\pi}\arctan\left( \frac{1.5\pi}{2}y\right).
\end{equation}
It can be shown that the system \eqref{EQ: HopfieldExampleComparisonPaper} is dissipative in the whole space and has 3 equilibrium states: $\phi^{0} \equiv (0,0)^{T}$ and $\phi^{\pm} \equiv (0.2139...,0.2090...)^{T}$.

As is shown in \cite{AnikushinRomanovQBC2026}, the quadratic Bendixson criterion ensures the global stability of \eqref{EQ: HopfieldExampleComparisonPaper} for $\tau \leq 1.2$, which significantly improves upon the bound  $\tau < 0.579$ obtained in \cite{YeMichelWang1994}. Since \cite{AnikushinRomanovQBC2026} considers only single vector blocks in the feedback part, the rough estimate for the derivatives ($\Lambda = 1.5$) is used in this result. Utilizing the general subroutine presented in \cite{AnikushinRomanov2026Subroutine}, we can treat the model as a system with two monotone scalar blocks, as in \nameref{DESC: MonC} with $\Lambda^{(1)}_{1}=\Lambda^{(2)}_{1}=0$, $\Lambda^{(1)}_{2}=1.4$, and $\Lambda^{(2)}_{2}=1.5$ and verify the corresponding frequency inequality \eqref{EQ: FrequencyInequalityCompoundDelayMixedConstraint}. This improves the global stability bound up to $\tau \leq 1.4$. For details, see \cite{AnikushinRomanov2026Subroutine}. It should also be noted that \cite{AnikushinRomanovQBC2026} reports the existence of hidden periodic attractors in this model for $\tau \in [2.13, 2.16]$, even though these delays exceed the threshold $\tau \approx 1.93$ at which a focus-type instability emerges at $\phi^{0}$.

In \cite{AnikushinRomanov2023FreqConds}, we also examined applications to the Suarez--Schopf delayed oscillator \cite{Suarez1988}:
\begin{equation}
	\label{EQ: SuarezSchopfStandard}
	\dot{x}(t) = x(t) - \alpha x(t-\tau) - x^{3}(t),
\end{equation}
where $\alpha \in (0,1)$ is a parameter. This equation generates a dissipative semiflow in the whole space, and there are $3$ equilibria in the model: $\phi^{0} \equiv 0$ and $\phi^{\pm} = \pm \sqrt{1-\alpha}$. As is shown in \cite{AnikushinRom2023SS}, interesting dynamics occurs in the parameter region where $\phi^{\pm}$ are linearly stable. Specifically, there exist curves along which the system undergoes nonlocal bifurcations, associated with the emergence of attracting periodic orbits (via saddle-node bifurcations) and the birth of unstable periodic orbits (by way of a homoclinic "figure-eight" bifurcation).

It can also be shown here that the quadratic Bendixson criterion allows for the improvement of certain results from the works \cite{Anikushin2023LyapExp} (based on explicit metrics) and \cite{Anikushin2022Semigroups} (based on the comparison principle from \cite{MalletParretNussbaum2013}). However, we recently established that for \eqref{EQ: SuarezSchopfStandard}, global stability can be proven in the region $\tau < 1/\alpha$ by using the following Lyapunov functional:
\begin{equation}
	V(x_{t}) = g(x(t)) + \frac{\alpha}{2}\int_{t-\tau}^{t}(s-t+\tau)x'(s)^{2}ds,
\end{equation}
where $g(x) = x^{4}/4 - (1-\alpha)x^{2}/2$. In \cite{AnikushinRom2023SS}, it was conjectured  that \eqref{EQ: SuarezSchopfStandard} is globally stable, provided that $\lambda_{1}(\phi^{0}) + \lambda_{2}(\phi^{0}) < 0$ for the leading pair $\lambda_{1}(\phi^{0})$ and $\lambda_{2}(\phi^{0})$ of characteristic roots at $\phi^{0}$ (they are always real). It can be shown that the relevant parameters lie within the region $\tau < 1/\alpha$, which is strictly larger. Thus, the conjecture is true, and no method based on the uniform exponential stability of $\Xi_{2}$ can improve upon this result, at least when such a method is applied to the global attractor.

%% file: DiagonalTranslateSemigroups.tex
\section{Diagonal translation semigroups}
\label{SEC: DiagonalTranslationSemigroups}
Throughout this section, we fix a separable Hilbert space $\mathbb{F}$, a positive integer $m > 0$, and some reals $\tau > 0$ and $p \geq 1$. Let $\Omega$ be an open bounded subset of $\mathbb{R}^{m}$. Consider the \textit{preparatory diagonal Sobolev space}\footnote{Here the subscript $D$ stands for the ``diagonal derivative''.}:
\begin{equation}
	\label{EQ: W2PreparatoryDiagonalDefinition}
	\widehat{\mathcal{W}}^{p}_{D}(\Omega;\mathbb{F})\coloneq \left\{ \Phi \in L_{p}(\Omega;\mathbb{F}) \mid \left(\sum_{j=1}^{m}\frac{\partial}{\partial \theta_{j}}\right)\Phi \in L_{p}(\Omega;\mathbb{F}) \right\}.
\end{equation}
Here the diagonal derivative $(\sum_{j=1}^{m}\frac{\partial}{\partial \theta_{j}})\Phi$ of $\Phi$ is understood in the sense of distributions, but we aim to define it independently in a way more convenient for the analysis to come. For this, let $\mathcal{L}_{0} \coloneq \{ \underline{t}=(t,\ldots,t) \in \mathbb{R}^{m} \mid t \in \mathbb{R} \}$ be the diagonal line in $\mathbb{R}^{m}$, and let $\mathcal{L}^{\perp}_{0}$ be its orthogonal complement. For $\bar{s} \in \mathcal{L}^{\perp}_{0}$, we set $\Omega(\bar{s}) \coloneq (\mathcal{L}_{0} + \bar{s}) \cap \Omega$.

Any function $\Phi$ on $\Omega(\bar{s})$ can be identified with the function $\Psi(t) \coloneq \Phi(\underline{t}+\bar{s})$ defined on an open subset of $t \in \mathcal{I}(\bar{s}) \subset \mathbb{R}$ such that $\Omega(\bar{s})$ is the union of $\{\underline{t}+\bar{s}\}$ over all $t \in \mathcal{I}(\bar{s})$. Then, by definition, $\Phi \in W^{1,p}(\Omega(\bar{s});\mathbb{F})$ if and only if $\Psi \in W^{1,p}(\mathcal{I}(\bar{s});\mathbb{F})$. Moreover, the norm in the former space is induced by that correspondence. Below, we are interested in convex domains $\Omega$, for which $\mathcal{I}(\bar{s})$ is an open interval.

Then, a function $\Phi \in L_{p}(\Omega;\mathbb{F})$ belongs to $\widehat{\mathcal{W}}^{p}_{D}(\Omega;\mathbb{F})$ if and only if there exists $\Psi \in L_{p}(\Omega;\mathbb{F})$ such that\footnote{We assume that any identities between functions restricted to the empty set are satisfied by definition.}
\begin{equation}
	\label{EQ: DiagonalSobolevPropEquivalent}
	\restr{\Phi}{\Omega(\bar{s})} \in W^{1,p}(\Omega(\bar{s});\mathbb{F}) \qquad \text{and} \qquad \frac{d}{d\underline{t}}\restr{\Phi}{\Omega(\bar{s})} = \restr{\Psi}{\Omega(\bar{s})} \quad \text{in} \ L_{p}(\Omega(\bar{s});\mathbb{F}).
\end{equation}
hold for $\mu^{m-1}_{L}$-almost all $\bar{s} \in \mathcal{L}^{\perp}_{0}$, where $\frac{d}{d\underline{t}}$ is the derivative along $\underline{1} \in \mathbb{R}^{m}$. In this context, we set $(\sum_{j=1}^{m} \frac{\partial}{\partial \theta_{j}}) \Phi \coloneq \Psi$.

We consider the natural norm $\| \cdot \|_{\widehat{\mathcal{W}}^{p}_{D}(\Omega;\mathbb{F})}$ on $\widehat{\mathcal{W}}^{p}_{D}(\Omega;\mathbb{F})$, defined as follows:
\begin{equation}
	\label{EQ: NormW2Diagonal}
	\| \Phi \|^{p}_{\widehat{\mathcal{W}}^{p}_{D}(\Omega;\mathbb{F})} \coloneq \|\Phi\|^{p}_{L_{p}(\Omega;\mathbb{F})} + \left\| \left(\sum_{j=1}^{m} \frac{\partial}{\partial \theta_{j}}\right) \Phi \right\|^{p}_{L_{p}(\Omega;\mathbb{F})}.
\end{equation}
Clearly, the space $\widehat{\mathcal{W}}^{p}_{D}(\Omega;\mathbb{F})$, equipped with the above norm, becomes a Banach space. Moreover, for $p=2$ it is a Hilbert space.

For $\Omega = \mathbb{R}^{m}$, we simply write $\mathcal{W}^{p}_{D}(\mathbb{R}^{m};\mathbb{F})$ and call this space the \textit{diagonal Sobolev space} on $\mathbb{R}^{m}$. From here, we define the \textit{diagonal Sobolev space} $\mathcal{W}^{p}_{D}(\Omega;\mathbb{F})$ on a general open domain $\Omega \subset \mathbb{R}^{m}$ as follows:
\begin{equation}
	\label{EQ: W2DiagonalDefinition}
	\mathcal{W}^{p}_{D}(\Omega;\mathbb{F}) \coloneq \left\{ \Phi \in L_{p}(\Omega;\mathbb{F}) \mid \ \Phi = \restr{\Psi}{\Omega} \ \text{for some} \ \Psi \in \mathcal{W}^{p}_{D}(\mathbb{R}^{m};\mathbb{F}) \right\}
\end{equation}
and endow it with the norm
\begin{equation}
	\label{EQ: SobolevDiagonalTrueNorm}
	\| \Phi \|_{\mathcal{W}^{p}_{D}(\Omega;\mathbb{F})} \coloneq \inf \| \Psi \|_{\mathcal{W}^{p}_{D}(\mathbb{R}^{m};\mathbb{F})},
\end{equation}
where the infimum is taken over all $\Psi \in \mathcal{W}^{p}_{D}(\mathbb{R}^{m};\mathbb{F})$ such that $\Phi = \restr{\Psi}{\Omega}$. Note that the right-hand side of \eqref{EQ: SobolevDiagonalTrueNorm} is the norm in the quotient space of $\mathcal{W}^{p}_{D}(\mathbb{R}^{m};\mathbb{F})$ over the closed subspace of functions vanishing on $\Omega$. Thus, $\mathcal{W}^{p}_{D}(\Omega;\mathbb{F})$ is a Banach space.

\begin{lemma}
	\label{LEM: ExtensionOperatorFromSobolevDiagonalToRm}
	Let $\Omega$ be a bounded convex open subset of $\mathbb{R}^{m}$ such that the lengths of $\mathcal{I}(\bar{s})$ are bounded from below uniformly in $\bar{s}$ for which $\mathcal{I}(\bar{s})$ is nonempty. Then there exists a bounded linear operator (an extension operator)
	\begin{equation}
		\label{EQ: ContinuationOperatorDiagonalTranslates}
		\mathfrak{C} \colon \widehat{\mathcal{W}}^{p}_{D}(\Omega;\mathbb{F}) \to \mathcal{W}^{p}_{D}(\mathbb{R}^{m};\mathbb{F})
	\end{equation}
    such that for any $\Phi \in \mathcal{W}^{p}_{D}(\Omega;\mathbb{F})$ we have the following properties:
    \begin{enumerate}
    	\item[1)] $(\mathfrak{C}\Phi)(\bar{s}) = \Phi(\bar{s})$ for almost all $\bar{s} \in \Omega$;
    	\item[2)] $(\mathfrak{C}\Phi)(\bar{s}) = 0$ for almost all $\bar{s} \in \mathbb{R}^{m}$ with $|\bar{s}|_{\infty} \geq r(\Omega)$, where $|\bar{s}|_{\infty}$ is the supremum norm of $\bar{s}$, and $r(\Omega)>0$ is a constant depending on $\Omega$.
    \end{enumerate} 
    In particular, the spaces $\widehat{\mathcal{W}}^{p}_{D}(\Omega;\mathbb{F})$ and $\mathcal{W}^{p}_{D}(\Omega;\mathbb{F})$ coincide as sets, and their norms are equivalent.
\end{lemma}
\begin{proof}
	To construct $\mathfrak{C}$, we extend $\Phi$ on each diagonal section $\Omega(\overline{s})$ to the entire diagonal line $\mathcal{L}_{0} + \overline{s}$, using the appropriate extension operator for the ordinary Sobolev space on the interval. More formal arguments are given below.
	
	Consider any extension operator 
	\begin{equation}
		E \colon W^{1,p}(0,1;\mathbb{F}) \to W^{1,p}(\mathbb{R};\mathbb{F})
	\end{equation}
    such that $E\Phi$ vanishes outside the interval $(-2,2)$ for any $\Phi \in W^{1,p}(0,1;\mathbb{F})$, see \cite[Section 2.2]{LionsMagenesBVP11972}. For any $-\infty \leq a < b \leq +\infty$ and $c$, let the operator $T_{c}$ take $\Phi \in W^{1,p}(a,b;\mathbb{F})$ into $T_{c}\Phi \in W^{1,p}(a+c,b+c;\mathbb{F})$ by $(T_{c}\Phi)(\theta) \coloneq \Phi(\theta-c)$ for any $\theta \in (a+c,b+c)$. For $l>0$, let $H_{l}$ take $\Phi \in W^{1,p}(a,b;\mathbb{F})$ into $H_{l}\Phi \in W^{1,p}(a\cdot l, b\cdot l;\mathbb{F})$ by $(H_{l}\Phi)(\theta) \coloneq \Phi(\theta / l)$ for any $\theta \in (a \cdot l,b \cdot l)$.
    
    Since $\Omega$ is convex, the line section $\Omega(\bar{s})$ is an open interval of $\mathcal{L}_{0} + \bar{s}$, identified, as above, with the open interval $\mathcal{I}(\bar{s}) = (a(\bar{s}), b(\bar{s}))$ of $\mathbb{R}$. Then such identifications yield isometric isomorphisms $R(\bar{s}) \colon W^{1,p}(\Omega(\bar{s});\mathbb{F}) \to W^{1,p}(a(\bar{s}), b(\bar{s});\mathbb{F})$ and $L(\bar{s}) \colon W^{1,p}(\mathcal{L}_{0} + \bar{s};\mathbb{F}) \to W^{1,p}(\mathbb{R};\mathbb{F})$.
    
    For any $\bar{s} \in \mathcal{L}^{\perp}_{0}$, we define $E(\bar{s}) \colon W^{1,p}(\Omega(\bar{s});\mathbb{F}) \to W^{1,p}(\mathcal{L}_{0} + \bar{s};\mathbb{F})$ as follows:
	\begin{equation}
		E(\bar{s}) \coloneq (L(\bar{s}))^{-1} \circ T_{a(\bar{s})} \circ H_{b(\bar{s})-a(\bar{s})} \circ E \circ  H_{1/(b(\bar{s})-a(\bar{s}))} \circ T_{-a(\bar{s})} \circ R(\bar{s}).
	\end{equation}
	According to our assumptions, there exist positive constants $l_{1} < l_{2}$ such that $b(\bar{s}) - a(\bar{s}) \in [l_{1},l_{2}]$ uniformly in $\bar{s}$ such that $b(\bar{s}) - a(\bar{s})>0$. Consequently, the norms of $E(\bar{s})$ are bounded uniformly in such $\bar{s}$.
	
	Now we define $\mathfrak{C}\Phi$ as follows:
	\begin{equation}
		\label{EQ: ExtensionOperatorDiagonalTranslatesDefinition}
		(\mathfrak{C} \Phi)(\bar{s} + \underline{t}) \coloneq \begin{cases}
			\left(E(\bar{s})\restr{\Phi}{\Omega(\bar{s})}\right)(\bar{s}+\underline{t}) \qquad &\text{if} \quad \Omega(\bar{s}) \not= \emptyset,\\
			0 \qquad &\text{if} \quad \Omega(\bar{s}) = \emptyset,
		\end{cases}
	\end{equation}
    what makes sense for $\mu^{m-1}_{L}$-almost all $\bar{s} \in \mathcal{L}^{\perp}_{0}$ and all $t \in \mathbb{R}$. By the construction and the Fubini theorem, we obtain that $\mathfrak{C}$ is the desired extension operator. 
    
    Note that for any $\Phi \in \widehat{\mathcal{W}}^{p}_{D}(\Omega;\mathbb{F})$ we have the following relations:
    \begin{equation}
    	\|\Phi\|_{\widehat{\mathcal{W}}^{p}_{D}(\Omega;\mathbb{F})} \leq \| \Phi \|_{\mathcal{W}^{p}_{D}(\Omega;\mathbb{F})} \leq \|\mathfrak{C} \Phi \|_{\mathcal{W}^{p}_{D}(\mathbb{R}^{m};\mathbb{F})} \leq \| \mathfrak{C} \| \cdot \| \Phi \|_{\widehat{\mathcal{W}}^{p}_{D}(\Omega;\mathbb{F})},
    \end{equation}
    which show the equivalence of norms.
\end{proof}

In the study of delay equations, we encounter the domain $\Omega = (-\tau,0)^{m}$ for some $\tau > 0$. Clearly, it does not satisfy the assumptions of Lemma \ref{LEM: ExtensionOperatorFromSobolevDiagonalToRm}. As will be seen, this precludes the existence of an extension operator, since the functions from $\widehat{\mathcal{W}}^{p}_{D}(\Omega;\mathbb{F})$ do not necessarily have $L_{2}$-summable traces on the boundary of $\Omega$. In this case, the space $\mathcal{W}^{p}_{D}((-\tau,0)^{m};\mathbb{F})$ is more appropriate to work with, and it will be described in Proposition \ref{PROP: DiagonalSobolevCubeDescription} below.

\begin{remark}
	\label{REM: DelayCompoundLemmaExtensionApplication}
	We will apply Lemma \ref{LEM: ExtensionOperatorFromSobolevDiagonalToRm} in the case where $\Omega = \mathring{\mathcal{C}}^{m}_{T}$ is the interior of $\mathcal{C}^{m}_{T}$ from \eqref{EQ: TheSetDiagonalDomainDefinition}, i.e., $\Omega$ is given by the union of $(-\tau,0)^{m} + \underline{t}$ over all $t \in [0,T]$ for some $T>0$.
\end{remark}

Next, consider the \textit{diagonal translation group} $\mathcal{T}_{m}$ in $L_{p}(\mathbb{R}^{m};\mathbb{F})$ with its time-$t$ mappings $\mathcal{T}_{m}(t)$ for $t \in \mathbb{R}$ given by
\begin{equation}
	(\mathcal{T}_{m}(t)\Phi)(\bar{s}) \coloneq \Phi(\bar{s}+\underline{t}) \qquad \text{for} \quad \bar{s} = (s_{1},\ldots, s_{m}) \in \mathbb{R}^{m}.
\end{equation}
Recall that for $t \in \mathbb{R}$, the vector $\underline{t}$ has identical components, all of which are equal to $t$.

Since the action of $\mathbb{R}^{m}$ by translations (with respect to arguments) in $L_{p}(\mathbb{R}^{m};\mathbb{F})$ is strongly continuous, $\mathcal{T}_{m}$ is a $C_{0}$-group. For any $\Psi \in L_{p}(\mathbb{R}^{m};\mathbb{F})$, this yields the following relation:
\begin{equation}
	\label{EQ: DiagonalTranslationRmSmoothing}
	\lim_{h \to 0+} \frac{1}{h} \int_{0}^{h}\mathcal{T}_{m}(\theta)\Psi d\theta = \Psi \qquad \text{in} \quad \ L_{p}(\mathbb{R}^{m};\mathbb{F}).
\end{equation}

In the following theorem, we describe the generator of $\mathcal{T}_{m}$ in connection with the diagonal Sobolev space $\mathcal{W}^{p}_{D}(\mathbb{R}^{m};\mathbb{F})$.
\begin{theorem}
	\label{TH: DiagonalTranslationInRm}
	Let $A_{\mathcal{T}_{m}}$ be the generator of $\mathcal{T}_{m}$ in $L_{p}(\mathbb{R}^{m};\mathbb{F})$. Then it has the domain $\mathcal{D}(A_{\mathcal{T}_{m}}) = \mathcal{W}^{p}_{D}(\mathbb{R}^{m};\mathbb{F})$, for which the subspace $C^{\infty}_{0}(\mathbb{R}^{m};\mathbb{F})$ is a core\footnote{That is a subspace dense in the graph norm.}, and acts according to the formula\footnote{Here $\Phi$ is considered as a function of $(s_{1},\ldots,s_{m}) \in \mathbb{R}^{m}$.}:
	\begin{equation}
		A_{\mathcal{T}_{m}}\Phi = \left(\sum_{j=1}^{m}\frac{\partial}{\partial s_{j}}\right) \Phi \qquad \text{for} \quad \Phi \in \mathcal{D}(A_{\mathcal{T}_{m}}).
	\end{equation}
	In addition, let $\Gamma$ be an affine hyperplane in $\mathbb{R}^{m}$ that intersects transversely\footnote{In the sense that there exists a unique intersection point.} the diagonal line $\mathcal{L}_{0}$. Then there is a linear (trace) operator $\operatorname{Tr}_{\Gamma} \colon \mathcal{W}^{p}_{D}(\mathbb{R}^{m};\mathbb{F}) \to L_{p}(\Gamma;\mathbb{F})$ defined on compactly supported functions $\Phi$. It is given for $\mu^{m-1}_{L}$-almost all $\bar{s} \in \Gamma$ as follows:
	\begin{equation}
		\label{EQ: TraceOperatorDiagonalTranslate}
		\operatorname{Tr}_{\Gamma}\Phi(\bar{s}) = \int_{-\infty}^{0} (A_{\mathcal{T}_{m}}\mathcal{T}(t)\Phi)(\bar{s})dt = \restr{\Phi}{\Gamma}(\bar{s}).
	\end{equation}
    In particular, for any $r > 0$ there exists a constant $C(r) > 0$ such that
    \begin{equation}
    	\label{EQ: EstimateTraceOperatorWithRadius}
    	\| \operatorname{Tr}_{\Gamma}\Phi \|_{L_{p}(\Gamma;\mathbb{F})} \leq C(r) \cdot \| \Phi \|_{\mathcal{W}^{p}_{D}(\mathbb{R}^{m};\mathbb{F})}
    \end{equation}
    is satisfied for any $\Gamma$ and any $\Phi$ whose support is contained in the ball of radius $r$.
\end{theorem}
\begin{proof}
	Clearly, the subspace of smooth functions $C^{\infty}_{0}(\mathbb{R}^{m};\mathbb{F})$ with compact support is dense in $L_{p}(\mathbb{R}^{m};\mathbb{F})$ and invariant with respect to $\mathcal{T}_{m}(t)$ for any $t \in \mathbb{R}$. It is also obvious that for $\Phi \in C^{\infty}_{0}(\mathbb{R}^{m};\mathbb{F})$ there exists the limit
	\begin{equation}
		\label{EQ: DiagonalTranslateSemigroupRmLimitFiniteSmooth}
		\lim_{h \to 0+}\frac{1}{h}\left(\mathcal{T}_{m}(h)\Phi - \Phi\right) = \left(\sum_{j=1}^{m}\frac{\partial}{\partial s_{j}}\right) \Phi.
	\end{equation}
    Therefore, $\Phi \in \mathcal{D}(A_{\mathcal{T}_{m}})$. By  \cite[Chapter II, Proposition 1.7]{EngelNagel2000}, we obtain that $C^{\infty}_{0}(\mathbb{R}^{m};\mathbb{F})$ is a core for $\mathcal{D}(A_{\mathcal{T}_{m}})$. Now we aim to show that the limit \eqref{EQ: DiagonalTranslateSemigroupRmLimitFiniteSmooth} also exists for $\Phi \in \mathcal{W}^{p}_{D}(\mathbb{R}^{m};\mathbb{F})$. From this, $\mathcal{W}^{p}_{D}(\mathbb{R}^{m};\mathbb{F})$ must also be a core by the same argument and hence coincide with $A_{\mathcal{T}_{m}}$, since it is closed in the graph norm.
    
    Consider $\Phi \in \mathcal{W}^{p}_{D}(\mathbb{R}^{m};\mathbb{F})$. For convenience, we set $D\coloneq\sum_{j=1}^{m}\frac{\partial}{\partial s_{j}}$, and let $L_{p}$ stand for $L_{p}(\mathbb{R}^{m};\mathbb{F})$. By \eqref{EQ: DiagonalTranslationRmSmoothing}, for any $\varepsilon>0$ there exists $\delta > 0$ such that
    \begin{equation}
    	\frac{1}{h} \int_{0}^{h}\mathcal{T}_{m}(\theta)D\Phi d\theta = D\Phi + R_{h},
    \end{equation} 
    where $\|R_{h}\|_{L_{p}} < \varepsilon$, provided that $0 < h < \delta$. 
    
    Moreover, for $\mu^{m-1}_{L}$-almost all $\bar{s} \in \mathcal{L}^{\bot}_{0}$ and any $t \in \mathbb{R}$, we have that
    \begin{equation}
    	\begin{split}
    		\frac{1}{h}\left[\Phi(\bar{s}+\underline{t} + \underline{h}) - \Phi(\bar{s}+\underline{t})\right] = \frac{1}{h}\int_{0}^{h}D\Phi(\bar{s}+\underline{t}+\underline{\theta})d\theta =\\= \frac{1}{h}\int_{0}^{h}(\mathcal{T}_{m}(\theta)D\Phi)(\bar{s}+\underline{t}) d\theta = D\Phi(\bar{s} + \underline{t}) + R_{h}(\bar{s}+\underline{t}).
    	\end{split}
    \end{equation}
    From this, using the Fubini theorem, we obtain that
    \begin{equation}
    	\begin{split}
    		\left\| \frac{1}{h}\left(\mathcal{T}_{m}(h)\Phi - \Phi\right) - D\Phi \right\|^{p}_{L_{p}} =\\ \int_{\bar{s} \in \mathcal{L}^{\bot}_{0}}d\mu^{m-1}_{L}(\bar{s})\int_{t \in \mathbb{R}} \left| \frac{1}{h}\left[\Phi(\bar{s}+\underline{t}+\underline{h}) - \Phi(\bar{s}+\underline{t})\right] - D\Phi(\bar{s}+\underline{t}) \right|^{p}_{\mathbb{F}}dt= \\
    		\int_{\bar{s} \in \mathcal{L}^{\bot}_{0}}d\mu^{m-1}_{L}(\bar{s})\int_{t \in \mathbb{R}} \left| R_{h}(\bar{s}+\underline{t}) \right|^{p}_{\mathbb{F}}dt = \|R_{h}\|^{p}_{L_{p}} < \varepsilon^{p},
    	\end{split}
    \end{equation}
    which shows the required.
	
	For the second part, consider $\Phi \in C^{\infty}_{0}(\mathbb{R}^{m};\mathbb{F})$. Then the Newton--Leibniz formula gives that the restriction of $\Phi$ to $\Gamma$ can be described by \eqref{EQ: TraceOperatorDiagonalTranslate}, and the estimate \eqref{EQ: EstimateTraceOperatorWithRadius} is valid. Moreover, since such functions are dense in $\mathcal{D}(A_{\mathcal{T}_{m}}) = \mathcal{W}^{p}_{D}(\mathbb{R}^{m};\mathbb{F})$ by the previous, the formula can be extended for $\Phi \in \mathcal{W}^{p}_{D}(\mathbb{R}^{m};\mathbb{F})$ with compact support thanks to \eqref{EQ: EstimateTraceOperatorWithRadius}. Here the last equality in \eqref{EQ: TraceOperatorDiagonalTranslate} is satisfied because the restrictions of $\Phi$ to almost every line parallel to the diagonal $\mathcal{L}_{0}$ are well defined elements of appropriate $W^{1,p}$-spaces and can therefore be identified with continuous functions. This allows us to determine the values of $\Phi$ at $\mu^{m-1}_{L}$-almost every point of $\Gamma$.
\end{proof}

Now consider a hyperplane $\Gamma_{0}$ that transversely intersects the diagonal line $\mathcal{L}_{0}$. Then hyperplanes $\Gamma$ close to $\Gamma_{0}$ also transversely intersect $\mathcal{L}_{0}$. This allows us to identify $\Gamma$ and $\Gamma_{0}$ along the diagonal line, i.e., any $\bar{s} \in \Gamma$ is identified with the unique element from the intersection $\Gamma_{0} \cap (\mathcal{L}_{0}+\bar{s})$. This defines a linear isomorphism $E_{\Gamma,\Gamma_{0}}$ from $L_{p}(\Gamma;\mathbb{F})$ to $L_{p}(\Gamma_{0};\mathbb{F})$.
\begin{lemma}
	\label{LEM: ContinuityTraceOperatorGamma}
	Consider $\Gamma_{0}$ as above. Then for any $\Phi \in \mathcal{W}^{p}_{D}(\mathbb{R}^{m};\mathbb{F})$ with compact support, the mapping $\Gamma \mapsto E_{\Gamma,\Gamma_{0}} \circ \operatorname{Tr}_{\Gamma} \Phi \in L_{p}(\Gamma_{0};\mathbb{F})$ is continuous at $\Gamma_{0}$.
\end{lemma}
\begin{proof}
	Let $\mathcal{S}(\Gamma_{0},\Gamma)$ denote the sector between $\Gamma_{0}$ and $\Gamma$, i.e., the symmetric difference between $\bigcup_{t = -\infty}^{0}( \Gamma_{0} + \underline{t})$ and $\bigcup_{t = -\infty}^{0}( \Gamma + \underline{t})$. Let $\mathcal{B}(r)$ be the ball in $\mathbb{R}^{m}$ of radius $r>0$ centered at $0$ and containing the support of $\Phi$. Then from \eqref{EQ: TraceOperatorDiagonalTranslate}, the H\"{o}lder inequality, and the Fubini theorem, for some $C(r)>0$ we have
	\begin{equation}
		\begin{split}
			\| E_{\Gamma,\Gamma_{0}} \circ \operatorname{Tr}_{\Gamma}\Phi - \operatorname{Tr}_{\Gamma_{0}}\Phi \|^{p}_{L_{p}(\Gamma_{0};\mathbb{F})} \leq \\ \leq C(r) \cdot \int\limits_{\mathcal{S}(\Gamma_{0};\Gamma) \cap \mathcal{B}(r)} \left\| \left(\sum_{j=1}^{m} \frac{\partial}{\partial s_{j}}\right)\Phi(\bar{s}) \right\|^{p}_{\mathbb{F}}d\bar{s},
		\end{split}
	\end{equation}
    where the integral tends to $0$ as $\Gamma \to \Gamma_{0}$ due to absolute continuity of the integral.
\end{proof}

Next, we consider the case where $\Omega = (-\tau,0)^{m}$ for some $\tau > 0$. For $j \in \{1,\ldots,m\}$, recall the subsets $\mathcal{B}_{\hat{j}} = \mathcal{B}^{(m)}_{\hat{j}}$ from \eqref{EQ: DefinitionOfBoundaryFace} consisting of all $\bar{\theta}=(\theta_{1},\ldots,\theta_{m}) \in [-\tau,0]^{m}$ with $\theta_{j} = 0$.

\begin{proposition}
	\label{PROP: DiagonalSobolevCubeDescription}
	The space $\mathcal{W}^{p}_{D}((-\tau,0)^{m};\mathbb{F})$ consists of exactly those elements $\Phi$ from $\widehat{\mathcal{W}}^{p}_{D}((-\tau,0)^{m};\mathbb{F})$ for which the restriction $\Phi_{j}$ of $\Phi$ to $\mathcal{B}^{(m)}_{\hat{j}}$ is an element of $L_{p}(\mathcal{B}^{(m)}_{\hat{j}};\mathbb{F})$ for any $j \in \{1,\ldots,m\}$. Furthermore, the norm
	\begin{equation}
		\label{EQ: NormInDiagonalSobolevOnCube}
		\|\Phi\|^{p} \coloneq \| \Phi \|^{p}_{\widehat{\mathcal{W}}^{p}_{D}((-\tau,0)^{m};\mathbb{F})} + \sum_{j=1}^{m}\| \Phi_{j} \|^{p}_{L_{p}(\mathcal{B}^{(m)}_{\hat{j}};\mathbb{F})}.
	\end{equation}
	is equivalent to the norm in $\mathcal{W}^{p}_{D}((-\tau,0)^{m};\mathbb{F})$. 
	
	In addition, there exists an extension operator
	\begin{equation}
		\label{EQ: ExtensionDiagonalSobolevCube}
		\mathfrak{C} \colon \mathcal{W}^{p}_{D}((-\tau,0)^{m};\mathbb{F}) \to \mathcal{W}^{p}_{D}(\mathbb{R}^{m};\mathbb{F})
	\end{equation}
	with the same properties as in items 1) and 2) of Lemma \ref{LEM: ExtensionOperatorFromSobolevDiagonalToRm}.
\end{proposition}
\begin{proof}
	By Theorem \ref{TH: DiagonalTranslationInRm}, any $\Phi \in \mathcal{W}^{p}_{D}((-\tau,0)^{m};\mathbb{F})$ has $L_{p}$-summable traces on the boundary faces. So, it is required to show that any $\Phi \in \widehat{\mathcal{W}}^{p}_{D}((-\tau,0)^{m};\mathbb{F})$ with $L_{p}$-summable traces $\Phi_{j}$ on each $\mathcal{B}^{(m)}_{\hat{j}}$, where $j \in \{1,\ldots,m\}$, belongs to $\mathcal{W}^{p}_{D}((-\tau,0)^{m};\mathbb{F})$.
	
	For such $\Phi$, let $\Phi_{0}$ be defined on $\mathring{\mathcal{C}}^{m}_{\tau}$ from Remark \ref{REM: DelayCompoundLemmaExtensionApplication} as follows:
	\begin{equation}
		\Phi_{0}(\bar{s}) \coloneq \begin{cases}
			\Phi(\bar{s}) \qquad &\text{for} \quad \bar{s} \in (-\tau,0)^{m},\\
			\Phi_{j}(\bar{s}-\underline{t}) \qquad &\text{for} \quad \bar{s}-\underline{t} \in \mathcal{B}^{(m)}_{\hat{j}} \ \text{and} \ t \in (0,\tau],
		\end{cases}
	\end{equation}
	where the second condition is taken over $j \in \{1,\ldots,m\}$. Clearly, $\Phi_{0} \in \widehat{\mathcal{W}}^{p}_{D}(\mathring{\mathcal{C}}^{m}_{\tau};\mathbb{F})$, and hence, by Lemma \ref{LEM: ExtensionOperatorFromSobolevDiagonalToRm}, there exists an extension $\mathfrak{C}\Phi_{0} \in \mathcal{W}^{p}_{D}(\mathbb{R}^{m};\mathbb{F})$ with compact support lying in the ball of radius $r$ depending only on $\tau$. In particular, $\Phi \in \mathcal{W}^{p}_{D}(\mathbb{R}^{m};\mathbb{F})$.
	
	Using the Fubini theorem, for some constant $C_{1} > 0$ depending only on $\tau$, we obtain
	\begin{equation}
		\| \Phi_{0} \|^{p}_{L_{p}(\mathring{\mathcal{C}}^{m}_{\tau} \setminus (-\tau,0)^{m};\mathbb{F})} \leq C_{1} \sum_{j=1}^{m}\| \Phi_{j} \|^{p}_{L_{p}(\mathcal{B}^{(m)}_{\hat{j}};\mathbb{F})}.
	\end{equation}
	Consequently, there exists a constant $C_{2}>0$ such that
	\begin{equation}
		\begin{split}
			\|\Phi\|^{p}_{\mathcal{W}^{p}_{D}(\mathbb{R}^{m};\mathbb{F})} \leq \| \mathfrak{C}\Phi_{0} \|^{p}_{\mathcal{W}^{p}_{D}(\mathbb{R}^{m};\mathbb{F})} \leq \| \mathfrak{C} \|^{p} \cdot \| \Phi_{0} \|^{p}_{\widehat{\mathcal{W}}^{p}_{D}(\mathring{\mathcal{C}}^{m}_{\tau};\mathbb{F})} = \\ = \| \mathfrak{C} \|^{p} \left( \|\Phi\|^{p}_{\widehat{\mathcal{W}}^{p}_{D}((-\tau,0)^{m};\mathbb{F})} + \| \Phi_{0} \|^{p}_{L_{p}(\mathring{\mathcal{C}}^{m}_{\tau} \setminus (-\tau,0)^{m};\mathbb{F})} \right) \\
			\leq C_{2} \|\Phi\|^{p}.
		\end{split}
	\end{equation}
	
	On the other hand, let $\Gamma_{j}$ be the hyperplane in $\mathbb{R}^{m}$ such that $\Gamma_{j} \cap [-\tau,0]^{m} = \mathcal{B}^{(m)}_{\hat{j}}$. Then for any extension $\hat{\Phi} \in \mathcal{W}^{p}_{D}(\mathbb{R}^{m};\mathbb{F})$ of $\Phi$ with support contained in the ball of radius $r$, from \eqref{EQ: EstimateTraceOperatorWithRadius} we obtain
	\begin{equation}
		\begin{split}
			\|\Phi\|^{p} \leq  \| \hat{\Phi} \|^{p}_{\mathcal{W}^{p}_{D}(\mathbb{R}^{m};\mathbb{F})} + \sum_{j=1}^{m}\| \operatorname{Tr}_{\Gamma_{j}}\hat{\Phi} \|^{p}_{L_{p}(\Gamma_{j};\mathbb{F})} \leq \\ \leq
			 \left(1 + mC^{p}(r)\right) \| \hat{\Phi} \|^{p}_{\mathcal{W}^{p}_{D}(\mathbb{R}^{m};\mathbb{F})}.
		\end{split}
	\end{equation}
	Since taking such $\hat{\Phi}$ is sufficient to compute the norm of $\Phi$ in $\mathcal{W}^{p}_{D}(\mathbb{R}^{m};\mathbb{F})$ up to a uniform constant (depending only on $\tau$ and $m$ or, more precisely, on derivatives of a proper cut-off function), this shows the equivalence of norms and hence the boundedness of the extension operator $\Phi \mapsto \mathfrak{C}\Phi_{0}$.
\end{proof}

It is convenient to consider the norm $\|\cdot\|$ from \eqref{EQ: NormInDiagonalSobolevOnCube} as the basic norm in the space $\mathcal{W}^{p}_{D}((-\tau,0)^{m};\mathbb{F})$ and denote it by $\| \cdot  \|_{\mathcal{W}^{p}_{D}((-\tau,0)^{m};\mathbb{F})}$.

We deduce the trace theorem for $\mathcal{W}^{p}_{D}((-\tau,0)^{m};\mathbb{F})$ as follows.

\begin{theorem}
	\label{TH: TraceOperatorForDiagonalTranslatesFinite}
	Let $\Gamma$ be an affine hyperplane that transversely intersects the diagonal line, and let the intersection $\mathcal{I} \coloneq \Gamma \cap [-\tau,0]^{m}$ have a nonempty interior in $\Gamma$. Then there exists a bounded linear operator
	\begin{equation}
		\operatorname{Tr}_{\mathcal{I}} \colon \mathcal{W}^{p}_{D}( (-\tau,0)^{m};\mathbb{F}) \to L_{p}(\mathcal{I}; \mathbb{F})
	\end{equation}
    defined by the restriction of $\Phi$ to $\mathcal{I}$. Moreover, its norm admits an upper bound that depends only on $\tau$ and $m$ and does not depend on $\mathcal{I}$.
\end{theorem}
\begin{proof}
	Let $R_{\mathcal{I}} \colon L_{p}(\Gamma;\mathbb{F}) \to L_{p}(\mathcal{I};\mathbb{F})$ be the operator that restricts functions from $\Gamma$ to $\mathcal{I}$. We define $\operatorname{Tr}_{\mathcal{I}}$ as follows:
	\begin{equation}
		\label{EQ: DefinitionTraceOperatorDiagonalSobolevOnCube}
		\operatorname{Tr}_{\mathcal{I}}\Phi \coloneq R_{\mathcal{I}}\operatorname{Tr}_{\Gamma}\mathfrak{C}\Phi,
	\end{equation}
    where $\mathfrak{C}$ and $\operatorname{Tr}_{\Gamma}$ are defined in \eqref{EQ: ExtensionDiagonalSobolevCube} and \eqref{EQ: TraceOperatorDiagonalTranslate}, respectively. From \eqref{EQ: EstimateTraceOperatorWithRadius} and the construction of $\mathfrak{C}$, we obtain that the norm of $\operatorname{Tr}_{\mathcal{I}}$ can be estimated only in terms of $\tau$ and $m$. Moreover, $\operatorname{Tr}_{\mathcal{I}}$ is indeed determined by the restriction of $\Phi$ to $\mathcal{I}$, thanks to the last identity in \eqref{EQ: TraceOperatorDiagonalTranslate}.
\end{proof}

Let $e_{j}$ be the $j$th vector in the standard basis of $\mathbb{R}^{m}$. Then for  $\theta \in [-\tau,0]$, each subset $\mathcal{B}_{\hat{j}} + \theta e_{j}$ is naturally identified with $[-\tau,0]^{m-1}$ by omitting the $j$th coordinate.
\begin{lemma}
	\label{LEM: TraceContinuityAlongShiftedFaces}
	Under the above given identifications, the mapping
	\begin{equation}
		[-\tau,0] \ni \theta \mapsto \operatorname{Tr}_{\mathcal{B}_{\hat{j}} +\theta e_{j}}\Phi \in L_{p}((-\tau,0)^{m-1};\mathbb{F})
	\end{equation}
	is continuous for any $\Phi \in \mathcal{W}^{p}_{D}((-\tau,0)^{m};\mathbb{F})$ and $j = 1,\ldots, m$.
\end{lemma}
\begin{proof}
	Let $\Gamma_{j}(\theta)$ be the hyperplane consisting of $(s_{1},\ldots,s_{m}) \in \mathbb{R}^{m}$ with $s_{j} = \theta$, i.e., $\Gamma_{j}(\theta) \cap [-\tau,0]^{m} = \mathcal{B}_{\hat{j}} +\theta e_{j}$. According to \eqref{EQ: DefinitionTraceOperatorDiagonalSobolevOnCube}, $\operatorname{Tr}_{\mathcal{B}_{\hat{j}} +\theta e_{j}}\Phi$ is obtained by restricting the trace $\operatorname{Tr}_{\Gamma_{j}(\theta)}\mathfrak{C}\Phi$ of the extension $\mathfrak{C}\Phi$ to $\mathcal{B}_{\hat{j}} +\theta e_{j}$. Then Lemma \ref{LEM: ContinuityTraceOperatorGamma} yields the continuity of $\operatorname{Tr}_{\Gamma_{j}(\theta)}\Phi$ in $\theta$ if the identification of $\Gamma_{j}(\theta)$ along the diagonal line $\mathcal{L}_{0}$ is used. Note that this identification differs from the identification along the $j$th axis in $\mathbb{R}^{m}$ only by a translation in the argument, which becomes arbitrarily small for hyperplanes $\Gamma_{j}(\theta)$ with close $\theta$. Since the action by translates is strongly continuous, this implies that the mapping $[-\tau,0] \ni \theta \mapsto \operatorname{Tr}_{\Gamma_{j}(\theta)}\mathfrak{C}\Phi$ is continuous if the identification along the $j$th axis is used.
\end{proof}

Next, we introduce certain spaces related to the continuous dependence of traces, as in Lemma \ref{LEM: TraceContinuityAlongShiftedFaces}, and operators on such spaces. In our study of delay equations, such spaces are used to construct intermediate spaces, as in \eqref{EQ: MyLovelyScaleDelayCompoundGeneral}, and to study pointwise measurement operators in Appendix \ref{SEC: MeasurementOperatorsOnAgalmanated}.

Let $\mathbb{M}_{\gamma}$ be a separable Hilbert space over the same field as $\mathbb{F}$, and consider an $\mathcal{L}(\mathbb{F};\mathbb{M}_{\gamma})$-valued function $\gamma(\cdot)$ of bounded variation on $[-\tau,0]$. For each $J \in \{1, \ldots, m\}$, we associate with $\gamma$ a bounded linear operator $C^{\gamma}_{J}$ from $C([-\tau,0]^{m};\mathbb{F})$ to $C([-\tau,0]^{m-1};\mathbb{M}_{\gamma})$, defined by the formula:
\begin{equation}
	\label{EQ: WeigthedFunctionsOperatorCgammaDef}
	(C^{\gamma}_{J}\Phi)(\theta_{1},\ldots, \hat{\theta}_{J}, \ldots, \theta_{m}) = \int_{-\tau}^{0}d\gamma(\theta_{J}) \Phi(\theta_{1},\ldots,\theta_{m}),
\end{equation}
where $(\theta_{1},\ldots,\hat{\theta}_{J},\ldots,\theta_{m}) \in [-\tau,0]^{m-1}$ and the integral is understood as the Riemann--Stieltjes integral.

We need to consider $C^{\gamma}_{J}$ in a broader context. For any $p \geq 1$, we define the space $\mathbb{E}^{p}_{m}(\mathbb{F})$ of all functions $\Phi \in L_{p}((-\tau,0)^{m}; \mathbb{F})$ such that for any $j \in \{1, \ldots, m\}$ there exists $\Phi^{b}_{j} \in C([-\tau,0];L_{p}((-\tau,0)^{m-1};\mathbb{F})$, called the \textit{function of the $j$th section}, satisfying the identity in $L_{p}( (-\tau,0)^{m-1};\mathbb{F} )$ as follows:
\begin{equation}
	\label{EQ: SpaceEmDelayPhiBDefinition}
	\restr{\Phi}{\mathcal{B}_{\hat{j}}+\theta e_{j}} = \Phi^{b}_{j}(\theta) \qquad \text{for almost all} \quad \theta \in [-\tau,0],
\end{equation}
where we have naturally identified $\mathcal{B}_{\hat{j}}+\theta e_{j}$ with $[-\tau,0]^{m-1}$ by omitting the $j$th argument.

We endow $\mathbb{E}^{p}_{m}(\mathbb{F})$ with the natural norm
\begin{equation}
	\label{EQ: NormInSpaceEmDelay}
	\| \Phi \|_{\mathbb{E}^{p}_{m}(\mathbb{F})} \coloneq \sup_{ j \in \{1,\ldots,m\}} \sup_{\theta \in [-\tau,0]}\| \Phi^{b}_{j}(\theta) \|_{L_{p}( (-\tau,0)^{m-1};\mathbb{F} )},
\end{equation}
which makes it a Banach space.

Since $\Phi^{b}_{j}(\theta)$ continuously depend on $\theta \in [-\tau,0]$, it is not hard to show that the space $C([-\tau,0]^{m};\mathbb{F})$ is dense in $\mathbb{E}^{p}_{m}(\mathbb{F})$. We have the following theorem.
\begin{theorem}
	\label{TH: OperatorCExntesionOntoWDiagonal}
	The operator $C^{\gamma}_{J}$ from \eqref{EQ: WeigthedFunctionsOperatorCgammaDef} uniquely extends by continuity to a bounded operator from $\mathbb{E}^{p}_{m}(\mathbb{F})$ to $L_{p}((-\tau,0)^{m-1};\mathbb{M}_{\gamma})$ whose norm does not exceed the total variation $\operatorname{Var}_{[-\tau,0]}(\gamma)$ of $\gamma$ on $[-\tau,0]$.
\end{theorem}
\begin{proof}
	For convenience, by $d\gamma$ we denote the corresponding $\mathbb{M}_{\gamma}$-valued linear functional on $C([-\tau,0];\mathbb{F})$, obtained by integration as in \eqref{EQ: WeigthedFunctionsOperatorCgammaDef} with $m=1$. Let $\delta^{J}_{\tau_{0}}$ be the operator $C^{\gamma}_{J}$ in the case where $d\gamma = \delta_{\tau_{0}}$ is the $\mathbb{F}$-valued $\delta$-functional $\delta_{\tau_{0}}$ at some point $\tau_{0} \in [-\tau,0]$. In these terms, for all $\Phi \in C([-\tau,0]^{m};\mathbb{F})$ we have that
	\begin{equation}
		\label{EQ: LemmaFunctionalCextensionWDiagonal1}
		\delta^{J}_{\tau_{0}}\Phi = \restr{\Phi}{\mathcal{B}_{\hat{J}}+\tau_{0} e_{J}} = \Phi^{b}_{J}(\tau_{0}).
	\end{equation}
	
	Using \eqref{EQ: LemmaFunctionalCextensionWDiagonal1} and \eqref{EQ: NormInSpaceEmDelay}, we obtain that
	\begin{equation}
		\label{EQ: LemmaFunctionalCextensionWDiagonal2}
		\| \delta^{J}_{\tau_{0}}\Phi \|_{L_{p}((-\tau,0)^{m-1};\mathbb{F}) } \leq \|\Phi\|_{\mathbb{E}^{p}_{m}(\mathbb{F})} \qquad \text{for any} \quad \tau_{0} \in [-\tau,0],
	\end{equation}
	which shows the statement for $\delta$-functionals.

    For general $d\gamma$, we use a particular approximation by $\delta$-functionals. For $k=1,2,\ldots$, consider a partition of $[-\tau,0]$ by $N_{k}+1$ points $-\tau = \theta^{(k)}_{0} < \theta_{2} < \cdots < \theta^{(k)}_{N_{k}} = 0$ such that $\max_{1 \leq l \leq N_{k}}|\theta^{(k)}_{l}-\theta^{(k)}_{l-1}|$ tends to $0$ as $k \to \infty$. For each $l \in \{1,\ldots,N_{k}\}$, we set $\alpha^{(k)}_{l} \coloneq \gamma(\theta^{(k)}_{l}) - \gamma(\theta^{(k)}_{l-1})$ and $\delta^{(k)}_{l} \coloneq \delta_{\theta^{(k)}_{l}}$. Then
    \begin{equation}
    	\label{EQ: ApproximationByDeltaFunctionalsGamma}
    	d\gamma_{k} \coloneq \sum_{l=1}^{N_{k}} \alpha^{(k)}_{l} \delta^{(k)}_{l} \to d\gamma \qquad \text{pointwise in} \quad C([-\tau,0];\mathbb{F}).
    \end{equation}

	Using \eqref{EQ: ApproximationByDeltaFunctionalsGamma} and \eqref{EQ: LemmaFunctionalCextensionWDiagonal2}, for any $\Phi \in C([-\tau,0]^{m};\mathbb{F})$ we obtain that
	\begin{equation}
		\begin{split}
			\| C^{\gamma}_{J} \Phi \|_{L_{p}( (-\tau,0)^{m-1};\mathbb{M}_{\gamma} )} =\\= \lim_{k \to \infty} \| C^{\gamma_{k}}_{J}\Phi \|_{L_{p}( (-\tau,0)^{m-1};\mathbb{M}_{\gamma} )} \leq \operatorname{Var}_{[-\tau,0]}(\gamma) \cdot \|\Phi\|_{\mathbb{E}^{p}_{m}(\mathbb{F})},
		\end{split}
	\end{equation}
	which shows the statement.
\end{proof}

Combining Theorem \ref{TH: TraceOperatorForDiagonalTranslatesFinite} and Lemma \ref{LEM: TraceContinuityAlongShiftedFaces}, we immediately obtain the following.
\begin{proposition}
	\label{PROP: EmbeddingWDiagToEmDelay}
	There is a natural continuous and dense embedding of $\mathcal{W}^{p}_{D}((-\tau,0)^{m};\mathbb{F})$ into $\mathbb{E}^{p}_{m}(\mathbb{F})$, and the embedding constant can be estimated only in terms of $\tau$ and $m$.
\end{proposition}

Let $T_{m}$ be the diagonal translation semigroup in $L_{p}((-\tau,0)^{m};\mathbb{F})$ with the time-$t$ mappings $T_{m}(t)$ for $t \geq 0$ given by
\begin{equation}
	\label{EQ: SemigroupDiagonalTranslateFiniteCubeDef}
	(T_{m}(t)\Phi)(\bar{\theta}) = \begin{cases}
		\Phi(\bar{\theta} + \underline{t}) \qquad &\text{if} \quad \bar{\theta} + \underline{t} \in (-\tau,0)^{m},\\
		0 \qquad &\text{otherwise},
	\end{cases}
\end{equation}
where $\bar{\theta} = (\theta_{1},\ldots,\theta_{m}) \in (-\tau,0)^{m}$ and $\underline{t} = (t,\ldots,t) \in \mathbb{R}^{m}$.

Since the action of $\mathbb{R}^{m}$ by translations (in arguments) in $L_{p}(\mathbb{R}^{m};\mathbb{F})$ is strongly continuous, $T_{m}$ is a $C_{0}$-semigroup. Its generator is described in the following theorem. 
\begin{theorem}	
	\label{TH: DiagonalTranslatesSquareDelay}
	Let $A_{T_{m}}$ be the generator of $T_{m}$ in $L_{p}((-\tau,0)^{m};\mathbb{F})$. Then its domain $\mathcal{D}(A_{T_{m}})$ is described as follows: (see Theorem \ref{TH: TraceOperatorForDiagonalTranslatesFinite})
	\begin{equation}
		\label{EQ: DomainNilponentDiagonalTranslates}
		\begin{split}
			\mathcal{D}(A_{T_{m}}) = \mathcal{W}^{p}_{D_{0}}((-\tau,0)^{m};\mathbb{F}) \coloneq \\ = \left\{ \Phi \in \mathcal{W}^{p}_{D}((-\tau,0)^{m};\mathbb{F}) \mid \operatorname{Tr}_{\mathcal{B}_{\hat{j}}}\Phi = 0 \quad \text{for all} \quad j \in \{1, \ldots, m\} \right\},
		\end{split}
	\end{equation}
    and it acts on $\Phi \in \mathcal{D}(A_{T_{m}})$ according to the formula:
    \begin{equation}
    	\label{EQ: GeneratorDiagonalTranslateSemigroupCubeAct}
    	A_{T_{m}}\Phi = \left(\sum_{j=1}^{m}\frac{\partial}{\partial \theta_{j}}\right) \Phi.
    \end{equation}
\end{theorem}
\begin{proof}
	It is not hard to see that the space $\mathcal{W}^{p}_{D_{0}}$ given by the right-hand side of \eqref{EQ: DomainNilponentDiagonalTranslates} is invariant with respect to $T_{m}$ and dense in $L_{p}((-\tau,0)^{m};\mathbb{F})$. Consider $\Phi \in \mathcal{W}^{p}_{D_{0}}$, and let $\hat{\Phi} \in \mathcal{W}^{p}_{D}(\mathbb{R}^{m};\mathbb{F})$ be any extension of $\Phi$, which exists due to Proposition \ref{PROP: DiagonalSobolevCubeDescription}. For convenience, we set $D_{\theta} \coloneq \sum_{j=1}^{m}\frac{\partial}{\partial \theta_{j}}$ and $D_{s} \coloneq \sum_{j=1}^{m}\frac{\partial}{\partial s_{j}}$. Then, by Theorem \ref{TH: DiagonalTranslationInRm}, we have that
	\begin{equation}
		\begin{split}
			\left\| \frac{1}{h}\left[T_{m}(h)\Phi - \Phi \right] - D_{\theta}\Phi \right\|_{L_{p}((-\tau,0)^{m};\mathbb{F})} \leq \\ \leq \left\| \frac{1}{h}\left[\mathcal{T}_{m}(h)\hat{\Phi} - \hat{\Phi} \right] - D_{s}\hat{\Phi} \right\|_{L_{p}(\mathbb{R}^{m};\mathbb{F})} \to 0 \qquad \text{ as } h \to 0+.
		\end{split}
	\end{equation}
	Therefore, $\Phi$ lies in $\mathcal{D}(A_{T_{m}})$ and satisfies \eqref{EQ: GeneratorDiagonalTranslateSemigroupCubeAct}. By \cite[Chapter II, Proposition 1.7]{EngelNagel2000}, the space $\mathcal{W}^{p}_{D_{0}}$ is dense in $\mathcal{D}(A_{T_{m}})$ in the graph norm. Since it is also closed in the graph norm, it must coincide with the domain.
\end{proof}

%% file: WeigthedFunctionsCompound.tex
\section{Pointwise measurement operators}
\label{SEC: MeasurementOperatorsOnAgalmanated}

Throughout this section, we fix real or complex separable Hilbert spaces $\mathbb{F}$ and $\mathbb{M}_{\gamma}$. For $\tau > 0$, consider an $\mathcal{L}(\mathbb{F};\mathbb{M}_{\gamma})$-valued function $\gamma(\cdot)$ of bounded variation on $[-\tau,0]$. Let a positive integer $m$ be fixed. For each $J \in \{ 1,\ldots, m \}$, let $C^{\gamma}_{J}$ be the operator defined in \eqref{EQ: WeigthedFunctionsOperatorCgammaDef}, i.e., $C^{\gamma}_{J}$ takes continuous functions of $m$ arguments into continuous functions of $m-1$ arguments by integrating over $d\gamma$ with respect to the $J$th argument. By Theorem \ref{TH: OperatorCExntesionOntoWDiagonal}, it can be extended to a bounded linear operator from $\mathbb{E}^{p}_{m}(\mathbb{F})$ to $L_{p}((-\tau,0)^{m-1};\mathbb{M}_{\gamma})$, where $\mathbb{E}^{p}_{m}(\mathbb{F})$ is defined above \eqref{EQ: NormInSpaceEmDelay}.

In what follows, we are interested in interpreting the pointwise measurement operator $\Phi(\cdot) \mapsto C^{\gamma}_{J}\Phi(\cdot)$ for some classes of $L_{2}((-\tau,0)^{m};\mathbb{F})$-valued functions $\Phi(\cdot)$ of time, whose values do not belong to the space $\mathbb{E}^{p}_{m}(\mathbb{F})$ in general. For $m=1$, such a theory was constructed in our paper \cite{Anikushin2020FreqDelay}, and below we present its generalization.

\subsection{Pointwise measurement operators on embracing spaces}

We first construct a space that is, in a sense, maximal on which pointwise measurement operators can be defined. For this, consider two real numbers $-\infty \leq a < b \leq +\infty$ and the corresponding time interval $(a,b)$. For any $p \geq 1$, we define the \textit{embracing space} $\mathcal{E}_{p}(a,b;L_{p}((-\tau,0)^{m};\mathbb{F}))$ or, for short, $\mathcal{E}_{p}(a,b;L_{p})$ by taking the completion of the space $L_{p}(a,b;\mathbb{E}^{p}_{m}(\mathbb{F}))$ in the norm
\begin{equation}
	\label{EQ: NormEmbracingSpaceCube}
	\| \Phi(\cdot) \|_{\mathcal{E}_{p}(a,b;L_{p})} \coloneq \sup_{J \in \{1,\ldots,m\}}\sup_{\theta \in [-\tau,0]} \| (\mathcal{I}_{\delta^{J}_{\theta}}\Phi)(\cdot) \|_{L_{p}(a,b;L_{p}((-\tau,0)^{m-1};\mathbb{F}))},
\end{equation}
where $(\mathcal{I}_{\delta^{J}_{\theta}}\Phi)(t) \coloneq \delta^{J}_{\theta}\Phi(t)$ for almost all $t \in (a,b)$, where $\delta^{J}_{\theta}$ coincides with $C^{\gamma}_{J}$ for $\gamma$ such that $d\gamma = \delta_{\theta}$ is the $\mathbb{F}$-valued $\delta$-functional at $\theta$, i.e., $\mathbb{M}_{\gamma} = \mathbb{F}$ and $\gamma(\theta') = H(\theta' - \theta) \operatorname{Id}_{\mathbb{F}}$, where $H$ is the Heaviside step function. Since the total variation of such $\gamma$ equals to $1$, from Theorem \ref{TH: OperatorCExntesionOntoWDiagonal} for any $\Phi \in L_{p}(a,b;\mathbb{E}^{p}_{m}(\mathbb{F}))$, $\theta \in [-\tau,0]$, and $J \in \{1,\ldots, m\}$, we have that
\begin{equation}
	\begin{split}
		\int_{a}^{b}\| (\mathcal{I}_{\delta^{J}_{\theta}}\Phi)(t) \|^{p}_{L_{p}((-\tau,0)^{m-1};\mathbb{F})}dt \leq \\ \leq \int_{a}^{b}\| \Phi(t) \|^{p}_{\mathbb{E}^{p}_{m}(\mathbb{F})}dt = \| \Phi(\cdot) \|^{p}_{L_{p}(a,b;\mathbb{E}^{p}_{m}(\mathbb{F}))}.
	\end{split}
\end{equation}
Thus, the norm in \eqref{EQ: NormEmbracingSpaceCube} is well defined.

To characterize $\mathcal{E}_{p}(a,b;L_{p})$, consider the space $\mathbb{E}^{p}_{m}(a,b;\mathbb{F})$ consisting of all $\Phi \in L_{p}(a,b;L_{p}((-\tau,0)^{m};\mathbb{F}))$ such that for any $J \in \{1,\ldots, m\}$ there exists a continuous function
\begin{equation}
	\label{EQ: EmbracingFunctionOfJthSection}
	[-\tau,0] \ni \theta \mapsto \mathcal{R}^{J}_{\Phi}(\theta) \in L_{p}(a,b;L_{p}((-\tau,0)^{m-1};\mathbb{F})), 
\end{equation}
called the \textit{function of the $J$th section} of $\Phi$, which satisfies
\begin{equation}
	\Phi(t)(\bar{\theta}) = \mathcal{R}^{J}_{\Phi}(\theta)(t)(\bar{\theta}_{\hat{J}})
\end{equation}
for $\mu^{1}_{L}$-almost all $\theta \in [-\tau,0]$, $\mu^{m-1}_{L}$-almost all $\bar{\theta} = (\theta_{1},\ldots,\theta_{m}) \in (-\tau,0)^{m}$ with $\theta_{J} = \theta$, and $\mu^{1}_{L}$-almost all $t \in (a,b)$.

We endow the space $\mathbb{E}^{p}_{m}(a,b;\mathbb{F})$ with the norm
\begin{equation}
	\| \Phi(\cdot) \|_{\mathbb{E}^{p}_{m}(a,b;\mathbb{F})} \coloneq \sup_{J \in \{1,\ldots, m\}} \sup_{ \theta \in [-\tau,0] } \|\mathcal{R}^{J}_{\Phi}(\cdot)\|_{L_{p}(a,b;L_{p}((-\tau,0)^{m-1};\mathbb{F}))}
\end{equation}
which obviously makes it a Banach space.

\begin{lemma}
	\label{LEM: CharacterizationEmbacingSpace}
	The spaces  $\mathcal{E}_{p}(a,b;L_{p})$ and $\mathbb{E}^{p}_{m}(a,b;\mathbb{F})$ are naturally isometrically isomorphic.
\end{lemma}
\begin{proof}
	Consider $\Phi \in L_{p}(a,b;\mathbb{E}^{p}_{m}(\mathbb{F}))$. For each $J \in \{1,\ldots, m\}$, there is a well-defined function $\mathcal{R}^{J}_{\Phi}(\theta) \coloneq \mathcal{I}_{\delta^{J}_{\theta}}\Phi \in L_{p}(a,b;L_{p}((-\tau,0)^{m-1};\mathbb{F}))$ of $\theta \in [-\tau,0]$. From the dominated convergence theorem, it is not hard to see that the corresponding mapping as in \eqref{EQ: EmbracingFunctionOfJthSection} is continuous. Therefore, $\Phi \in \mathbb{E}^{p}_{m}(a,b;\mathbb{F})$.
	
	Note that the norms of $\mathcal{E}_{p}(a,b;L_{p})$ and $ \mathbb{E}^{p}_{m}(a,b;\mathbb{F})$ are identical on the common subspace $L_{p}(a,b;\mathbb{E}^{p}_{m}(\mathbb{F}))$. Since this subspace is dense in $\mathcal{E}_{p}(a,b;L_{p})$ (by definition) and in $ \mathbb{E}^{p}_{m}(a,b;\mathbb{F})$ (by the approximation argument), the conclusion of the lemma follows.
\end{proof}

\begin{lemma}
	\label{LEM: EbracingSpaceEmbeddingLp}
	There is a natural continuous embedding
	\begin{equation}
		\label{EQ: EmbeddibngEmbracingLp}
		\mathcal{E}_{p}(a,b;L_{p}((-\tau,0)^{m};\mathbb{F})) \hookrightarrow L_{p}(a,b;L_{p}((-\tau,0)^{m};\mathbb{F})),
	\end{equation}
	whose constant does not exceed $\tau^{1/p}$.
\end{lemma}
\begin{proof}
	For $\Phi \in L_{p}(a,b;\mathbb{E}^{p}_{m}(\mathbb{F}))$ and $J \in \{1,\ldots, m\}$, the Fubini theorem yields
	\begin{equation}
		\begin{split}
			\int_{a}^{b}\|\Phi(t)\|^{p}_{ L_{p}((-\tau,0)^{m};\mathbb{F}) } dt  =\\= \int_{-\tau}^{0} \|(\mathcal{I}_{\delta^{J}_{\theta}}\Phi)(\cdot) \|^{p}_{L_{p}(a,b;L_{p}((-\tau,0)^{m-1};\mathbb{F}))}d\theta \leq \tau \|\Phi(\cdot)\|^{p}_{\mathcal{E}_{p}(a,b;L_{p})}.
		\end{split}
	\end{equation}
	By Lemma \ref{LEM: CharacterizationEmbacingSpace}, the embedding from the dense subspace extends to the entire space.
\end{proof}

\begin{theorem}
	\label{TH: PointwiseMeasurementOperatorOnEmbracingSpace}
	Let $\gamma$ and $C^{\gamma}_{J}$ with $J \in \{1,\ldots,m\}$ be as in Theorem \ref{TH: OperatorCExntesionOntoWDiagonal} with $p \geq 1$. Then there exists a bounded linear operator
	\begin{equation}
		\mathcal{I}_{C^{\gamma}_{J}} \colon \mathcal{E}_{p}(a,b;L_{p}((-\tau,0)^{m};\mathbb{F})) \to L_{p}(a,b; L_{p}((-\tau,0)^{m-1};\mathbb{M}_{\gamma})),
	\end{equation}
    whose norm does not exceed the total variation $\operatorname{Var}_{[-\tau,0]}(\gamma)$ of $\gamma$ on $[-\tau,0]$ and such that for any $\Phi(\cdot) \in L_{p}(a,b;\mathbb{E}^{p}_{m}(\mathbb{F}))$ we have that 
    \begin{equation}
    	\label{EQ: OperatorIGammaIdentityContEmbr}
    	(\mathcal{I}_{C^{\gamma}_{J}}\Phi)(t) = C^{\gamma}_{J} \Phi(t) \qquad \text{for almost all} \quad t \in (a,b).
    \end{equation}
\end{theorem}
\begin{proof}
	Consider the approximation of $\gamma$ by $\gamma_{k}$ as in \eqref{EQ: ApproximationByDeltaFunctionalsGamma}. Using the Fatou lemma and the Minkowski inequality, for any $\Phi \in L_{p}(a,b;\mathbb{E}_{m}(\mathbb{F}))$ we obtain that
	\begin{equation}
		\begin{split}
			\left(\int_{a}^{b}\| C^{\gamma}_{J} \Phi(t) \|^{p}_{L_{p}((-\tau,0)^{m-1};\mathbb{M}_{\gamma})}dt\right)^{1/p} \leq \\ \leq \liminf_{k \to \infty}  \left(\int_{a}^{b}\| C^{\gamma_{k}}_{J} \Phi(t) \|^{p}_{L_{p}((-\tau,0)^{m-1};\mathbb{M}_{\gamma})}dt\right)^{1/p} \leq \\ \leq \operatorname{Var}_{[-\tau,0]}(\gamma) \cdot \| \Phi(\cdot) \|_{\mathcal{E}_{p}(a,b;L_{p})}.
		\end{split}
	\end{equation}
\end{proof}

Now consider two intervals $[a,b] \subset [c,d]$, where $-\infty \leq c \leq a \leq b \leq d \leq +\infty$, and the operators  $R^{1}_{T} \colon \mathcal{E}_{p}(c,d;L_{p}) \to \mathcal{E}_{p}(a,b;L_{p})$ and $R^{2}_{T} \colon L_{p}(c,d;\mathbb{M}_{\gamma}) \to L_{p}(a,b;\mathbb{M}_{\gamma})$, which act by restricting functions from $[c,d]$ to $[a,b]$. Using \eqref{EQ: OperatorIGammaIdentityContEmbr}, we immediately have the following.
\begin{lemma}
	\label{LEM: CommutativeOperatorIGammaEmbr}
	In the above notations, the diagram
	\begin{equation}
		\label{EQ: CommDiagramICgammaAdorned}
		\begin{tikzcd}
			\mathcal{E}_{p}(c,d;L_{p})\ar[d, "R^{1}_{T}"] \ar[r,"\mathcal{I}_{C^{\gamma}_{J}}"] 
			& L_{p}(c,d;\mathbb{M}_{\gamma}) \arrow[d,"R^{2}_{T}"] 
			\\
			\mathcal{E}_{p}(a,b;L_{p})  \arrow[r,"\mathcal{I}_{C^{\gamma}_{J}}"] 
			& L_{p}(a,b;\mathbb{M}_{\gamma})
		\end{tikzcd}
	\end{equation}
	is commutative. Here, the operators $\mathcal{I}_{C^{\gamma}_{J}}$ are given by Theorem \ref{TH: PointwiseMeasurementOperatorOnEmbracingSpace}.
\end{lemma}

Using this lemma and the fact that $\mathcal{E}_{p}(a,b;L_{p}) \subset \mathcal{E}_{1}(a,b;L_{1})$ for finite $a$ and $b$, we obtain the following relaxed version of \eqref{EQ: OperatorIGammaIdentityContEmbr}.
\begin{corollary}
	\label{COR: RelaxedComputationMeasurementOperatorsOnEmbracing}
	Let $\mathcal{I}_{C^{\gamma}_{J}}$ be given by Theorem \ref{TH: PointwiseMeasurementOperatorOnEmbracingSpace}.  Then \begin{equation}
		\label{EQ: OperatorIGammaIdentityContEmbrRelaxed}
		(\mathcal{I}_{C^{\gamma}_{J}}\Phi)(t) = C^{\gamma}_{J} \Phi(t) \qquad \text{for almost all} \quad t \in (a,b)
	\end{equation}
    is satisfied for any $\Phi \in \mathcal{E}_{p}(a,b;L_{p}) \cap L_{1,loc}(a,b;\mathbb{E}^{1}_{m}(\mathbb{F}) )$.
\end{corollary}

Next, we discuss the differentiability properties of $\mathcal{I}_{C^{\gamma}_{J}}\Phi$. Although they are not used in this paper, such results may be useful for developing a similar theory for neutral delay equations, see \cite{Anikushin2020FreqDelay} for $m=1$.

Let $\mathcal{E}_{p}(a,b;W^{1,p})$ be the subspace consisting of all $\Phi \in \mathcal{E}_{p}(a,b;L_{p})$ such that for any $J \in \{1,\ldots,m\}$ in terms of Lemma \ref{LEM: CharacterizationEmbacingSpace} we have\footnote{Note that in the definition of $\mathcal{E}_{p}(a,b;W^{1,p})$, the symbol $W^{1,p}$ reflects the space of values for the functions $\mathcal{R}^{J}_{\Phi}(\cdot)$ of the $J$th sections, rather than for $\Phi$.}
\begin{equation}
	\mathcal{R}^{J}_{\Phi}(\cdot) \in C([-\tau,0]; W^{1,p}(a,b;L_{p}((-\tau,0)^{m-1};\mathbb{F}))).
\end{equation} 
For such $\Phi$, by $\Phi'$ we denote the element of $\mathcal{E}_{p}(a,b;L_{p})$ satisfying $\mathcal{R}^{J}_{\Phi'}(\theta) = \frac{d}{dt} \mathcal{R}^{J}_{\Phi}(\theta)$ for any $\theta \in [-\tau,0]$ and $J \in \{1,\ldots,m\}$, where $\frac{d}{dt}$ denotes the derivative in the space $W^{1,p}(a,b;L_{p}((-\tau,0)^{m-1};\mathbb{F}))$. There is a natural norm in $\mathcal{E}_{p}(a,b;W^{1,p})$ given by
\begin{equation}
	\label{EQ: NormEmbracingSobolev}
	\| \Phi(\cdot) \|^{p}_{\mathcal{E}_{p}(a,b;W^{1,p})} \coloneq \|\Phi(\cdot)\|^{p}_{\mathcal{E}_{p}(a,b;L_{p})} + \|\Phi'(\cdot)\|^{p}_{\mathcal{E}_{p}(a,b;L_{p})},
\end{equation}
which obviously makes it a Banach space.

\begin{theorem}
	\label{TH: DifferentiabilityMeasurementOperatorsOnEmbracingSobolev}
	Let $\gamma$ and $C^{\gamma}_{J}$ with $J \in \{1,\ldots,m\}$ be as in Theorem \ref{TH: OperatorCExntesionOntoWDiagonal} with $p \geq 1$. Then for any $\Phi \in \mathcal{E}_{p}(a,b;W^{1,p})$ we have that $\mathcal{I}_{C^{\gamma}_{J}}\Phi$ belongs to $ W^{1,p}(a,b;L_{p}((-\tau,0)^{m-1};\mathbb{M}_{\gamma}))$ and satisfies
	\begin{equation}
		\label{EQ: DerivativeOfPointwiseMeasurementOperatorFormula}
		\frac{d}{dt}(\mathcal{I}_{C^{\gamma}_{J}}\Phi)(t) = (\mathcal{I}_{C^{\gamma}_{J}}\Phi')(t) \qquad \text{for almost all} \quad t \in (a,b),
	\end{equation}
    where $\Phi'$ is as in \eqref{EQ: NormEmbracingSobolev}. In particular, the operator
    \begin{equation}
    	\mathcal{I}_{C^{\gamma}_{J}} \colon \mathcal{E}_{p}(a,b;W^{1,p}) \to W^{1,p}(a,b;L_{p}((-\tau,0)^{m-1};\mathbb{M}_{\gamma}))
    \end{equation}
    is bounded, and its norm does not exceed the total variation $\operatorname{Var}_{[-\tau,0]}(\gamma)$ of $\gamma$ on $[-\tau,0]$.
\end{theorem}
\begin{proof}
	By definition, \eqref{EQ: DerivativeOfPointwiseMeasurementOperatorFormula} is satisfied for $C^{\gamma}_{J} = \delta^{J}_{\theta}$ and any $\theta \in [-\tau,0]$. For general $C^{\gamma}_{J}$, one can use the approximations of $\gamma$ by $\gamma_{k}$ as in \eqref{EQ: ApproximationByDeltaFunctionalsGamma} and the pointwise convergence of $\mathcal{I}_{C^{\gamma_{k}}_{J}}$ to $\mathcal{I}_{C^{\gamma}_{J}}$ in $\mathcal{E}_{p}(a,b;L_{p})$ as $k \to \infty$.
\end{proof}

We now establish a key property of embracing spaces and pointwise measurement operators concerned with the Fourier transform. For the following theorem, $\mathbb{F}$ and $\mathbb{M}_{\gamma}$ are complex Hilbert spaces.
\begin{theorem}
	\label{TH: EmbracingFourierCommutesIGamma}
	Let $\mathfrak{F}_{1}$ be the Fourier transform  in $L_{2}(\mathbb{R};L_{2}((-\tau,0)^{m};\mathbb{F}))$. Then $\mathfrak{F}_{1}$ provides an isometric automorphism of $\mathcal{E}_{2}(\mathbb{R};L_{2}((-\tau,0)^{m};\mathbb{F}))$.
	
	Furthermore, let $\gamma$ and $C^{\gamma}_{J}$ be as in Theorem \ref{TH: OperatorCExntesionOntoWDiagonal}. Then the diagram
	\begin{equation}
		\label{EQ: EmbracingSpaceIGammaFourierIdentity}
		\begin{tikzcd}
			\mathcal{E}_{2}(\mathbb{R};L_{2}((-\tau,0)^{m};\mathbb{F}))\ar[d, "\mathfrak{F}_{1}"] \ar[r,"\mathcal{I}_{C^{\gamma}_{J}}"] 
			& L_{2}(\mathbb{R};L_{2}((-\tau,0)^{m-1};\mathbb{M}_{\gamma})) \arrow[d,"\mathfrak{F}^{\gamma}_{2}"] 
			\\
			\mathcal{E}_{2}(\mathbb{R};L_{2}((-\tau,0)^{m};\mathbb{F}))  \arrow[r,"\mathcal{I}_{C^{\gamma}_{J}}"] 
			& L_{2}(\mathbb{R};L_{2}((-\tau,0)^{m-1};\mathbb{M}_{\gamma}))
		\end{tikzcd}
	\end{equation}
	is commutative. Here, $\mathcal{I}_{C^{\gamma}_{J}}$ is given by Theorem \ref{TH: PointwiseMeasurementOperatorOnEmbracingSpace}, and $\mathfrak{F}^{\gamma}_{2}$ denotes the Fourier transform in $L_{2}(\mathbb{R};L_{2}((-\tau,0)^{m-1};\mathbb{M}_{\gamma}))$.
\end{theorem}
\begin{proof}
	First, we show that $\mathfrak{F}_{1} \Phi \in \mathcal{E}_{2}(\mathbb{R};L_{2})$ for any $\Phi \in \mathcal{E}_{2}(\mathbb{R};L_{2})$. Let $\mathcal{W}^{2}_{D}((-\tau,0)^{m};\mathbb{F})$ be the diagonal Sobolev space from \eqref{EQ: W2DiagonalDefinition}, see Proposition \ref{PROP: DiagonalSobolevCubeDescription} for its characterization. Using the definition of $\mathcal{E}_{2}(\mathbb{R};L_{2})$ and the fact that $L_{2}(\mathbb{R}; \mathcal{W}^{2}_{D}((-\tau,0)^{m};\mathbb{F}))$ is dense in $L_{2}(\mathbb{R}; \mathbb{E}^{2}_{m}(\mathbb{F}))$, we obtain a sequence $\Phi_{k} \in L_{2}(\mathbb{R}; \mathcal{W}^{2}_{D}((-\tau,0)^{m};\mathbb{F}))$, where $k=1,2,\ldots$, which tends to $\Phi$ in $\mathcal{E}_{2}(\mathbb{R};L_{2})$ as $k \to \infty$. By Lemma \ref{LEM: CharacterizationEmbacingSpace}, for any $J \in \{1,\ldots,m\}$ we have as $k \to \infty$ that
	\begin{equation}
		\label{EQ: UniformConvergenceFourierPropertyEmbracing}
		\mathcal{R}^{J}_{\Phi_{k}}(\cdot) \to \mathcal{R}^{J}_{\Phi}(\cdot) \text{ in } C([-\tau,0]; L_{2}(\mathbb{R}; L_{2}((-\tau,0)^{m-1};\mathbb{F}))).
	\end{equation}
    Since $\mathcal{W}^{2}_{D}((-\tau,0)^{m};\mathbb{F})$ is continuously embedded into $L_{2}( (-\tau,0)^{m};\mathbb{F} )$ and it is a Hilbert space, we have that $\mathfrak{F}_{1}\Phi_{k} \in L_{2}(\mathbb{R}; \mathcal{W}^{2}_{D}((-\tau,0)^{m};\mathbb{F}))$.
    
    Let $\mathfrak{F}_{2}$ be the Fourier transform in $L_{2}(\mathbb{R};L_{2}((-\tau,0)^{m-1};\mathbb{F}))$. Then for any $\theta \in [-\tau,0]$, $J \in \{1,\ldots,m\}$, and $k$, we have the following identities in $L_{2}(\mathbb{R}; L_{2}((-\tau,0)^{m-1};\mathbb{F}))$ with respect to $\omega \in \mathbb{R}$:
    \begin{equation}
    	\label{EQ: FirstFourierIdentityEmbracingFourierProperty}
    	\begin{split}
    		(\mathfrak{F}_{2}\mathcal{R}^{J}_{\Phi_{k}}(\theta))(\omega) = \lim_{T \to +\infty} \frac{1}{\sqrt{2\pi}}\int_{-T}^{T}e^{-i \omega t}\delta^{J}_{\theta}\Phi_{k}(t)dt =\\= \lim_{T \to +\infty}\delta^{J}_{\theta} \frac{1}{\sqrt{2\pi}}\int_{-T}^{T}e^{-i \omega t}\Phi_{k}(t)dt = \delta^{J}_{\theta} (\mathfrak{F}_{1}\Phi_{k})(\omega),
    	\end{split}
    \end{equation}
    where we used that $\mathcal{W}^{2}_{D}((-\tau,0)^{m};\mathbb{F})$ is continuously embedded into $\mathbb{E}^{2}_{m}(\mathbb{F})$, thanks to Proposition \ref{PROP: EmbeddingWDiagToEmDelay}.
    
    Using \eqref{EQ: FirstFourierIdentityEmbracingFourierProperty} and \eqref{EQ: UniformConvergenceFourierPropertyEmbracing}, we obtain the following relations in $L_{2}(\mathbb{R};L_{2}((-\tau,0)^{m-1};\mathbb{F}))$:
    \begin{equation}
    	\label{EQ: EmbracingFourierPropertySecondFourierIdentity}
    	\mathfrak{F}_{2}\mathcal{R}^{J}_{\Phi}(\theta) = \lim_{k \to \infty}\mathfrak{F}_{2}\mathcal{R}^{J}_{\Phi_{k}}(\theta)=\lim_{k \to \infty} \mathcal{I}_{\delta^{J}_{\theta}}\mathfrak{F}_{1}\Phi_{k},
    \end{equation}
    where the convergence is uniform in $\theta \in [-\tau,0]$. In other words, $\mathfrak{F}_{1}\Phi_{k}$ converges in $\mathcal{E}_{2}(\mathbb{R};L_{2})$ as $k \to \infty$. Since the embracing space can be continuously embedded into $L_{2}(\mathbb{R};L_{2}((-\tau,0)^{m};\mathbb{F}))$ due to Lemma \ref{LEM: EbracingSpaceEmbeddingLp}, and $\mathfrak{F}_{1}\Phi_{k}$ converges to $\mathfrak{F}_{1}\Phi$ as $k \to \infty$ in the latter space, we get that $\mathfrak{F}_{1}\Phi$ must belong to $\mathcal{E}_{2}(\mathbb{R};L_{2})$.
    
    From \eqref{EQ: EmbracingFourierPropertySecondFourierIdentity}, we obtain 
    \begin{equation}
    	\label{EQ: FourierPropertyEmbracingCommutativityDeltaFunc}
    	\mathcal{R}^{J}_{\mathfrak{F}_{1}\Phi}(\theta) \coloneq \mathcal{I}_{\delta^{J}_{\theta}}\mathfrak{F}_{1}\Phi = \mathfrak{F}_{2}\mathcal{I}_{\delta^{J}_{\theta}}\Phi \qquad \text{for any} \quad \Phi \in \mathcal{E}_{2}(\mathbb{R};L_{2}).
    \end{equation}
    This yields that $\mathfrak{F}_{1}$ is an isometry of $\mathcal{E}_{2}(\mathbb{R};L_{2})$. Since it bijectively takes the dense subspace $L_{2}(\mathbb{R}; \mathcal{W}^{2}_{D}((-\tau,0)^{m};\mathbb{F}))$ into itself, it must be an isometric automorphism of the embracing space. 

    Finally, we also note that \eqref{EQ: FourierPropertyEmbracingCommutativityDeltaFunc} gives the commutativity of the diagram from \eqref{EQ: EmbracingSpaceIGammaFourierIdentity} for $C^{\gamma}_{J} = \delta^{J}_{\theta}$ and any $\theta \in [-\tau,0]$. For general $C^{\gamma}_{J}$, one may use the approximations of $\gamma$ by $\gamma_{k}$ as in \eqref{EQ: ApproximationByDeltaFunctionalsGamma} and the pointwise convergence argument.
\end{proof}

In the forthcoming subsections, we introduce special spaces that can be continuously embedded into appropriate embracing spaces. In the present paper, these spaces are involved into the structural Cauchy formula, see Theorem \ref{TH: StructuralCauchyFormulaCompoundDelay}. This, together with the just established Theorem \ref{TH: EmbracingFourierCommutesIGamma}, is used to prove Lemma \ref{LEM: FourierTransformCompoundDelay}, which is a key computational ingredient in the derivation of frequency conditions.

\input{AdornedFunctions}
\input{TwistedFunctions}
\input{AgamlanatedFunctions}

%% file: AdornedFunctions.tex
\subsection{Spaces of adorned functions}
Recall that $\underline{t}$ denotes the diagonal vector $(t,\ldots,t)$ in $\mathbb{R}^{m}$ for any $t \in \mathbb{R}$. For any reals $\tau > 0$ and $T>0$, we consider the subset $\mathcal{C}^{m}_{T}$ of $\mathbb{R}^{m}$ defines as follows:
\begin{equation}
	\label{EQ: TheSetDiagonalDomainDefinition}
	\mathcal{C}^{m}_{T} = \bigcup_{t \in [0,T]} \left( [-\tau,0]^{m} + \underline{t} \right).
\end{equation}
We will also consider the case $T = \infty$, where the interval $[0,T]$ is understood as $[0,\infty)$.

Consider a continuous function $\rho \colon [0,+\infty) \to \mathbb{R}$ that have a constant sign and for some $\rho_{0} = \rho_{0}(\rho,\tau) > 0$ satisfies the inequality:
\begin{equation}
	\label{EQ: WeightFunctionProperty}
	|\rho(t+s)| \leq \rho_{0} \cdot |\rho(t)| \qquad \text{for all} \ t \geq 0 \ \text{and} \ s \in [0, \tau].
\end{equation}
Any such $\rho(\cdot)$ is called a \textit{weight function}. For the purposes of this paper, it is sufficient to consider $\rho(t) = \rho_{\nu}(t) = e^{\nu t}$ for some $\nu \in \mathbb{R}$, although there are important examples arising from differentiation of such weights, as in Theorems \ref{TH: DifferentiationAdornedFunctions} and \ref{TH: DifferentiationTwistedFunctions}, or with time-dependent exponents, as in the following remark.
\begin{remark}
	Let $\pi$ be a semiflow on a metric space $\mathcal{P}$ and $\nu_{0} \colon \mathcal{P} \to \mathbb{R}$ be a bounded continuous function. For $\wp \in \mathcal{P}$ and $t \geq 0$, set $\nu(t;\wp) \coloneq \int_{0}^{t}\nu_{0}(\pi^{s}(\wp))ds$. Then $\rho(t) = \rho(t;\wp) \coloneq e^{\nu(t;\wp)}$ also satisfies the assumption \eqref{EQ: WeightFunctionProperty}. Such general exponents $\nu$ arise in the nonstationary quadratic optimization and its applications in the theory of exponential dichotomies, see \cite{Anikushin2020Geom}. This is why we prefer to use the more general condition \eqref{EQ: WeightFunctionProperty} rather than limiting ourselves to the case of exponential weights with constant exponents.
\end{remark}

Recall the Hilbert space $\mathbb{F}$ and let $p \geq 1$. For any $T > 0$ and $X \in L_{p}(\mathcal{C}^{m}_{T}; \mathbb{F})$, we define a function $\Phi(t)$ of $t \in [0,T]$ as follows:
\begin{equation}
	\label{EQ: WindowFunctionDefinition}
	\Phi(t) = \Phi_{X,\rho}(t) \coloneq \rho(t) X_{t} \in L_{p}((-\tau,0)^{m};\mathbb{F}),
\end{equation}
where $X_{t}(\bar{\theta}) \coloneq X(\bar{\theta} + \underline{t})$ for almost all $\bar{\theta} = (\theta_{1},\ldots,\theta_{m}) \in (-\tau,0)^{m}$. Any such $\Phi$ is called a \textit{$\rho$-adorned $L_{p}((-\tau,0)^{m};\mathbb{F})$-valued  function} on $[0,T]$ or, more simply, \textit{$\rho$-adorned}, if the spaces are understood from the context. It is also convenient to say that $\Phi$ is the $\rho$-\textit{adornment of $X$} over $\mathcal{C}^{m}_{T}$.

For any $\rho$-adorned function $\Phi$, the mapping
\begin{equation}
	\label{EQ: AgalmanatedValueAtTMappingContinuous}
	[0,T] \ni t \mapsto \Phi(t) \in L_{p}((-\tau,0)^{m}; \mathbb{F})
\end{equation}
is continuous, since the action of $\mathbb{R}^{m}$ by translates in $L_{p}(\mathbb{R}^{m};\mathbb{F})$ is strongly continuous.

For each $j \in \{1,\ldots, m\}$, we consider the $(m-1)$-face $\mathcal{B}_{\hat{j}}$ defined by
\begin{equation}
	\label{EQ: FaceBjHatNearAdorned}
	\mathcal{B}_{\hat{j}} \coloneq \{ \bar{\theta} = (\theta_{1},\ldots,\theta_{m} ) \in [-\tau,0]^{m} \mid \theta_{j} = 0 \},
\end{equation}
which is consistent with \eqref{EQ: DefinitionOfBoundaryFace}. Recall that $\mu^{m-1}_{L}$ denotes the $(m-1)$-dimensional Lebesgue measure.

For what follows, $L_{p}$ stands for $L_{p}((-\tau,0)^{m};\mathbb{F})$. For any $T>0$, we introduce the space $\mathcal{Y}^{p}_{\rho}(0,T;L_{p})$ of all $\rho$-adorned $L_{p}$-valued functions $\Phi$ on $[0,T]$ and endow it with the norm defined as follows:
\begin{equation}
	\label{EQ: NormInWindowsSpaces}
	\begin{split}
		\| \Phi(\cdot) \|^{p}_{\mathcal{Y}^{p}_{\rho}(0,T;L_{p})} \coloneq \\ = \int_{(-\tau,0)^{m}}\left|X(\bar{\theta})\right|^{p}_{\mathbb{F}}d\bar{\theta} + \sum_{j=1}^{m}\int_{\mathcal{B}_{\hat{j}}}d\mu^{m-1}_{L}(\bar{\theta}) \int_{0}^{T} \left|\rho(t)X(\bar{\theta}+\underline{t})\right|^{p}_{\mathbb{F}} dt,
	\end{split}
\end{equation}
where $\Phi = \Phi_{X,\rho}$ is as in \eqref{EQ: WindowFunctionDefinition}. For $T = \infty$, instead of $X \in L_{p}(\mathcal{C}^{m}_{T};\mathbb{F})$, we require that the restriction of $X$ to $\mathcal{C}^{m}_{T_{0}}$ lies in $L_{p}(\mathcal{C}^{m}_{T_{0}};\mathbb{F})$ for any $T_{0} > 0$ and the corresponding norm in \eqref{EQ: NormInWindowsSpaces} with $T=\infty$ is finite. Since $\rho(t) \not=0$ for any $t \geq 0$, and hence any $\Phi$ uniquely determines $X$ via \eqref{EQ: WindowFunctionDefinition}, the norm is well defined. Clearly, $\mathcal{Y}^{p}_{\rho}(0,T;L_{p})$, equipped with this norm, becomes a Banach space.

\begin{lemma}
	\label{LEM: DeltaFuncEstimateForContAdornedFunc}
	For $T>0$ and $p \geq 1$, let $\Phi_{X,\rho}$ be associated with $X \in C(\mathcal{C}^{m}_{T};\mathbb{F})$ via \eqref{EQ: WindowFunctionDefinition}. Then the estimate
	\begin{equation}
		\left(\int_{0}^{T}\| \delta^{J}_{\tau_{0}} \Phi_{X,\rho}(t) \|^{p}_{L_{p}((-\tau,0)^{m-1};\mathbb{F})}dt\right)^{1/p} \leq \kappa(\rho) \cdot \| \Phi_{X,\rho}(\cdot) \|_{\mathcal{Y}^{p}_{\rho}(0,T;L_{p})}
	\end{equation}
	is satisfied for any $\tau_{0} \in [-\tau,0]$ and $J \in \{1,\ldots, m\}$. Here $\kappa(\rho)$ is defined in \eqref{EQ: AdornedConstantKappa}.
\end{lemma}
\begin{proof}
	Let $e_{J}$ be the $J$-th vector in the standard basis of $\mathbb{R}^{m}$. Then
	\begin{equation}
		\label{EQ: WeightedFunctionsEstimateDeltaFunc}
		\begin{split}
			\int_{0}^{T}\| \delta^{J}_{\tau_{0}}\Phi_{X,\rho}(t) \|^{p}_{L_{p}((-\tau,0)^{m-1}; \mathbb{F})}dt =\\= \int_{\mathcal{B}_{\hat{J}}+\tau_{0} e_{J}}\int_{0}^{T}d\mu^{m-1}_{L}(\bar{\theta}) \left|\rho(t) X(\bar{\theta}+\underline{t}) \right|^{p}_{\mathbb{F}}dt \leq \kappa(\rho)^{p} \cdot \left\| \Phi_{X,\rho} (\cdot) \right\|^{p}_{\mathcal{Y}^{p}_{\rho}(0,T;L_{p})},
		\end{split}
	\end{equation}
	where the last inequality follows from \eqref{EQ: NormInWindowsSpaces} and \eqref{EQ: WeightFunctionProperty} with $\kappa(\rho)$ defined in \eqref{EQ: AdornedConstantKappa}. Indeed, in the integral over $[0,T]$ from \eqref{EQ: NormInWindowsSpaces}, the value of $X$ at $(\bar{\theta}+\underline{t}) \in \mathcal{C}^{m}_{T} \setminus (-\tau,0)^{m}$, where $t \in [0,T]$ and $\bar{\theta} \in (\mathcal{B}_{\hat{J}}+\tau_{0} e_{J})$, is weighted by $\rho(s)$ for some $s=s(\bar{\theta},t)$ such that $t - s \in [0,\tau]$ and hence $|\rho(t)| \leq \rho_{0} |\rho(s)|$. For $\bar{\theta}+\underline{t} \in (-\tau,0)^{m}$, we use the inequality $|\rho(t)| \leq \rho_{0} |\rho(0)|$, since $t \in [0,\tau]$ in this case, and estimate the corresponding part of the integral from \eqref{EQ: WeightedFunctionsEstimateDeltaFunc} through the first term in \eqref{EQ: NormInWindowsSpaces}. Thus, for 
	\begin{equation}
		\label{EQ: AdornedConstantKappa}
		\kappa(\rho) \coloneq \operatorname{max}\{ \rho_{0}, \rho_{0} |\rho(0)| \},
	\end{equation}
    the estimate in \eqref{EQ: WeightedFunctionsEstimateDeltaFunc} is valid.
\end{proof}

Since the subspace of all $\Phi_{X,\rho}$ with $X \in C(\mathcal{C}^{m}_{T};\mathbb{F})$ is dense in $\mathcal{Y}^{p}_{\rho}(0,T;L_{p})$, Lemma \ref{LEM: DeltaFuncEstimateForContAdornedFunc} yields the following.

\begin{lemma}
	\label{LEM: EmbeddingAdornedIntoEmbracing}
	Let $T > 0$ or $T=\infty$ and $p \geq 1$. Then there is a natural embedding of the space $\mathcal{Y}^{p}_{\rho}(0,T;L_{p})$ into $\mathcal{E}_{p}(0,T;L_{p})$ whose embedding constant does not exceed $\kappa(\rho)$, defined in \eqref{EQ: AdornedConstantKappa}.
\end{lemma}

Combining Lemma \ref{LEM: EmbeddingAdornedIntoEmbracing}, Theorem \ref{TH: PointwiseMeasurementOperatorOnEmbracingSpace}, and Corollary \ref{COR: RelaxedComputationMeasurementOperatorsOnEmbracing} yields the following.
\begin{theorem}
	\label{TH: PointwiseMeasurementOperatorOnAdornedSpace}
	Let $\gamma$ and $C^{\gamma}_{J}$ with $J \in \{1,\ldots,m\}$ be as in Theorem \ref{TH: OperatorCExntesionOntoWDiagonal}, and consider $T > 0$ or $T=\infty$ and $p \geq 1$. Then there exists a bounded linear operator
	\begin{equation}
		\mathcal{I}_{C^{\gamma}_{J}} \colon \mathcal{Y}^{p}_{\rho}(0,T;L_{p}((-\tau,0)^{m};\mathbb{F})) \to L_{p}(0,T; L_{p}((-\tau,0)^{m-1};\mathbb{M}_{\gamma})),
	\end{equation}
	whose norm does not exceed $\kappa(\rho) \cdot \operatorname{Var}_{[-\tau,0]}(\gamma)$, where $\kappa(\rho)$ is defined in \eqref{EQ: AdornedConstantKappa}, and 
	\begin{equation}
		\label{EQ: OperatorIGammaIdentityContAdorned}
		(\mathcal{I}_{C^{\gamma}_{J}}\Phi)(t) = C^{\gamma}_{J} \Phi(t) \qquad \text{for almost all} \quad t \in (0,T)
	\end{equation}
	is satisfied for any $\Phi(\cdot) \in  \mathcal{Y}^{p}_{\rho}(0,T;L_{p}) \cap L_{1, loc}(0,T;\mathbb{E}^{1}_{m}(\mathbb{F}))$.
\end{theorem}

Next, we describe conditions for the differentiability of $\mathcal{I}_{C^{\gamma}_{J}}\Phi_{X,\rho}$ in terms of $X$. For this, we assume that the weight function $\rho(\cdot)$ is $C^{1}$-differentiable and that its derivative $\dot{\rho}(\cdot)$ is either identically zero or is also a weight function. In this case, we say that $\rho(\cdot)$ is a \textit{proper $C^{1}$-weight function}.

For $T>0$ or $T=\infty$, let $\mathcal{Y}^{p}_{\rho}(0,T;\mathcal{W}^{p}_{D})$ be the subspace of all $\Phi_{X,\rho} \in \mathcal{Y}^{p}_{\rho}(0,T;L_{p})$ such that the restriction of $X$ to the interior $\mathring{\mathcal{C}}_{T_{0}}^{m}$ of $\mathcal{C}_{T_{0}}^{m}$ belongs to the diagonal Sobolev space $\mathcal{W}^{p}_{D}(\mathring{\mathcal{C}}_{T_{0}}^{m};\mathbb{F})$, defined in \eqref{EQ: W2DiagonalDefinition}, for any finite $T_{0} \leq T$, and the norm
\begin{equation}
	\label{EQ: NormAdornedSobolev}
	\begin{split}
		\|\Phi_{X,\rho}(\cdot)\|^{p}_{\mathcal{Y}^{p}_{\rho}(0,T;\mathcal{W}^{p}_{D})} \coloneq \\
		=\| \Phi_{X,\rho}(\cdot) \|^{p}_{\mathcal{Y}^{p}_{\rho}(0,T;L_{p})} + \|\Phi_{\dot{X},\rho}(\cdot)\|^{p}_{\mathcal{Y}^{p}_{\rho}(0,T;L_{p})} + \| \Phi_{X,\dot{\rho}}(\cdot) \|^{p}_{\mathcal{Y}^{p}_{\rho}(0,T;L_{p})},
	\end{split}
\end{equation}
where $\dot{X}$ is the diagonal derivative of $X$, is finite. For $\dot{\rho}(\cdot) \equiv 0$, the last term in \eqref{EQ: NormAdornedSobolev} is assumed to be zero. Clearly, $\mathcal{Y}^{p}_{\rho}(0,T;\mathcal{W}^{p}_{D})$, equipped with the above norm, is a Banach space.

As a byproduct of the following theorem, we obtain that $\mathcal{Y}^{p}_{\rho}(0,T;\mathcal{W}^{p}_{D})$ is continuously embedded into the space $\mathcal{E}_{p}(0,T;W^{1,p})$, defined above \eqref{EQ: NormEmbracingSobolev}. This places the result into the context of Theorem \ref{TH: DifferentiabilityMeasurementOperatorsOnEmbracingSobolev}.
\begin{theorem}
	\label{TH: DifferentiationAdornedFunctions}
	In the context of Theorem \ref{TH: PointwiseMeasurementOperatorOnAdornedSpace}, assume that $\rho(\cdot)$ is a proper $C^{1}$-weight function. Then the operator
	\begin{equation}
		\label{EQ: AdornedFunctionsPMODifferMap}
		\mathcal{I}_{C^{\gamma}_{J}} \colon \mathcal{Y}^{p}_{\rho}(0,T;\mathcal{W}^{p}_{D}) \to W^{1,p}(0,T;L_{p}((-\tau,0)^{m-1};\mathbb{M}_{\gamma}))
	\end{equation}
	is well defined and bounded, and its norm admits an upper estimate only in terms of  $\operatorname{Var}_{[-\tau,0]}(\gamma)$, $\rho$, $\dot{\rho}$, and $\tau$. Furthermore, for any $\Phi_{X,\rho} \in \mathcal{Y}^{p}_{\rho}(0,T;\mathcal{W}^{p}_{D})$, the function $\mathcal{I}_{C^{\gamma}_{J}}\Phi_{X,\rho}$, as an element of the $W^{1,p}$-space from \eqref{EQ: AdornedFunctionsPMODifferMap}, satisfies
	\begin{equation}
		\frac{d}{dt}(\mathcal{I}_{C^{\gamma}_{J}}\Phi_{X,\rho})(t) = (\mathcal{I}_{C^{\gamma}_{J}}\Phi_{\dot{X},\rho})(t) + (\mathcal{I}_{C^{\gamma}_{J}}\Phi_{X,\dot{\rho}})(t)
	\end{equation}
	for almost all $t \in (0,T)$, where $\dot{X}$ is as in \eqref{EQ: NormAdornedSobolev}.
\end{theorem}
\begin{proof}
	First, we suppose that $C^{\gamma}_{J} = \delta^{J}_{\theta}$ for some $\theta \in [-\tau,0]$. Then, by the Leibniz rule, we obtain for almost all $t \in (0,T)$ that
	\begin{equation}
		\begin{split}
			\frac{d}{dt}\left(\mathcal{I}_{\delta^{J}_{\theta}}\Phi_{X,\rho}\right)(t) = \frac{d}{dt}( \rho(t) \delta^{J}_{\theta} X_{t} ) =\\ = \dot{\rho}(t)\delta^{J}_{\theta}X_{t} + \rho(t) \delta^{J}_{\theta}\dot{X}_{t} = (\mathcal{I}_{\delta^{J}_{\theta}}\Phi_{X,\dot{\rho}})(t) + (\mathcal{I}_{\delta^{J}_{\theta}}\Phi_{\dot{X},\rho})(t).
		\end{split}
	\end{equation}
    This shows the statement for $C^{\gamma}_{J} = \delta^{J}_{\theta}$ and proves that $\mathcal{Y}^{p}_{\rho}(0,T;\mathcal{W}^{p}_{D})$ is continuously embedded into $\mathcal{E}_{p}(0,T;W^{1,p})$. For general $C^{\gamma}_{J}$, one can use the approximations of $\gamma$ by $\gamma_{k}$ as in \eqref{EQ: ApproximationByDeltaFunctionalsGamma} or simply appeal to Theorem \ref{TH: DifferentiabilityMeasurementOperatorsOnEmbracingSobolev}.
\end{proof}

%% file: TwistedFunctions.tex
\subsection{Spaces of twisted functions}
Recall the diagonal translation semigroup $T_{m}$ in $L_{p}((-\tau,0)^{m};\mathbb{F})$ defined by \eqref{EQ: SemigroupDiagonalTranslateFiniteCubeDef}, and let $L_{p}$ stand for $L_{p}((-\tau,0)^{m};\mathbb{F})$. Consider a weight function $\rho(\cdot)$, as in \eqref{EQ: WeightFunctionProperty}, and any $T>0$. We introduce the space $\mathcal{T}^{p}_{\rho}(0,T;L_{p})$ of functions $\Psi(\cdot)$ on $[0,T]$ with values in $L_{p}$ that satisfy
\begin{equation}
	\label{EQ: TwistedFunctionDefinition}
	\Psi(t) = \Psi_{Y,\rho}(t) \coloneq \rho(t) \int_{0}^{t}T_{m}(t-s)Y(s)ds \qquad \text{for all} \quad t \in [0,T]
\end{equation}
for some $Y(\cdot) \in L_{p}(0,T;L_{p})$. 

Any such $\Psi$, as in \eqref{EQ: TwistedFunctionDefinition}, is called the \textit{$\rho$-twisting} of $Y$. We will also say that $\Psi$ is an \textit{$L_{p}$-valued $\rho$-twisted function} on $[0,T]$ or, more briefly, that $\Psi$ is \textit{$\rho$-twisted}, if the spaces are understood from the context. As in \eqref{EQ: AgalmanatedValueAtTMappingContinuous}, we have that the mapping
\begin{equation}
	\label{EQ: TwistedValueAtMappingContinuous}
	[0,T] \ni t \mapsto \Psi(t) \in L_{p}((-\tau,0)^{m};\mathbb{F})
\end{equation}
is continuous, since $T_{m}$ is a $C_{0}$-semigroup.

In the following lemma, we establish that $Y$ is uniquely determined by $\Psi$ via \eqref{EQ: TwistedFunctionDefinition}.
\begin{lemma}
	\label{LEM: TwistedFunctionUniquenessY}
	Let $T>0$, $p > 1$, and $Y \in L_{p}(0,T;L_{p})$ be given. Suppose that
	\begin{equation}
		\label{EQ: LemUniqTwistedIdentity}
		\int_{0}^{t}T_{m}(t-s)Y(s)ds = 0 \qquad \text{for all} \quad t \in [0,T].
	\end{equation}
    Then $Y = 0$.
\end{lemma}
\begin{proof}
	Let $\Psi=\Psi_{Y,\rho}$ be as in \eqref{EQ: TwistedFunctionDefinition} with $\rho \equiv 1$. Then $\Psi$ is a mild solution to the inhomogeneous problem $\dot{\Psi}(t) = A_{T_{m}}\Psi(t) + Y(t)$ on $[0,T]$ with $\Psi(0) = 0$, associated with the generator $A_{T_{m}}$ of $T_{m}$. By \cite[Lemma 3.5]{Anikushin2020FreqDelay}\footnote{The formulation of the cited lemma should be clarified by the requirement of reflexivity of the Banach space, since its proof uses the density of the domain of the adjoint operator. At this place the assumption $p > 1$ is used in the present lemma (any $L_{p}$-space of functions with values in a reflexive Banach space is reflexive if $p > 1$).}, the solution operator $Y \mapsto \Psi$ is injective, and hence \eqref{EQ: LemUniqTwistedIdentity} yields that $Y = 0$.
\end{proof}

We endow the space $\mathcal{T}^{p}_{\rho}(0,T;L_{p})$ with the norm defined as follows:
\begin{equation}
	\label{EQ: NormOnTwistedFunctionsDef}
	\| \Psi(\cdot) \|^{p}_{\mathcal{T}^{p}_{\rho}(0,T;L_{p})} \coloneq \int_{0}^{T} \| \rho(t)Y(t) \|^{p}_{L_{p}} dt,
\end{equation}
where $\Psi$ and $Y$ are related by \eqref{EQ: TwistedFunctionDefinition}. By Lemma \ref{LEM: TwistedFunctionUniquenessY}, the norm is well defined. We also introduce the corresponding space for $T=\infty$ by requiring that $Y \in L_{p}(0,T_{0};L_{p})$ for any $T_{0}>0$, and the norm in \eqref{EQ: NormOnTwistedFunctionsDef} with $T=\infty$ is finite. Clearly, $\mathcal{T}^{p}_{\rho}(0,T;L_{p})$, equipped with this norm, becomes a Banach space.

Next, we aim to show that $\mathcal{T}^{p}_{\rho}(0,T;L_{p})$ naturally embeds into the embracing space $\mathcal{E}_{p}(0,T;L_{p})$. For this, let $C_{0+}([-\tau,0]^{m};\mathbb{F})$ be the closed subspace of $C([-\tau,0]^{m};\mathbb{F})$ that consists of functions vanishing on the $(m-1)$-faces $\mathcal{B}_{\hat{j}}$, as in \eqref{EQ: FaceBjHatNearAdorned}, for each $j \in \{1,\ldots, m\}$. For short, this space is denoted by $C_{0+}$.

Clearly, $C_{0+}$ is an invariant subspace for $T_{m}$, and the restriction of $T_{m}$ to it is a $C_{0}$-semigroup in $C_{0+}$. In particular, for $T>0$ and $Y \in L_{p}(0,T;C_{0+})$, the function $\Psi_{Y,\rho}$ associated with $Y$ via \eqref{EQ: TwistedFunctionDefinition} belongs to $C([0,T];C_{0+})$.

\begin{lemma}
	\label{LEM: TwistedFunctionsEmbeddingIntoEmbracing}
	For $T>0$, $p > 1$, and $Y(\cdot) \in C([0,T];C_{0+})$, let $\Psi_{Y,\rho}$ be associated with $Y$ via \eqref{EQ: TwistedFunctionDefinition}. Then the estimate
	\begin{equation}
		\left(\int_{0}^{T}\|\delta^{J}_{\tau_{0}}\Psi_{Y,\rho}(t) \|^{p}_{L_{p}((-\tau,0)^{m-1};\mathbb{F})} dt\right)^{1/p} \leq \rho_{0} \tau^{1-1/p} \cdot \| \Psi_{Y,\rho}(\cdot) \|_{\mathcal{T}^{p}_{\rho}(0,T;L_{p})}.
	\end{equation}
    is satisfied for any $\tau_{0} \in [-\tau,0]$ and $J \in \{1,\ldots, m\}$. Here, $\rho_{0}$ is defined in \eqref{EQ: WeightFunctionProperty}.
\end{lemma}
\begin{proof}
	We set $\widetilde{Y}(s,\bar{\theta}) \coloneq Y(s)(\bar{\theta})$ for $s \in [0,T]$ and $\bar{\theta} \in [-\tau,0]^{m}$. Recall that $\bar{\theta}_{\hat{J}}$ denotes the $(m-1)$-vector obtained from $\bar{\theta}$ after eliminating the $J$-th component. Then for all $\bar{\theta}=(\theta_{1},\ldots,\theta_{m}) \in [-\tau,0]^{m}$ with $\theta_{J} = \tau_{0}$, we have
	\begin{equation}
		\label{EQ: DeltaFunctionActionOnTwistedFunctionCont}
		\begin{split}
		    \delta^{J}_{\tau_{0}}\Psi_{Y,\rho}(t)(\bar{\theta}_{\hat{J}}) = \rho(t)\Psi_{Y,\rho}(t)(\bar{\theta}) =
			\rho(t)\int_{t_{0}(t,\bar{\theta})}^{t} \widetilde{Y}(s, \bar{\theta} + \underline{t}-\underline{s})ds,
		\end{split}
	\end{equation}
	where $t_{0}(t,\bar{\theta})$ is the maximum among $0$ and $\theta_{j}+t$ for $j \in \{1,\ldots, m\}$. Note that we always have $t - t_{0}(t,\bar{\theta}) \in [0, \tau]$. Thus, from \eqref{EQ: WeightFunctionProperty} we obtain $|\rho(t)| \leq \rho_{0} |\rho(s)|$ for all $t \in [0,T]$ and $s \in [t_{0}(t,\bar{\theta}), t]$.
	
	Recall the Lebesgue measure $\mu^{m-1}_{L}$, which can be considered on the subsets $\mathcal{B}_{\hat{J}} - \tau_{0} e_{J}$, where $e_{J}$ is the $J$-th vector in the standard basis of $\mathbb{R}^{m}$. Using the H\"{o}lder inequality and monotonicity of the integral, we obtain
	\begin{equation}
		\label{EQ:  TwistedFunctionsDeltaFunctionalEstimate}
		\begin{split}
			\int_{0}^{T}\|\delta^{J}_{\tau_{0}}\Psi_{Y,\rho}(t) \|^{p}_{L_{p}((-\tau,0)^{m-1};\mathbb{F})} dt = \\ = \int_{0}^{T} dt|\rho(t)|^{p} \int_{\mathcal{B}_{\hat{J}} - \tau_{0} e_{J} } d\mu^{m-1}_{L}(\bar{\theta})\left|\int_{t_{0}(\bar{\theta},t)}^{t} \widetilde{Y}(s, \bar{\theta} + \underline{t}-\underline{s})ds \right|^{p}_{\mathbb{F}} \leq \\ \leq
			\tau^{p-1} \rho^{p}_{0} \int_{0}^{T} dt \int_{\mathcal{B}_{\hat{J}} - \tau_{0} e_{J} } d\mu^{m-1}_{L}(\bar{\theta})\int_{t_{0}(\bar{\theta},t)}^{t} \left|\rho(s)\widetilde{Y}(s, \bar{\theta} + \underline{t}-\underline{s})\right|^{p}_{\mathbb{F}} ds \leq \\ \leq
			\tau^{p-1} \rho^{p}_{0} \int_{[0,T] \times [-\tau,0]^{m} } \left| \rho(s)\widetilde{Y}(s,\bar{\theta}) \right|^{p}_{\mathbb{F}}dsd\bar{\theta} =\\= \tau^{p-1} \rho^{p}_{0}\int_{0}^{T}\|\rho(t) Y(t)\|^{p}_{L_{p}( (-\tau,0)^{m};\mathbb{F} )} dt,
		\end{split}
	\end{equation}
    where for the last inequality we applied the change of variables $(t, \bar{\theta}, s) \mapsto (s, \bar{\theta} + \underline{t}-\underline{s}) \in [0,T] \times [-\tau,0]^{m}$, whose determinant is $\pm 1$, and then used the monotonicity.
\end{proof}

Since the subspace of $\Psi_{Y,\rho}$ with $Y(\cdot) \in C([0,T];C_{0+})$ is dense in $\mathcal{T}^{p}_{\rho}(0,T;L_{p})$, Lemma \ref{LEM: TwistedFunctionsEmbeddingIntoEmbracing} gives the following.
\begin{lemma}
	\label{LEM: EmbeddinTwistedntoEmbracing}
	Let $T > 0$ or $T=\infty$ and $p > 1$. Then there is a natural embedding of the space $\mathcal{T}^{p}_{\rho}(0,T;L_{p})$ into $\mathcal{E}_{p}(0,T;L_{p})$ whose constant does not exceed $\rho_{0} \tau^{1-1/p}$.
\end{lemma}

Combining Lemma \ref{LEM: EmbeddinTwistedntoEmbracing}, Theorem \ref{TH: PointwiseMeasurementOperatorOnEmbracingSpace}, and Corollary \ref{COR: RelaxedComputationMeasurementOperatorsOnEmbracing} yields the following.
\begin{theorem}
	\label{TH: PointwiseMeasurementOperatorOnTwistedSpace}
	Let $\gamma$ and $C^{\gamma}_{J}$ with $J \in \{1,\ldots,m\}$ be as in Theorem \ref{TH: OperatorCExntesionOntoWDiagonal}, and consider $T > 0$ or $T=\infty$ and $p > 1$. Then there exists a bounded linear operator
	\begin{equation}
		\mathcal{I}_{C^{\gamma}_{J}} \colon \mathcal{T}^{p}_{\rho}(0,T;L_{p}((-\tau,0)^{m};\mathbb{F})) \to L_{p}(0,T; L_{p}((-\tau,0)^{m-1};\mathbb{M}_{\gamma})),
	\end{equation}
	whose norm does not exceed $\rho_{0} \tau^{1-1/p} \cdot \operatorname{Var}_{[-\tau,0]}(\gamma)$, and 
	\begin{equation}
		\label{EQ: OperatorIGammaIdentityContTwisted}
		(\mathcal{I}_{C^{\gamma}_{J}}\Phi)(t) = C^{\gamma}_{J} \Phi(t) \qquad \text{for almost all} \quad t \in (0,T)
	\end{equation}
	is satisfied for any $\Phi(\cdot) \in \mathcal{T}^{p}_{\rho}(0,T;L_{p}) \cap L_{1, loc}(0,T;\mathbb{E}^{1}_{m}(\mathbb{F}))$.
\end{theorem}

Next, we describe conditions for the differentiability $\mathcal{I}_{C^{\gamma}_{J}}\Psi_{Y,\rho}$ in terms of $Y$. For this, recall the generator $A_{T_{m}}$ of $T_{m}$ in $L_{p}$ and its domain $\mathcal{D}(A_{T_{m}}) = \mathcal{W}^{p}_{D_{0}}((-\tau,0)^{m};\mathbb{F})$ defined in \eqref{EQ: DomainNilponentDiagonalTranslates}.

As in \eqref{EQ: NormAdornedSobolev}, we assume that $\rho(\cdot)$ is a proper $C^{1}$-weight function. Then, for $T>0$ or $T = \infty$, we define the space $\mathcal{T}^{p}_{\rho}(0,T;\mathcal{W}^{p}_{D})$ to be the subspace of all $\Psi_{Y,\rho} \in \mathcal{T}^{p}_{\rho}(0,T;L_{p})$ such that $Y \in L_{p}(0,T_{0}; \mathcal{D}(A_{T_{m}}))$ for any finite $T_{0} \leq T$ and the norm 
\begin{equation}
	\label{EQ: NormTwistedSobolev}
	\begin{split}
		\| \Psi_{Y,\rho}(\cdot) \|^{p}_{\mathcal{T}^{p}_{\rho}(0,T;\mathcal{W}^{p}_{D})} \coloneq\\= \| \Psi_{Y,\rho}(\cdot) \|^{p}_{\mathcal{T}^{p}_{\rho}(0,T;L_{p})} + \| \Psi_{Y',\rho}(\cdot) \|^{p}_{\mathcal{T}^{p}_{\rho}(0,T;L_{p})} + \| \Psi_{Y,\dot{\rho}}(\cdot) \|_{\mathcal{T}^{p}_{\rho}(0,T;L_{p})},
	\end{split}
\end{equation}
is finite. Here $Y'(t) \coloneq (\sum_{j=1}^{m}\frac{\partial}{\partial\theta_{j}})Y(t)$ is the diagonal derivative of $Y(t)$ for almost all $t \in [0,T]$. Clearly, $\mathcal{T}^{p}_{\rho}(0,T;\mathcal{W}^{p}_{D})$, equipped with the norm \eqref{EQ: NormTwistedSobolev}, is a Banach space.

As a byproduct of the following theorem, we obtain that $\mathcal{T}^{p}_{\rho}(0,T;\mathcal{W}^{p}_{D})$ is continuously embedded into the space $\mathcal{E}_{p}(0,T;W^{1,p})$, defined above \eqref{EQ: NormEmbracingSobolev}. This places the result into the context of Theorem \ref{TH: DifferentiabilityMeasurementOperatorsOnEmbracingSobolev}.
\begin{theorem}
	\label{TH: DifferentiationTwistedFunctions}
	In the context of Theorem \ref{TH: PointwiseMeasurementOperatorOnTwistedSpace}, let $\rho(\cdot)$ be a proper $C^{1}$-weight function. Then the operator
	\begin{equation}
		\label{EQ: TwistedPMODifferOpDef}
		\mathcal{I}_{C^{\gamma}_{J}} \colon \mathcal{Y}^{p}_{\rho}(0,T;\mathcal{W}^{p}_{D}) \to W^{1,p}(0,T;L_{p}((-\tau,0)^{m-1};\mathbb{M}_{\gamma}))
	\end{equation}
	is bounded, and its norm admits an upper estimate only in terms of $\operatorname{Var}_{[-\tau,0]}(\gamma)$, $\rho$, $\dot{\rho}$, and $\tau$. Furthermore, for any $\Psi_{Y,\rho} \in \mathcal{T}^{p}_{\rho}(0,T;\mathcal{W}^{p}_{D})$, the function $\mathcal{I}_{C^{\gamma}_{J}}\Psi_{Y,\rho}$, as an element of the $W^{1,p}$-space from \eqref{EQ: TwistedPMODifferOpDef}, satisfies
	\begin{equation}
		\label{EQ: DerivativeMeasurementOperatorTwostedFormula}
		\frac{d}{dt}(\mathcal{I}_{C^{\gamma}_{J}}\Psi_{Y,\rho})(t) = (\mathcal{I}_{C^{\gamma}_{J}}\Psi_{Y',\rho})(t) + (\mathcal{I}_{C^{\gamma}_{J}}\Psi_{Y,\dot{\rho}})(t) + \rho(t)C^{\gamma}_{J}Y(t),
	\end{equation}
	for almost all $t \in (0,T)$, where $Y'$ is as in \eqref{EQ: NormTwistedSobolev}.
\end{theorem}
\begin{proof}
	It suffices to show the statement for finite $T$. Moreover, we can consider only $\Psi_{Y,\rho}$ for which $Y$ lies in $C^{1}([0,T];C^{1}([-\tau,0]^{m};\mathbb{F}))$ and $Y(t)$ vanishes on $\mathcal{B}_{\hat{j}}$ for any $j \in \{1,\ldots,m\}$ and $t \in [0,T]$, since the subspace of such $\Psi_{Y,\rho}$ is dense in $\mathcal{T}^{p}_{\rho}(0,T;\mathcal{W}^{p}_{D})$. First, we show \eqref{EQ: DerivativeMeasurementOperatorTwostedFormula} for such $\Psi_{Y,\rho}$ and $C^{\gamma}_{J} = \delta^{J}_{\tau_{0}}$ with $\tau_{0} \in [-\tau,0]$. Indeed, differentiating \eqref{EQ: DeltaFunctionActionOnTwistedFunctionCont}, we obtain
	\begin{equation}
		\frac{d}{dt}\left(\delta^{J}_{\tau_{0}}\Psi_{Y,\rho}(t)\right)(\bar{\theta}_{\hat{J}}) = \delta^{J}_{\tau_{0}}\Psi_{Y,\dot{\rho}}(t) + \rho(t)\frac{d}{dt}\int_{t_{0}(t,\bar{\theta})}^{t} \widetilde{Y}(s, \bar{\theta} + \underline{t}-\underline{s})ds.
	\end{equation}
    for all $t \in [0,T]$ and $\bar{\theta} \in [-\tau,0]^{m}$ with $\theta_{J} = \tau_{0}$.
    
    Since $Y(t)$ vanish on any $\mathcal{B}_{\hat{j}}$, we have $\widetilde{Y}(s,\bar{\theta} + \underline{t} - \underline{s}) = 0$ for $s = t_{0}(t,\bar{\theta})$. This gives
    \begin{equation}
    	\begin{split}
    		\rho(t)\frac{d}{dt}\int_{t_{0}(t,\bar{\theta})}^{t} \widetilde{Y}(s, \bar{\theta} + \underline{t}-\underline{s})ds =\\= \rho(t)\widetilde{Y}(t,\bar{\theta}) +  \rho(t)\int_{0}^{t}\frac{d}{dt}T_{m}(t-s)Y(s)(\bar{\theta})ds =\\= \rho(t) (\delta^{J}_{\tau_{0}}Y(t))(\bar{\theta}_{\hat{J}}) + \delta^{J}_{\tau_{0}}\Psi_{Y',\rho}(t),
    	 \end{split}
    \end{equation} 
    where for the last term we used that 
    \begin{equation}
    	\begin{split}
    		 \frac{d}{dt}T_{m}(t-s)Y(s) = A_{T_{m}}T_{m}(t-s)Y(s) =\\= T_{m}(t-s)A_{T_{m}}Y(s) = T_{m}(t-s)Y'(s).
    	\end{split}
    \end{equation}
    
     Thanks to the density of such $\Psi_{Y,\rho}$ and Theorems \ref{TH: PointwiseMeasurementOperatorOnTwistedSpace} and \ref{TH: OperatorCExntesionOntoWDiagonal}, this proves the statement for $C^{\gamma}_{J} = \delta^{J}_{\tau_{0}}$ and hence establishes a continuous embedding into $\mathcal{E}_{p}(0,T;W^{1,p})$. For general $C^{\gamma}_{J}$, one can use the approximations of $\gamma$ by $\gamma_{k}$ as in \eqref{EQ: ApproximationByDeltaFunctionalsGamma} or simply refer to Theorem \ref{TH: DifferentiabilityMeasurementOperatorsOnEmbracingSobolev}.
\end{proof}

%% file: AgamlanatedFunctions.tex
\subsection{Spaces of agalmanated functions}

As above, let $L_{p}$ stand for $L_{p}(0,T;L_{p}((-\tau,0)^{m};\mathbb{F}))$. We begin this subsection by showing that the spaces $\mathcal{Y}^{p}_{\rho}(0,T;L_{p})$ of $\rho$-adorned and $\mathcal{T}^{p}_{\rho}(0,T;L_{p})$ of $\rho$-twisted functions defined in \eqref{EQ: NormInWindowsSpaces} and \eqref{EQ: NormOnTwistedFunctionsDef}, respectively, are linearly independent for $p > 1$. This is explained by the fact that $\Psi_{Y,\rho}(t)$ according to \eqref{EQ: TwistedFunctionDefinition} must have a small $L_{p}$-norm near each boundary face $\mathcal{B}_{\hat{j}}$, and this smallness is uniform in $t$. The correct development of this reasoning gives the following.
\begin{proposition}
	\label{PROP: UniquenessOfSumAdornedTwisted}
	Let $T > 0$ and $p > 1$. Suppose that for some $X \in L_{p}(\mathcal{C}^{m}_{T};\mathbb{F})$ and $Y \in L_{p}(0,T;L_{p})$ we have that\footnote{Recall that both $\Phi_{X,\rho}(t)$ and $\Psi_{Y,\rho}(t)$ depend continuously on $t$, see \eqref{EQ: AgalmanatedValueAtTMappingContinuous} and \eqref{EQ: TwistedValueAtMappingContinuous}.}
	\begin{equation}
		\label{EQ: UniquenessOfSumAdornedTwistedZeroIdentity}
		0 = \Phi_{X,\rho}(t) + \Psi_{Y,\rho}(t) \qquad \text{for all} \quad t \in [0,T],
	\end{equation}
    where $\Phi_{X,\rho}$ is the $\rho$-adornment of $X$ and $\Psi_{Y,\rho}$ is the $\rho$-twisting of $Y$ defined in \eqref{EQ: WindowFunctionDefinition} and \eqref{EQ: TwistedFunctionDefinition}, respectively. Then $\Phi_{X,\rho}(t) = \Psi_{Y,\rho}(t) = 0$ for all $t \in [0,T]$.
\end{proposition}
\begin{proof}
	It suffices to consider the case $\rho \equiv 1$ and show that $X = 0$. Let $h \in (0,\tau)$ be fixed, and let $\mathcal{D}_{h}$ be the subset of $(-\tau,0)^{m}$ consisting of all $(\theta_{1},\ldots,\theta_{m}) \in (-\tau,0)^{m}$ such that $\theta_{j} \geq -h$ holds at least for one $j \in \{1,\ldots, m\}$. Set also
	\begin{equation}
		\mathcal{D}^{T}_{h} \coloneq \bigcup_{k=1}^{\lfloor \frac{T}{h} \rfloor} \left(\mathcal{D}_{h} + \underline{kh}\right)
	\end{equation}
	and note that the Lebesgue measure of $\mathcal{C}^{m}_{T} \setminus (\mathcal{D}^{T}_{h} \cup (-\tau,0)^{m} )$ tends to zero as $h \to 0+$. Since $\Psi_{Y,\rho}(0) = 0$, from \eqref{EQ: UniquenessOfSumAdornedTwistedZeroIdentity} we have that $X(\bar{s}) = 0$ for almost all $\bar{s} \in (-\tau,0)^{m}$. Thus,
	\begin{equation}
		\label{EQ: UniquenessOfSumAdornedTwistedApproximationDomain}
		\int_{\mathcal{C}^{m}_{T}}|X(\bar{s})|^{p}_{\mathbb{F}}d\bar{s} = \lim\limits_{h \to 0+} \int_{\mathcal{D}^{T}_{h}}|X(\bar{s})|^{p}_{\mathbb{F}}d\bar{s}
	\end{equation}
	For $t \in [h,T]$, the values of $\Psi_{Y,\rho}(t)$ on $\mathcal{D}_{h}$ are determined by the last term in the formula
	\begin{equation}
		\Psi_{Y,\rho}(t) = \int_{0}^{t-h}T_{m}(t-s)Y(s)ds + \int_{t-h}^{t} T_{m}(t-s)Y(s)ds.
	\end{equation}
	Using this, \eqref{EQ: UniquenessOfSumAdornedTwistedZeroIdentity}, \eqref{EQ: WindowFunctionDefinition}, and the H\"{o}lder inequality, we obtain that
	\begin{equation}
		\label{EQ: UniquenessOfSumAdornedTwistedEstimationIntegral}
		\begin{split}
			\int_{\mathcal{D}^{T}_{h}}|X(\bar{s})|^{p}_{\mathbb{F}}d\bar{s} = \sum_{k=1}^{\lfloor \frac{T}{h} \rfloor} \int_{\mathcal{D}_{h}}|X(\bar{\theta}+\underline{kh})|^{p}_{\mathbb{F}}d\bar{\theta} =\\= \sum_{k=1}^{\lfloor \frac{T}{h} \rfloor} \int_{\mathcal{D}_{h}}|\Phi_{X,\rho}(kh)(\bar{\theta})|^{p}_{\mathbb{F}}d\bar{\theta} \leq \sum_{k=1}^{\lfloor \frac{T}{h} \rfloor} \left\|\int_{(k-1) h}^{kh}T_{m}(t-s)Y(s)ds\right\|^{p}_{L_{p}} \leq \\ \leq  h^{p-1}\sum_{k=1}^{\lfloor \frac{T}{h} \rfloor} \int_{(k-1) h}^{kh}\|Y(s)\|^{p}_{L_{p}}ds \leq h^{p-1} \int_{0}^{T}\|Y(s)\|^{p}_{L_{p}}ds.
		\end{split}
	\end{equation}
	By combining \eqref{EQ: UniquenessOfSumAdornedTwistedApproximationDomain} with \eqref{EQ: UniquenessOfSumAdornedTwistedEstimationIntegral}, we obtain that $X = 0$.
\end{proof}

For $T>0$ or $T = \infty $ and $p > 1$, we introduce the space of $\rho$-agalmanated functions as the following direct sum:
\begin{equation}
	\label{EQ: AgalmanatedSpaceDefinition}
	\mathcal{A}^{p}_{\rho}(0,T;L_{p}) \coloneq \mathcal{Y}^{p}_{\rho}(0,T;L_{p}) \oplus \mathcal{T}^{p}_{\rho}(0,T;L_{p}),
\end{equation}
which is naturally endowed with the norm
\begin{equation}
	\label{EQ: NormAgalmanatedFunctions}
	\| (\Phi_{X,\rho}(\cdot),\Psi_{Y,\rho}(\cdot)) \|^{p}_{	\mathcal{A}^{p}_{\rho}(0,T;L_{p})} \coloneq \| \Phi_{X,\rho}(\cdot)\|^{p}_{	\mathcal{Y}^{p}_{\rho}(0,T;L_{p})} + \| \Psi_{Y,\rho}(\cdot)\|^{p}_{	\mathcal{T}^{p}_{\rho}(0,T;L_{p})}
\end{equation}
that makes it a Banach space.

Combining Proposition \ref{PROP: UniquenessOfSumAdornedTwisted}, Lemma \ref{LEM: EmbeddingAdornedIntoEmbracing}, and Lemma \ref{LEM: EmbeddinTwistedntoEmbracing} yields the following.
\begin{theorem}
	\label{TH: EmbeddingAgalmanatedIntoEmbracing}
	Let $T>0$ or $T = \infty$ and $p > 1$. Then the correspondence
	\begin{equation}
		\label{EQ: EmbeddingAgalmanatedIntoEmbracing}
		\mathcal{A}^{p}_{\rho}(0,T;L_{p}) \ni (\Phi_{X,\rho}, \Psi_{Y,\rho}) \mapsto \Phi_{X,\rho} + \Psi_{Y,\rho} \in \mathcal{E}_{p}(0,T;L_{p}),
	\end{equation}
    is an embedding, and its norm admits an upper estimate in terms of $\rho_{0}$ from \eqref{EQ: WeightFunctionProperty} and $\tau$.
\end{theorem}

In virtue of the above theorem, it is convenient to identify $\mathcal{A}^{p}_{\rho}(0,T;L_{p})$ with its image under \eqref{EQ: EmbeddingAgalmanatedIntoEmbracing}.

Combining Theorem \ref{TH: EmbeddingAgalmanatedIntoEmbracing}, Theorem \ref{TH: PointwiseMeasurementOperatorOnEmbracingSpace}, and Corollary \ref{COR: RelaxedComputationMeasurementOperatorsOnEmbracing} yields the following.
\begin{theorem}
	\label{TH: PointwiseMeasurementOperatorOnAgalmanatedSpace}
	Let $\gamma$ and $C^{\gamma}_{J}$ with $J \in \{1,\ldots,m\}$ be as in Theorem \ref{TH: OperatorCExntesionOntoWDiagonal}, and consider $T > 0$ or $T=\infty$ and $p > 1$. Then there exists a bounded linear operator
	\begin{equation}
		\mathcal{I}_{C^{\gamma}_{J}} \colon \mathcal{A}^{p}_{\rho}(0,T;L_{p}((-\tau,0)^{m};\mathbb{F})) \to L_{p}(0,T; L_{p}((-\tau,0)^{m-1};\mathbb{M}_{\gamma})),
	\end{equation}
	whose norm admits an upper estimate in terms of $\operatorname{Var}_{[-\tau,0]}(\gamma)$, $\rho_{0}$ from \eqref{EQ: WeightFunctionProperty}, and $\tau$, and such that for $\Phi = (\Phi_{X,\rho}, \Psi_{Y,\rho})$, where $\Phi_{X,\rho} \in \mathcal{Y}^{p}_{\rho}(0,T;L_{p})$ and $\Psi_{Y,\rho} \in \mathcal{T}^{p}_{\rho}(0,T;L_{p})$, it is defined by
	\begin{equation}
		\mathcal{I}_{C^{\gamma}_{J}}\Phi \coloneq \mathcal{I}_{C^{\gamma}_{J}}\Phi_{X,\rho} + \mathcal{I}_{C^{\gamma}_{J}}\Psi_{Y,\rho},
	\end{equation}
	where the action on $\Phi_{X,\rho}$ and $\Psi_{Y,\rho}$ may be understood according to Theorems \ref{TH: PointwiseMeasurementOperatorOnAdornedSpace} and \ref{TH: PointwiseMeasurementOperatorOnTwistedSpace}, respectively, or Theorem \ref{TH: PointwiseMeasurementOperatorOnEmbracingSpace}. Moreover, if $\Phi_{X,\rho}+\Psi_{Y,\rho} \in L_{1, loc}(0,T;\mathbb{E}^{1}_{m}(\mathbb{F}))$, then
	\begin{equation}
		\label{EQ: OperatorIGammaIdentityContAgalmanated}
		(\mathcal{I}_{C^{\gamma}_{J}}\Phi)(t) = C^{\gamma}_{J}(\Phi_{X,\rho}(t)+\Psi_{Y,\rho}(t)) \qquad \text{for almost all} \quad t \in (0,T).
	\end{equation}
\end{theorem}